\documentclass[10pt,a4paper]{article}

\usepackage{amsmath}
\usepackage{amsfonts}
\usepackage{amsthm}
\usepackage{amssymb}
\usepackage{url}
\usepackage[usenames,dvips]{color}
\usepackage{stmaryrd}
\usepackage[dvips]{graphicx}
\usepackage{psfrag, float}  %per al text dels gr�fics

\DeclareMathAlphabet{\mathpzc}{OT1}{pzc}{m}{it}

\def\figlabel#1{\label{#1}}
\newtheorem{theorem}{Theorem}[section]
\newtheorem{lemma}[theorem]{Lemma}
\newtheorem{corollary}[theorem]{Corollary}

\newtheorem{remark}[theorem]{Remark}

\newtheorem{proposition}[theorem]{Proposition}

\newcommand{\be}{\begin{equation}}
\newcommand{\ee}{\end{equation}}
\newcommand{\eps}{\varepsilon}
\newcommand{\ga}{\gamma}
\newcommand{\dps}{\displaystyle}
\newcommand{\RR}{\mathbb{R}}
\newcommand{\NN}{\mathbb{N}}
\newcommand{\CC}{\mathbb{C}}
\newcommand{\TT}{\mathbb{T}}
\newcommand{\ZZ}{\mathbb{Z}}
\newcommand{\Mm}{\mathcal{M}}
\newcommand{\MM}{M}
\newcommand{\WW}{\mathcal{W}}
\newcommand{\PP}{\mathcal{P}}
\newcommand{\GG}{\mathcal{G}}

\newcommand{\LL}{\mathcal{L}}
\newcommand{\QQ}{\mathcal{Q}}
\newcommand{\YY}{\mathcal{Y}}
\newcommand{\KK}{\mathcal{K}}
\newcommand{\SSS}{\mathcal{S}}
\newcommand{\DD}{\mathcal{D}}
\newcommand{\XX}{\mathcal{X}}
\newcommand{\OO}{\mathcal{O}}
\newcommand{\oo}{\mathpzc{o}}
\newcommand{\FF}{\mathcal{F}}
\newcommand{\HH}{\mathcal{H}}
\newcommand{\VV}{\mathcal{V}}
\newcommand{\RRR}{\mathcal{R}}
\newcommand{\TTT}{\mathcal{T}}
\newcommand{\UU}{\mathcal{U}}
\newcommand{\NNN}{\mathcal{N}}
\newcommand{\EE}{\mathcal{E}}
\newcommand{\JJ}{\mathcal{J}}
\newcommand{\AAA}{\mathcal{A}}
\newcommand{\ZZZ}{\mathcal{Z}}
\newcommand{\CCC}{\mathcal{C}}
\newcommand{\Id}{\mathrm{Id}}

\newcommand{\xp}{x_p}
\newcommand{\yp}{y_p}

\newcommand{\ii}{^{-1}}
\newcommand{\de}{\delta}
\newcommand{\pa}{\partial}
\newcommand{\la}{\lambda}

\newcommand{\inn}{\mathrm{in}}
\newcommand{\out}{\mathrm{out}}
\newcommand{\q}{\beta}
\newcommand{\p}{\alpha}

\newcommand{\overM}{\overline M_1}
\newcommand{\vinfty}{v_\infty}

\newcommand{\kk}{\kappa}
\newcommand{\rr}{\rho}
\newcommand{\tro}{I}
\newcommand{\tri}{\mathcal{I}}
\newcommand{\ups}{\Upsilon}

\newcommand{\C}{c}

\newcommand{\tet}{\theta}
\newcommand{\ol}{\overline}

\newcommand{\mat}{A_0}

\newcommand{\ppi}{\phi}

\newcommand{\wH}{H_1}
\renewcommand{\Re}{\mathrm{Re\, }}
\renewcommand{\Im}{\mathrm{Im\,}}
\newcommand{\wt}{\widetilde}
\newcommand{\wh}{\widehat}
\newcommand{\Lip}{\mathrm{Lip}\,}
\newcommand{\hmu}{\hat\mu}
\newcommand{\Cte}{A}
\begin{document}
\title{Exponentially small splitting of separatrices beyond Melnikov analysis: rigorous results}
\author{Inmaculada Baldom\'a\thanks{\tt immaculada.baldoma@upc.edu}, Ernest Fontich\thanks{\tt fontich@ub.edu}, Marcel Guardia\thanks{\tt marcel.guardia@upc.edu}\, and Tere M. Seara\thanks{\tt tere.m-seara@upc.edu}}
\maketitle
\medskip
\begin{center}$^{*\S}$
Departament de Matem\`atica Aplicada I\\
Universitat Polit\`ecnica de Catalunya\\
Diagonal 647, 08028 Barcelona, Spain
\end{center}
\smallskip
\begin{center}$^\dagger$
Departament de Matem\`atica Aplicada i An\`alisi\\
Universitat de Barcelona\\
Gran Via 585, 08007 Barcelona, Spain
\end{center}
\smallskip
\begin{center}$^\ddagger$
Department of Mathematics\\
Mathematics Building, University of Maryland\\
College Park, MD 20742-4015
\end{center}

\begin{abstract}
We study the  problem of exponentially small splitting of separatrices of one degree of freedom classical Hamiltonian systems with a non-autonomous perturbation fast and periodic in time. We provide a result valid for general systems which are algebraic or trigonometric polynomials in the state variables. It consists on obtaining a rigorous proof of the asymptotic formula for the measure of the splitting. We obtain that the splitting has the asymptotic behavior  $K \varepsilon^{\beta} \text{e}^{-a/\varepsilon}$, identifying the constants $K,\beta,a$ in terms of the system features.

We consider several cases. In some cases, assuming the perturbation is small enough, the values of $K,\beta$ coincide with the classical Melnikov approach. We identify the limit size of the perturbation for which this theory holds true. However for the limit  cases, which appear naturally both in averaging and bifurcation 
theories, we encounter that, generically, $K$ and $\beta$ are not well predicted by Melnikov theory.

%We have found
%systems for which the values of $K$ and $\beta$ do not coincide with the
%ones given by Melnikov approach. 

%Such a result is valid  for both the so-called
%regular and singular cases.
%In the former our result improves some previous works existing in the literature validating the Melnikov method.
%In the latter  we show that Melnikov theory fails to
%predict correctly the first order of the splitting of separatrices and we provide an alternative formula for it.
%In this paper we study the  problem of exponentially small splitting of
%separatrices of one degree of freedom classical Hamiltonian systems with a
%non-autonomous perturbation which is fast and periodic in time. We provide the
%asymptotic formula for the measure of the splitting for both the so-called
%regular and singular cases. In the latter  we show that Melnikov theory fails to
%predict correctly the first order of the splitting of separatrices.
\end{abstract}

\section{Introduction}\label{sec:intro}
In this  paper we consider the familiy of Hamiltonian systems of the form
\begin{equation}  \label{eq:model}
H\left(x,y,\frac{t}{\eps};\eps\right)=H_0(x,y)+\mu\eps^{\eta}H_1\left(x,y,\frac{
t}{\eps};\eps\right), \qquad (x,y) \in \RR ^{2},
\end{equation}
where $H_0(x,y)$ is given by a classical Hamiltonian
\[
H_0(x,y)= \frac{y^2}{2}+V(x)
\]
and $H_1(x,y,\tau;\eps)$
is a $2 \pi$-periodic time dependent Hamiltonian with zero average:
\[
\langle H_1 \rangle = \frac{1}{2\pi} \int _{0}^{2\pi} H_1(x,y,\tau;\eps) \,
d\tau = 0 .
\]
We
%assume $H(x,y,\tau;\eps)$ is  analytic and we
study the problem of the splitting of separatrices. The
parameter $\eps$ is a small parameter but this is not the case for $\mu$,
which may be  of order one.
The results in this paper are valid not only for $\mu$ small, but also for
finite values of $\mu$.  We will see
that the results are significantly different depending on the other parameter
$\eta\geq 0$, which
appears in \eqref{eq:model}, and on the analytic properties of $H$. Depending of these properties our results are valid even for (the non perturbative case) $\eta=0$ and we will see that, in this case, Melnikov theory gives a wrong prediction of the measure of the splitting.

The perturbative setting is when  $\mu \eps ^{\eta}$ is small, that is when
$\eta >0$. In this case,
the Hamiltonian system associated to $H$ is a small perturbation
of the Hamiltonian system associated to $H_0$:
\begin{eqnarray}\label{eq:sistemanopertorbat}
\dot x &=& y\nonumber\\
\dot y &=& - V'(x).
\end{eqnarray}

Our first observation is that, being the Hamiltonian $H$ fast in time, averaging
theory \cite{ArnoldKN88, LochakM88}
tells us that, even for $\mu \eps ^{\eta}=\OO (1)$, that is for $\eta =0$, the
solutions of the Hamiltonian system associated to \eqref{eq:model} are close to
the solutions of
\eqref{eq:sistemanopertorbat}.

We assume that system \eqref{eq:sistemanopertorbat} has a hyperbolic
or parabolic critical point at the origin with stable and unstable manifolds
which coincide along
a separatrix $(q_0(u), p_0(u))$. The coincidence of the stable and unstable
invariant manifolds is not a generic phenomenon for Hamiltonian systems of
one and half degrees of freedom as \eqref{eq:model}. Therefore, one can expect
that the homoclinic connection of \eqref{eq:sistemanopertorbat} breaks down when
we add the non-autonomous part  to the system. Nevertheless, the symplectic
structure ensures the existence of intersections between the perturbed invariant
manifolds. Hence a natural question is whether these intersections are
transversal or not.

As it is well known, the transversal intersection of invariant manifolds is an
obstruction for  the integrability of  the system as well as  one of the  main causes
of the appearance of chaos.  Even if this transversality is a generic
phenomenon, it is difficult to check it in a concrete given system of type
\eqref{eq:model}. In this paper we give checkable conditions (see Section \ref{sec:Hypotheses} for the concrete hypotheses)
which  ensure that
transversality and, moreover, we provide an asymptotic formula, as $\eps \to 0$,
which measures this transversality and shows that it is exponentially small with
respect to $\eps$.

To check this transversality there are several quantities that can be
considered.
Due to the $2\pi\eps$-periodicity with respect to $t$ of the Hamiltonian $H$, it
is convenient to consider the
Poincar\'{e} map $P_{t_{0}}$ defined in a Poincar\'{e}  section
$\Sigma_{t_{0}}=\{ (x,y,t_{0}); \,(x,y) \in \RR^{2}\}$.
If $\mu=0$, the phase portrait of $P_{t_{0}}$ is given by
the level curves of the Hamiltonian $ H_0(x,y)=\frac{y^2}{2}+V(x)$.
Therefore,  the homoclinic connection $(q_0(u), p_0(u))$ is contained in the 
stable and unstable curves of  the fixed point $(0,0)$ of $P_{t_{0}}$.

%The origin is a hyperbolic or parabolic fixed point of $P_{t_{0}}$ with stable and unstable curves
%given by the homoclinic connection $(q_0(u), p_0(u))$.

In the hyperbolic case, a classical result of averaging theory \cite{ArnoldKN88,
LochakM88}
is that, for $\eps$  small enough,
there exists a
hyperbolic fixed point of $P_{t_{0}}$, corresponding to a hyperbolic periodic
orbit of $H$, which has  stable and unstable
invariant curves $C^{s}(t_0)$ and $C^{u}(t_0)$. These curves  remain close to the
unperturbed separatrix. In the parabolic case our (standard) hypotheses will
ensure that the origin
will still be a fixed point with similar properties.

As $P_{t_{0}}$ is a symplectic map,  the curves $C^{s}(t_0)$ and $C^{u}(t_0)$ intersect
giving rise to some homoclinic points $z_h$.
The natural  quantity that can be used at  homoclinic points to
measure the transversality of the intersection is the angle between
the curves $C^{s}(t_0)$ and $C^{u}(t_0)$.

Once we have proved that this intersection is transversal at two consecutive
homoclinic points,
we can measure the splitting by computing the area $\mathcal{A}$ enclosed by the
invariant curves between these two points. This area does not depend on the
chosen homoclinic points  (see Figure~\ref{fig:splitting}) and is  also
invariant under symplectic changes of coordinates.
For these reasons, in Theorems \ref{th:MainGeometric:regular}  and
\ref{th:MainGeometric:singular} we
measure this area  instead of measuring the angle.
Another invariant quantity, related to the angle, is the so-called
\emph{Lazutkin invariant} (see, for
instance~\cite{GelfreichLT91}).  From now on, we will use the expression
\emph{splitting of separatrices} to refer to any of these quantities.

One model where our results can be applied is a classical $2\pi\varepsilon$-periodic time dependent Hamiltonian system:
\begin{equation}\label{eq:modelclassic}
H\left(x,y,\frac{t}{\eps}\right)=\frac{y^2}{2}+ \wt
V\left(x,\frac{t}{\eps}\right)
\end{equation}
taking $ V(x) = \frac{1}{2\pi}\int _0^{2\pi} \wt V(x,\tau)\, d\tau$ and
$H_1\left(x,y,\tau\right)=\wt V(x,\tau)- V(x)$.
In this case, under certain hypotheses about $V$, which are specified in
Section \ref{sec:Hypotheses}, our result in Theorem
\ref{th:MainGeometric:singular} provides a formula for the splitting even if in this case $\mu=1$ and
$\eta  =0$. In this case, our result improves several partial results \cite{DelshamsS97,Gelfreich97a,BaldomaF04} which,
applied to \eqref{eq:model},
needed to consider an artificial factor
$\eps^{\eta}$, $\eta >\eta _0>0$, in front of the term $H_1$ to prove  an asymptotic formula for the splitting.
Moreover, it occurs that this formula is wrong for the natural case $\eta =0$.

%Hamiltonian \eqref{eq:modelclassic} is a particular case of periodic time
%dependent Hamiltonians $H(x,y,t/\eps)$.
%After the change of time $t= \eps \tau$, these systems become  time
%dependent Hamiltonian systems with slow dynamics and classical averaging theory
%can be applied to them. So it is natural to split
%\[
%H\left(x,y,\frac{t}{\eps}\right)=H_0(x,y)+H_1\left(x,y,\frac{t}{\eps}\right)
%\]
%with $H_0(x,y)= \frac{1}{2\pi }\int _0^{2\pi} H\left(x,y,\tau\right) \, d\tau$
%and
%$H_1\left(x,y,\tau\right) = H\left(x,y,\tau\right)-H_0(x,y)$.

%Averaging theory tells us that, even if in this model $\eta=0$, the solutions of
%the whole Hamiltonian system associated
%to $H$ can be approximated, up to long times, by the solutions of $H_0$.
%Moreover, in the case that $H_0(x,y)$ has a hyperbolic critical point the stable
%and unstable invariant manifolds of $H$ are $\eps$-close to the stable and
%unstable manifolds of $H_0$.
%It is then natural to study if the homoclinic orbit of the Hamiltonian system
%\eqref{eq:sistemanopertorbat} splits and to give an asymptotic formula, for
%$\eps$
%small enough, of this splitting.
%Our result in Theorem \ref{th:MainGeometric:singular} gives the splitting of
%these manifolds in this
%case (where   $\eta =0$ ) for a wide class of Hamiltonians.

One also encounters the  case  $\eta =0$,
when one studies the splitting of separatrices phenomenon
near a resonance of one and a half degrees of freedom Hamiltonian
systems which are close to completely integrable ones (in the sense of
Liouville-Arnold).
This setting does not fit exactly in our hypotheses but,
as we will see in a forthcoming paper, the methods used in this paper can be
easily adapted to that case (see Section \ref{sec:mainresult:resonance} for a discussion of this problem).

\begin{figure}[ht]
\begin{center}
\psfrag{A}{$\AAA$}\psfrag{x}{$x$}\psfrag{y}{$y$}\psfrag{z}{$z_h$}
\includegraphics[height=2in]{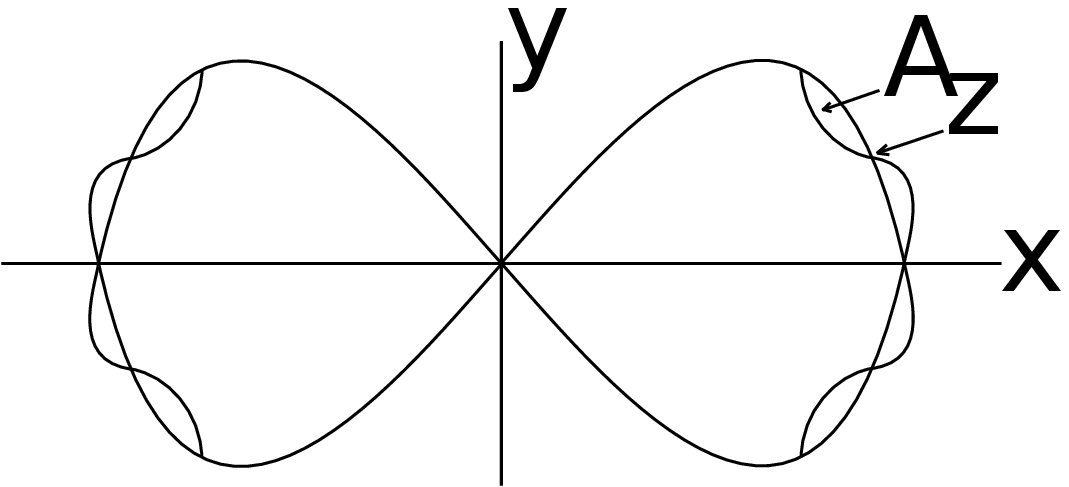}
\end{center}
\caption{\figlabel{fig:splitting} Splitting of separatrices.}
\end{figure}

Classical perturbation theory applied to our problem provides the so-called
Melnikov potential
(called also sometimes Poincar\'e Function, see for instance
\cite{DelshamsG00}),
which is given by
\[
L\left(t_0 \right)=\int_{-\infty}^{+\infty}H_1\left(q_0(u), p_0(u),\eps\ii (t_0
+ u);0\right)du.
\]
Using this function, Poincar\'{e} \cite{Poincare90,Poincare99}, and later Melnikov \cite{Melnikov63}, proved
that, if $\mu \eps ^{\eta}$
is small enough, non-degenerate critical points of $L$ give rise to transversal
intersections between
the invariant curves $C^{s}(t_0)$ and $C^{u}(t_0)$, and the area of the lobes is
given asymptotically
by $L(t_0^1) - L(t_0^2)$, being $t_0^1$ and $t_0^2$ two consecutive critical
points of $L$.

If  $H_0(x,y)$ and $H_1(x,y,\tau;0)$ are 
either algebraic or algebraic in $y$ and trigonometric
polynomials in $x$,
the Poincar\'{e} function $L$ is asymptotically given by:
\begin{equation} \label{eq:formulamelnikov}
L\left(t_0 \right) \simeq K \eps ^{\beta} e ^{-a/\eps} \sin
\left(\frac{t_0}{\eps}+\phi \right), \qquad \eps \to 0
\end{equation}
being $a>0$, $K, \phi , \beta \in \RR $ some computable constants. The constant $a$
is independent of the perturbation: it turns out that the time parameterization
of the unperturbed separatrix has always singularities in the complex plane (see
\cite{Fontich95, BaldomaF04}) and the constant $a$ is nothing but the imaginary
part of the singularity closest to the real axis. It is clear that $L\left(t_0
\right)$ has non-degenerate critical points if  $K\ne 0$.

We want to emphasize that the asymptotic size with respect to $\eps$ of the
Melnikov potential is given by \eqref{eq:formulamelnikov} provided $H_0(x,y)$
and $H_1(x,y,\tau;0)$ are either algebraic or algebraic in $y$ and trigonometric
polynomials in $x$. The study of the Melnikov potential for general analytic
Hamiltonian systems with fast periodic perturbations strongly depends on the
analyticity properties of the Hamiltonian $H$. Even if the Melnikov potential
can be estimated for some concrete systems
\cite{LlibreS80,MartinezP94,HolmesMS91}, a general study of this function seems
to require more powerful analytic tools and, as far as the authors know, has not
been done.

The straightforward application of Melnikov method to Hamiltonian
\eqref{eq:model} provides a
formula for the area of the lobes which reads:
\begin{equation} \label{eq:formulanocorrecta}
 \mathcal{A} = \mu\eps^\eta\left(\mathcal{A} _0 +
\OO\left(\mu \eps^{\eta}\right)\right), \qquad \eps \to 0 ,
\end{equation}
where
\begin{equation} \label{eq:formulanocorrecta0}
 \mathcal{A} _0\simeq 2 K \eps ^{\beta} e ^{-a/\eps}
\end{equation}
is the prediction for the area given by the Melnikov potential \eqref{eq:formulamelnikov}.

Therefore, either for general algebraic or algebraic in $y$ and trigonometric
polynomials in $x$ Hamiltonians, the Melnikov potential is exponentially small in $\eps$ and a direct
application
of classical perturbation theory  only ensures the validity of such an
approximation if
 $K\ne 0$ and $\mu\eps^\eta=\oo(\eps^\beta e^{-a/\eps})$.

To compute the first asymptotic order of the splitting of separatrices for
general analytic Hamiltonian systems seems nowadays a problem out of reach. Nevertheless,
 (non-sharp) exponentially small upper bounds were already obtained by
Neishtadt in \cite{Neishtadt84} using averaging techniques and by
\cite{FontichS90,Fontich95} using complex extensions of the invariant manifolds.

%llll.scitotpmysa tcerroc sti sevig hcaorppa vokinleM> fi tuo dnif ot si
%noitseuq larutan a llams
%ylliatnenopxe eb ot sah gnittilps eht taht wonk ew ecnO

Once we know that the splitting is exponentially small, a natural question which
arises is whether the Melnikov
potential gives the correct asymptotic first order of the splitting. In
comparison with the problem of giving exponentially small upper bounds for the
splitting, this problem is much more intricate. The results in this direction
strongly depend on the behavior of the homoclinic orbit $(q_0(u),p_0(u))$ around
its complex singularities and on the analytical properties of the perturbation.

The previous considerations lead us to consider the problem of splitting of separatrices for general systems which are
either algebraic in $(x,y)$ or trigonometric polinomial in $x$ and algebraic in $y$.

As we have already explained, inspecting formula \eqref{eq:formulanocorrecta},
one sees that Melnikov theory works provided $\mu\eps^\eta=\oo(\eps^\beta
e^{-a/\eps})$. Namely, one needs the size of the perturbation to be
exponentially small with respect to $\eps$. This is not the natural setting and
therefore the first works dealing with this problem
\cite{HolmesMS88} (see also Section \ref{sec:Historical} about historical
remarks) tried to enlarge the size of the perturbation $\mu\eps^\eta H_1$ for
which Melnikov theory actually measures the splitting.
In fact, under certain non-degeneracy conditions, it suffices to take $\eta$ big
enough and $\mu$ of order 1.

In this work we have obtained, for  Hamiltonians \eqref{eq:model}
satisfying the hypotheses given in Section \ref{sec:Hypotheses}, the open
set of values of $\eta$ for which the Melnikov prediction works.

Studying the phenomenon of splitting in general Hamiltonian systems, for $\eta$ in the boundary of this set,
we have
found examples where
the Melnikov theory does not predict correctly the formula for the area
of the lobes  \eqref{eq:formulanocorrecta} in several aspects.

%In section \ref{sec:examples} we will see models
There are cases where the constant $K$ is not
correctly given by the Melnikov formula. This phenomenon has been found before
in concrete examples \cite{Gelfreich00,Treshev97,Olive06, GuardiaOS10}.
In these cases, the correct value of the constant $K$ is obtained  from the study of the so called \textit{inner equation}.

Moreover, we have found a more surprising phenomenon, namely, there are cases where the Melnikov prediction  \eqref{eq:formulanocorrecta0}  does not give  the correct order of the splitting. More concretely, it fails to predict the constant $K$ but also the correct power $\beta$ in
\eqref{eq:formulanocorrecta0}.
In section \ref{sec:examples} we provide a concrete model where this phenomenon happens.

Our work shows that   all the results validating the prediction of the Melnikov
approach require some artificial conditions about the smallness of the
perturbation.
The reason, roughly speaking, is the following.
To prove that  Melnikov theory gives asymptotically the  first order of the
splitting one needs to perform ``complex perturbation theory''. Namely, one
looks for complex parameterizations $Z_\mu^{u,s}(u,t_0)$ of the perturbed
invariant curves $C^{u,s}(t_0)$  of the Poincar\'e map $P_{t_0}$ as a
perturbation of the time-parameterization of the unperturbed separatrix
$Z_0(u)=(q_0(u),p_0(u))$.
This is the main novelty in the proofs of exponentially small splitting,
and was discovered by Lazutkin in his pioneer paper
\cite{Lazutkin84russian}:  the perturbed and unperturbed manifolds, as well as
the solutions of the variational equations along them, need to be close enough
when one considers complex times in a domain which contains a suitable real interval and
which reaches a neighborhood of order $\eps$ of the singularities
of the unperturbed homoclinic orbit. Clearly, when time is real, the homoclinic
orbit is a bounded solution
and it is easy to see that the perturbed invariant manifolds are close to it in suitable intervals.
However, when we reach a neighborhood of its  singularities, the homoclinic
orbit itself
blows up, and it is not always the case that the perturbed invariant manifolds are close
to it anymore.
Of course assuming artificially that the perturbation is small enough (increasing $\eta$
in the perturbative term in \eqref{eq:model}) one can see that the perturbed
manifolds are close to the unperturbed homoclinic orbit in a complex domain which
reaches a neighborhood
of size $\eps$ of the singularities of the unperturbed homoclinic trajectory.
Consequently the Melnikov
approach, that is based on the fact that the perturbed
manifolds are well approximated by the unperturbed homoclinic orbit, still works.
This was the approach used in \cite{DelshamsS97,Gelfreich97a,BaldomaF04} for
$\eta >\ell$, were the constant $\ell$ was called the \emph{order}
of the perturbation $H_1$. Roughly speaking, it is
the order of the singularities  of the unperturbed homoclinic trajectory
$(q_0(u),p_0(u))$ closest to the real axis
of the function $h_1(u)=H_1(q_0(u),p_0(u), t/\eps;0)$, for any $t \in \RR$.

In the aforementioned works, the condition $\eta>\ell$ ensures that the
perturbed  parameterizations $Z_\mu^{u.s}$ are close to the parameterization of
the unperturbed separatrix $Z_0$ even up to  a distance of order $\eps$ of the
singularities of $Z_0$ closest to the real axis. Nevertheless, as we will see in
this paper, the condition $\eta>\ell$ is sufficient but not necessary to ensure
that Melnikov approach still predicts correctly the size of the splitting.  What
is important is the relative size between the homoclinic orbit $Z_0$ and the
difference between
the homoclinic orbit and the perturbed manifolds, and analogously between the
solutions of the corresponding variational equations.
In other words, as the parameterizations of the invariant manifolds can be
written as
$Z^{u,s}_{\mu}= Z_0+ (Z^{u,s}_{\mu}- Z_0)$,
the  Melnikov method
gives the correct asymptotic term for the size of the splitting provided
the homoclinic $Z_0$ is bigger than the  difference $Z^{u,s}_{\mu}- Z_0$.
For systems of type \eqref{eq:model} this condition can be easily stated as
follows.
Call $r$ to the order of the singularities of $p_0(u)$ closest to the real axis.
Then, the size of $p_0(u)$ at points $u$ which are $\eps$-close to the
singularities is  $\OO(\eps^{-r})$.
Looking at the relative size of  $\mathrm{grad}H_0(q_0(u),p_0(u))$ and
$\mu \eps ^{\eta}\mathrm{grad}H_1(q_0(u),p_0(u),\tau;\eps)$, one can guess that
the first one is strictly bigger than the second if $\eta-(\ell-r)>-r$.
Working with the equations associated to Hamiltonian System \eqref{eq:model},
we prove in this paper that $Z_0(u)$ is strictly
bigger than $Z^{u,s}_{\mu}(u,t_0)- Z_0(u)$ provided $\eta >\ell -2r$, even if
$u$ is at  a distance $\eps$ of the singularity.

For $\ell\geq 2r$, the condition for both the parameterizations and  the solutions
of the variational equations to be  relatively close coincides and is given by 
$\eta>\eta^\ast=\ell-2r$.
For  $\ell<2r$ we will not consider  values of $\eta$ such that $\ell-2r<\eta  <0$.
In fact, decreasing $\eta$, we  will reach first the ``natural'' limit
$\eta =0$, where $\mathrm{grad}H_0(q_0(u),p_0(u))$ and
$\mu \mathrm{grad}H_1(q_0(u),p_0(u),\tau;\eps)$
%the  partial derivatives of $H_0(x,y)$ and $H_1(x,y,\tau;\eps)$
%evaluated on the unperturbed homoclinic
are not close even for real values of
$u$. Even if for concrete examples \cite{Gelfreich00,GuardiaOS10} one can prove
the existence of invariant manifolds and compute the size of their splitting for
negative values of $\eta$, in this paper we deal with general Hamiltonians and
$\eta \ge 0$.  This  means that we deal with cases for which the unperturbed system and the perturbation
can have the same size.

% (see Section \ref{sec:Motivation} for some problems which
%correspond to this natural limit $\eta=0$).

When $\eta=0$, one can apply classical averaging theory to see that we are still
in a perturbative setting and the real perturbed  invariant manifolds are
$\mu\eps$-close to the real unperturbed separatrix and it makes sense  to study the splitting of separatrices in this case.
Nevertheless, as we will see in this paper, the solutions of the variational
equations are not close enough near the singularity in this case. This implies that, as is stated in
Theorems \ref{th:MainGeometric:regular} and \ref{th:MainGeometric:singular},  Melnikov formula \eqref{eq:formulanocorrecta0} generically
does not give the
correct first asymptotic term of the splitting.

In conclusion, under certain non-degeneracy conditions, the previous considerations suggest, and we actually will prove in Theorem
\ref{th:MainGeometric:regular} and Corollary \ref{coro:MainGeometric:regular},
that Melnikov theory gives
the correct prediction provided
\[
 \eta>\eta^*=\max\{\ell-2r,0\}.
\]

%\marginpar{Comentar C}

The so called ``singular'' case occurs when the difference
$Z^{u,s}_{\mu}(u,t_0)- Z_0(u)$ has the same size as the unperturbed homoclinic
$Z_0(u)$ when $u$ reaches a neighborhood at a distance $\eps$ of the
singularities of $Z_0$. Consequently,
the invariant manifolds are not well approximated by the unperturbed homoclinic
in this complex region. Let us note that this singular case can only happen if
$\ell\geq 2r$ and $\eta=\eta^\ast$.
In this case, we need to obtain a different approximation
of the manifolds in this region of the complex plane.
Close to a singularity of the homoclinic orbit,
an equation for the leading term is obtained and it is called the
\emph{inner equation}.
This is  a non-integrable equation  whose  study is done in \cite{Baldoma06}.

Summarizing, on the one hand, the invariant manifolds are well approximated by
the unperturbed homoclinic orbit
in a complex region containing an interval of the real line. On the other hand,
the inner equations provide
good approximations of the invariant manifolds near the singularities of the
unperturbed homoclinic.
Finally, matching techniques are required to match the different
approximations obtained for the invariant manifolds. Roughly
speaking, the difference between two suitable solutions of the inner equations
replaces the Melnikov potential in the asymptotic formula for the
splitting.

We want to emphasize that, as far as the authors know, there are no general
results
dealing with the singular case.
The previous results in the singular case (see \cite{Lazutkin84russian,
Lazutkin84,Gelfreich00,
Treshev97,Olive06, GuardiaOS10})
only dealt with particular examples.

In this paper we give results that contain the
so-called regular  case $\eta>\eta^\ast$ (see Section
\ref{sec:Hypotheses}), in which the Melnikov formula predicts
correctly the splitting  between the manifolds, but we also consider the
so-called singular case
$\eta=\eta^\ast$, in which the Melnikov formula does not predict
correctly the splitting between the perturbed manifolds anymore.
In this singular case we provide and prove  an alternative formula for the splitting.

We have seen that the behavior of the splitting is extremely sensitive on the
sign of $\ell-2r$ and the value of $\eta$.
%In consequence, one has to obtain the
%different first asymptotic orders separately, taking into account the properties
%of each case.
We summarize the main features of each case:
\begin{itemize}
\item
$\eta>\eta^\ast= \max \{\ell-2r, 0\}$: under certain non-degeneracy conditions,
the Melnikov formula \eqref{eq:formulanocorrecta0} gives the correct first order of the splitting,
that is, the correct constants $K$, $\beta$ and $a$.
Moreover, the transversality of the splitting is a direct consequence of the
existence of non-degenerate critical points of the Melnikov potential, which is
ensured if $K\ne 0$.
\item
$\ell-2r<0$ and $\eta=0$: it appears a (depending on $\mu$) constant  correcting
term which multiplies  $K$ in the Melnikov formula \eqref{eq:formulanocorrecta0}.
This term can be obtained through classical perturbation theory techniques. This
correcting term  does not vanish for any value of $\mu$. Therefore, the first
asymptotic order is non-degenerate if and only if $K\ne 0$. Note that in this
case, for real values of the variables, $H$ is not a perturbation of $H_0$.
\item
$\ell-2r>0$ and $\eta=\eta^\ast=\ell-2r$: it appears a (depending on $\mu$) constant  correcting term which replaces
$K$ in the Melnikov formula \eqref{eq:formulanocorrecta0}.
This correcting term has a significantly different origin from the one in the previous case, since it
comes from the study of the aforementioned \emph{inner equation}. In particular, it can vanish for some values of $\mu$. Then, the
transversality of the invariant manifolds is guaranteed provided this correcting term does not vanish. Let us note that for the range $\eta\in
[0,\ell-2r)$ the problem of the splitting of separatrices remains open.
\item
$\ell-2r=0$ and $\eta=0$: as in the previous case, we need to consider an
\emph{inner equation} to obtain a candidate for the  first asymptotic order of
the splitting. This candidate differs from the Melnikov formula by both the
constant $K$ and the exponent $\beta$.
%Nevertheless, to obtain the true first
%order, one has to make still one more modification to the formula given by the
%inner equation, which implies a change in the constant provided by the inner equation.
Note, that the change in the exponent $\beta$ is a substantial
qualitative change in the behavior of the splitting. Even if this fact was
already pointed out in \cite{Baldoma06}, the present paper, as far as the
authors know, is the first work that rigorously proves that this phenomenon
actually happens.
\end{itemize}

This work concludes the general problem, initiated
and partially solved in
\cite{DelshamsS97,Gelfreich97a,BaldomaF04,BaldomaF05} for $\eta >\ell$, of the
splitting of separatrices in the singular and regular cases $\eta \ge \eta ^*$,
for the general mentioned perturbations $H_1$ of  classical polynomial or trigonometric polynomial Hamiltonian systems
$H_0(x,y)=\frac{y^2}{2}+V(x)$.

\subsection{Historical remarks}\label{sec:Historical}

Historically, the results about exponentially small splitting of separatrices
can be classified into three groups: upper bounds, validation of the Melnikov
approach and asymptotics for  the singular case.

Some results, dealing with quite general systems,  obtain exponentially
small upper
bounds for the splitting for Hamiltonian systems.
Neishtadt in \cite{Neishtadt84} gave exponentially small upper bounds for the
splitting for two degrees of freedom Hamiltonian systems.
For second order equations with a rapidly forced periodic term, several authors
gave sharp
exponentially small upper bounds in~\cite{Fontich93,Fontich95,FiedlerS96} and,
for the higher dimensional case, the papers~\cite{Sauzin01, Simo94} gave
(non-sharp)
exponentially small upper bounds.

The Poincar\'{e} map of a non-autonomous Hamiltonian in the plane is a
particular case of a
planar area preserving map. For the Hamiltonian \eqref{eq:model} the
Poincar\'{e} map $P$ is  a near
the identity area preserving map.
Rigorous upper bounds for the splitting of area
preserving maps close to the identity were given in \cite{FontichS90}.

The second group of results is concerned with the question of the validity of
the
asymptotics provided by the Melnikov theory.
Several authors in the last 15 years  have tried to ensure the validity of
the formula provided by the Melnikov
potential \eqref{eq:formulanocorrecta0} to compute the  asymptotic formula for the
area $\mathcal{A}$.
As we have already said, the results in this direction strongly depend on the
behavior of the homoclinic orbit around
its complex singularities and on the analytical properties of the perturbation.
For this reason, the existing results  in this direction mostly deal with
specific examples.

%llll.scitotpmysa tcerroc sti sevig hcaorppa vokinleM> fi tuo dnif ot si
%noitseuq larutan a llams
%ylliatnenopxe eb ot sah gnittilps eht taht wonk ew ecnO

The most studied example in the literature has been the rapidly perturbed
pendulum with a
perturbation only depending on time,
\[
 \ddot x=\sin x+\mu\eps^\eta\sin\frac{t}{\eps},
\]
which in our notation corresponds to  $H_0(x,y) = y^2/2+\cos x -1$ and
$H_1(x,t/\eps) = -x \sin (t/\eps)$.
The first result concerning this system was obtained by Holmes, Marsden and
Scheurle
in~\cite{HolmesMS88} (followed by~\cite{Scheurle89,Angenent93}),
where they  confirmed the prediction of the Melnikov potential establishing
exponentially
small upper and lower bounds for the area $\mathcal{A}$
provided $\eta\geq 8$, which coincide with the Melnikov prediction.
Later the work~\cite{EllisonKS93} validated the same result for $\eta\geq 3$.
Delshams and Seara established rigourosly the result in~\cite{DelshamsS92}
for $\eta>0$ and an analogous result for $\eta>5$ was obtained by Gelfreich in
\cite{Gelfreich94}.
The latter two  papers used a different approach inspired by
the work of Lazutkin \cite{GelfreichLT91}.
For a simplified perturbation  an alternative proof, using
Parametric Resurgence, was
done in  \cite{Sauzin95}.

%\marginpar{mirar si Gelfreich es $\eta >5$}

The only works which provide (partial) results for some general Hamiltonian as
\eqref{eq:model}
taking $\eta$ big enough, are \cite{DelshamsS97,Gelfreich97a,BaldomaF04,
BaldomaF05}. In \cite{DelshamsS97,Gelfreich97a},
a proof for the validity of the Melnikov method for
general rapidly periodic Hamiltonian perturbations of a class of
second order equations was given.
The case of a perturbed second order equation with a parabolic point was studied
in~\cite{BaldomaF04,
BaldomaF05}.

In the papers~\cite{Sauzin01,LochakMS03} the authors introduced a
different approach that avoided the ``flow box coordinates'' of Lazutkin's
method.
The authors worked with  the original
variables of the problem and were able to measure the distance between the
manifolds without using
``flow box coordinates''.
The idea was the following: being both manifolds given by the graphs of suitable
functions that are
solutions of the same equation, their difference satisfies a linear
equation and is bounded in some complex strip. Studying the
properties of bounded solutions of this linear equation, where
periodicity also plays a role, one obtains exponentially small
results.

The method in~\cite{Sauzin01,LochakMS03} uses the fact that, in the  considered
systems,
the manifolds  can be written as graphs of the gradient
of generating functions in suitable domains. These generating functions are
solutions of the
Hamilton-Jacobi equation associated to system \eqref{eq:model}. Solving
these partial differential equations one can obtain parameterizations of the
global manifolds.

%Even if in this paper we deal with  general  Hamiltonian systems and then the
%invariant manifolds may not
%be expressed as graphs globally in the original variables, we have adapted the
%method in \cite{Sauzin01,LochakMS03}.
%We work with the Hamilton-Jacobi equation in suitable domains
%where the manifolds are given by graphs and then we measure the splitting there.

A Melnikov theory for twist maps can be found in~\cite{DelshamsR97} and
some results about the validity of the prediction given
by the Poincar\'e function for area preserving maps were given in
\cite{DelshamsR98}.

The generalization of the splitting  problem to higher dimensional systems has
been
achieved by several authors, mainly in the Hamiltonian case.
See, for instance,~\cite{Eliasson94,Treshev94,LochakMS03,DelshamsG00} and
references
therein.
Some results about the validity of the Melnikov method for higher dimensional
Hamiltonian
systems can be found in
\cite{Gal94, ChierchiaG94, DelshamsGJS97, GalGM99, Sauzin01, DelshamsGS04}.
Finally, in a non Hamiltonian setting, in \cite{BaldomaS06} the splitting of a
heteroclinic orbit for some degenerate unfoldings of the Hopf-zero
singularity of vector fields in $\RR ^{3}$ was found.

As we have already explained, all the results validating the prediction of the
Melnikov
approach require some artificial condition about the smallness of the
perturbation.

The third group of results deals with the so called ``singular case''  $\eta =
\eta ^*$ for which one needs to study the \emph{inner equation} and use matching
techniques to relate  different
approximations for the invariant manifolds.

The first author who dealt with this singular case was Lazutkin
in~\cite{Lazutkin84russian,
Lazutkin84}.
He studied the splitting of separatrices of the
Chirikov standard map, and gave the main idea that inspired most of  the works in the
subject:
as we explained above, one needs to deal with suitable
complex parameterizations of the invariant manifolds.
A complete proof was published years later by Gelfreich in~\cite{Gelfreich99}. A
fundamental tool in Lazutkin's work is the use of ``flow box
coordinates", called ``straightening the flow"
in~\cite{Gelfreich00}, around one of the manifolds. In this way, one
obtains a periodic function whose values are related with the
distance between the manifolds and whose zeros correspond to the
intersections between them. Consequently, the result about
exponentially small splitting is derived from some properties of
analytic periodic functions bounded in complex strips (see, for
instance, Proposition 2.7  in~\cite{DelshamsS97}).

After Lazutkin's work, some authors used his method and
obtained results for the inner equation of several specific equations.
In~\cite{GelfreichS01} there is a rigorous study of the inner equation of the
H\'enon map using Resurgence Theory \cite{Ecalle81a,Ecalle81b}, and
in~\cite{BaldomaS08}
the authors studied  the inner system associated to the
Hopf-zero singularity using functional analysis techniques.
The corresponding inner equation for several periodically perturbed second order
equations
was given by Gelfreich in~\cite{Gelfreich97} and he called them Reference
Systems.
In~\cite{OliveSS03} there is a rigorous analysis
of the inner equation for the Hamilton-Jacobi equation
associated to a pendulum equation with perturbation term $H_1(x,t/\eps) = (\cos
x -1) \sin (t/\eps)$ by
using Resurgence Theory.
The only result which deals with the inner equation associated to general
polynomial Hamiltonian
like \eqref{eq:model} is~\cite{Baldoma06}, where this analysis is done using
functional analysis
techniques. Finally, in \cite{MartinSS10a}, the authors study the inner equation
of the McMillan Map.

Besides the work of Lazutkin, there are very few works with rigorous proofs in
the singular case.
In~\cite{Gelfreich00} there is a detailed
sketch of the proof for the splitting of separatrices of the
equation of a pendulum with perturbation $H_1(x,t/\eps)= x \sin (t/\eps)$  and
$\eta ^*=-2$.
A complete rigorous proof which also cover some ``under the limit'' cases, that
is $\eta <\eta ^*=-2$
is done in \cite{GuardiaOS10}.
Numerical results about the splitting for this problem can be found
in~\cite{BensenyO93, Gelfreich97}.
In~\cite{Olive06} it was obtained a rigorous proof for the pendulum with
perturbation $H_1(x,t/\eps) = (\cos x -1) \sin (t/\eps)$, for which
$\eta^\ast=0$.
Treschev, in a remarkable paper~\cite{Treshev97}, gave  an asymptotic formula for the
splitting in the case of a pendulum with
%$H_1 = (\cos x -1) \left( B_{-1} e^{-i t/\eps}+ B_{-1} e^{-i t/\eps}\right)$,
certain perturbations, for which $\eta^\ast=0$,
using a different method called Continuous Averaging.
Concerning 2-dimensional symplectic maps, a detailed numerical study of the splitting can be found in \cite{DelshamsR99, GelfreichSimoSagaro}. 
The study of the splitting for the H\'enon and McMillan
maps  have recently been completed in \cite{BrannstromG10} and
\cite{MartinSS10b} respectively.
Both cases correspond to  $\eta^*=0$.
Finally, in \cite{GaivaoG10}, combining numerical and analytical techniques,
the authors study the Hamiltonian-Hopf bifurcation.

Another work dealing with a singular case is~\cite{Lombardi00},  where the
author
proves the splitting of separatrices
for a certain class of reversible systems in $\RR^4$. 
A related problem
about adiabatic invariants for the harmonic oscillator is studied in \cite{Slutskin64}.
See also \cite{ArnoldKN88}.
The study of this problem using matching
techniques and Resurgence Theory was done in \cite{BonetSSV98}.

The structure of this paper goes as follows. First in Section
\ref{sec:HypsAndMainResults} we introduce some notation, the
hypotheses and we state the main results. In Section \ref{sec:Heuristic} we give
some heuristic ideas of the proof and we compare our methods to those of some of
the aforementioned previous results.
Section \ref{sec:SketchProof} is devoted to
describe the proof of the main theorems. To make this section more
readable, the proof of the partial results obtained in this section
are deferred to the following sections, that is, Sections
\ref{sec:periodica}-\ref{sec:CanviFinal}.

\section{Notation and main results}\label{sec:HypsAndMainResults}
In this section we present the main problem we consider, the
hypotheses we assume and the rigorous statement of the main results.

\subsection{Notation and hypotheses}\label{sec:Hypotheses}
We consider Hamiltonian systems with Hamiltonian function of the
form
\begin{equation}\label{def:Ham:Original0}H\left(x,y,\frac{t}{\eps}
;\eps\right)=H_0(x,y)+\mu\eps^{\eta}H_1\left(x,y,\frac{t}{\eps};\eps\right),
\end{equation}
where
\begin{equation}\label{def:Ham:Integrable}
H_0(x,y)=\frac{y^2}{2}+V(x)
\end{equation}
and $V$ is either a polynomial or a trigonometric polynomial.  In
the first case we assume that
\begin{equation}\label{def:Ham:Original:perturb:poli}
H_1\left(x,y,\tau;\eps\right)=\sum_{k+l=n}^N a_{kl}(\tau;\eps)x^ky^l
\end{equation}
and in the second one
\begin{equation}\label{def:Ham:Original:perturb:trig}
H_1\left(x,y,\tau;\eps\right)=a(\tau;\eps)x+\sum_{\substack{k=-N,\ldots,
N\\l=0,\ldots, N}} a_{kl}(\tau;\eps)e^{k i x}y^l=\sum_{i+j\geq n}\wh
a_{ij}(\tau;\eps)x^iy^j ,
\end{equation}
where the second equality defines $n$ and $\wh
a_{ij}$.
Even if in the second case $H_1$ can have terms of the form
$a(\tau;\eps)x$, we will refer to $H_1$ as a trigonometric polynomial. In both
cases
we will refer to  $n$ as the order of $H_1$.

The equations associated to the Hamiltonian
\eqref{def:Ham:Original0} are
\begin{equation}\label{eq:ode:original0}
\left\{\begin{array}{l}\dps \dot x=y+\mu\eps^\eta \pa_y
H_1\left(x,y,\frac{t}{\eps};\eps\right)\\
\dps\dot y=-V'(x)-\mu\eps^\eta \pa_x
H_1\left(x,y,\frac{t}{\eps};\eps\right).
\end{array}\right.
\end{equation}
From now on, we call  unperturbed system to the system defined by
the Hamiltonian $H_0$ and we refer to $H_1$ as the perturbation. Let
us observe that the term $a(\tau;\eps)x$ in
\eqref{def:Ham:Original:perturb:trig} corresponds to a term in
\eqref{eq:ode:original0} which only depends on time (and on the parameter
$\eps$).

We devote the rest of the section to state the hypotheses we assume
on $H$.

\subsubsection{Hypotheses on the unperturbed system}
We assume the following hypotheses corresponding to the unperturbed
system
\begin{description}
\item[\textbf{HP1}] $H_0(x,y)=y^2/2+V(x)$, where $V$ is either a polynomial or a
trigonometric polynomial and satisfies one of the following conditions
\begin{description}
\item[\textbf{HP1.1}] $H_0$ has a hyperbolic critical point at $(0,0)$ with
eigenvalues $\{\la,-\la\}$ with $\la>0$, and then
\[
V(x)=-\frac{\la^2}{2}x^2+\OO\left(x^3\right)\,\quad\text{ as
}x\rightarrow 0.
\]
%In this case  we allow $n\geq 1$.
\item[\textbf{HP1.2}] $H_0$ has a parabolic critical point at $(0,0)$ and then
\begin{equation}\label{def:Potencial:Parabolic}
V(x)=v_m x^m+\OO\left(x^{m+1}\right)\,\quad\text{ as }x\rightarrow
0,
\end{equation}
for certain $m \in \NN$,   $m\geq 3$, which is called the order of
$V$ and $v_m\in\RR$.
% In this case we allow $n$ to be such that $2n-2\geq m$.
\end{description}
\item [\textbf{HP2}] The critical point $(0,0)$ has stable
and unstable invariant manifolds which coincide along a separatrix.

We denote by $(q_0(u),p_0(u))$ a real-analytic time parameterization of the
separatrix with some chosen (fixed) initial condition.
It is well known (see \cite{Fontich95} for the hyperbolic case and
\cite{BaldomaF04} for the parabolic one) that  there exists $\rho>0$ such that
the parameterization $(q_0(u), p_0(u))$ is analytic in the complex strip $\{ |\Im
u |<\rho\}$.

We assume that there exists a real-analytic time parameterization of the
separatrix $(q_0(u),p_0(u))$ analytic on $\{|\Im u | < a\}$
such that the only singularities of $(q_0(u),p_0(u))$ in the lines $\{\Im u=\pm
a\}$ are $\pm ia$.

More precisely, Hypothesis \textbf{HP2} implies that one of the two following
situations is satisfied (see the remarks in Section \ref{remrkshipotesis}):
\begin{description}
\item[\textbf{HP2.1}] In the polynomial case, the singularities $\pm ia$ of the
homoclinic orbit are branching points
(or poles) of the same order, i.e. there exists an irreducible rational number
$r=\p/\q>1$ (independent of the singularity) and $\nu>0$ such that
$(q_0(u),p_0(u))$ can be expressed as
\begin{equation}\label{eq:SepartriuAlPol}
\begin{split}
\dps q_0(u)&=-\frac{C_\pm}{(r-1)(u\mp
ia)^{r-1}}\left(1+\OO\left((u\mp ia)^{1/\q}\right)\right)\\
\dps p_0(u)&=\frac{C_\pm}{(u\mp ia)^{r}}\left(1+\OO\left((u\mp
ia)^{1/\q}\right)\right)
\end{split}
\end{equation}
for $u\in\CC$ and either $|u-ia|<\nu$ and $\mathrm{arg}(u-ia)\in
(-3\pi/2,\pi/2)$ or  $|u+ia|<\nu$ and
$\mathrm{arg}(u+ia)\in(-\pi/2,3\pi/2)$ respectively. Let us point
out that the real-analytic character of $(q_0(u),p_0(u))$ implies that
$C_-=\ol C_+$.
\item[\textbf{HP2.2}] In the trigonometric case, $q_0(u)$ has logarithmic
singularities at $\pm ia$
of the form  $q_0(u)\sim \ln(u\mp ia)$ (where we  take different
branches of the logarithm whether we are close to $+ia$ or $-ia$: we
take $\mathrm{arg}(u-ia)\in (-3\pi/2,\pi/2)$ and
$\mathrm{arg}(u+ia)\in(-\pi/2,3\pi/2)$ respectively). In this case, one can see
that there exists $M\in \NN$ such that, if
$u\in\CC$, $|u\mp ia|<\nu$,
\begin{equation}\label{eq:SepartriuAlPolTrig}
\begin{split}
\dps \cos(q_0(u))&=\frac{\wh C^1_\pm}{(u\mp
ia)^{2/M}}\left(1+\OO\left((u\mp ia)^{2/M}\right)\right)\\
\dps \sin(q_0(u))&=\frac{\wh C^2_\pm}{(u\mp
ia)^{2/M}}\left(1+\OO\left((u\mp ia)^{2/M}\right)\right)\\
\dps p_0(u)&=\frac{C_\pm}{(u\mp ia)}\left(1+\OO\left((u\mp
ia)^{2/M}\right)\right)
\end{split}
\end{equation}
with  $\mathrm{arg}(u-ia)\in (-3\pi/2,\pi/2)$ and
$\mathrm{arg}(u+ia)\in(-\pi/2,3\pi/2)$ if we are dealing with the
singularity $+ia$ or $-ia$ respectively. We also have that $C_+ =
\overline{C_-}= \pm i 2/M $.

For convenience, in the trigonometric case, we take the convention
$r=1$ and $\beta=M$.
\end{description}
\end{description}

\subsubsection{Hypotheses on the perturbation}
\begin{description}
\item[\textbf{HP3}] The function $H_1(x,y,\tau;\eps)$ is $2\pi$-periodic in
$\tau$ and real-analytic in $(x,y,\tau,\eps)\in \CC^2\times\TT\times
(-\eps^*,\eps^*)$, for certain $\eps^*>0$.
Furthermore,  either it is a polynomial  of the form
\eqref{def:Ham:Original:perturb:poli} if $V(x)$ is a polynomial or it is a
trigonometric
polynomial  of the form \eqref{def:Ham:Original:perturb:trig} if $V(x)$ is a
trigonometric polynomial.
Moreover, it has zero mean
\[
\int_0^{2\pi}H_1(x,y,\tau;\eps)\, d\tau=0.
\]
\item[\textbf{HP4}] Let us consider the order of $H_1$, $n$ given in
\eqref{def:Ham:Original:perturb:poli} or
\eqref{def:Ham:Original:perturb:trig}. We ask $H_1$ to satisfy:
\begin{description}
\item[\textbf{HP4.1}] In the hyperbolic case ($H_0$ satisfies \textbf{HP1.1}),
$n\geq 1$.
\item[\textbf{HP4.2}] In the parabolic case  ($H_0$ satisfies \textbf{HP1.2}),
$2n-2\geq m$.
\end{description}
\end{description}
\begin{remark}
Let us point out that, in fact, \textbf{HP4.1} does not add any
extra hypothesis on the Hamiltonian, since it can always be taken
with $n\geq 1$ (the constant terms in $(x,y)$ do not play any role).
\end{remark}
%\begin{remark}
%In the trigonometric case, the condition to be satisfied is the same
%taking into account that in order to define the order $n$ of $H_1$
%it is enough to Taylor-expand $H_1$ in $(x,y)$ around $(0,0)$ and to
%take the lowest degree among the monomials in $(x,y)$.
%\end{remark}
Let us consider the function $H_1(q_0(u),p_0(u),\tau;\eps)$ that is:
$H_1$ evaluated on the separatrix. Then, we define $\ell$ to be the
 order of the branching points $\pm ia$, namely,
the maximum of the  orders of the branching points of the monomials of $H_1$.
This parameter was already defined in \cite{DelshamsS97,BaldomaF04}.
Let us point out that $\ell$ can be simply defined  as
\begin{equation}\label{def:ell}
\begin{split}
\ell(\eps)&=\max_{ n \le k+l \le N}\left\{k(r-1)+l r;
a_{kl}(\tau;\eps)\not\equiv
0\right\}\quad\quad\text{(polynomial case)}\\
\ell(\eps)&=\max_{|k|\le N, \ 0\le l\le N}\left\{2|k|/M+l;
a_{kl}(\tau;\eps)\not\equiv
0\right\}\quad\quad\text{(trigonometric case).}
\end{split}
\end{equation}
Note that in the trigonometric case, if $H_1(x,y,\tau;\eps)=a(\tau;\eps) x$,
then  $H_1(q_0(u),p_0(u),\tau;\eps)$ has a logarithmic singularity (see
Hypothesis \textbf{HP2.2}). In this case we make the convention $\ell(\eps)=0$.
\begin{description}
\item[\textbf{HP5}]
We assume $\ell=\ell(0)=\ell(\eps)$ for all $\eps\in (-\eps^*,\eps^*)$ and
$\eta\geq \eta ^*= \max\{0,\ell-2r\}$.
\end{description}

%Since we are assuming $\eta\geq\ell-2r$, we rename the Hamiltonian
%function \eqref{def:Ham:Original0} as
%\begin{equation}\label{def:Ham:Original}
%H\left(x,y,\frac{t}{\eps}\right)=H_0(x,y)+\mu\eps^{\ell-2r}H_1\left(x,y,\frac{t
%}{\eps}\right).
%\end{equation}
%Therefore, with this notation, we give results for $\mu$ constant
%and independent of $\eps$ or small. The corresponding system is
%given by
%\begin{equation}\label{eq:ode:original}
%\left\{\begin{split} \dot x&\dps=y+\mu\eps^{\ell-2r} \pa_y
%H_1\left(x,y,\frac{t}{\eps}\right)\\
%\dot y&\dps =-V'(x)-\mu\eps^{\ell-2r} \pa_x
%H_1\left(x,y,\frac{t}{\eps}\right).
%\end{split}\right.
%\end{equation}

\subsubsection{Some remarks about the hypotheses} \label{remrkshipotesis}

\begin{itemize}

\item Let us point out
that the time parameterization of the separatrix
has always singularities for complex time
(see \cite{Fontich95} for the hyperbolic case and \cite{BaldomaF04} for the
parabolic one).
The real restriction in \textbf{HP2} is that there exists only one singularity
in the lines $\{\Im u=\pm a\}$. In Remark \ref{remark:DosSing} we explain how to generalize the
results obtained in this paper to systems whose
separatrix has more than one singularity with the same  minimum
imaginary part.

\item The conditions satisfied in \textbf{HP2.1} and \textbf{HP2.2} are
consequence of \textbf{HP2}. Indeed,
let $u^{*}$ be a singularity of $(q_0(u),p_0(u))$. We have that:
\begin{itemize}
\item If $V$ is a polynomial, let $M$ be its degree. Then $u^*$ is a branching
points
(or pole) of order $2/(M-2)$. That is, if $u$ belongs to a neighborhood of
$u^*$, then $(q_0(u),p_0(u))$ can be expressed as
\begin{equation*}
\begin{split}
\dps
q_0(u)&=-\frac{C(M-2)}{2(u-u^*)^{2/(M-2)}}\left(1+\OO\left((u-u^*)^{2/(M-2)}
\right)\right)\\
\dps
p_0(u)&=\frac{C}{(u-u^*)^{M/(M-2)}}\left(1+\OO\left((u-u^*)^{2/(M-2)}
\right)\right)
\end{split}
\end{equation*}
with $C\neq 0$ some adequate constant. This fact is proved in \cite{BaldomaF04}.

From the above equalities, taking into account that the homoclinic connection
is a solution of the unperturbed Hamiltonian system and identifying
terms of the same order in $(u-ia)$, one can deduce that the degree of $V$ is $2r/(r-1)$. In fact,
%as long as \textbf{HP2.1} is satisfied we could deal with more general
%potentials $V$, and then Hypothesis \textbf{HP2.1} would imply that
there exists a constant $v_\infty\in\RR$ such that
\begin{equation}\label{eq:PotencialInfinit}
V(x)= \vinfty x^{\frac{2r}{r-1}}(1+\oo(1))\qquad \text{ as }\ x\rightarrow
\infty.
\end{equation}
\item If $V$ is a trigonometric polynomial, let us
call $M$ to its degree. Then, for $u$ belonging to a neighborhood of $u^*$,
$(q_0(u), p_0(u))$ are of the form
\begin{equation*}
\begin{split}
\dps q_0(u)&= C \log \big (-i(u-u^*) \big )+ \OO\big( (u-u^*)^{2/M}\big) \\
\dps p_0(u)&= \frac{C}{(u-u^*)} + \OO\big( (u-u^*)^{2/M}\big)
\end{split}
\end{equation*}
with the constant $C= \pm i 2/M$ depending on $\Im q_0(u) \to \mp \infty$
respectively.
Indeed, first we note that, due to the fact that $\Re q_0(u) \in [0,2\pi]$, the
condition $|q_0(u)|\to +\infty$ as $u\to u^*$ forces to
$|\Im q_0(u)|\to +\infty$ as $u$ goes to $u^*$.
Assume that $\Im q_0(u) \to -\infty$ as $u\to u^*$. We note that in this case,
since $q_0(u)$ is a real analytic function,
then $\overline{u^*}$ is also a singularity of $q_0$ and it satisfies $\Im
q_0(u) \to +\infty$ as $u\to \overline{u^*}$.
We perform the change of variables $x=i\log w$ and we emphasize that, if $\Im x
\to -\infty$, then $w \to 0$. From the fact that
\begin{equation*}
\frac{dx}{du} = \sqrt{-2 V(x)},
\end{equation*}
we obtain that
\begin{equation*}
\frac{du}{dw} = i w^{M/2 -1} (c_0+\OO(w))
\end{equation*}
for some constant $c_0$.
Henceforth, integrating both sides of the previous differential equation,
we obtain $u-u^* = i w^{M/2} (c_1+ \OO(w))$, for some constant $c_1$, which
implies that $w = \big (-i(u-u^*)\big )^{2/M} \big(c_2 + \OO\big ((u-
u^*)^{2/M}\big )\big)$ for a suitable constant $c_2$.
and the results follows going back to the original variables.
\end{itemize}

\item In fact, let us observe that the hypotheses considered about  the
expansions of
$(q_0(u),p_0(u))$ given in \eqref{eq:SepartriuAlPol} and
\eqref{eq:SepartriuAlPolTrig} (\textbf{HP2.1} and \textbf{HP2.2}) are weaker
than what usually happens
when the potential $V$ is a polynomial or a trigonometric
polynomial as we have seen previously. This weakness comes from the fact that
the second terms
in the expansions are, in fact, of  greater order. We assume this
weaker hypothesis to show that our results could be applied to more
general potentials as long as Hypothesis \textbf{HP2} is satisfied.

\item Hypothesis \textbf{HP4.2} is to ensure that the parabolic critical point
$(0,0)$ of the unperturbed system persists when we add the perturbation and that
it keeps its parabolic character. Therefore it is the natural hypothesis to deal
with and it is the same one that was considered in \cite{BaldomaF04}. Namely, if
the perturbation has order $n$ with $2n-2<m$, when the perturbation is added the
system might undergo bifurcations and the invariant manifolds might even
disappear. The only study done in one of these bifurcation cases can be found in
\cite{BaldomaF05}.

\item
The class of  the perturbed Hamiltonian $H_1$ considered is more restrictive than
necessary.
In fact, our result can be applied to any Hamiltonian of the form
\[
H_1(x,y,\tau;\eps) = \sum _{n=0} ^{N}\eps ^{n} H_1^n (x,y,\tau)
\]
if the functions $H_1^n (q_0(u),p_0(u),\tau)$ have a singularity of order less
or equal than $\ell +n$. In this case, the order $\ell (\eps)$ in  \eqref{def:ell} does depend on
$\eps$ ($\ell(0)= \ell$, and $\ell(\eps)= \ell +N$ if $\eps \ne 0$) and then Hypothesis \textbf{HP5} is not satisfied.
The  result in this case would be the same but one has
to slightly adapt the definition of the constant $b$ in Theorem
\ref{th:MainGeometric:singular}.
\item Note that the hypothesis requiring $\ell(\eps)$ constant is nothing but a
non-degeneracy condition on the coefficients $a_{kl}(\tau;\eps)$. This condition
is equivalent to ask that one of the pairs $(k,l)$ reaching the maximum in the
definition of $\ell(\eps)$ in \eqref{def:ell} for any value of $\eps$  must reach also the maximum for $\eps=0$.

\item
Recall the Hamiltonian
\[
H\left(x,y,\frac{t}{\eps};\eps\right)=H_0(x,y)+\mu\eps^{\eta}H_1\left(x,y,\frac{
t}{\eps};\eps\right).
\]
Let us point out that in the case $\ell-2r\leq0$, Hypothesis
\textbf{HP5} corresponds to $\eta \ge 0$, which is optimal in the sense that it
includes the case such that
the perturbation is of the same order as the unperturbed system.

The case $\ell=2r$ is what typically happens in near  integrable
Hamiltonian systems close to a resonance and in general periodic systems with
slow dynamics, therefore, in this
sense Hypothesis \textbf{HP5} is optimal in the generic case.

In the case $\ell-2r>0$ one may think to also ask
$\eta\geq 0$. Nevertheless, our techniques only provide
optimal exponentially upper bounds if $\eta-\ell+2r\geq 0$.

For lower values of $\eta$, that is
$0\leq \eta<\ell-2r$, using similar tools as the ones presented in
this paper, one could easily prove the existence of the perturbed
invariant manifolds and  obtain (non-optimal) exponentially small
upper bounds for the difference between them. This case can be
called \emph{below the singular case} (see \cite{GuardiaOS10}). To
obtain an asymptotic formula for the difference between the
invariant manifolds in the \emph{below the singular case} is a
problem which remains open. Some ideas to deal with this case by
using averaging theory can be found in \cite{GuardiaOS10}.

%\item Hypothesis \textbf{HP6} for the case $\ell <2r$ has an analogous
%condition for the
%general case $\ell \geq 2r$ which is stated in Theorem
%\ref{th:MainGeometric:lmajorigual}.
%Therefore we are not adding an extra hypothesis for the ``easier" case $\ell
%<2r$.

\end{itemize}

%\marginpar{potser treure l'ultim remark i camviar la hipotesi HP6}

\subsection{Main results}\label{sec:MainResults}
By Hypothesis \textbf{HP1}, system \eqref{def:Ham:Original0} with
$\mu=0$ has either a hyperbolic or parabolic  point at the origin.
In the second case, Hypothesis \textbf{HP4.2} ensures that the
origin is also a critical point of the perturbed system ($\mu\neq 0$) which is
also parabolic.
% $2\pi$-periodic orbit which remains at the origin.
In the hyperbolic case, the next
theorem ensures that the  hyperbolic critical point of
the unperturbed system becomes a hyperbolic periodic orbit which is
close to the origin.

\begin{theorem}\label{th:MainPO}
Let us  assume Hypotheses \textbf{HP1.1}, \textbf{HP3},
\textbf{HP4.1}. Take  $\eta \ge 0$  and fix any value $\mu_0>0$. Then,
there exists $\eps_0>0$ such that for any $|\mu|<\mu_0$ and
$\eps\in(0,\eps_0)$, system \eqref{def:Ham:Original0} has a hyperbolic
periodic orbit $(x_p(t/\eps),y_p(t/\eps))$ which satisfies that, for
$t\in\RR$,
\[
\left
|x_p\left(\frac{t}{\eps}\right)\right|+\left|y_p\left(\frac{t}{\eps}
\right)\right|\leq
K |\mu|\eps^{\eta+1}
\]
for a constant $K>0$ independent of $\eps$ and $\mu$.
\end{theorem}
The proof of this theorem, which was done in \cite{DelshamsS97} for
$\eta >\ell$, is given in Section
\ref{sec:periodica}. An alternative proof for values of $\eta>-1/2$ without
explicit bounds for the
periodic orbit can be found in \cite{Fontich95}. For the case when perturbation
only depends on time in \cite{Fontich93} the existence of the periodic orbit
with explicit bounds was given for $\eta >-2$.

To use the same notation in both the hyperbolic and parabolic cases,
in the latter one we define $(x_p,y_p)=(0,0)$.

The next step is to study the stable and unstable invariant
manifolds of the periodic orbit $(x_p,y_p)$. In the unperturbed case
(that is $\mu=0$) we know that they coincide along the separatrix
$(q_0,p_0)$ given in \textbf{HP2}. When $\mu\neq 0$ they generically split.

To measure the splitting of the invariant manifolds let us consider the
$2\pi\eps$-Poincar\'e map $P_{t_0}$ in a transversal section
$\Sigma_{t_0}=\left\{(x,y,t_0); (x,y)\in\RR^2\right\}$. This
Poincar\'e map has a (hyperbolic or parabolic) fixed point
$(x_p(t_0/\eps),y_p(t_0/\eps))$. We will see that this fixed point
has stable and unstable invariant curves.

As $P_{t_0}$ is an area
preserving map, we measure the splitting giving an asymptotic
formula for the area of the lobes generated by these curves between
two transversal homoclinic points. Moreover, by the area preserving
character of $P_{t_0}$, the area $\AAA$ of these lobes does not
depend on the choice of the homoclinic points. Other quantities measuring
the splitting, as the distance along a transversal section to the unperturbed
separatrix, or the angle between these curves at an homoclinic point, can be
easily derived from our work.

Assuming \textbf{HP5}, we have that $\eta\geq \eta^\ast=\max\{\ell-2r,0\}$  (see
Hypothesis \textbf{HP2}
for the definition of $r$ and \eqref{def:ell} for the definition of $\ell$). The
quantitative measure of the splitting depends substantially
on the sign of $\eta-(\ell-2r)$. Therefore, we split these results into two
different theorems.
First,  Theorem \ref{th:MainGeometric:regular} deals with the regular case
$\eta>\ell-2r$ and then
Theorem \ref{th:MainGeometric:singular} deals with the singular case
$\eta=\ell-2r$, which can only happen provided $\ell-2r\geq 0$.
We will give a complete description of the proof of the two theorems in Section
\ref{sec:SketchProof}.
We also refer to Section \ref{sec:Heuristic} for an heuristic idea of the main
features of the proof
of our main results.

\subsubsection{Main result for the regular case}\label{sec:reg}

In this section we will give results concerning the regular case. This case
appears in two different settings. The first one is  when
$\eta>\eta^*=\max\{\ell-2r,0\}$ and we will see in Theorem
\ref{th:MainGeometric:regular} that Melnikov predicts the splitting correctly.
The second case is when $\ell-2r<0$ and $\eta=\eta^\ast=0$. In this case, we
reach the natural value  $\eta=0$ before we reach the singular limit
$\eta=\ell-2r<0$.  We will see in Theorem \ref{th:MainGeometric:regular} that
even if we are in a regular setting, one has to modify slightly the Melnikov
function to obtain the true first asymptotic order.

Since the asymptotic coefficient for the area of the lobe between two
consecutive homoclinic points
is strongly related with the Melnikov Potential, first of all
we are going to obtain an asymptotic formula for it.

The Melnikov Potential (called also sometimes Poincar\'e Function, see for
instance \cite{DelshamsG00}),
is given by
\begin{equation}\label{def:MelnikovPotentialEnt}
L\left(u,\frac{t}{\eps};\eps\right)=\int_{-\infty}^{+\infty}H_1\left(q_0(u+s),
p_0(u+s),\eps\ii (t+s);\eps\right)ds.
\end{equation}

Let us point out that, by Hypothesis \textbf{HP4}, this integral is uniformly
convergent. Moreover
\begin{equation}\label{LM}
 L(u,\tau ;\eps )= \MM(\tau-\eps^{-1} u, \eps  ),
\end{equation}
where $\MM$ is the $2\pi$-periodic function
\begin{equation*}
\MM(s;\eps)=\int_{-\infty}^{+\infty}H_1\left(q_0(r), p_0(r),\eps \ii r + s
;\eps \right)dr = \sum_{k \neq 0} \MM^{[k]}(\eps)
e^{ik s}
\end{equation*}
which, by \textbf{HP3}, has zero mean. Here  $\MM^{[k]}$ denotes the $k$-Fourier
coefficient of $\MM$.

In \cite{DelshamsS97} (polar case) and \cite{BaldomaF04} (branching point case),
it was seen that
Hypotheses \textbf{HP3} and \textbf{HP4} allow us to give an asymptotic formula
for the Fourier coefficients of $\MM$ and henceforth we will obtain an
asymptotic formula for the functions
$\MM $ and $L$. To state the lemma, we first define the following Fourier
expansion
\[
H_1(q_0(u),p_0(u),\tau;0)=\sum_{k\in\ZZ\setminus\{0\}}H_1^{[k]}(q_0(u),
p_0(u);0)e^{ik\tau}.
\]
Note that, by the definition of $\ell$ in \eqref{def:ell}, all the Fourier
coefficients $H_1^{[k]}(q_0(u),p_0(u);0)$ have at $u=\pm ia$ a branching point
of  order less than or equal to $\ell$.

\begin{lemma}[\cite{DelshamsS97,BaldomaF04}]\label{lemma:Melnikov}
Let us assume Hypotheses \textbf{HP2}, \textbf{HP3} and \textbf{HP4}. Let
\begin{equation*}
 f_0= \frac{\Cte i^{-\ell -1}}{\Gamma(\ell)},
\end{equation*}
where $\Cte$ is the constant defined as
\begin{equation}
 \Cte= \lim_{u \to ia} (u-ia)^\ell H_1^{[1]}(q_0(u), p_0(u);0).
\end{equation}
Then:
\begin{enumerate}
\item The first Fourier coefficients of $\MM$ are given by:
\begin{equation*}
\overline{\MM^{[1]}}= \MM^{[-1]} = -\frac{1}{\eps^{\ell-1}}
e^{-\dps\tfrac{a}{\eps}}\left(f_0 +\OO\left(\eps^{\frac{1}{\q}}\right)\right ).
\end{equation*}
\item If $|k|\neq 1$,
\begin{equation*}
\MM^{[k]} = \OO\left ( \frac{1}{\eps^{\ell-1}} e^{-|k| \frac{a}{\eps}} \right ).
\end{equation*}
\item For $u\in \RR$ and $t\in\RR$,
\[
L\left(u,\frac{t}{\eps};\eps\right)=-\frac{2}{\eps^{\ell-1}}e^{-\dps\tfrac{a}{\eps}}
\left(\Re \left(f_0
e^{-i\left({\dps\tfrac{u-t}{\eps}}\right)}\right)+\OO\left(\eps^{\frac{1}{\q}}
\right)\right),
\]
where $a$ and  $\q$ are the constants defined in Hypothesis \textbf{HP2}.
\end{enumerate}
\end{lemma}

%CANVIAR AIXO
%In the case $\ell <2r$, we still have to assume  a generic  extra hypothesis.
%By the definition of $\ell$ in  \eqref{def:ell}, we already know that
%$H_1(q_0(u),p_0(u),\tau)$ has branching points of degree $\ell$  at $u=\pm ia$.
%Therefore, if we %consider the Fourier expansion
%\[
%H_1(q_0(u),p_0(u),\tau)=\sum_{k\in\ZZ\setminus\{0\}}H_1^{[k]}(q_0(u),p_0(u))e^{
%ik\tau},
%\]
%all the Fourier coefficients $H_1^{[k]}(q_0(u),p_0(u))$ have at $u=\pm ia$ a
%branching point of  order less than or equal to $\ell$.
%Next hypothesis, which will be only used in the case $\ell-2r<0$, requires a
%slightly stronger condition on the Fourier
%coefficients  $H_1^{[\pm 1]}(q_0(u),p_0(u))$, namely, that these coefficients
%also have these  branchings points   $u=\pm ia$ of order $\ell$ and not less.
%\begin{description}
%\item[\textbf{HP6}] If $\ell<2r$, the Fourier coefficients $H_1^{[\pm
%1]}(q_0(u),p_0(u))$ have branching
%points of order \emph{exactly} $\ell$ at $u=\pm ia$. That is we are assuming
%that
%\end{description}

\begin{theorem}[Main Theorem: Regular case]\label{th:MainGeometric:regular}
Let us assume Hypotheses \textbf{HP1}-\textbf{HP5} and $\eta>\ell-2r$. Then,
given  any $\mu_0>0$, there exists $\eps_0>0$ such that for any $\mu \in
\{|\mu|\leq
\mu_0\}$ and $\eps\in(0,\eps_0)$ the area of the lobes between the invariant
manifolds of the periodic orbit given in Theorem \ref{th:MainPO}  is given by,
\begin{itemize}
\item If $\eta>\eta^\ast$,
\begin{equation}\label{def:formulaArea:regular}
\AAA=4|\mu|\eps^{\eta+1-\ell}e^{-{\dps\tfrac{a}{\eps}}}
\left(\left|f_0\right|+\OO\left(\frac{1}{|\ln\eps|^{\nu}}\right)\right),
\end{equation}
where $f_0$ is the constant given in Lemma \ref{lemma:Melnikov}, $\nu=1$ if
$\ell-2r\leq0$ and $\nu=\ell-2r$ if $\ell-2r>0$.
\item If $\eta=0$ (which can only happen if $\ell-2r<0$),
\begin{equation}\label{def:formulaArea:regular:lmenor}
\AAA=4|\mu|\eps^{1-\ell}e^{-{\dps\tfrac{a}{\eps}}}\left(\left|f_0e^{iC(\mu)}
\right|+\OO\left(\frac{1}{|\ln\eps|}\right)\right),
\end{equation}
where $f_0$ is the constant given in Lemma \ref{lemma:Melnikov} and $C(\mu)$ is
an entire analytic function  which satisfies $C(\mu)=\OO(\mu)$.
\end{itemize}
\end{theorem}

Note that if $f_0=0$, this theorem only gives exponentially small upper bounds
for of the area $\AAA$.

\begin{corollary}\label{coro:MainGeometric:regular}
Let us assume the hypotheses of Theorem \ref{th:MainGeometric:regular} and
$f_0\neq 0$, where $f_0$ is the constant given in Lemma \ref{lemma:Melnikov}.
Then, the invariant manifolds intersect transversally and the area of the lobes
of the Poincar\'e map between two consecutive transversal homoclinic points is
asymptotically given by the formulas stated in  Theorem
\ref{th:MainGeometric:regular}.
\end{corollary}

\begin{remark}
In Corollary \ref{coro:MainGeometric:regular} we have asked for the hypothesis
$f_0\neq 0$, which by Lemma \ref{lemma:Melnikov} corresponds to $A\neq 0$. This
condition is equivalent to ask that the Fourier coefficients $H_1^{[\pm
1]}(q_0(u),p_0(u);0)$ have branching points of order \emph{exactly} $\ell$ at
$u=\pm ia$. Note that this hypothesis is generic since it is equivalent to
assume that some coefficient in the Laurent expansions of $H_1^{[\pm
1]}(q_0(u),p_0(u);0)$ at the points $u=\pm ia$ is non-zero.
\end{remark}

\subsubsection{Main result for the singular case}

The case $\ell \geq 2r$ and $\eta=\ell-2r$ is essentially different from the
previous cases in the sense that
we are not able to have ``a priori" estimates for the asymptotic coefficient of
the area of the lobes
between two consecutive homoclinic points. Such asymptotic coefficient depends
on an unknown function
($f(\mu)$ in Theorem \ref{th:MainGeometric:singular}) which comes from the study
of the difference between
adequate approximations of the invariant manifolds near  the singularities $\pm
ia$.
%For that reason we note that, in this case, the analogous condition to
%Hypothesis \textbf{HP6} is $f(\hmu) \neq 0$,
%which is assumed in Theorem \ref{th:MainGeometric:lmajorigual}.

%\marginpar{la funcio $f$ es entera?, cal posar interseccio els reals?}
%\marginpar{la $f$ es entera?}

\begin{theorem}[Main Theorem: singular case]\label{th:MainGeometric:singular}
Let us assume Hypotheses \textbf{HP1}-\textbf{HP5}, $\ell-2r\geq 0$ and
$\eta=\ell-2r$. Then, given any fixed $\mu$,
there exists $\eps_0>0$ such that if $\eps\in (0 ,\eps_0)$, the area of the
lobes between the invariant manifolds of the periodic orbit given in Theorem
\ref{th:MainPO}  is given by
\begin{itemize}
 \item If $\ell-2r>0$,
\begin{equation}\label{def:formulaArea:singular:lmajor}
\AAA=4|\mu|\eps^{1-2r}e^{-{\dps\tfrac{a}{\eps}}}
\left(\left|f\left(\mu\right)\right|+\OO\left(\frac{1}{|\ln\eps|^{\ell-2r}}
\right)\right)
\end{equation}
where $f(\mu)$ is an entire analytic function.
\item If $\ell-2r=0$,
\begin{equation}\label{def:formulaArea:singular:ligual}
\AAA=4|\mu|\eps^{1-2r}e^{-{\dps\tfrac{a}{\eps}}+\mu^2\Im
b\ln\frac{1}{\eps}}\left(\left| f
\left(\mu\right)e^{iC\left(\mu\right)}\right|+\OO\left(\frac{1}{|\ln\eps|}
\right)\right),
\end{equation}
where $b\in \CC$ is a constant, whose explicit expression is given in
\eqref{def:Constantb}, $f(\mu)$ is an entire analytic function and $C(\mu)$ is
an entire analytic function such
that $C(\mu)=\OO(\mu)$.
\end{itemize}
\end{theorem}

\begin{corollary}\label{coro:MainGeometric:singular}
Let us  assume the hypotheses of Theorem \ref{th:MainGeometric:singular} and
$f(\mu)\neq 0$. Then, the invariant manifolds intersect transversally and the
area of the lobes of the Poincar\'e map between two consecutive transversal
homoclinic points is asymptotically given by the formulas of Theorem
\ref{th:MainGeometric:singular}.
\end{corollary}

\subsubsection{Some comments about the results}\label{sec:mainresults:remark}
\begin{itemize}
\item It is important to mention that, by applying Theorems \ref{th:MainGeometric:regular} and \ref{th:MainGeometric:singular},
 we do not need to compute exactly a parameterization
$(q_0(u), p_0(u))$ of the homoclinic orbit in order to know the size of the splitting.
What we need is the behavior of the homoclinic connection around its singularities $\pm ia$, which
as we pointed out in Section \ref{remrkshipotesis}, can be computed explicitly.

\item The constant $b$ appearing in Theorem \ref{th:MainGeometric:singular} can
be
computed explicitly as it is showed in formula \eqref{def:Constantb}
in Proposition \ref{coro:gInner}. In particular, $b=0$ when the
Hamiltonian $H_1$ in \eqref{def:Ham:Original:perturb:poli} and
\eqref{def:Ham:Original:perturb:trig} does not depend on $y$. For
this reason, in the previous results obtained in the singular case corresponding
to $\eta = \ell -2r =0$,
see \cite{Treshev97,Gelfreich00,Olive06,GuardiaOS10}, this term does not appear.
The
appearance of this logarithmic term in the asymptotic formula had
already been detected in \cite{Baldoma06}. Let us also point out
that an analogous phenomenon happens  in the analytic unfoldings of
the Hopf-zero singularity (see \cite{BaldomaS06,BaldomaS08}) and in weak
resonances
of area preserving maps \cite{SimoV10}.

\item The constant $C(\mu)$ appearing in Theorems \ref{th:MainGeometric:regular}
and \ref{th:MainGeometric:singular} also satisfies $C(\mu)=0$ if the
Hamiltonian $H_1$ in \eqref{def:Ham:Original:perturb:poli} and
\eqref{def:Ham:Original:perturb:trig} does not depend on $y$. In Section
\ref{Proof:Existence:C0:lmenor} we give an explicit expression of $C(\mu)$ in
terms of several explicitly computable auxiliary functions.

\item If one weakens Hypothesis \textbf{HP3} to admit Hamiltonian systems
with $\CCC^1$ dependence on $\tau$, one can get analogous results to
the ones obtained in Theorems \ref{th:MainGeometric:regular} and
\ref{th:MainGeometric:singular}.

\item \textbf{Comparison with Melnikov}.
Observe that when $\eta>\eta^\ast$, Theorem \ref{th:MainGeometric:regular} gives
a natural result which generalizes the previous results dealing with the regular
case (see Section \ref{sec:Historical} about historical remarks): if one
artificially assumes that the perturbation is small enough, the splitting of
separatrices is given in first order by the Melnikov function.

If $\ell-2r<0$ and $\eta=0$, the Melnikov function does not predict the area
correctly in general. Nevertheless,
since $C(\mu)\equiv 0$ when the perturbation does not depend on $y$, in this
case Melnikov theory gives
the asymptotic size of the area of the lobes even if $\eta=0$, that is, when the
perturbation has the
same size as the integrable system.

In the singular cases $\ell-2r\geq0$ and $\eta=\eta^\ast=\ell-2r$, we know that
the function $f(\mu)$ appearing in Theorem
\ref{th:MainGeometric:singular},  satisfies that for $\mu$ small
\[
f(\mu)=f_0+\OO(\mu),
\]
where $f_0\in \CC$ is a constant independent of $\mu$. In
\cite{Baldoma06}, it is seen that the constant $f_0$ coincides with
the constant  that Melnikov theory gives in front of the exponential
term (see Lemma \ref{lemma:Melnikov}).

In other words, this means that for the case $\ell-2r>0$, if $\mu$ is a small
parameter and $f_0\neq 0$, Melnikov
theory also predicts the asymptotic behavior of the area of the
lobes correctly.

In the case $\ell-2r=0$, $f_0$ also corresponds to the Melnikov
theory prediction. Nevertheless, since a logarithmic term appears in the
exponential,
the Melnikov prediction is valid provided
\[
|\mu | \ll     \frac{1}{\sqrt{|\ln\eps|}}.
\]
Of course, if $b=0$, as happens when the perturbation does not
depend on $y$, the Melnikov prediction is valid for any $\mu$ small
and independent of $\eps$.
\end{itemize}

\subsubsection{Examples}\label{sec:examples}
In this section we apply Theorems \ref{th:MainGeometric:regular} and
\ref{th:MainGeometric:singular} to some examples.
We consider the Duffing equation
\[
 H_0(x,y)=\frac{y^2}{2}-\frac{x^2}{2}+\frac{x^4}{4}
\]
with different perturbations. The Duffing equation has two separatrices forming
a figure eight, which are parameterized by
\[
 \Gamma^\pm (u)=(\pm q_0(u),p_0(u))=\left(\pm \frac{\sqrt{2}}{\cosh u},\mp
\frac{\sqrt{2}\sinh u}{\cosh^2u}\right).
\]
The singularities of these separatrices which are closer to the real axis are
$u=\pm i\pi/2$ and $r=2$ (see the definition of $r$ in Hypothesis \textbf{HP2}).

We consider two different types of perturbations and we study how the separatrix
$\Gamma^+$ splits.
The first perturbation is
\[
 H(x,y)=\frac{y^2}{2}-\frac{x^2}{2}+\frac{x^4}{4}+\mu\eps^\eta
x^n\sin\frac{t}{\eps}
\]
for $n\in\NN$ and $\eta\ge 0$.
Then the order of the perturbation is $\ell=n$ (see the definition of $\ell$ in
\eqref{def:ell}).

Applying Melnikov theory to these Hamiltonian systems, one obtains the
following prediction for the area of the lobes
\begin{equation}\label{def:Duffing:Melnikov:1}
 \AAA=|\mu|\eps^\eta  \frac{2^{\frac{n}{2}+2}\pi}{(n-1)!\,\eps^{n-1}}e^{-{\dps
\tfrac{\pi}{2\eps}}}+\OO\left(\mu^2\eps^{2\eta}\right).
\end{equation}

For $\eta>\eta^*=\max\{n-4,0\}$ or $\eta=0$ and $n<4$ (which corresponds to
$\ell-2r<0$), one can apply Theorem \ref{th:MainGeometric:regular}  to see that
Melnikov theory predicts correctly the area of the lobes. Note that
$C(\mu)\equiv 0$ since the perturbation does not depend on $y$. Then,
\begin{equation} \label{eq:Duffing:FirstOrderArea:Melnikov}
 \AAA\simeq|\mu|\eps^\eta
\frac{2^{\frac{n}{2}+2}\pi}{(n-1)!\,\eps^{n-1}}e^{-{\dps \tfrac{\pi}{2\eps}}}.
\end{equation}

The case  $n\geq 4$ corresponds to $\ell\geq 2r$. In this case for
$\eta=\eta^*=n-4$, since the perturbation does not depend on $y$,
we have that $b=0$ and $C(\mu)\equiv 0$. Then, applying Theorem
\ref{th:MainGeometric:singular}, the area is given by the formula
\begin{equation}\label{def:Duffing:FirstOrderArea:1}
 \AAA=|\mu|  \frac{4\left|f\left(\mu\right)\right|}{\eps^{n-1}}e^{-{\dps
\tfrac{\pi}{2\eps}}}\left(1+\OO\left(\frac{1}{|\ln\eps|}\right)\right),
\end{equation}
where $f(\mu)$ satisfies
\begin{equation}\label{eq:Duffing:Expansio:f}
 f(\mu)=\frac{2^{\frac{n}{2}}\pi i}{(n-1)!}+\OO\left(\mu\right).
\end{equation}
Therefore, for $\eta=n-4$ and fixed $\mu$ independent of $\eps$, the first order
depends on the full
jet of $f(\mu)$ and then the Melnikov function does not predict it correctly.

To see how the first asymptotic order of the area of the lobes changes when the
perturbation depends on $y$, we consider the following perturbation of the
Duffing equation, where $\ell=2r=4$ and $\eta=\ell-2r=0$,
\[
 H(x,y)=\frac{y^2}{2}-\frac{x^2}{2}+\frac{x^4}{4}+\mu \left(
x^4\sin\frac{t}{\eps}+\lambda x^2y\cos\frac{t}{\eps} \right)
\]
with $\lambda\in \RR$. For this example, Melnikov theory predicts that the area
of the lobes is
\[
 \AAA=|\mu|\frac{4\pi}{3\eps^3}|2+\sqrt{2}\lambda|e^{-{\dps
\tfrac{\pi}{2\eps}}}+\OO\left(\mu^2\right).
\]
Note that if one takes $\lambda=0$, $\AAA$ coincides with
\eqref{def:Duffing:Melnikov:1} with $n=4$ and $\eta=0$.
On the other hand, if one takes $\lambda=-\sqrt{2}$ the Melnikov function is
degenerate since the first order vanishes.

Since $\ell=2r$ and $\eta=0$,  one can apply Theorem
\ref{th:MainGeometric:singular}.
Using formula \eqref{def:Constantb} for the definition of $b$, one can easily
see that $b=-4\sqrt{2}\lambda i$. Therefore, the true first asymptotic order of
the area of the lobes is given by
\begin{equation}\label{def:Duffing:FirstOrderArea:2}
 \AAA=|\mu|\frac{4}{\eps^3}e^{-{\dps
\tfrac{\pi}{2\eps}}-4\sqrt{2}\lambda\mu^2\ln\frac{1}{\eps}}\left(\left|f(\mu)e^{
iC(\mu)}\right|+\OO\left(\frac{1}{|\ln\eps|}\right)\right),
\end{equation}
where $f(\mu)$ satisfies
\[
 f(\mu)=\frac{\pi i }{3}\left(2+\sqrt{2}\lambda\right)+\OO\left(\mu\right).
\]
One can take, for instance, $\mu=1$ and write formula
\eqref{def:Duffing:FirstOrderArea:2} as
\[
  \AAA=\frac{4}{\eps^{3-4\sqrt{2}\lambda}}e^{-{\dps
\tfrac{\pi}{2\eps}}}\left(\left|f(1)e^{iC(1)}\right|+\OO\left(\frac{1}{|\ln\eps|
}\right)\right).
\]
Therefore, the correcting logarithmic term in the exponential implies a drastic
change in the power  of $\eps$ in  the asymptotics. 
Note that one can take any $\lambda\in \RR$ and then the
power of $\eps$ in the first order can  change
arbitrarily, both increasing or decreasing. Finally, if one takes
$\lambda=0$, one recovers formula
\eqref{def:Duffing:FirstOrderArea:1}.

\subsection{Near integrable Hamiltonian systems of $1\tfrac{1}{2}$ degrees of
freedom close to a resonance}\label{sec:mainresult:resonance}
The results obtained in this work can be easily  adapted to study
near integrable Hamiltonian systems of $1\tfrac{1}{2}$ degrees
of freedom close to a resonance. Let us consider an analytic
Hamiltonian system with Hamiltonian
\begin{equation}\label{def:HamResonant}
h(x,I,\tau)=h_0(I)+\de h_1(x,I,\tau),
\end{equation}
where $\de\ll 1$ is a small parameter, $(x,\tau)\in\TT^2$,
$I\in\RR$ and $h_1$ is a trigonometric polynomial as a function of $x$.
When $\de=0$, the Hamiltonian system is completely integrable (in
the sense of Liouville-Arnold) and the phase space is foliated by
invariant tori with frequency $\omega(I)=(\pa_I h_0(I),1)$.

In particular, if for certain $I$, there exists $k\in\ZZ^2$ such that
$\omega(I)\cdot k=0$, the corresponding torus is foliated by
periodic orbits. When $\de>0$ (but small enough), it is a well known
fact that typically this torus, a resonant torus, breaks down.

Let us consider the simplest setting and let us assume that
\[
h_0(I)=\frac{I^2}{2}+G(I)\qquad \text{ with }
G(I)=\OO\left(I^3\right).
\]
Then $I=0$ corresponds to the resonant vector $\omega(0)=(0,1)$. To
study the dynamics of the perturbed system around this resonance,
one usually performs the rescaling
\[
I=\sqrt{\de} y\quad\text{ and }\quad \tau=\frac{t}{\sqrt{\de}}
\]
and takes  $\eps=\sqrt{\de}$ as a new parameter. Then, one obtains
the Hamiltonian
\[
H(x,y,t)=\frac{y^2}{2}+\frac{1}{\eps^2}G(\eps
y)+V(x)+F\left(x,\frac{t}{\eps}\right)+R\left(x,\eps
y,\frac{t}{\eps}\right),
\]
where
\[
\begin{split}
\dps V(x)&=\langle h_1(x,0,\tau)\rangle=\frac{1}{2\pi}\int_0^{2\pi}
h_1(x,0,\tau)\,d\tau\\
\dps F(x,\tau)&= h_1(x,0,\tau)-\langle h_1(x,0,\tau)\rangle\\
\dps R(x,I,\tau)&=h_1(x,I,\tau)-h_1(x,0,\tau),
\end{split}
\]
which can be written as
\[
H\left(x,y,\frac{t}{\eps}\right)=H_0(x,y)+\mu
H_1\left(x,y,\frac{t}{\eps},\eps\right)
\]
with
\[
\begin{split}
H_0(x,y)&=\frac{y^2}{2}+V(x)\\
H_1(x,y,\tau,\eps)&=F(x,\tau)+\frac{1}{\eps^2}G(\eps y)+R(x,\eps y,
\tau).
\end{split}
\]
Here $\mu$ is in fact a fake parameter, since we are interested in
the case $\mu=1$. This system  is similar to the ones considered in
this paper. Let us point out also that, by definition,
$\eps^{-2}G(\eps y)$ and $R(x,\eps y,\tau)$ are of order $\eps$.

Let us assume that the Hamiltonian $H$ satisfies Hypotheses
\textbf{HP1}-\textbf{HP4} and instead of \textbf{HP5} satisfies the
alternative hypothesis that $V$, which is a trigonometric polynomial,
has the same degree as $h_1$ in \eqref{def:HamResonant} as a
function of $x$. Then, using the tools considered in this paper,
one can  give an asymptotic formula analogous to the one given
in Theorem \ref{th:MainGeometric:singular}. Let us point out that in this
setting, even if the terms $\eps^{-2}G(\eps y)$ and $R(x,\eps
y,\tau)$ are of order $\eps$ and therefore smaller than $F(x,\tau)$,
the function $f(\mu)$ appearing in Theorem \ref{th:MainGeometric:singular}
depends not only on $F$ but also on the full jet in $y$ of $G$ and
$R$. The reason is that  these terms become of the same order as $V(x)$ and
$F(x,\tau)$ close to the singularities of the unperturbed
separatrix. Moreover,  for these systems,
the first asymptotic order also has the logarithmic term in the
exponential as it happens in Theorem \ref{th:MainGeometric:singular} for
$\ell-2r=0$. We plan to study rigorously these kind of systems in
future work.

%\subsection{Structure of the paper}
%We devote the rest of the paper to prove Theorem ??. First, in
%Section \label{sec:Unperturbed} we state several properties of the
%unperturbed system and its separatrix.

%Section \label{sec:parameterizations} is devoted to state the
%existence of a $2\pi$-periodic orbit close to the critical point of
%the unperturbed system. In this Section we also explain the two
%different parameterizations of the invariant manifolds of the
%periodic orbit that we are going to use.

%Next step is to prove the existence of the invariant manifolds and
%study their difference. Then in Section
%\ref{sec:SketchProofPolinomi} we give the description of the proof
%of Theorem ??. Nevertheless, since the proof in the trigonometric
%case can be done in a considerably simpler way, first in Section
%\ref{sec:SketchProofTrigonometric}, we describe this simpler proof.

%The rest of the sections are devoted to prove the partial results
%stated in Sections \ref{sec:SketchProofTrigonometric} and
%\ref{sec:SketchProofPolinomi}.

\section{Heuristic ideas of the proof} \label{sec:Heuristic}

The rigorous proofs  of asymptotic formulas for measuring the splitting of
separatrices require a significant amount of technicalities. For the convenience
of the reader, even though in Section \ref{sec:SketchProof} we give a precise
description of the entire proof of
Theorems \ref{th:MainGeometric:regular} and \ref{th:MainGeometric:singular}, we
first devote this
section to give an heuristic description of our strategy explaining the main
differences respect to the ones already
used in the literature. We also explain
the main novelties we have  introduced to overcome the difficulties that our
general setting involves.

\subsection{Measuring the splitting by using generating
functions}\label{sec:heuristic:splitting}
To measure the splitting using generating functions we use the method in
\cite{LochakMS03, Sauzin01}, based on ideas by Poincar\'e \cite{Poincare99}.
Roughly speaking, if the invariant manifolds can
be expressed in a suitable way, then the area of the lobes  generated by the
perturbed manifolds between two consecutive homoclinic points and also the
distance between the manifolds
can be simply computed by the difference between two functions.

Let us  explain this approach in more detail. 
%Assume that we are in the Hamiltonian setting given in Section \ref{sec:HypsAndMainResults}.
As the main goal is to measure the distance of the stable and unstable manifolds
of
the periodic orbit $(x_p(t/\eps), y_p(t/\eps))$ in a Poincar\'{e} section
$\Sigma _{t_0}$,
it is useful to obtain these manifolds as graphs. The stable and unstable
manifolds of the perturbed system can be expressed as
graphs as
\[
y = \varphi(x,t/\eps) = y_p(t/\eps) + \pa_x S^{s,u}(x - x_p(t/\eps),t/\eps)
\]
in some complex domains, where the functions  $S^{s,u}$ are called generating
functions.
%As it is pointed out in \cite{Sauzin01} 
The generating functions
$S^{s,u}(q,\tau)$ are solutions of the Hamilton-Jacobi
equation associated to our Hamiltonian system after the change of variables
\[
q=x-x_p (t/\eps), \quad p=y-y_p (t/\eps)
\]
and the change of time $\tau =t/\eps$.

Note that for $\mu=0$, as the Hamiltonian is autonomous, the Hamilton-Jacobi equation reads:
\[
\frac{(\pa _q S(q))^2}{2} + V(q)=0
\]
which gives  $\pa _q S^{s}(q,\tau)= \pa _q S^{u}(q,\tau)=\pa _q S_0(q)=\sqrt{-2V(q)}$ as the homoclinic connection.

Then, to measure the distance between the stable and the unstable manifolds in a
Poincar\'{e} section
we just need to compute:
\begin{equation}\label{distancia0}
d(q,t_0 )=\pa_q S^{u}(q,t_0/\eps)-\pa_q S^{s} (q,t_0/\eps)
\end{equation}
and it is standard that the area of the lobes is given by
\begin{equation}\label{area0}
\AAA= S^{u}(q_2,t_0/\eps)-S^{s}(q_2,t_0/\eps)- \left(S^{u}(q_1
,t_0/\eps)-S^{s}(q_1 ,t_0/\eps)\right),
\end{equation}
where $q_1$, $q_2$ are the coordinates of two consecutive homoclinic points in
the section $\Sigma_{t_0}$. Note that, thanks to the symplectic structure,
$\AAA$ does not depend on $t_0$.

We perform the change of variables
$q= q_0(u)$, where $q_0(u)$ is the first component of the unperturbed homoclinic
orbit. In this way, we work with the function
\[
T^{u,s}(u,\tau)= S^{u,s}(q_0(u), \tau)
\]
that is, we write the perturbed manifolds as functions of the time $\tau$ and the
``time over the homoclinic orbit" $u$, which
parameterizes the unperturbed homoclinic orbit. These functions satisfy a new
Hamilton-Jacobi equation,  which is easier to deal with.

We consider the difference
\begin{equation*}
\Delta(u,\tau)=T^u(u,\tau)-T^s(u,\tau).
\end{equation*}

The first observation is that, when $\mu=0$, we have 
$p_0(u)= \pa _q S_0(q_0(u))$. Therefore
$\pa_uT^{u,s}(u,\tau)= \pa_uT_0 (u)
=p_0(u) \partial_q S^{u,s}(q_0(u),\tau)= (p_0(u))^2$ which corresponds to the
parameterization of the unperturbed separatrix. Then, by analyticity with
respect to
the regular parameter $\mu$, we have that $\Delta(u,\tau)= \OO(\mu)$.

The second observation is that, as the experts in this area know,
$\Delta(u,\tau)$ is exponentially small in the
singular parameter $\eps$. To obtain sharp estimates of  $\Delta(u,\tau)$, we
need to bound it, and
consequently $T^{u}(u,\tau)$ and $T^{s}(u,\tau)$, in a region of the complex
plane that,
on one hand, contains a segment of the real line having two values of $u$ giving
rise to two consecutive homoclinic points
and, on the other hand,  intersects a neighborhood sufficiently close to the
singularities $\pm ia$ of $T_0(u)$.

Assume that we can construct parameterizations $T^{u,s}(u,\tau)$ of the
perturbed invariant manifolds satisfying both that  they are $2\pi$-periodic
with respect to $\tau$ and that they are real-analytic and bounded in some
complex
domain which contains two real values of $u$ which give rise to two consecutive
homoclinic points.
Now we are going to explain how an exponentially small upper bound of the
difference $\Delta$ can be derived.
The first point is that, being $T^{u}$ and $T^{s}$
solutions of the same partial differential equation (with different boundary
conditions),
$\Delta(u,\tau)$ satisfies a homogeneous linear partial differential
equation. One can see that this equation is conjugated to
$(\eps \partial _u + \partial_\tau ) Y(u,\tau)=0$. Let us assume for a moment
that $\Delta$  is a
solution of this equation. In fact, in Theorems \ref{th:CanviFinal:lmenor} and
\ref{th:CanviFinal},
we will see that this is true after a suitable change of variables.
Then,  we obtain that
$\Delta(u,\tau)= \Lambda (\tau -u/ \eps )$ and, since $\Delta$ is $2\pi$-periodic in $\tau$,  
$\Lambda (s)$ is a $2\pi$-periodic function in $s$.
This fact implies that
\[
\Delta(u,\tau)=\sum _{k\in\ZZ}\Lambda _k e^{-i k \frac{u}{\eps}} e^{i k \tau}.
\]
Now, a bound $|\Delta(u,\tau)|\le M$ for $|\Im u |\le a'$, automatically gives
\[
|\Lambda _k| \le M e^{- |k| \frac{a'}{\eps}},  \quad k \ne 0
\]
which implies that $|\Delta(u,\tau)-\Lambda _0|\le 4 M e^{- \frac{a'}{\eps}}$ for
real values of $u$.
The bigger the size of the strip where we can bound $|\Delta(u,\tau)|$ the
smaller the exponential that
gives the bound  for real values of $u$.
Note that the constant $\Lambda_0$ does not appear neither in the formula of the
area \eqref{area0}, nor in the formula of the distance \eqref{distancia0}
If we use Melnikov theory the expected exponential exponent is $a$, where $\pm a
 i$ are the
singularities of $T_0$. Then,  to obtain sharp bounds, it would be enough to
take $a'= a-\eps$.

In some cases, which correspond to $\eta=0$ in \eqref{eq:model}, the change of
variables which conjugates the original partial differential equation for
$\Delta(u,\tau)$ with $(\eps\pa_u+\pa_\tau)Y(u,\tau)=0$ is not close enough to
the identity.
This fact implies the appearance of the constant $C(\mu)$ and the logarithmic
term in the asymptotic formulas obtained in Theorems
\ref{th:MainGeometric:regular} and \ref{th:MainGeometric:singular}.
This change of variables is obtained, essentially, studying the variational
equation along the perturbed invariant manifolds. Therefore, the existence of
these terms, which were not present in the Melnikov prediction, shows that, to
study the exponentially small splitting of separatrices, it is not enough to
look for the first order approximations of the invariant manifolds close to the
singularities. One has  to look also for the first order of certain solutions of
the variational equation of the perturbed invariant manifolds close to the
singularities. In fact, these terms appear when these certain solutions of the
variational equation of the perturbed invariant manifolds close to the
singularities are not well approximated by the solutions of the variational
equation of the unperturbed separatrix.

Then, roughly speaking one can conclude that Melnikov theory gives the correct
answer if:
\begin{itemize}
 \item The perturbed invariant manifolds are well approximated by the
unperturbed separatrix close to the singularity.
\item The solutions of the variational equation along the perturbed invariant
manifold are well approximated by certain solutions of the  variational equation
along the unperturbed separatrix.
\end{itemize}
In all the other cases, the splitting is given by an alternative formula. This
fact, is explained in more  detail Section
\ref{sec:heuristic:asymptotic}.

\subsection{The \emph{boomerang domains}}
For the  Hamiltonians considered in this paper, the  invariant manifolds, in
general, are not global graphs over
$q$.
Therefore, the approach explained in the previous section cannot be used
straightforwardly. Nevertheless, we will see that there are always regions in
the phase space where both manifolds are graphs and we will use one of these
regions to measure the splitting.
Consequently, being the area of the lobes  an invariant quantity, this will give
the wanted result.

As we have explained, we are forced to find parameterizations $T^{u,s}$ of the
invariant manifolds which have to be
analytic in a common complex domain which reaches points at a distance $\eps$ of
the singularities.
Moreover we also need to guarantee that our domain contains an open set of real
values of $u$ (this will be enough to ensure that the domain contains $u_1$ and
$u_2$ that give rise to homoclinic points since they are $\eps$ close).

To this end let us observe that we have no hope to construct parameterizations
$T^{u,s}(u,\tau)$
for values of $u$ such that $p_0(u) = 0$, at least in a general case.
In fact, the unperturbed homoclinic connection can be expressed as
$\text{graph}\{ p=\sqrt{-2 V(q)}\} \cup \text{graph}\{p=-\sqrt{-2V(q)}\}$.
Then if $p_0(u_0)=0$, for some value $u_0$, the unperturbed homoclinic
connection cannot be expressed as a graph over the base
in the original variables $(q,p)$ in a neighborhood of $(q_0(u_0), 0)$. This
fact implies that the Hamilton-Jacobi equation
that $T^{u,s}$ has to satisfy is not defined for $u=u_0$.

We will always keep in mind that we need to check this condition ($p_0(u)\neq
0$) if we want to use the parameterizations $T^{u,s}$.

For this reason we define the following
\emph{boomerang domains} (see Figure \ref{fig:BoomerangDomains}), in
which $p_0(u)\neq 0$, and hence the functions $T^{s,u}$ will be well defined on
them.
\begin{equation}\label{def:DominisRaros}
\begin{split}
D^{s}_{\kk,d}=&\left\{u\in\CC;\right.\left. |\Im u|<\tan\beta_1\Re
u+a-\kk\eps,
|\Im u|<\tan\beta_2\Re u+a-\kk\eps,\right.\\
& \left.|\Im u|>\tan\beta_2\Re u+a-d\right\}\\
D^u_{\kk,d}=&\left\{u\in\CC;\right.\left.|\Im u|<-\tan\beta_1\Re
u+a-\kk\eps,
|\Im u|<\tan\beta_2\Re u+a-\kk\eps,\right.\\
& \left.|\Im u|>\tan\beta_2\Re u+a-d\right\}\\
&\cup \left\{ u\in\CC; \right.|\Im u|<-\tan\beta_1\Re
u+a-\kk\eps, |\Im u|>-\tan\beta_2\Re u+a-d,\\
& \left.\Re u<0\right\},
\end{split}
\end{equation}
where $\beta_1\in (0,\pi/2)$ is any fixed angle.

To choose $\beta_2$ we use the following.
First we point out that the zeros of $p_0(u)$ are isolated in
$\CC$. Moreover, close to the singularities $u=\pm ia$, $p_0(u)$ can
not vanish. Then, in order to assure that $p_0(u)$ does not vanish
in the whole domains $D^s_{\kk,d}$ and $D^u_{\kk,d}$, one has to
choose an angle $\beta_2$ such that $\beta_2>\beta_1$
%has a positive lower bound independent of $\eps$ and $\mu$  and 
and the lines $|\Im
u|=\tan\beta_2\Re u+a$ do not contain any zero of $p_0(u)$. Then,
taking $\eps>0$ and $d>0$ independent of $\eps$, both small enough,
one can guarantee that $p_0(u)$ does not vanish neither in
$D^s_{\kk,d}$ nor in $D^u_{\kk,d}$.

We will use these \emph{boomerang domains} as fundamental domains to measure the
splitting.
It is important to emphasize that both $D^{s}_{\kk,d}$ and $D^{u}_{\kk,d}$ reach
a neighborhood of
the singularities $\pm ia$ of size $\eps$.

\begin{figure}[H]
\begin{center}
\psfrag{u1}{$u_1$}\psfrag{u2}{$\bar u_1$}
\psfrag{b1}{$\beta_1$}\psfrag{b2}{$\beta_2$}\psfrag{a1}{$ia$}
\psfrag{a2}{$-ia$}\psfrag{a3}{$i(a-d)$}\psfrag{a4}{$i(a-\kk\eps)$}
\psfrag{D1}{$D^{\out,s}_{\rr,
\kk}$}\psfrag{D}{$D^{s}_{\kk,d}$}\psfrag{D4}{$D^{\out,u}_{\rr,\kk}$}\psfrag{D3}{
$D^{u}_{\kk,d}$}
\includegraphics[height=6cm]{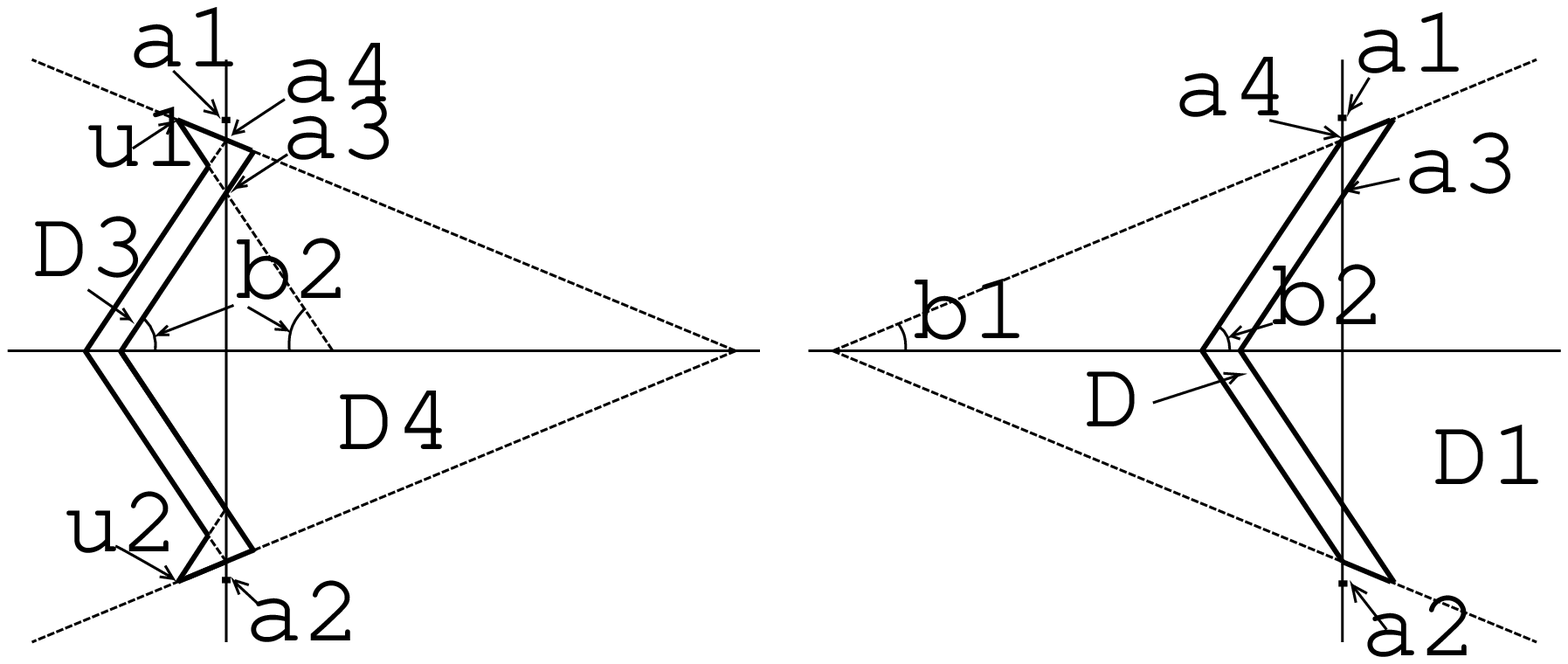}
\end{center}
\caption{\figlabel{fig:BoomerangDomains} The \emph{boomerang
domains} $D^{u}_{\kk,d}$ and $D^{s}_{\kk,d}$ defined in
\eqref{def:DominisRaros}.}
\end{figure}

\begin{remark}
Let us observe that the domains $D^u_{\kk,d}$ and $D^s_{\kk,d}$ have different
shape.
We will give all the proofs in the unstable case. All
of them are analogous, and even simpler, in the stable one.
\end{remark}

To study the difference between the manifolds, we consider
$\Delta(u,\tau)=T^u(u,\tau)-T^s(u,\tau)$
in the domain
$R_{\kk,d}=D^s_{\kk,d}\cap D^u_{\kk,d}$
which is defined as
\begin{equation}\label{def:DominiRaro:Interseccio}
\begin{split}
R_{\kk,d}=\left\{u\in\CC; \right.&\left. |\Im u|<\tan\beta_2\Re
u+a-\kk\eps, |\Im u|>\tan\beta_2\Re u+a-d , \right. \\
&\left.|\Im u|<-\tan \beta_1\Re u+a-\kk\eps\right\}.
\end{split}
\end{equation}

%We define also the intersection of the inner domains and $R_{\kk,d}$
%as
%\begin{equation}\label{def:DominiRaro:Interseccio:Inner}
%\begin{split}
%R_{\kk,d}^{\inn,+}&=R_{\kk,d}\cap D^{\inn,+,u}_{\kk,\C_0}\cap
%D^{\inn,+,s}_{\kk,\C_0}\\
%R_{\kk,d}^{\inn,-}&=R_{\kk,d}\cap D^{\inn,-,u}_{\kk,\C_0}\cap
%D^{\inn,-,s}_{\kk,\C_0}\\
%\end{split}
%\end{equation}
%where $\C_0>0$ is the constant given in Theorem \ref{th:MatchingHJ}.
We recall that $p_0(u)\neq 0$ if $u\in R_{\kk,d}$ and hence we can
use the functions $T^{s,u}$ in this domain.

\begin{figure}[H]
\begin{center}
\psfrag{D6}{$\wt D^{\out,s}_{\rr,d,\kk}$} \psfrag{D5}{$\wt
D^{\out,u}_{\rr,d,\kk}$}
\psfrag{b1}{$\beta_1$}\psfrag{b2}{$\beta_2$}
\psfrag{a1}{$ia$}\psfrag{a2}{$-ia$}\psfrag{a3}{$i(a-d)$}\psfrag{a4}{
$i(a-\kk\eps)$}
\psfrag{D3}{$R_{\kk,d}$}\psfrag{D}{$D^{s}_{\kk,d}$}\psfrag{D1}{$D^{u}_{\kk,d}$}
\includegraphics[height=6cm]{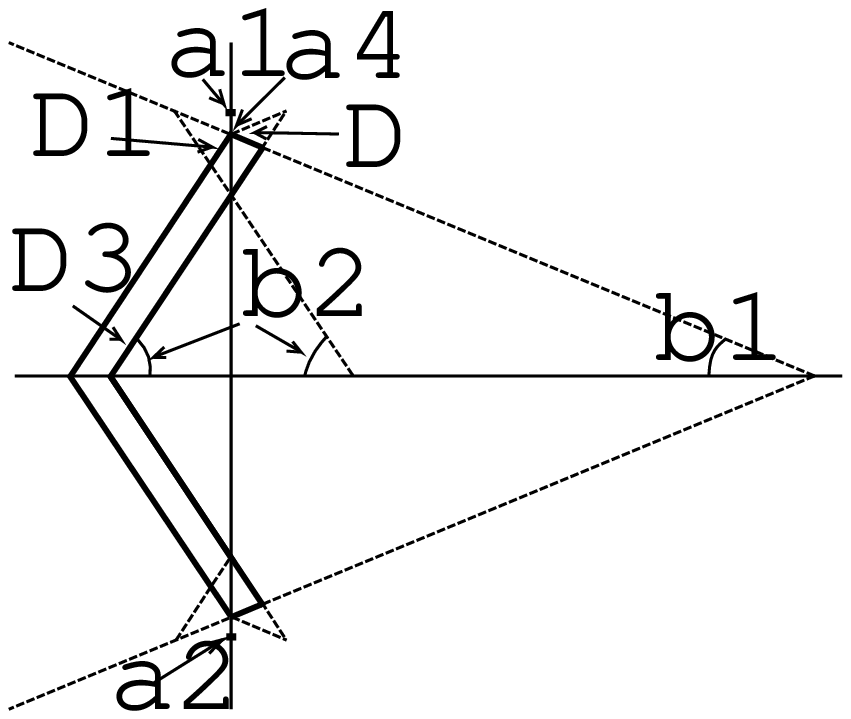}
\end{center}
\caption{\figlabel{fig:BoomInter} The domain $R_{\kk,d}$  defined in
\eqref{def:DominiRaro:Interseccio}.}
\end{figure}

The domain $R_{\kk,d}$, where we measure the difference between the invariant
manifolds,
is considerably different from  the ones used in previous works (see for
instance \cite{Sauzin01}), where the analogous domains  look like diamonds. In
\cite{Sauzin01}, the author considers systems for which the unperturbed
separatrix is a graph globally and then he can work in such wide domains.
%Nevertheless, we have to restrict ourselves to this narrower \emph{fundamental
%domain}  $R_{\kk,d}$ to ensure that $p_0(u)\neq 0$ on it.

Once we have the difference $\Delta $ in $R_{\kk,d}$, using the arguments
exposed in the previous subsection
one can obtain exponentially small upper bounds for $\Delta$.

Recall that our goal is to give an asymptotic formula for the area of the lobe
between two consecutive homoclinic points. Henceforth, once we find the first
asymptotic term of $\Delta$, which we call $\Delta_0$, we use the
arguments indicated in the previous section to bound the difference $\Delta
(u,\tau) - \Delta_0(u,\tau)$.
We will come back to the problem of finding $\Delta_0$ in Section
\ref{sec:heuristic:asymptotic}.

\subsection{Parameterizations of the invariant manifolds of the perturbed
system} \label{sec:heuristic:parametrization}
In this section we are going to explain the strategy we use to prove the
existence of $T^{u,s}$
in the corresponding boomerang domains $D^{u,s}_{\kk,d}$.
In fact we will always deal with $\partial_u T^{u,s}$.

We begin our construction near the origin $(q,p)=(0,0)$. In terms of the new
variable
$u$ this corresponds to take $\Re u$ near $-\infty$ for the unstable invariant
manifold and near $+\infty$ for the stable one.

Given $\rr_1\ge 0$, we consider the following domains:
\begin{equation}\label{def:DominsInfinit}
\begin{array}{l}
D^{u}_{\infty,\rr_1}=\{u\in\CC; \Re u<-\rr_1\}\\
D^{s}_{\infty,\rr_1}=\{u\in\CC; \Re u>\rr_1\}.
\end{array}
\end{equation}
It is not difficult to prove that  the constant $\rr_1$
can be taken big enough so that $p_0(u)$ does not vanish in these domains.
Henceforth the
Hamilton-Jacobi formulation is allowed in these domains (see
\eqref{eq:separatrix:hyp:infty} and
\eqref{eq:separatrix:hyp:parab}).
The first result is Theorem \ref{th:ExistenceCloseInfty}, where we prove the
existence of $\partial_ u T^{s,u}$
and we see that both are well approximated by $\partial_u T_0$ in
$D^{u,s}_{\infty,\rr_1}$.
This result gives the existence of local invariant manifolds and, moreover,
provides suitable properties of them.

In the case that $p_0(u) \ne 0$ the next step is to extend $\partial_u T^{u,s}$
to the
so-called \emph{outer domains} (see Figure \ref{fig:OuterDomains}) defined by
\begin{equation}\label{def:DominisOuter}
\begin{array}{l}
\dps D^{\out,u}_{\rr,\kk}=\left\{u\in\CC; |\Im u|<-\tan \beta_1\Re u+a-\kk\eps,
\Re u>-\rr\right\}\\
\dps D^{\out,s}_{\rr,\kk}=\left\{u\in\CC; -u\in
D^{\out,u}_{\rr,\kk}\right\},
\end{array}
\end{equation}
where $\kappa >0$, which might depend on $\eps$,
is such that $a-\kappa \eps >0$.
The constant $\rr$ will be taken $\rr>\rr_1$, in order to
ensure that $D^{*}_{\infty,\rr_1}\cap D^{\out,*}_{\rr,\kk}\neq
\emptyset$ for $*=u,s$.
Since  we have already proved the existence of local invariant manifolds defined
in $D^{u,s}_{\infty,\rr_1}$,
therefore $\partial_u T^{u,s}$ are defined in $D^*_{\infty,\rr_1}\cap
D^{\out,*}_{\rr,\kk}$ for $*=u,s$.

\begin{figure}[H]
\begin{center}
\psfrag{d}{$\beta_1$}\psfrag{a}{$ia$}\psfrag{b}{$-ia$}
\psfrag{D2}{$D^{\out,u}_{\rr,\kk}$}\psfrag{D3}{$D^{\out,s}_{\rr,\kk}$}\psfrag{D}
{$D^{u}_{\infty,\rr_1}$}
\psfrag{D1}{$D^{s}_{\infty,\rr_1}$}\psfrag{r}{-$\rr$}\psfrag{rr}{-$\rr_1$}
\psfrag{r1}{$\rr$}\psfrag{rr1}{$\rr_1$}
\psfrag{s}{$i(a-\kk\eps)$}\psfrag{u}{$u_1$}\psfrag{v}{$\ol
u_1$}
\includegraphics[height=6cm]{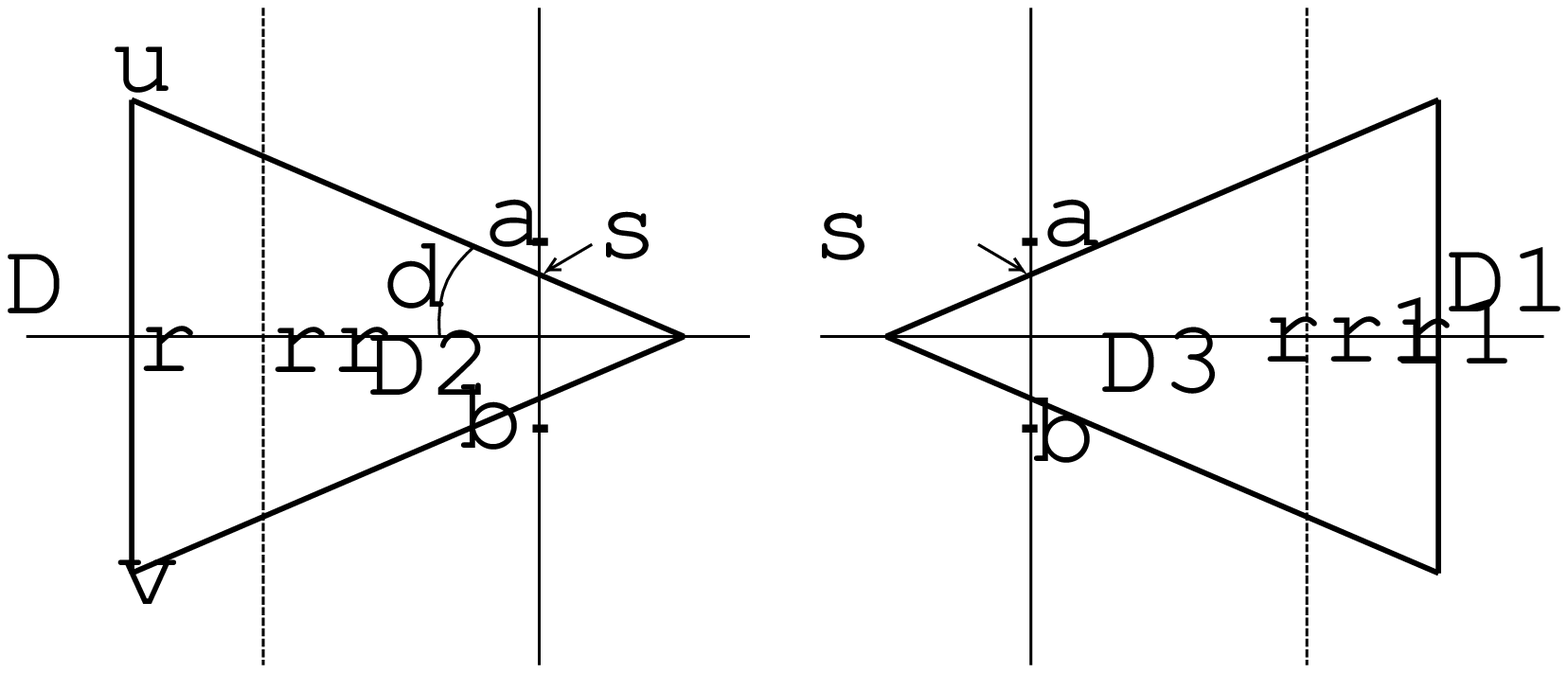}
\end{center}
\caption{\figlabel{fig:OuterDomains} The outer domains
$D^{\out,u}_{\rr,\kk}$ and $D^{\out,s}_{\rr,\kk}$ defined in
\eqref{def:DominisOuter}.}
\end{figure}

In Theorem \ref{th:Extensio:Trig} it is proved that $\partial_u T^{u,s}(u,\tau)$
can be extended
to the \emph{outer domain} $\dps D^{\out,*}_{\rr,\kk}$, $*=u,s$, and that is
well approximated (in some norm) by $\pa_u T_0(u)$ there.

In the case that $p_0(u)$ vanishes in the outer domains the procedure becomes
a little technical.
The main idea is to use  parameterizations of the invariant manifolds of the
form $(Q(u,\tau),P(u,\tau))$
%(introduced by Lazutkin in \cite{Lazutkin84})
to extend them to a new domain where $p_0(u)$ does not vanish anymore and
that
overlaps with the \emph{boomerang domain} $D^{u,s}_{\kk,d}$ (see Theorem
\ref{th:Extensio}).
We point out that these new domains are still far away
from the singularities $\pm ia$ of $T_0(u)$, henceforth the obtention of the parameterizations
defined in these domains is straightforward (see Theorem \ref{th:ParamtoHJ}).
Once we have proved the existence of the parameterizations of the  invariant
manifolds for values of $u$ far from the singularities but inside the
\emph{boomerang domains} $D^{u,s}_{\kk,d}$, we can recover the generating
functions $\partial_u T^{u,s}(u,\tau)$ and extend them to the whole
\emph{boomerang domains} $D^{u,s}_{\kk,d}$ in Theorem \ref{th:ExtensioFinal}.

We want to emphasize here that
\begin{itemize}
\item We are able  to extend the manifolds up to a distance of order $\eps$
of the singularities in all the cases without using any inner equation even in
the singular case
$\ell-2r \ge 0$ and $\eta=\ell-2r$.
\item
The outer domain $D^{\out,\ast}_{\rr,\kk}$ contains the \emph{boomerang domain}
$D^{\ast}_{\kk,d}$ for $\ast=u,s$.
\end{itemize}

\subsection{The asymptotic first order of $\Delta$
}\label{sec:heuristic:asymptotic}
Even though we have proved the existence of the invariant manifolds in the
\emph{boomerang domains}, we need some
extra information to detect the asymptotic first order of their difference. The
main idea is that functions which are
of algebraic order with respect to $\eps$ near the singularities $\pm ia$ are
exponentially small for real
values of $u$. Thus, the main point to compute the difference and capture the
asymptotic  first order
is to be able to give the main terms of this difference close to the
singularities, concretely,
up to distance of order $\eps$ of the singularities. For that we need to give
better approximations of the
generating functions $T^{u,s}(u,\tau)$ near the singularities $\pm ia$ of the
homoclinic connection.

To this end, we define the so-called \emph{inner domains} (see Figure
\ref{fig:Inners}), which are defined as
\begin{equation}\label{def:DominisInnerEnu}
\begin{split}
D_{\kk,\C}^{\inn,+,u}=&\left\{ u\in\CC; \Im
u>-\tan\beta_1 (\Re u+\C\eps^\gamma)+a, \Im u<-\tan \beta_2\Re
u+a-\kk\eps,\right.\\
& \left.\Im u<-\tan\beta_0\Re u+a-\kk\eps\right\}\\
D_{\kk,\C}^{\inn,-,u}=&\left\{u\in\CC; \bar u \in
D_{\kk,\C}^{\inn,+,u}\right\}\\
D_{\kk,\C}^{\inn,+,s}=&\left\{u\in\CC; -\bar u \in
D_{\kk,\C}^{\inn,+,u}\right\}\\
D_{\kk,\C}^{\inn,-,s}=&\left\{u\in\CC; - u \in
D_{\kk,\C}^{\inn,+,u}\right\}
\end{split}
\end{equation}
for $\kk>0$, $\C>0$ and $\gamma \in (0,1)$.
On the other hand, $\beta_1$  and $\beta _2$ are the angles  considered in
the definition of the boomerang domains in \eqref{def:DominisRaros}  and
$\beta_0$ is  any angle
satisfying that $\beta_1-\beta_0$ has a positive lower bound independent of
$\eps$ and $\mu$.
%Later on, in  the definition of the domains \eqref{def:DominisRaros} in Section
\ref{sec:sketch:Extensio:General}, %we will impose more conditions on the angle
$\beta_2$.
Let us observe that, if $u\in D_{\kk,\C}^{\inn,\pm,\ast}$, $\ast=u,s$, then
$\OO(\kk \eps) \leq |u\mp ia| \leq \OO(\eps^{\gamma})$.

\begin{figure}[H]
\begin{center}
\psfrag{D}{$D_{\kk,\C}^{\inn,+,u}$}\psfrag{D1}{$D_{\kk,\C}^{\inn,-,u}$}
\psfrag{D3}{$D_{\kk,\C}^{\inn,+,s}$}\psfrag{D4}{$D_{\kk,\C}^{\inn,-,s}$}
\psfrag{b0}{$\beta_0$}\psfrag{b1}{$\beta_1$}\psfrag{b2}{$\beta_2$}
\psfrag{a}{$ia$}\psfrag{a1}{$-ia$}\psfrag{a2}{$i(a-\kk\eps)$}
\includegraphics[height=6cm]{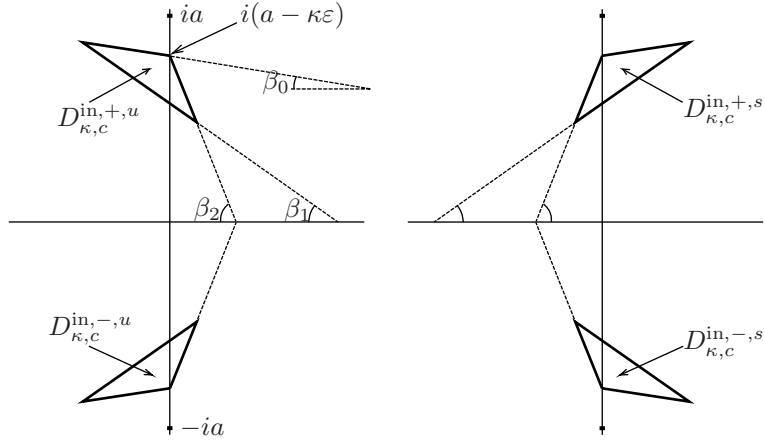}
\end{center}
\caption{\figlabel{fig:Inners} The \emph{inner domains} defined in
\eqref{def:DominisInnerEnu}.}
\end{figure}

Let us observe that simply rewriting $\mu:=\mu\eps^{\eta-\eta^\ast}$, one can
include the regular case ($\eta>\eta^\ast$) into the singular one.
This is very convenient since one can prove the results for both cases at the
same time.
Therefore, from now on in this section, we will focus on the singular case.

When studying the functions $\partial_u T^{u,s}$ evaluated in the inner domains,
one can distinguish the cases $\ell - 2r<0$ or $\ell - 2r \ge 0$.
The difference between these two cases, roughly speaking, is that, when $\ell -
2r<0$, the approximation of the manifolds in the inner domain is still given by
the first order perturbation theory as is stated in Proposition
\ref{coro:Canvi:FirstOrder:lmenor}.
In the  case $\ell - 2r \ge 0$ this fact  is not true anymore.

Analyzing $\partial_u T^{u,s}$ close to the singularity $ia$, one can see that,
if $u-ia = \OO(\eps)$, then
$\partial_u T^{u,s}$ is of order $\OO(1/\eps^{2r})$. For this reason we perform
the change of variables
$
u= ia +\eps z
$
and we study the functions $\psi^{u,s}(z,\tau) = \eps ^{2r-1} T^{u,s}(ia + \eps
z ,\tau)$.
The first order in $\eps$ of these functions verifies the so called \emph{inner
equation}.  Their  solutions $\psi ^{u,s}_{0}(z,\tau)$  were studied in
\cite{Baldoma06}.
Then, in Theorem \ref{th:MatchingHJ} we provide a bound for $|\psi
^{u,s}(z,\tau)-\psi ^{u,s}_{0}(z,\tau)|$.
This is known as \emph{complex matching}.

We emphasize that we have not used the inner
solutions $\psi ^{u,s}_{0}(z,\tau)$
to extend our functions $T^{u,s}$ to the inner domains since we already knew
their existence.
Henceforth to bound  $|\psi ^{u,s}(z,\tau)-\psi ^{u,s}_{0}(z,\tau)|$ we have
exploited the same idea as the one used to study the difference $\Delta =
T^{u}-T^s$.
Let us  explain it in more detail.
As we have explained in Section \ref{sec:heuristic:parametrization}, we have
already
proved the existence of generating functions $T^{u,s}$ in the whole
\emph{boomerang domains}. Henceforth, 
the new functions $\psi^{u,s}(z,\tau) =
\eps ^{2r-1} T^{u,s}(ia + \eps z ,\tau)$ have the corresponding properties
coming from the ones of $T^{u,s}$.
Now we consider the difference $\Delta \psi^{u,s} = \partial_z \psi^{u,s} -
\partial_z \psi_0$.
Such functions (which are known)  satisfy a non-homogeneous linear equation
which can be ``easily'' studied.
Summarizing, we just obtain an ``a posteriori'' bound of $\Delta
\psi^{u,s}$.
This makes our \emph{complex matching}
considerably simpler because we just need to use Gronwall-like techniques.

In both cases $\ell-2r<0$ and $\ell-2r\ge 0$, we have now accurate
approximations for $T^{u,s}$ near the singularities.
Let us call them $T_0^{u,s}$.
The first order asymptotics for the difference $\Delta = T^{u}-T^s$ comes from
$T_0^{u} - T_0^s$ after a change of variables.
Recall that, as we have explained in Section \ref{sec:heuristic:splitting}, in
some cases, this change of variables implies an additional correcting term in
$T_0^{u} - T_0^s$.
Finally, we bound the remainder by using the techniques explained in Section
\ref{sec:heuristic:splitting}.

\section{Description of the proofs of Theorems~\ref{th:MainGeometric:regular}
and \ref{th:MainGeometric:singular}}\label{sec:SketchProof}
We devote this section to prove Theorems~\ref{th:MainGeometric:regular} and
\ref{th:MainGeometric:singular}.

\subsection{Basic notations}\label{sec:basicnotation}
First, we introduce some basic notations which will
be used through the paper.

We denote by $\TT=\RR/(2\pi\ZZ)$ the real 1-dimensional torus and by
\[
\TT_\sigma=\left\{\tau\in \CC/(2\pi\ZZ); |\Im \tau|<\sigma\right\},
\]
with $\sigma>0$, the torus with a complex strip.

Given a function $h:D\times\TT_\sigma\rightarrow \CC$,
where $D\subset\CC$ is an open set, we denote its Fourier
series by
\[
h(u,\tau)=\sum_{k\in\ZZ}h^{[k]}(u)e^{ik\tau}
\]
and its average by
\[
\langle h\rangle (u)
=h^{[0]}(u)=\frac{1}{2\pi}\int_0^{2\pi}h(u,\tau)\,d\tau .
\]

In any Banach space $(\XX, \|\cdot\|)$, we define the following
balls
\[
\begin{split}
B(R)&=\left\{x\in\XX; \|x\|<R\right\}\\
\ol B(R)&=\left\{x\in\XX; \|x\|\leq R\right\}.\\
\end{split}
\]

By Hypothesis \textbf{HP3},  the
Hamiltonian $H$ in \eqref{def:Ham:Original0} is analytic in $\tau=t/\eps$.
By the compactness of $\TT$,  there exists a constant
$\sigma_0$ such that $H$ is continuous in $\ol \TT_{\sigma_0}$ and
analytic in $\TT_{\sigma_0}$.  From now on, we fix  $0<\sigma<\sigma_0$.
%All the results will be stated for $\tau\in \TT_\sigma$ for any fixed
%$\sigma<\sigma_0$.

Throughout the proof of  Theorems \ref{th:MainGeometric:regular} and
\ref{th:MainGeometric:singular} we will use the analyticity in $\mu$. We fix an
arbitrary value  $\mu_0>0$. Even if we do not write it explicitly, all
functions we will encounter from now on will be analytic in $\mu \in B(\mu_0)$.

From now on, we work with the fast time $\tau=t/\eps$. Then,
denoting $'=d/d\tau$, we have the system
\begin{equation}\label{eq:ode:original:lent}
\left\{\begin{array}{rl} 
x'&\dps=\eps\left(y+\mu\eps^{\eta} \pa_y
H_1\left(x,y,\tau;\eps\right)\right)\\
y'&\dps =-\eps\left(V'(x)+\mu\eps^{\eta} \pa_x
H_1\left(x,y,\tau;\eps\right)\right).
\end{array}\right.
\end{equation}

In order to simplify the notation, through the rest of this paper
we will denote by $K$ any constant  independent of $\mu$ and $\eps$ to state all
the bounds.

\subsection{The periodic orbit}
In the parabolic case, Hypothesis \textbf{HP4.2} on $H_1$ implies
that the origin is still a critical point of the perturbed system
\eqref{eq:ode:original:lent}
%This does not happen in the hyperbolic case when the order of the perturbation
%$n$ is equal to $1$ (see
%Hypothesis \textbf{HP4.1}).
In the hyperbolic case, the next theorem
states the existence and useful properties of a hyperbolic periodic
orbit close to the origin of the perturbed system.

\begin{theorem}\label{th:Periodica}
Let us  assume Hypotheses \textbf{HP1.1}, \textbf{HP3},
\textbf{HP4.1} and $\eta \ge 0$. Then,
there exists $\eps_0>0$ such that for any $|\mu|<\mu_0$ and
$\eps\in(0,\eps_0)$,  system \eqref{eq:ode:original:lent} has a
$2\pi$-periodic orbit $(\xp(\tau),\yp(\tau)):\TT_\sigma\rightarrow
\CC^2$ which is real-analytic and satisfies
\[
\sup_{\tau\in\TT_\sigma}
\left(\left|\xp(\tau)\right|+\left|\yp(\tau)\right|\right)\leq
b_0|\mu|\eps^{\eta+1},
\]
where $b_0>0$ is a constant independent of $\eps$ and $\mu$.
\end{theorem}

This theorem is proved in Section \ref{sec:periodica}.

\begin{remark}
The Hamiltonian $H_1$, the periodic orbit $(\xp(\tau),\yp(\tau))$, and
consequently the Hamiltonians $\wh H$,
$\wh H_1$, $\wh H^1_1$, $\wh H^2_1$, which will be defined below, depend on the
parameters $\mu$, $\eps$.
From now on, we will not write  this dependence explicitly  but we will
emphasize it when necessary.
%In particular, the Hamiltonian $\wh H^1_1$ is simply the limit, when $\eps \to
%0$, of $\wh H^1$.
\end{remark}

Once we know the existence of the periodic orbit, we perform the
time dependent change of variables
\begin{equation}\label{eq:CanviShiftPO}
\left\{\begin{array}{l} q=x-\xp(\tau)\\ p=y-\yp(\tau)
\end{array}\right.
\end{equation}
which transforms system \eqref{eq:ode:original:lent} into a
Hamiltonian system with Hamiltonian function $\eps\widehat
H(q,p,\tau)$:
\begin{equation}\label{def:HamPeriodicaShiftada}
\begin{split}
\widehat
H(q,p,\tau)=&\frac{p^2}{2}
+V\left(q+\xp(\tau)\right)-V\left(\xp(\tau)\right)-V'\left(\xp(\tau)\right)q\\
&+\mu\eps^{\eta}\widehat H_1(q,p,\tau)
\end{split}
\end{equation}
with
\begin{equation}\label{def:ham:ShiftedOP:perturb}
\begin{split}
\widehat
H_1(q,p,\tau)=&H_1(x_p(\tau)+q,y_p(\tau)+p,\tau)-H_1(x_p(\tau),y_p(\tau),\tau)\\
&-DH_1(x_p(\tau),y_p(\tau),\tau)\left(\begin{array}{c}q\\p\end{array}\right),
\end{split}
\end{equation}
where we have denoted $DH_1=(\pa_xH_1,\pa_yH_1)$. We have added the terms
$V\left(\xp(\tau)\right)$ and
$H_1(x_p(\tau),y_p(\tau),\tau)$ for convenience.
Note that they do not generate any term in the differential equations associated
to $\wh H$.

Since
$|(x_p(\tau),y_p(\tau))|=\OO\left(\mu \eps^{\eta+1}\right)$, $\wh
H_1$ can be split as
\[
\wh H_1(q,p,\tau)=\wh H_1^1(q,p,\tau)+\eps \wh H^2_1(q,p,\tau),
\]
where
\[
\wh
H_1^1(q,p,\tau)=H_1(q,p,\tau)-H_1(0,0,\tau)-DH_1(0,0,\tau)\left(\begin{array}{c}
q\\p\end{array}\right)
\]
and $ \wh H^2_1(q,p,\tau)$ is the remaining part.
In fact, we can give a more precise formula for $\wh H_1^1$ and $\wh H_1^2$ in
both the polynomial and the
trigonometric cases:
\begin{align}
\wh H_1^1(q,p,\tau)=&\dps\sum_{2\leq k+l\leq N}
a_{kl}(\tau)q^kp^l&\text{(polynomial case)}\notag\\
\wh H_1^1(q,p,\tau) =&\dps\sum_{k=-N,\ldots,
N}a_{k0}(\tau)\left(e^{ikq}-1-ikq\right)&\label{def:HamPertorbat:H1}\\
&+\sum_{k=-N,\ldots,
N}a_{k1}(\tau)\left(e^{ikq}-1\right)p+\sum_{\substack{k=-N,\ldots,
N\\l=2,\ldots, N}}a_{kl}(\tau)e^{ikq}p^{l}, & \text{(trigonometric
case)}\notag
\end{align}
where $a_{kl}$ are the functions defined in
\eqref{def:Ham:Original:perturb:poli} and
\eqref{def:Ham:Original:perturb:trig} and have zero average, that is
\begin{equation}
 \langle\wh H_1^1\rangle=0.
\end{equation}
Let us point out that $\wh H_1^1$ is $H_1$
subtracting its linear terms in $(x,y)$, and hence it is of order
$n=2$.

The Hamiltonian $\wh H_1^2$ is given by:
\begin{align}
 \wh
H_1^2(q,p,\tau)=&\dps\sum_{2\leq k+l\leq N-1}
c_{kl}(\tau)q^kp^l&{\text{(polynomial case)}}\notag\\
 \wh
H_1^2(q,p,\tau)=&\dps\sum_{k=-N,\ldots,
N}c_{k0}(\tau)\left(e^{ikq}-1-ikq\right)&\label{def:HamPertorbat:H2}\\
&\dps+\sum_{k=-N,\ldots,
N}c_{k1}(\tau)\left(e^{ikq}-1\right)p+\sum_{\substack{k=-N,\ldots,
N\\l=2,\ldots, N-1}}c_{kl}(\tau)e^{ikq}p^{l} , &\text{(trigonometric
case)}\notag
\end{align}
where $c_{kl}$ are $2\pi$-periodic functions which, in general, do not
have zero average.
As we will see in Corollary \ref{coro:ShiftPeriodica} the functions
$c_{kl}$ are $2\pi$-periodic and satisfy
\begin{equation}
|c_{kl}(\tau)|\leq
K|\mu|\eps^{\eta}.
\end{equation}

In the case that the unperturbed Hamiltonian has
a parabolic point at the origin, since $(\xp,\yp)=(0,0)$, we have
that $c_{kl}=0$.

\subsection{Different parameterizations of the invariant
manifolds}\label{sec:DiferentsParam}
The next step is to prove the existence of the unstable and stable invariant
manifolds of the
periodic orbit given in Theorem \ref{th:Periodica}.

We will consider two different strategies to find suitable parameterizations of
these invariant manifolds
depending on the domain we are.
On the one hand, when it is possible, we will  follow  \cite{LochakMS03,
Sauzin01} (see also \cite{GuardiaOS10}),
and we will write the invariant manifolds as graphs of suitable generating
functions which are solutions
of a Hamilton-Jacobi equation in appropriate variables.
On the other hand, when this is not possible, we will obtain 
parameterizations of invariant manifolds formed by families of solutions of the differential equations.

To introduce the first method, let us consider the
symplectic change of variables (see \cite{Baldoma06})
\begin{equation}\label{eq:CanviSimplecticSeparatriu}
\left\{\begin{array}{l}
q=q_0(u)\\
\dps p=\frac{w}{p_0(u)},
\end{array}\right.
\end{equation}
where $(q_0(u),p_0(u))$ is the parameterization of the homoclinic orbit given in
Hypothesis \textbf{HP2}.
This is a well defined change for any $u\in\CC$ such that $p_0(u)\neq 0$ and
leads to a new Hamiltonian given  by
\begin{equation}\label{def:Hbarra}
\eps \overline H(u,w,\tau)=\eps \widehat
H\left(q_0(u),\frac{w}{p_0(u)}, \tau\right),
\end{equation}
where $\widehat H$ is the Hamiltonian defined in
\eqref{def:HamPeriodicaShiftada}.

Let us recall that when $\mu=0$, $\wh H$ becomes $H_0$ defined in
\eqref{def:Ham:Integrable}. Then, the separatrix of the unperturbed
system ($\mu=0$) for $\ol H$ can be parameterized as a graph as
$w=p_0(u)^2$.

To obtain parameterizations of the perturbed invariant
manifolds, we can take into account the well known fact that,  locally,  they
are Lagrangian and can be obtained as graphs of some functions which are
solutions of the
Hamilton-Jacobi equation associated to the Hamiltonian $\eps\overline H$.
That is, we look for  $w= \partial _{u} T^{u,s}(u,\tau)$, where the
functions $T^{u,s}$ satisfy
\begin{equation}\label{eq:HamJacGeneral}
\pa_\tau T(u,\tau)+ \eps \overline H(u,\pa_u T(u,\tau),\tau)=0
\end{equation}
and certain limiting properties.

The solutions of this equation give parameterizations of the
invariant manifolds, which, in the original variables, read
\begin{equation}\label{eq:ParameterizationHJ}
(q,p)=\left(q_0(u), \frac{\pa_u T^{u,s}(u,\tau)}{p_0(u)}\right).
\end{equation}

Notice that in variables $(q,p)$ the condition $p_0(u)=\dot q_0(u)\ne 0$ ensures
that the manifolds can be written as graphs over the variable $q$ through the
functions
$ S^{u,s}(q,\tau) = T^{u,s}(q_0^{-1}(q),\tau)$ which verify the classical
Hamilton-Jacobi equation associated to the
Hamiltonian $\widehat H(q,p,\tau)$.

When this method cannot be used, that is when $p_0(u)$ can vanish,
we  look for the invariant manifolds as parameterizations:
\begin{equation}\label{def:ParamByFlow}
(q,p)=(Q(v,\tau),P(v,\tau))
\end{equation}
in such a way that $(q(s), p(s))=(Q(u+\eps s,s),P(u+\eps s,s))$
are solutions of the differential equation associated to the Hamiltonian
\eqref{def:HamPeriodicaShiftada}. These kind of parameterizations were used in
\cite{DelshamsS92, DelshamsS97,Gelfreich97a,Gelfreich00,BaldomaF04,BaldomaF05}.

Then, it is straightforward to see
(\cite{Gelfreich97a}) that $(Q,P)$ has to satisfy
\begin{equation}\label{eq:PDEParametritzacions}
\LL_\eps\left(\begin{array}{c}Q\\P\end{array}\right)=\left(\begin{array}{c}
P+\mu\eps^{\eta}\pa_p \widehat H_1(Q,P,\tau)\\
-\left(V'(Q+\xp(\tau))-V'(\xp(\tau))\right)-\mu\eps^{\eta}\pa_q
\widehat H_1(Q,P,\tau)\end{array}\right),
\end{equation}
where $\LL_\eps$ is the operator
\begin{equation}\label{def:Lde}
\LL_\eps=\eps\ii\pa_\tau+\pa_v
\end{equation}
and $\widehat H_1$ is the Hamiltonian defined in
\eqref{def:ham:ShiftedOP:perturb}.

Both parameterizations \eqref{eq:ParameterizationHJ} and
\eqref{eq:PDEParametritzacions} satisfy that,
fixing $\tau=\tau_\ast$, they give
parameterizations of the invariant curves of the fixed point of the
$2\pi$-Poincar\'e map from the section $\tau=\tau^\ast$ to the
section $\tau=\tau^\ast+2\pi$.

\subsection{Existence of the local invariant manifolds}\label{sec:sketch:outer}

In this section we will find the local invariant manifolds of the origin of the
Hamiltonian system
\eqref{def:HamPeriodicaShiftada}.

First, we recall the behavior of the separatrix $(q_0(u),p_0(u))$
as $\Re u\rightarrow \pm\infty$, which is substantially different
depending on whether $(0,0)$ is a hyperbolic or a parabolic point of
the unperturbed system.

In the hyperbolic case, by
Hypothesis \textbf{HP1.1}, close to $x=0$ the potential behaves as
\begin{equation}\label{def:Potencial:Hyperbolic:PropZero}
V(x)=-\frac{\lambda^2}{2} x^2+\OO(x^3).
\end{equation}
Therefore,  $\{\lambda, -\lambda\}$
are the eigenvalues of the critical point.
Moreover, there exist constants $c_{\pm }\neq 0$ such that as $\Re
u\rightarrow\mp\infty$ the separatrix behaves as
\begin{equation}\label{eq:separatrix:hyp:infty}
\begin{split}
q_0(u)&=c_{\pm} e^{\pm\lambda u}+\OO\left(e^{\pm 2\lambda
u}\right),\\
p_0(u)&=\pm\lambda c_{\pm} e^{\pm\lambda u}+\OO\left(e^{\pm 2\lambda
u}\right).
\end{split}
\end{equation}
In the parabolic case, using Hypothesis \textbf{HP1.2}, in
\cite{BaldomaF04} it is seen that  there exists a constant $c_0$
such that as $\Re u\rightarrow\mp\infty$ the separatrix behaves as
\begin{equation}\label{eq:separatrix:hyp:parab}
\begin{split}
q_0(u)&=\frac{c_0}{u^\frac{2}{m-2}}+\OO\left(\frac{1}{u^\nu}\right),\\
p_0(u)&=-\frac{2c_0}{(m-2)
u^\frac{m}{m-2}}+\OO\left(\frac{1}{u^{\nu+1}}\right),
\end{split}
\end{equation}
where $m$ is the order of the potential
\eqref{def:Potencial:Parabolic} and $\nu>2/(m-2)$.

We look for the parameterizations of the local invariant manifolds in the
domains $D^{u,s}_{\infty,\rr}$
defined in \eqref{def:DominsInfinit}.

By \eqref{eq:separatrix:hyp:infty} and \eqref{eq:separatrix:hyp:parab},
the constant $\rr$ can be taken big enough so that $p_0(u)$ does not vanish in
these domains.
Then, as we explained in Section \ref{sec:DiferentsParam}, we can look for the
invariant manifolds by means
of generating functions $T^{u,s}$ (see \eqref{eq:ParameterizationHJ}) defined in
$D^\ast_{\infty,\rr}$ with $\ast=u,s$ respectively, which are
solutions of the Hamilton-Jacobi equation \eqref{eq:HamJacGeneral}.
Moreover, we impose the asymptotic conditions
\begin{align}
\dps \lim_{\Re u\rightarrow-\infty}p_0\ii(u)\cdot\pa_u T^u(u,\tau)=0 &
\text{\quad (for the unstable manifold)}
\label{eq:AsymptCondFuncioGeneradora:uns}\\
\dps \lim_{\Re u\rightarrow+\infty}p_0\ii(u)\cdot \pa_u
T^s(u,\tau)=0 & \text{\quad (for the stable manifold)}.
\label{eq:AsymptCondFuncioGeneradora:st}
\end{align}

We note that when $\mu=0$ a solution of \eqref{eq:HamJacGeneral}
satisfying both asymptotic conditions \eqref{eq:AsymptCondFuncioGeneradora:uns}
and
\eqref{eq:AsymptCondFuncioGeneradora:st} is
\begin{equation}\label{def:T00}
T_0(u)=\int_{-\infty}^u p_0^2(v)\,dv,
\end{equation}
which corresponds to the  the unperturbed separatrix.

The next theorem gives the existence of the invariant manifolds in
the domains $D^{\ast}_{\infty,\rr}$ with $\ast=u,s$ defined in
\eqref{def:DominsInfinit}. We state the results for the unstable
invariant manifold. The stable one has analogous properties.

\begin{theorem}\label{th:ExistenceCloseInfty}
Let us  assume Hypotheses \textbf{HP1.1}, \textbf{HP3},
\textbf{HP4} and take   $\eta \ge 0$.
Let  $\rr_1>0$ be a real number big enough such that $p_0(u)\neq0$
for $u\in D^{u}_{\infty,\rr_1}$. Then, there exists $\eps_0>0$ such
that for $\eps\in(0,\eps_0)$ and $\mu\in B(\mu_0)$, the
Hamilton-Jacobi equation \eqref{eq:HamJacGeneral} has  a unique
(modulo an additive constant) real-analytic solution in
$D^{u}_{\infty,\rr_1}\times\TT_\sigma$ satisfying the asymptotic
condition \eqref{eq:AsymptCondFuncioGeneradora:uns}.

Moreover, there exists a real constant $b_1>0$ independent of $\eps$
and $\mu$, such that for $(u,\tau)\in
D^{u}_{\infty,\rr_1}\times\TT_\sigma$,
\[
\left|\pa_u T^u(u,\tau)-\pa_u T_0(u)\right|\leq
b_1|\mu|\eps^{\eta+1}.
\]
\end{theorem}
The asymptotic behavior of the invariant manifolds when $\Re u\to
+\infty$ is qualitatively different for the hyperbolic case and the
parabolic case. For this reason we prove separately Theorem
\ref{th:ExistenceCloseInfty}
for  these two cases.
We deal with the hyperbolic case  in Section \ref{sec:InftyHyperbolic} and with
the parabolic case in
Section \ref{sec:InftyParabolic}.\\

In the rest of the paper we will assume the whole set of Hypotheses
\textbf{HP1}, \textbf{HP2}, \textbf{HP3}, \textbf{HP4} and \textbf{HP5}.

\subsection{The global invariant manifolds}\label{sec:sketch:Extensio}

The next step is to extend the invariant manifolds to a wider domain
which contains a region close to the singularities $\pm ia$ of the separatrix
(see Hypothesis \textbf{HP2}).
In the general case the function $p_0(u)$ can vanish and therefore,
the symplectic change \eqref{eq:CanviSimplecticSeparatriu} is not
well defined. For this reason one cannot use the Hamilton-Jacobi equation
\eqref{eq:HamJacGeneral} anymore. Instead we  look for
parameterizations
\[(q,p)=(Q^{u,s}(v,\tau),P^{u,s}(v,\tau))\] which are solutions of the
partial differential equation \eqref{eq:PDEParametritzacions}.

Nevertheless, there are some cases, as happens for the classical
pendulum, where $p_0(u)$ does not vanish for $u\in\CC$, and then
one can  use  the Hamilton-Jacobi equation in the whole domain,
which makes the proof of Theorems \ref{th:MainGeometric:regular} and
\ref{th:MainGeometric:singular} remarkably
simpler. Section \ref{sec:sketch:Extensio:Trig} is devoted to this
simpler case and Section \ref{sec:sketch:Extensio:General} to  the general one.

\subsubsection{The global invariant manifolds in the case $p_0(u)\neq
0$}\label{sec:sketch:Extensio:Trig}

In this section we extend the parameterizations \eqref{eq:ParameterizationHJ} of
the invariant manifolds to
the outer domains $\dps D^{\out,\ast}_{\rr,\kk}$, $\ast=u,s$,
(see Figure \ref{fig:OuterDomains}) defined by
\eqref{def:DominisOuter}, in the case that $p_0(u)\neq 0$.
We emphasize that these domains reach a region which is at a distance of
$\OO(\eps)$
of the singularities $u=\pm i a$ of the unperturbed separatrix.

The constant $\rho$ will be taken $\rr>\rr_1$, where $\rr_1$ is the
constant given by Theorem \ref{th:ExistenceCloseInfty}, in order to
ensure that $D^u_{\infty,\rr_1}\cap D^{\out,u}_{\rr,\kk}\neq
\emptyset$.

Since in this section we are assuming that $p_0(u)\neq 0$ in the whole
\emph{outer domain}, the
symplectic change of variables \eqref{eq:CanviSimplecticSeparatriu}
is still well defined there. Then, it is enough to look for the analytic
continuation of the generating functions $T^{u,s}$ obtained in
Theorem \ref{th:ExistenceCloseInfty}.

\begin{theorem}\label{th:Extensio:Trig}
Let $\rr_1$  be the constant considered in Theorem
\ref{th:ExistenceCloseInfty} and let us consider $\rr_2$ such that
$\rr_2>\rr_1$, $\kk_1>0$ big enough and $\eps_0>0$ small enough.
Then, for $\mu\in B(\mu_0)$, $\eps\in(0,\eps_0)$, the function
$T^u(u,\tau)$ obtained in Theorem \ref{th:ExistenceCloseInfty} can
be analytically extended to the domain $D^{\out,
u}_{\rr_2,\kk_1}\times\TT_\sigma$.

Moreover, there exists a real constant $b_2>0$ independent of $\eps$
and $\mu$, such that for $(u,\tau)\in D^{\out,
u}_{\rr_2,\kk_1}\times\TT_\sigma$,
\[
\left|\pa_uT^u(u,\tau)-\pa_u T_0(u)\right|\leq
\frac{b_2|\mu|\eps^{\eta+1}}{\left|u^2+a^2\right|^{\ell+1}}.
\]
\end{theorem}
The proof of this theorem is given in Section \ref{sec:Extensio}.
The results for the stable manifold  are analogous.

\subsubsection{The global invariant manifolds for the general
case}\label{sec:sketch:Extensio:General}
We devote this section to obtain parameterizations of the global
invariant manifolds for the general case, that is, considering
Hamiltonian systems for which $p_0(u)$ can vanish in the outer
domains defined in \eqref{def:DominisOuter}. We look  for parameterizations
\[(q,p)=(Q^{u,s}(v,\tau),P^{u,s}(v,\tau))\] which are solutions of the
partial differential equation \eqref{eq:PDEParametritzacions}. Our
strategy will be:
\begin{itemize}
\item To obtain the parameterizations $(Q^{u,s}(v,\tau),P^{u,s}(v,\tau))$ in a
transition
domain (Theorem \ref{th:HJtoParam}).
\item To extend them up to a region where we can ensure
that $p_0(u)$ does not vanish (Theorem \ref{th:Extensio}).
\item To recover in this new region the
representations \eqref{eq:ParameterizationHJ} through the generating function
$T^{u,s}$ of the manifolds,
which are solution of the Hamilton-Jacobi equation
\eqref{eq:HamJacGeneral} (Theorem \ref{th:ParamtoHJ}).
\item  To extend the generating function $\pa_u T^{u,s}(u,\tau)$ up to a
distance of order
$\eps$ of the singularity, as it was done in the easier case
$p_0(u)\neq 0$ in Theorem \ref{th:Extensio:Trig} (Theorem
\ref{th:ExtensioFinal}).
\end{itemize}

First we are going to construct the two dimensional parameterizations
of the invariant manifolds from the parameterizations of the local
invariant manifolds given in Theorem \ref{th:ExistenceCloseInfty},
which were obtained by using the Hamilton-Jacobi equation. We look
for them in  the transition domains
\begin{equation}\label{def:TransDomainInfty}
\begin{split}
\tro_{\rr,\bar\rr}^u&=D^{\out,u}_{\kk,\bar\rr}\cap
D^{u}_{\infty,\rr}\\
\tro_{\rr,\bar\rr}^s&=D^{\out,s}_{\kk,\bar\rr}\cap
D^{s}_{\infty,\rr}
\end{split}
\end{equation}
with $\bar\rr>\rr$ (see Figure \ref{fig:TransInfty}). Taking into
account the change of variables
\eqref{eq:CanviSimplecticSeparatriu}, it is natural to look for the
parameterizations of the invariant manifolds  $(Q^{u,s},P^{u,s})$ of
the form
\begin{equation}\label{eq:HJtoParam}
\begin{array}{ll}
\dps Q^{u,s}(v,\tau)=q_0\left(v+\UU^{u,s}(v,\tau)\right)\\
\dps P^{u,s}(v,\tau)=\frac{\pa_u
T^{u,s}\left(v+\UU^{u,s}(v,\tau)\right)}{p_0(v+\UU^{u,s}(v,\tau))},
\end{array}
\end{equation}
where $\UU^{u,s}$ define a change of variables
$u=v+\UU^{u,s}(v,\tau)$ in such a way that $(Q^{u,s},P^{u,s})$
satisfy the system of equations \eqref{eq:PDEParametritzacions}.
\begin{figure}[H]
\begin{center}
\psfrag{b}{$\beta_1$}\psfrag{D4}{$D^{\out,u}_{\bar\rr,\kk}$}\psfrag{D3}{$D^{\out
,s}_{\bar\rr,\kk}$}
\psfrag{D6}{$D^{u}_{\infty,\rr}$}\psfrag{D1}{$D^{s}_{\infty,\rr}$}\psfrag{D5}{
$\tro_{\rr,\bar\rr}^u$}\psfrag{D2}{$\tro_{\rr,\bar\rr}^s$}
\psfrag{r}{$-\bar\rr$}\psfrag{r1}{$-\rr$}\psfrag{r2}{$\rr$}\psfrag{r3}{$\bar\rr$
}\psfrag{s}{$i(a-\de)$}\psfrag{u}{$u_1$}\psfrag{v}{$\ol
u_1$}
\includegraphics[height=6cm]{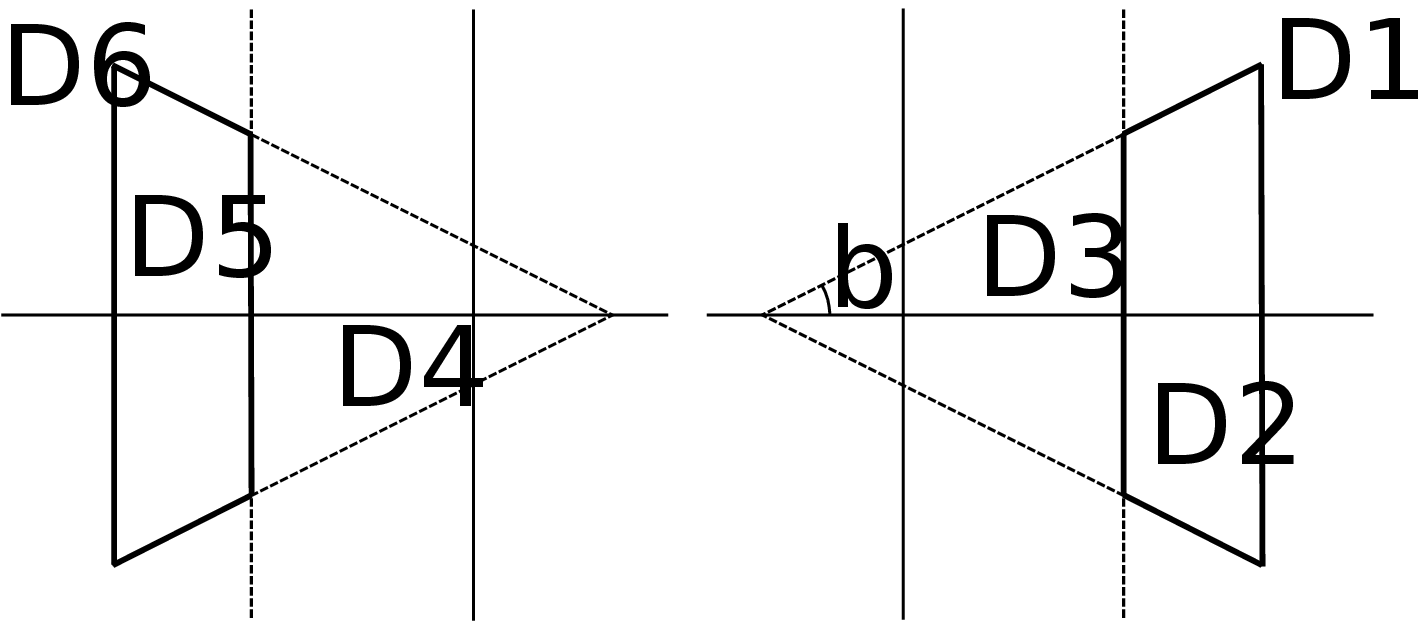}
\end{center}
\caption{\figlabel{fig:TransInfty} The transition domains
$\tro_{\rr,\bar\rr}^u$ and $\tro_{\rr,\bar\rr}^s$ defined in
\eqref{def:TransDomainInfty}.}
\end{figure}

The results in this section are only stated in the unstable case since the ones for the stable case are analogous.

The next theorem ensures that the change of variables $u=v+\UU^{u}(v,\tau)$ exists and it
is well defined in the transition domain $\tro_{\rr,\bar\rr}^u$. 
\begin{theorem}\label{th:HJtoParam}
Let $\rr_1$  be the constant considered in Theorem
\ref{th:ExistenceCloseInfty} and  let  $\rr_3$ and
$\rr_4$ such that $\rr_4>\rr_3>\rr_1$ and $\eps_0$ small enough
(which might depend on $\rr_i$, $i=1,2,3$). Then,  for
$\eps\in(0,\eps_0)$ and $\mu\in B(\mu_0)$, there exists a
real-analytic function $\UU^u: \tro_{\rr_3,
\rr_4}^u\times\TT_\sigma\rightarrow \CC$ such that
\begin{itemize}
\item There exists a constant $b_3>0$ independent of $\eps$ and $\mu$
such that for $(v,\tau)\in \tro_{\rr_3,\rr_4}^u\times\TT_\sigma$,
\[
|\UU^u(v,\tau)|\leq b_3|\mu|\eps^{\eta+1}.
\]
\item If $(v,\tau)\in \tro_{\rr_3,\rr_4}^u\times\TT_\sigma$, then
$v+\UU^u(v,\tau)\in
 D^{u}_{\infty,\rr_1}$.
\item The parameterizations of the invariant manifolds
$(Q^{u}(v,\tau),P^{u}(v,\tau))$ in \eqref{eq:HJtoParam} satisfy the system of
equations \eqref{eq:PDEParametritzacions} and there exists a
constant $b_4>0$ such that for $(v,\tau)\in
\tro_{\rr_3,\rr_4}^u\times\TT_\sigma$,
\[
\begin{split}
\left| Q^u(v,\tau)-q_0(v)\right|&\leq b_4|\mu|\eps^{\eta+1}\\
\left| P^u(v,\tau)-p_0(v)\right|&\leq b_4|\mu|\eps^{\eta+1},
\end{split}
\]
where $(q_0,p_0)$ is the parameterization of the unperturbed
separatrix given in Hypothesis \textbf{HP2}.
\end{itemize}
\end{theorem}
The proof of this theorem is deferred to section
\ref{sec:HJtoParam}.

Having the parameterizations $(Q^{u,s}(v,\tau),P^{u,s}(v,\tau))$ in
the transition domains $ \tro_{\rr_3,\rr_4}^\ast\times\TT_\sigma$
for $\ast=u,s$, we extend them until we arrive  to a region where we can ensure
that $p_0(u)$ does not vanish anymore.
This region consists of a piece of the
\emph{boomerang domains} defined in \eqref{def:DominisRaros} (see Figure
\ref{fig:BoomerangDomains}), in
which $p_0(u)\neq 0$, and hence the parameterizations
\eqref{eq:ParameterizationHJ} will be well defined in them.

The next step is to extend the parameterizations
$(Q^{u,s}(v,\tau),P^{u,s}(v,\tau))$ provided in Theorem
\ref{th:HJtoParam} up to domains  which intersect the \emph{boomerang domains}
$D^u_{\kk,d}$ and
$D^s_{\kk,d}$ respectively.
To this end, we define the following domains
\begin{equation}\label{def:DominOuterParam}
\begin{split}
\wt D^{\out,u}_{\rr,d,\kk}&=D^{\out,u}_{\rr,\kk}\cap \left\{u\in\CC
; |\Im
u|<-\tan\beta_2\Re u+a-\frac{d}{2}\right\}\\
\wt D^{\out,s}_{\rr,d,\kk}&=D^{\out,s}_{\rr,\kk}\cap \left\{u\in\CC
; |\Im u|>\tan\beta_2\Re u+a-\frac{d}{2}\right\},
\end{split}
\end{equation}
which are depicted in Figure \ref{fig:OuterParam}.

We want to emphasize that to extend the parameterizations
$(Q^{u,s}(v,\tau),P^{u,s}(v,\tau))$
to these new domains, has no technical difficulties since they are far from the
singularities $u=\pm ia$.
Actually the next theorem is a classical perturbative result.

\begin{figure}[h]
\begin{center}
\psfrag{D6}{$\wt D^{\out,s}_{\rr,d,\kk}$} \psfrag{D5}{$\wt
D^{\out,u}_{\rr,d,\kk}$}
\psfrag{b1}{$\beta_1$}\psfrag{b2}{$\beta_2$}\psfrag{r1}{$-\rr$}\psfrag{r2}{$\rr$
}
\psfrag{a1}{$ia$}\psfrag{a2}{$-ia$}\psfrag{a3}{$i(a-d)$}\psfrag{a4}{
$i(a-\kk\eps)$}
\psfrag{D1}{$D^{\out,s}_{\rr,\kk}$}\psfrag{D}{$D^{s}_{\kk,d}$}\psfrag{D4}{$D^{
\out,u}_{\rr,\kk}$}\psfrag{D3}{$D^{u}_{\kk,d}$}
\includegraphics[height=6cm]{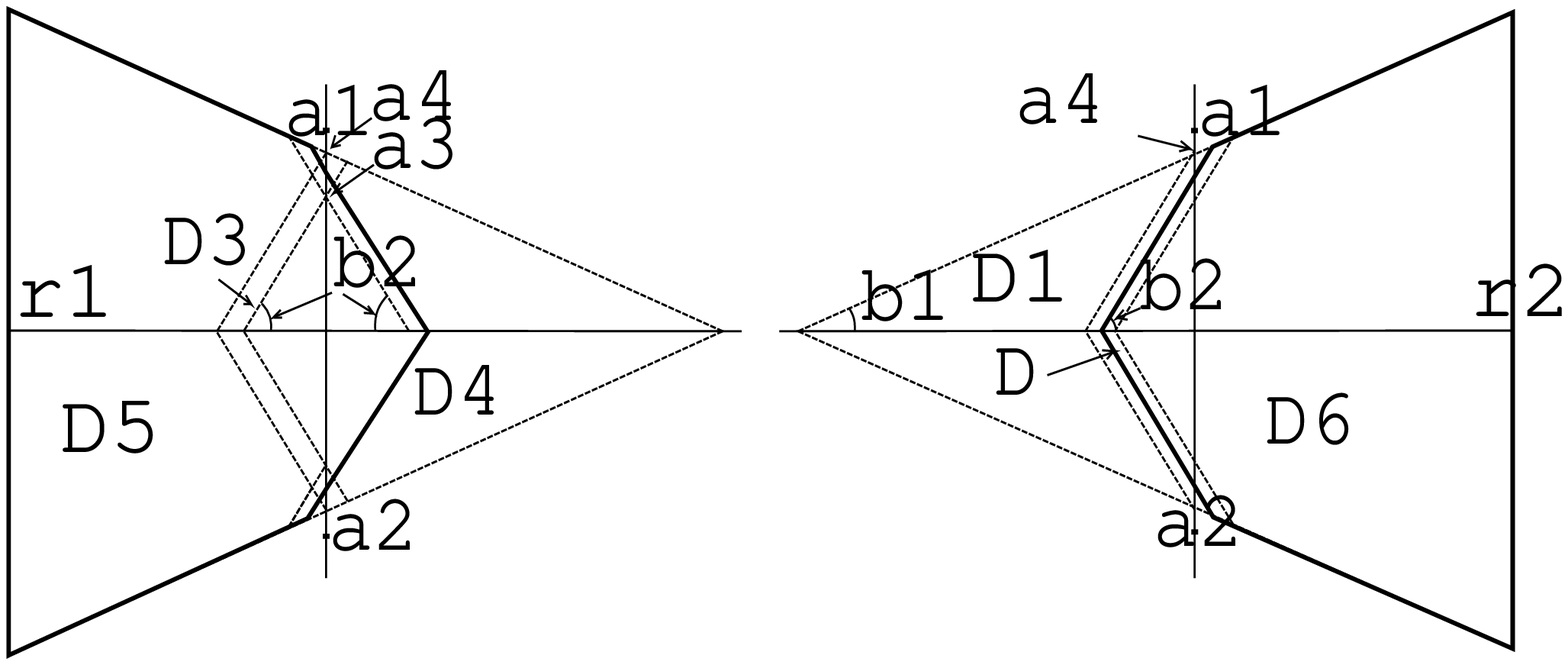}
\end{center}
\caption{\figlabel{fig:OuterParam} The domains $\wt
D^{\out,u}_{\rr,d,\kk}$ and $\wt D^{\out,s}_{\rr,d,\kk}$ defined in
\eqref{def:DominOuterParam}.}
\end{figure}

\begin{theorem}\label{th:Extensio}
Let $\rr_4$ and $\kk_1$  be the constants considered in Theorems
\ref{th:HJtoParam} and \ref{th:Extensio:Trig}, $d_0>0$ and
$\eps_0>0$ small enough. Then, for $\mu\in B(\mu_0)$ and
$\eps\in(0,\eps_0)$, there exist functions
$(Q^{u}(v,\tau),P^{u}(v,\tau))$ defined in $\wt
D^{\out,u}_{\rr_4,d_0,\kk_1}\times\TT_\sigma$ satisfying
equation \eqref{eq:PDEParametritzacions} and such that they are the analytic
continuation of the parameterizations of the invariant manifolds
obtained in Theorem \ref{th:HJtoParam}.

Moreover, there exists a constant $b_5>0$ independent of $\eps$ and
$\mu$ such that for  $(v,\tau)\in \wt
D^{\out,u}_{\rr_4,d_0,\kk_1}\times\TT_\sigma$,
\[
\begin{array}{l}
\dps\left|Q^u(v,\tau)-q_0(v)\right|\leq
b_5|\mu|\eps^{\eta+1}\\
\dps\left|P^u(v,\tau)-p_0(v)\right|\leq b_5|\mu|\eps^{\eta+1}.
\end{array}
\]
\end{theorem}
The proof of this theorem is given in Section \ref{subsec:Extensio:general}.\\

Theorem \ref{th:Extensio} provides parameterizations of the
invariant manifolds of the form \eqref{def:ParamByFlow} in the
domains $\wt D^{\out,u}_{\rr,d,\kk}$ and $\wt
D^{\out,s}_{\rr,d,\kk}$. In particular, they are defined in the
following transition domains, which are depicted  in Figure \ref{fig:TransBoom}.
\begin{equation}\label{def:Dominis:Trans:Raros}
\begin{split}
\tro^{\out,u}_{\kk,d}&=\wt D^{\out,u}_{\rr,d,\kk}\cap D^{u}_{\kk,d}\\
\tro^{\out,s}_{\kk,d}&=\wt D^{\out,s}_{\rr,d,\kk}\cap D^{s}_{\kk,d},
\end{split}
\end{equation}
where, by construction, $p_0(u)$ does not vanish. Then, we can use
these domains as  transition domains where we can go back to the
parameterizations \eqref{eq:ParameterizationHJ} and where the
Hamilton-Jacobi equation \eqref{eq:HamJacGeneral} can be used. To
obtain them, we look for changes of variables
$v=u+\VV^{u,s}(u,\tau)$ which satisfy
\begin{equation}\label{eq:CanviParamToHJ}
Q^{u,s}(u+\VV^{u,s}(u,\tau),\tau)=q_0(u),
\end{equation}
where $Q^{u,s}$ are the first components of the parameterizations
obtained in Theorem \ref{th:Extensio}. Once we have them, we will
define the generating functions $T^{u,s}$ which  give the
parameterizations \eqref{eq:ParameterizationHJ}. Let us observe that
if $p_0(u)$ does not vanish in the outer domains, the changes of
variables $v=u+\VV^{u,s}(u,\tau)$ are defined in the whole domain
and they are the inverse of the changes $u=v+\UU^{u,s}(v,\tau)$ obtained
in Theorem \ref{th:HJtoParam}.

\begin{figure}[h]
\begin{center}
\psfrag{D6}{$\wt D^{\out,s}_{\rr,d,\kk}$} \psfrag{D5}{$\wt
D^{\out,u}_{\rr,d,\kk}$}
\psfrag{b1}{$\beta_1$}\psfrag{b2}{$\beta_2$}
\psfrag{a1}{$ia$}\psfrag{a2}{$-ia$}\psfrag{a3}{$i(a-d)$}\psfrag{a4}{
$i(a-\kk\eps)$}
\psfrag{D1}{$D^{\out,s}_{\rr,\kk}$}\psfrag{D}{$D^{s}_{\kk,d}$}\psfrag{D4}{$D^{
\out,u}_{\rr,\kk}$}\psfrag{D3}{$D^{u}_{\kk,d}$}
\psfrag{D8}{$\tro^{\out,s}_{\kk,d}$}\psfrag{D7}{$\tro^{\out,u}_{\kk,d}$}
\includegraphics[height=6cm]{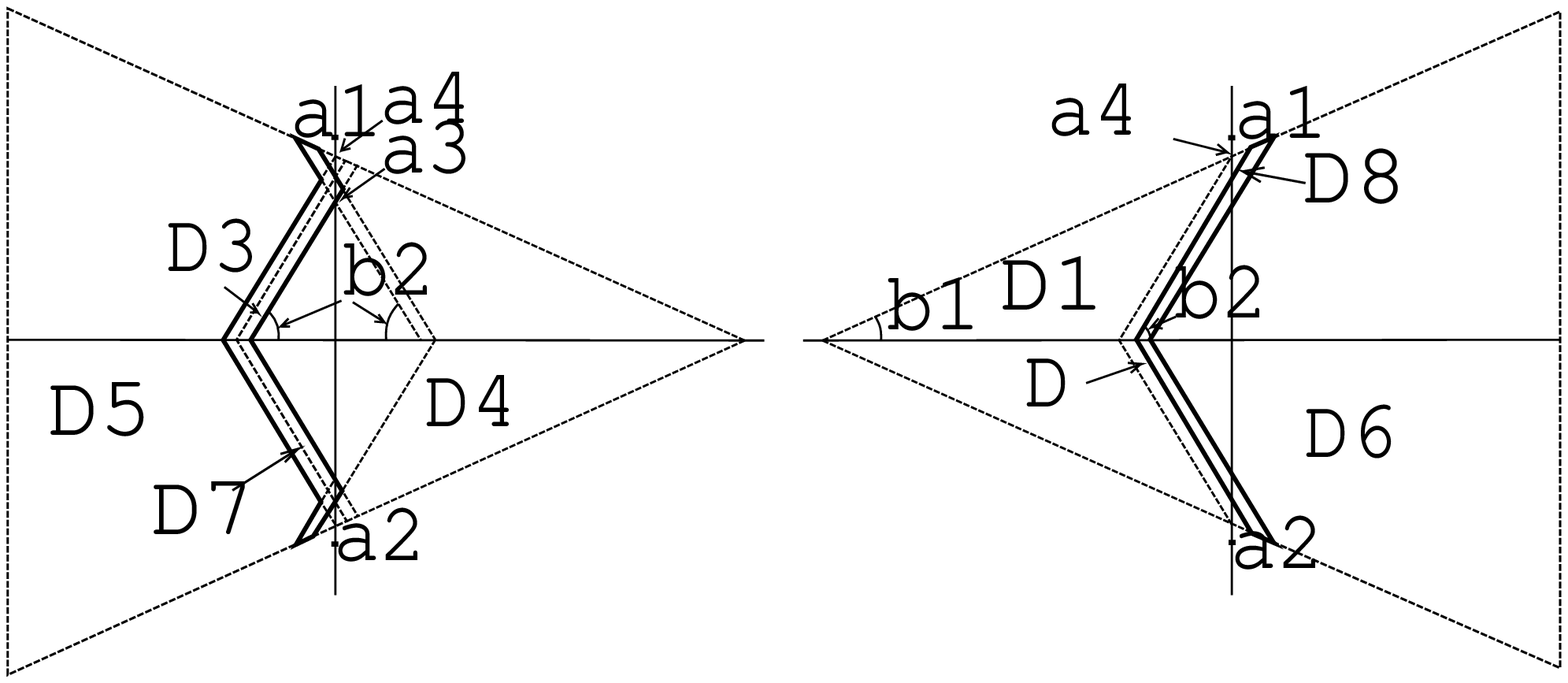}
\end{center}
\caption{\figlabel{fig:TransBoom} The domains
$\tro^{\out,u}_{\kk,d}$ and $\tro^{\out,s}_{\kk,d}$ defined in
\eqref{def:Dominis:Trans:Raros}.}
\end{figure}

\begin{theorem}\label{th:ParamtoHJ}
Let  $d_0$, $\kk_1$, $\rr_4$ be the constants given in Theorem
\ref{th:Extensio}, $\kk_2>\kk_1$, $d_1<d_0$ and $\eps_0>0$ small
enough. Then, for $\eps\in (0,\eps_0)$  and $\mu\in B(\mu_0)$, and
increasing $\kk_1$ if necessary,
\begin{itemize}
\item There exists a real-analytic function $\VV^u:
\tro^{\out,u}_{\kk_2,d_1}\times\TT_\sigma\rightarrow\CC$ which satisfies
\eqref{eq:CanviParamToHJ}. 
Moreover, if $(u,\tau)\in
\tro^{\out,u}_{\kk_2,d_1}\times\TT_\sigma$, then $u+\VV^u(u,\tau)\in
\tro^{\out,u}_{\kk_1,d_0}$ and
\[
\left|\VV^u(u,\tau)\right|\leq b_6|\mu|\eps^{\eta+1}
\]
with  $b_6$ a constant independent of $\mu$ and $\eps$.
\item There exists a generating function $T^u:
\tro^{\out,u}_{\kk_2,d_1}\times\TT_\sigma\rightarrow\CC$ such that
\[
\pa_uT^u(u,\tau)=p_0(u)P^u(u+\VV^u(u,\tau),\tau),
\]
where $P^u$ is the function obtained in Theorem \ref{th:Extensio},
and
 satisfies equation \eqref{eq:HamJacGeneral}. Then, we have that
$(q,p)=(q_0,p_0(u)\ii \pa_u T^u(u,\tau))$ is a
parameterization of the unstable invariant manifold of the form
\eqref{eq:ParameterizationHJ}. Moreover, there exists a constant
$b_7>0$ such that, for $(u,\tau)\in
\tro^{\out,u}_{\kk_2,d_1}\times\TT_\sigma$,
\[
\left| \pa_ uT^u(u,\tau)-\pa_ uT_0(u)\right|\leq b_7|\mu|
\eps^{\eta+1}.
\]
\end{itemize}
\end{theorem}
This theorem is proved in Section \ref{sec:Transicio}.

The final step is to extend the just obtained parameterizations of
the form \eqref{eq:ParameterizationHJ} to the whole \emph{boomerang
domains} $D^u_{\kk,d}$ and $D^s_{\kk,d}$ defined in
\eqref{def:DominisRaros} (see also Figure
\ref{fig:BoomerangDomains}). In particular the whole \emph{boomerang domains}
contain
points up to a distance $\kk \eps$ of the singularities $\pm ia$.

\begin{theorem}\label{th:ExtensioFinal}
Let $\kk_2$ and $d_1$ be the constants given in Theorem
\ref{th:ParamtoHJ},  $d_2<d_1$, $\kk_3>\kk_2$ big enough and
$\eps_0>0$ small enough. Then, for $\mu\in B(\mu_0)$ and
$\eps\in(0,\eps_0)$, the function $T^u(u,\tau)$ obtained in Theorem
\ref{th:ParamtoHJ} can be analytically extended to the domain
$D^{u}_{\kk_3,d_2}\times\TT_\sigma$.

Moreover, there exists a real constant $b_8>0$ independent of $\eps$
and $\mu$, such that for $(u,\tau)\in
D^{u}_{\kk_3,d_2}\times\TT_\sigma$,
\[
\left|\pa_uT^u(u,\tau)-\pa_u T_0(u)\right|\leq
\frac{b_8|\mu|\eps^{\eta+1}}{\left|u^2+a^2\right|^{\ell+1}},
\]
where $T_0$ is the unperturbed separatrix given in \eqref{def:T00}.
\end{theorem}
The proof of this theorem is given in Section
\ref{subsec:ExtensioFinal}.

\begin{remark}
Let us point out that these domains satisfy $D^u_{\kk,d}\subset
D^{\out,u}_{\rr,\kk}$ and $D^s_{\kk,d}\subset D^{\out,s}_{\rr,\kk}$
if $\rr$ is big enough. Therefore, in the case that $p_0(u)$ does
not vanish, Theorem \ref{th:Extensio:Trig} ensures that the
functions $T^{u,s}$ are already defined in $D^u_{\kk,d}$ and
$D^s_{\kk,d}$ respectively.

Let us observe that, if
$\eps$ is small enough, $D_{\kk,\C}^{\inn,\pm,s}\subset D^s_{\kk,d}$
and $D_{\kk,\C}^{\inn,\pm,u}\subset D^u_{\kk,d}$.
\end{remark}

After Theorem \ref{th:Extensio:Trig} and \ref{th:ExtensioFinal} there is no
difference between
the case $p_0(u)\neq 0$, when the invariant
manifolds can be written as graphs globally, and the general case
when $p_0$ can vanish: we have found \emph{boomerang domains}
which intersect the real line and which reach neighborhoods
of size $\kk \eps$ of the singularities where both manifolds
can be written as graphs. This will be the starting point in
our strategy to measure the distance between the invariant manifolds.

\subsection{The asymptotic first order of $\partial_u T^{u,s}$ close to the
singularities $\pm i a$}\label{sec:Aproximacio}

Theorems \ref{th:Extensio:Trig} and \ref{th:ExtensioFinal} are valid
for $\eta \geq \max\{0,\ell-2r\}$. Therefore, 
when $\ell\leq 2r$ the results are true for $\eta \geq 0$. Notice that  if $\ell < 2r$
Theorems \ref{th:Extensio:Trig} and \ref{th:ExtensioFinal} give a
\emph{classical} perturbative result
with respect to the singular parameter $\eps$, in the sense
that the main term of $\partial_u T^{u,s}$ is given by the unperturbed
separatrix $\partial_u T_0$
in the whole \emph{outer} domains.
This fact is not true anymore in the case $\ell-2r\geq 0$ and $\eta=\ell-2r$.
Then we will have to look for
different approximations of the invariant manifolds close to the
singularities $u=\pm ia$, by using suitable solutions of the so-called
\emph{inner} equations.
Consequently, the case $\ell <2r$ is easier to deal with, because it is always
regular
 and there is no need of using \emph{inner} equations to obtain a better
approximation
of $\partial_u T^{u,s}$ near  the singularities $\pm i a$ of $T_0$. 
When $\ell-2r\geq 0$, as we have mentioned in section \ref{sec:heuristic:asymptotic},
we include the regular case $\eta>\ell-2r$ in the
singular one $\eta=\ell-2r$ doing the change of parameter
$\hmu=\mu\eps^{\eta-(\ell-2r)}$.

We separate both cases $\ell <2r$ and $\ell \geq 2r$ in the corresponding
sections below.

\subsubsection{The asymptotic  first order of $\partial_u T^{u,s}$ for the case
$\ell<2r$}\label{sec:Aproximacio:lmenor}

In this section we will assume that $\ell <2r$ and henceforth we are dealing
with values of $\eta \geq 0$.

To obtain the main term of $\partial_u T^{u,s} - \partial_u T_0$
we just need to use classical perturbation theory even in the \emph{inner
domains}
$D_{\kk,\C}^{\inn,\pm,\ast}$, $\ast=u,s$,  defined in
\eqref{def:DominisInnerEnu} (see Figure \ref{fig:Inners}).
Let us observe that, if $u\in D_{\kk,\C}^{\inn,\pm,\ast}$, $\ast=u,s$, then
$\OO(\kk \eps) \leq |u\mp ia| \leq \OO(\eps^{\gamma})$.

The next proposition gives the first order asymptotic terms of
$\pa_u T^{u,s}-\pa_u T_0$ close to $u=ia$, that is in $
D_{\kk,\C}^{\inn,+,\ast}$, $\ast=u,s$.
The study close to $u=-ia$ can be done
analogously.
\begin{proposition}\label{coro:Varietat:FirstOrder:lmenor}
Let us assume $\ell-2r<0$ and $0<\ga< \min\{1,\frac{\ell+1}{r+1}\}$ where $\ga$
is the constant involved in the definition of the inner domains in
\eqref{def:DominisInnerEnu}.
Let us consider the constant $\kk_3$ given by Theorem
\ref{th:ExtensioFinal} and $\C_1>0$ and let us define the constant
\[
\nu^\ast=\min\left\{\nu_1^{\ast},\nu_2^\ast,1-\max\{0,\ell-2r+1\},r,\ell,
\ell+1- (r+1) \ga \right\}>0,
\]
where
\[
\begin{split}
\nu_1^{\ast}&=\min\{(2r-\ell)\ga,1\}\\
 \nu_2^\ast&=\left\{\begin{array}{ll} \ell(1-\gamma) &\text{if }\ell>0\\
                    1-\ga &\text{if }\ell=0
                   \end{array}\right..
\end{split}
\]
Let us also define  the functions
\begin{equation}\label{def:MigMelnikov}
\begin{split}
\TTT^u_0(u,\tau)=&-\mu \eps^{\eta} \int_{-\infty}^0
H_1(q_0(u+t),p_0(u+t),\tau+\eps\ii t)\, dt\\
\TTT^s_0(u,\tau)=&-\mu \eps^{\eta} \int_{+\infty}^0
H_1(q_0(u+t),p_0(u+t),\tau+\eps\ii t)\, dt ,
\end{split}
\end{equation}
where $H_1$ is the function defined in \eqref{def:Ham:Original:perturb:poli} and
\eqref{def:Ham:Original:perturb:trig} and $(q_0(u),p_0(u))$ is the
parameterization
of the unperturbed separatrix given in Hypothesis \textbf{HP2}.
Then, there exists  $\eps_0>0$ and a constant $b_{9}>0$ such that for any
$\eps\in (0,\eps_0)$ and $\mu\in B(\mu_0)$ the following bounds are satisfied.
\begin{itemize}
 \item  If $(u,\tau)\in
D_{\kk_3,\C_1}^{\inn,+,u}\times\TT_\sigma$,
\[
\left|\pa_u T^u(u,\tau)-\pa_u T_0(u)-\pa_u \TTT^u_0(u,\tau)\right|\leq
b_9|\mu|\eps^{\eta-\ell+\nu^\ast}.
\]
 \item  If $(u,\tau)\in D_{\kk_3,\C_1}^{\inn,+,s}\times\TT_\sigma$,
\[
\left|\pa_u T^s(u,\tau)-\pa_s T_0(u)-\pa_u \TTT^s_0(u,\tau)\right|\leq
b_9|\mu|\eps^{\eta-\ell+\nu^\ast}.
\]
\end{itemize}
\end{proposition}
This proposition is proved in Section \ref{sec:Extensio}.

\subsubsection{The first asymptotic order of $\partial_u T^{u,s}$ for the case
$\ell\geq 2r$}
\label{sec:Aproximacio:lmajor}

Theorems \ref{th:Extensio:Trig} and \ref{th:ExtensioFinal} give the existence
of parameterizations of the invariant manifolds of the form
\eqref{eq:ParameterizationHJ} in $D^s_{\kk,d_2}$ and $D^u_{\kk,d_2}$
for $\eps$ small enough and $\kk$ big enough. Nevertheless, when
$\eta=\ell-2r$ the parameterizations of the perturbed invariant
manifolds are not well approximated by the unperturbed separatrix
when $u$ is at a distance of order $\OO(\eps)$ of the singularities
$u=\pm ia$. For this reason, to obtain the first asymptotic order of
the difference between the manifolds, we need to look for better
approximations $T^{u,s}$ in the inner domains defined in
\eqref{def:DominisInnerEnu}. We obtain them through a singular
limit. Since we are dealing with the case $\eta\geq \ell-2r$, the
first step is to define a new parameter
\begin{equation}\label{def:mu:barret}
\hmu=\mu\eps^{\eta-(\ell-2r)}.
\end{equation}
Then, the Hamiltonian $\widehat H$ reads
\begin{equation}\label{def:HamPeriodicaShiftada:muhat}
\begin{split}
\widehat
H(q,p,\tau)=&\frac{p^2}{2}
+V\left(q+\xp(\tau)\right)-V\left(\xp(\tau)\right)-V'\left(\xp(\tau)\right)q\\
&+\hmu\eps^{\ell-2r}\widehat H_1(q,p,\tau)
\end{split}
\end{equation}
and, from $\widehat H$, one can define the Hamiltonian $\ol H$ in
\eqref{def:Hbarra} using again  the change
\eqref{eq:CanviSimplecticSeparatriu}. On the other hand, from
Theorems \ref{th:Extensio:Trig} and \ref{th:ExtensioFinal}, one can
obtain bounds for the parameterizations of the invariant manifolds
in terms of $\hmu$ and $\eps$. We state them for the unstable
manifold. The stable manifold satisfies analogous bounds.

\begin{corollary}\label{coro:CotesOuter:muhat}
Let us consider the constants $\kk_3$ and $d_2$ defined in Theorem
\ref{th:ExtensioFinal}.
Then the function $T^u$ obtained in Theorems  \ref{th:Extensio:Trig} and
\ref{th:ExtensioFinal}, which is defined for $(u,\tau)\in
D^{u}_{\kk_3,d_2}\times\TT_\sigma$, satisfies
\[
\left|\pa_uT^u(u,\tau)-\pa_u T_0(u)\right|\leq
\frac{b_8|\hmu|\eps^{\ell-2r+1}}{\left|u^2+a^2\right|^{\ell+1}},
\]
where $T_0$ is the unperturbed separatrix given in \eqref{def:T00}.
\end{corollary}

We want to study the invariant manifolds close to the singularities
$u=\pm ia$, that is, in the inner domains defined in
\eqref{def:DominisInnerEnu}. Since the study of both invariant
manifolds close either to $u=ia$ or $u=-ia$ is analogous, we only
study them in the domain $D_{\kk,\C}^{\inn,+,u}$. Then, we consider
the change of variables
\begin{equation}\label{def:CanviVarInnerOuter}
z=\eps\ii (u-ia).
\end{equation}
The variable $z$ is called the \emph{inner variable}, in
contraposition to the \emph{outer variable} $u$. We note that, by
definition of $T_0$ in \eqref{def:T00} and using the expansion
around the singularities of $p_0(u)$ in  \eqref{eq:SepartriuAlPol}
and \eqref{eq:SepartriuAlPolTrig}, we have that
\begin{equation*}
\partial_uT_0(\eps z+ ia) =\frac{C_+^{2}}{\eps^{2r}z^{2r}}
\left(1+\OO\left((\eps z)^{1/\beta}\right)\right)
\end{equation*}
and, using the results of Corollary  \ref{coro:CotesOuter:muhat}, we have that
\begin{equation*}
\left|\partial_u T^{u,s}(\eps z+ ia, \tau) - \partial_u T_0(\eps z+
ia)\right|\leq K\frac{|\hmu|}{\eps^{2r}|z|^{\ell+1}}.
\end{equation*}
Hence, in order to catch the terms of the same order in $\eps$, we
scale the generating function as
\begin{equation}\label{eq:FuncioGeneradoraInner}
\psi^{u,s}(z,\tau)=\eps^{2r-1}C_+^{-2} T^{u,s}(ia+\eps z,\tau).
\end{equation}

Then, the Hamilton-Jacobi equation \eqref{eq:HamJacGeneral} reads
\begin{equation}\label{eq:HJGeneralInner}
\pa_\tau\psi +\eps^{2r}C_+^{-2}\ol H\left(ia+\eps z,
\eps^{-2r}C^{2}_+\pa_z\psi,\tau\right)=0,
\end{equation}
where $\ol H$ is the Hamiltonian function defined in
\eqref{def:Hbarra}. The corresponding Hamiltonian is
\begin{equation}\label{Hamiltonia:varInner}
\HH (z,w,\tau)=\eps^{2r}C_+^{-2}\overline H\left(ia+\eps z,
\eps^{-2r}C^{2}_+w,\tau\right).
\end{equation}
We study equation \eqref{eq:HJGeneralInner} in the domain
$\DD_{\kk,\C}^{\inn,+,u}\times\TT_\sigma$, where
\begin{equation}\label{def:DominisInnerEnz}
\begin{split}
\DD_{\kk,\C}^{\inn,+,u}=&\left\{ z\in\CC; ia+\eps z \in
D_{\kk,\C}^{\inn,+,u}\right\}.
\end{split}
\end{equation}

To study equation \eqref{eq:HJGeneralInner}, as a first step it is
natural to study it in the limit case $\eps=0$.  In the polynomial
case it reads
\begin{equation}\label{eq:HJEqInner}
\pa_\tau\psi_0
+\frac{1}{2}z^{2r}\left(\pa_z\psi_0\right)^2-\frac{1}{2z^{2r}}+\frac{\hmu}{
z^\ell}
\sum_{(r-1)k+rl= \ell}\frac{C_+^{k+l-2}}{(1-r)^k}
a_{kl}(\tau)\left(z^{2r}\pa_z\psi_0\right)^l=0.
\end{equation}
The solutions of this equation were studied in detail in
\cite{Baldoma06}, where equation \eqref{eq:HJEqInner} was
rewritten as
\begin{equation}\label{eq:HJEqInner2}
\pa_\tau\psi_0
+\frac{1}{2}z^{2r}\left(\pa_z\psi_0\right)^2-\frac{1}{2z^{2r}}+
\frac{\hmu}{z^\ell}\sum_{l=0}^NA_l(\tau)\left(z^{2r}\pa_z\psi_0\right)^l=0,
\end{equation}
where
\begin{equation}\label{def:FuncionsA}
A_{l}(\tau)=\sum_{(r-1)k+rl=\ell}\frac{C_+^{k+l-2}}{(1-r)^k}
a_{kl}(\tau),
\end{equation}
and $a_{kl}$ are the coefficients of $H_1$ in
\eqref{def:Ham:Original:perturb:poli} and $C_{+}$ is
given in \textbf{HP2}.
This equation is in fact the Hamilton-Jacobi equation associated to
the non-autonomous Hamiltonian
\begin{equation}\label{Hamiltonia:EqInner}
\HH_0(z,w,\tau)=\frac{1}{2}z^{2r}w^2-\frac{1}{2z^{2r}}+\frac{\hmu}{z^\ell}\sum_{
l=0}^NA_l(\tau)\left(z^{2r}w\right)^l,
\end{equation}
which satisfies that $\HH\rightarrow\HH_0$ as $\eps\rightarrow 0$,
where $\HH$ is the Hamiltonian function defined in
\eqref{Hamiltonia:varInner}.

In the trigonometric case, an analogous equation to
\eqref{eq:HJEqInner} is obtained. There are only two differences.
First, one has to consider the definition of $\ell$ given in
\eqref{def:ell} associated to this type of systems. Secondly, in the
trigonometric case, the coefficients in front of $a_{kl}(\tau)$ are
expressed in terms of the coefficients $\wh C^1_\pm$, $\wh C^2_\pm$
and $C_\pm$ in \eqref{eq:SepartriuAlPolTrig}. Taking into account
these facts, one can also define the analogous functions $A_{l}$.

\begin{figure}[h]
\begin{center}
\psfrag{t}{$\arctan\tet$}\psfrag{k1}{$i\kk$}\psfrag{k2}{$-i\kk$}\psfrag{D1}{
$\DD_{\kk,\tet}^{+,u}$}\psfrag{D2}{$\DD_{\kk,\tet}^{+,s}$}
\includegraphics[height=6cm]{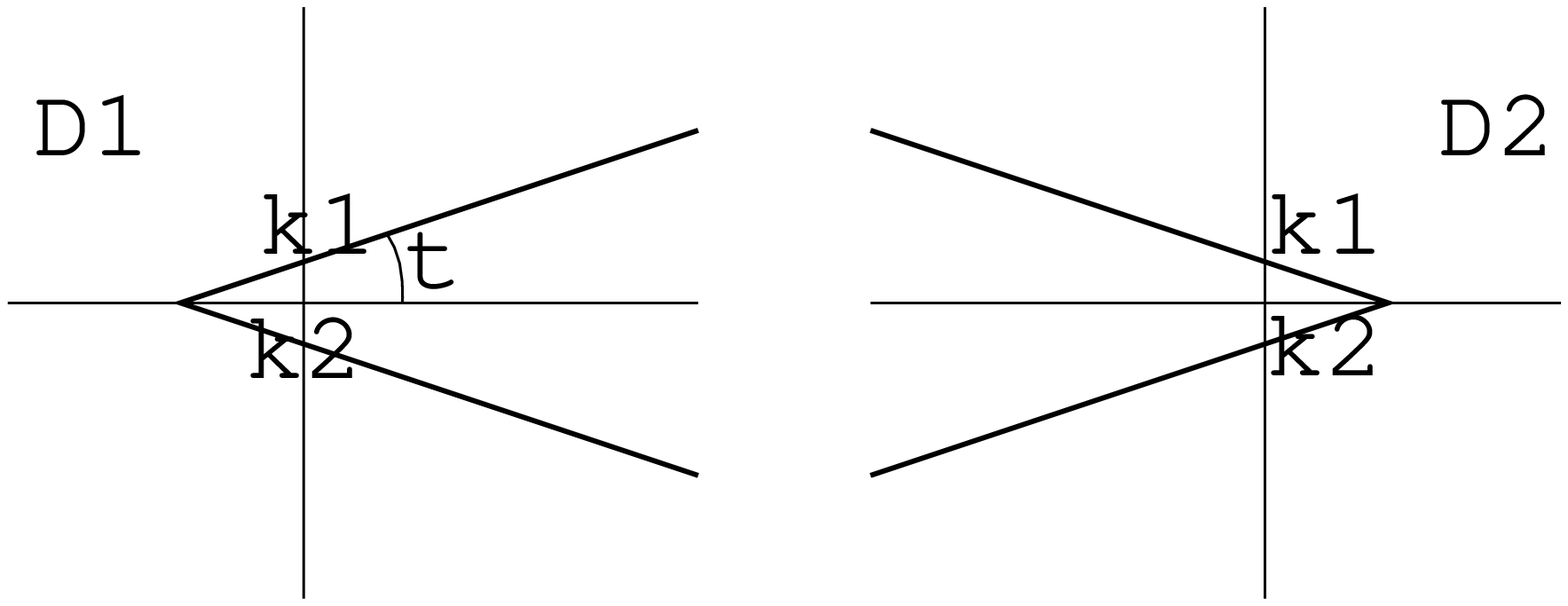}
\end{center}
\caption{\figlabel{fig:DomEqInner} The domains
$\DD_{\kk,\tet}^{+,u}$ and $\DD_{\kk,\tet}^{+,s}$ defined in
\eqref{def:Dominis:eqInner}.}
\end{figure}

The solutions of the Hamilton-Jacobi equation \eqref{eq:HJEqInner2}
were studied in \cite{Baldoma06} in the complex domains
\begin{equation}\label{def:Dominis:eqInner}
\begin{split}
\DD_{\kk,\tet}^{+,u}&=\left\{z\in\CC; \left|\Im z\right|>\tet \, \Re
z+\kk\right\}\\
\DD_{\kk,\tet}^{+,s}&=\left\{z\in\CC; -z\in \DD_{\kk,\tet}^{+,u}\right\}
\end{split}
\end{equation}
for $\kk>0$ and $\tet>0$. Let us observe that, for any $\C>0$,
$\DD_{\kk,\C}^{\inn,+,\ast}\subset\DD_{\kk, \tan\beta_2}^{+,\ast}$
for $\ast=u,s$. Nevertheless, since through the proof we will have
to change the slope  of the domains $\DD_{\kk,\tet}^{+,\ast}$, we
start with a certain fixed slope $\tet_0<\tan \beta_2$ which will be
determined \emph{a posteriori}.

The difference between the stable and unstable manifolds of the
inner equation was studied in the intersection domain
\begin{equation}\label{def:DominisInterseccio:eqInner}
\RRR_{\kk,\tet}^{+}=\DD_{\kk,\tet}^{+,u}\cap \DD_{\kk,\tet}^{+,s}
\cap \left\{z\in \CC; \Im z<0\right\}.
\end{equation}
\begin{figure}[h]
\begin{center}
\psfrag{k}{$-i\kk$}\psfrag{D1}{$\DD_{\kk,\tet}^{+,u}$}\psfrag{D2}{$\DD_{\kk,\tet
}^{+,s}$}\psfrag{R}{$\RRR_{\kk,\tet}^{+}$}
\includegraphics[height=6cm]{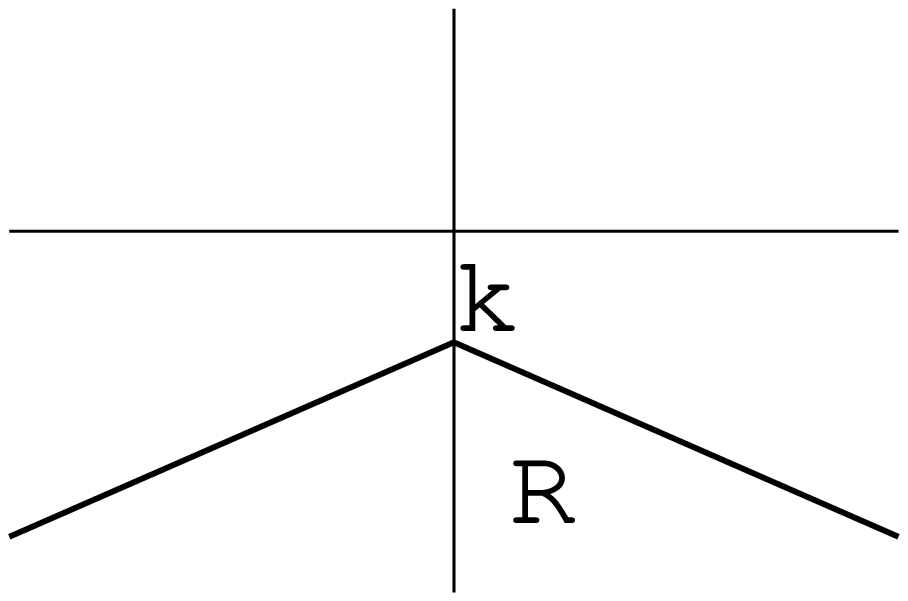}
\end{center}
\caption{\figlabel{fig:DomEqInnerInter} The domain
$\RRR_{\kk,\tet}^{+}$
 defined in \eqref{def:DominisInterseccio:eqInner}.}
\end{figure}
The next theorem gives the main results obtained in \cite{Baldoma06}
about the solutions of equation \eqref{eq:HJEqInner2} and their
difference.

\begin{theorem}\label{th:InnerImma}
Let us consider any fixed $\tet_0>0$. Then, for $\hmu\in B(\hmu_0)$
the following statements are satisfied:
\begin{enumerate}
\item There exists $\kk_4>0$ such that, equation \eqref{eq:HJEqInner2} has
solutions $\psi_0^{\ast}:
\DD_{\kk_4,\tet_0}^{+,\ast}\times\TT_\sigma\rightarrow\CC$,
$\ast=u,s$, of the form
\begin{equation}\label{eq:SolucioInnerHJ}
\psi_0^{u,s}(z,\tau)=-\frac{1}{(2r-1)z^{2r-1}}+\hmu\overline\psi_0^{u,s}
(z,\tau)+K^{u,s},\quad K^{u,s}\in\CC
\end{equation}
where $\overline\psi_0^{u,s}$ are analytic functions in all their
variables. Moreover, the derivatives of $\overline\psi_0^{u,s}$ are
uniquely determined by the condition
\[
\sup_{(z,\tau)\in
\DD_{\kk_4,\tet_0}^{+,\ast}\times\TT_\sigma}\left|z^{\ell+1}\pa_z\ol\psi_0^{\ast
}(z,\tau)\right|<\infty
\]
for $\ast=u,s$.  In fact, one can choose $\ol\psi_0^{u,s}$ such that
\[
\sup_{(z,\tau)\in
\DD_{\kk_4,\tet_0}^{+,\ast}\times\TT_\sigma}\left|z^{\ell}\ol\psi_0^{\ast}(z,
\tau)\right|<\infty
\]
for $\ast=u,s$.
\item There exists  $\kk_5>\kk_4$, analytic functions
$\left\{\chi^{[k]}(\hmu)\right\}_{k\in\ZZ^-}$
defined on $B(\hmu_0)$ and
$g:\RRR_{\kk_5,2\tet_0}^{+}\times\TT_\sigma\rightarrow \CC$ such
that two solutions $\psi_0^{u,s}$ of equation \eqref{eq:HJEqInner2}
of the form  given in \eqref{eq:SolucioInnerHJ} with $K^u=K^s$,
satisfy
\begin{equation}\label{eq:DifAsimptInner:HJ}
\left(\psi^u_0-\psi^s_0\right)(z,\tau)= \hmu\sum_{k<0}
\chi^{[k]}(\hmu)e^{ik\left(z-\tau+\hmu g(z,\tau)\right)}.
\end{equation}
Moreover, the function $g$ satisfies that
\begin{align*}
\sup_{(z,\tau)\in
\RRR_{\kk_5,2\tet_0}^{+}\times\TT_\sigma}\left|z^{\ell-2r}g(z,
\tau)\right|&<\infty
&\text{\,\,\,\,if
}&\ell>2r\\
\sup_{(z,\tau)\in
\RRR_{\kk_5,2\tet_0}^{+}\times\TT_\sigma}\left|\left(\ln|z|\right)\ii
g(z,\tau)\right|&<\infty &\text{\,\,\,\,if }&\ell=2r .
\end{align*}
\end{enumerate}
\end{theorem}

The proof of Theorem \ref{th:InnerImma} is given in
\cite{Baldoma06}.

\begin{remark}
Following the proofs of  \cite{Baldoma06}, it can be easily seen that the
analytic functions  $\left\{\chi^{[k]}(\hmu)\right\}_{k\in\ZZ^-}$ are entire.
\end{remark}

For the case $\ell-2r=0$ we will need
better knowledge of the function $g$ given by Theorem
\ref{th:InnerImma}. The next proposition gives its first asymptotic
terms. First, we define certain functions which will be used in the
statement of the next proposition. Let us consider the functions
$A_j$ defined in \eqref{def:FuncionsA}, then we define
\begin{equation}\label{def:FunctionsQ}
Q_j(\tau)=\sum_{k=j}^N\left(\begin{array}{c}k\\j\end{array}\right)A_k(\tau),
\end{equation}
and functions $F_j$ such that
\begin{equation}\label{def:FunctionsF}
\pa_\tau F_j=Q_j\,\,\text{ and }\,\,\langle F_j\rangle=0,
\end{equation}
which are periodic since $\langle Q_j\rangle=0$.

\begin{remark}\label{remark:DefQIntrinseca}
The functions $Q_j(\tau)$ can be also defined
intrinsically either $\wh H_1^1$ is a polynomial or a trigonometric
polynomial, as
\[
Q_j(\tau)=\frac{1}{j!}C_+^{j-2}\lim_{u\to ia}(u-ia)^{\ell-rj}\pa_p^j
\wh H_1^1(q_0(u),p_0(u),\tau),
\]
where $\wh H_1^1$ is the Hamiltonian defined in
\eqref{def:HamPertorbat:H1} and $C_+$ is given in
\eqref{eq:SepartriuAlPol} and \eqref{eq:SepartriuAlPolTrig}.
\end{remark}

\begin{proposition}\label{coro:gInner}
Let us consider the constant
\begin{equation}\label{def:Constantb}
b=2r\left\langle Q_0F_1+2F_0Q_2\right\rangle,
\end{equation}
where $Q_j$ and $F_j$ are the functions defined in
\eqref{def:FunctionsQ} and \eqref{def:FunctionsF} respectively.
Then, when $\ell-2r=0$, the function $g$ obtained in Theorem
\ref{th:InnerImma}, is of the form
\[
g(z,\tau)=- F_1(\tau)-\hmu b\ln z+\ol g(z,\tau)
\]
and $\ol g$ satisfies
\[
 \sup_{(z,\tau)\in
\RRR_{\kk_5,2\tet_0}^{+}\times\TT_\sigma}\left|z \ol
g(z,\tau)\right|<\infty.
\]
\end{proposition}
To have a better knowledge of the parameterizations of the invariant
manifolds in the inner domains $\DD_{\kk, \C}^{\inn, +,\ast}$,
$\ast=u,s$ in \eqref{def:DominisInnerEnz}, we need to compare the
parameterizations $\psi^{u,s}$, which are solutions of
\eqref{eq:HJGeneralInner} with $\psi_0^{u,s}$ which are solutions of
\eqref{eq:HJEqInner} and have been given in Theorem
\ref{th:InnerImma}.

 Since we have to use the functions and results
obtained in Theorem \ref{th:InnerImma}, we need that
$\DD_{\kk,\C}^{\inn,+,u}\subset\DD_{\kk,2\tet_0}^{+,u}$. To this
end, we impose
\[
\tet_0=\frac{\tan\beta_2}{2}.
\]

We state the next theorem for the unstable invariant manifold. The
stable manifold satisfies analogous properties.

\begin{theorem}\label{th:MatchingHJ}
Let $\ga\in (0,1)$, the constants $\kk_3$ and  $\kk_5$  defined in Theorems
\ref{th:ExtensioFinal} and \ref{th:InnerImma},
$\C_1>0$ and  $\eps_0>0$
 small enough and $\kk_6>\max\{\kk_3,\kk_5\}$ big enough, which might depend on
the previous constants.
Then, for $\eps\in (0,\eps_0)$ and $\hmu\in B(\hmu_0)$, there exists a
constant $b_{10}>0$ such that for $(z,\tau)\in \DD_{\kappa_6,
\C_1}^{\inn,+,u}\times\TT_\sigma$
\[
\left|\pa_z\psi^u(z,\tau)-\pa_z\psi_0^u(z,\tau)\right|\leq
\frac{b_{10}\eps^{\frac{1}{\q}}}{\left|z\right|^{2r-\frac{1}{\q}}},
\]
where $\ga$ enters in the definition of
$\DD_{\kappa_6, \C_1}^{\inn,+,u}$,  $r=\p/\q$ has been defined in
Hypothesis \textbf{HP2}, $\psi^u_0$ is given in Theorem
\ref{th:InnerImma} and $\psi^{u}$ is the scaling of the generating
function $T^u$ given in \eqref{eq:FuncioGeneradoraInner}.
\end{theorem}
The proof of this theorem is given in Section \ref{sec:matching}.

\subsection{Study of the difference between the invariant
manifolds}\label{sec:Difference}

Once we have obtained parameterizations of the invariant manifolds
of the form \eqref{eq:ParameterizationHJ} in the domains
$D^s_{\kk_3,d_2}$ and $D^u_{\kk_3,d_2}$ and studied their first order
approximations close to the singularities, the next step is to study their
difference.

We devote Section \ref{sec:Difference:lmenor} to study the (easier) case
$\ell-2r<0$ and then
in Section \ref{sec:Difference:lmajor} we consider the case $\ell-2r\geq 0$.

\subsubsection{Study of the difference between the invariant manifolds for the
case
$\ell-2r<0$}\label{sec:Difference:lmenor}

We are going to proceed to study the
difference $\partial_u T^u (u,\tau) - \partial_u T^s(u,\tau)$. Recall that in
the case $\ell-2r<0$,
Hypothesis \textbf{HP5} becomes $\eta\geq 0$. Therefore our study includes the
non perturbative case $\eta=0$.

To study the difference between the manifolds, we define
\begin{equation}\label{def:Diferencia:Original}
\Delta(u,\tau)=T^u(u,\tau)-T^s(u,\tau)
\end{equation}
 in the domain
$R_{\kk,d}=D^s_{\kk,d}\cap D^u_{\kk,d}$ which is defined in
\eqref{def:DominiRaro:Interseccio}.

We recall that $p_0(u)\neq 0$ if $u\in R_{\kk,d}$ and hence we can
use the Hamilton-Jacobi equation in this domain.

Subtracting equation \eqref{eq:HamJacGeneral} for both $T^u$ and
$T^s$, one can see that $\Delta$ satisfies the partial differential
equation
\begin{equation}\label{eq:DifferencePDE}
\widetilde\LL_\eps\xi=0,
\end{equation}
where
\begin{equation}\label{def:OperadorAnulador}
\widetilde\LL_\eps=\eps^{-1}\pa_\tau+(1+G(u,\tau))\pa_u
\end{equation}
with
\begin{equation}\label{def:FuncioLAnulador:lmenor}
\begin{split}
G(u,\tau)=&\dps\frac{1}{2p_0^2(u)}\left(\pa_u T_1^u(u,\tau)+\pa_u
T_1^s(u,\tau)\right)\\
&+\frac{\mu\eps^{\eta}}{p_0(u)}\int_0^1 \pa_p \wh H_1 \left(q_0(u),
p_0(u) +\frac{s\pa_u T_1^u(u,\tau)+(1-s)\pa_u
T_1^s(u,\tau)}{p_0(u)}, \tau\right)\, ds,
\end{split}
\end{equation}
where $\wh H_1$ is the function defined in
\eqref{def:ham:ShiftedOP:perturb} and
$T_1^{u,s}(u,\tau)=T^{u,s}(u,\tau)-T_0(u)$ with $\pa_u
T_0(u)=p_0^2(u)$ and $T^{u,s}$ are given in Theorems
\ref{th:Extensio:Trig} and \ref{th:ExtensioFinal}.

Following \cite{Baldoma06}, to obtain the asymptotic expression of
the difference $\Delta$, we take advantage from the fact that it is
a solution of the homogeneous linear partial differential equation
\eqref{eq:DifferencePDE}. In \cite{Baldoma06} it is seen that if
\eqref{eq:DifferencePDE} has a solution $\xi_0$ such that
$(\xi_0(u,\tau),\tau)$ is injective in $R_{\kk,d}\times\TT_\sigma$,
then any solution of equation \eqref{eq:DifferencePDE} defined in
$R_{\kk,d}\times\TT_\sigma$ can be written as
$\xi=\Upsilon\circ\xi_0$ for some function $\Upsilon$.

Following this approach, we begin by looking  for a solution of the form
\begin{equation}\label{def:Solucio:xi0}
\xi_0(u,\tau)=\eps\ii u-\tau+\CCC(u,\tau)
\end{equation}
being $\CCC$ a function $2\pi$-periodic in $\tau$, such that
$(\xi_0(u,\tau),\tau)$ is injective in $R_{\kk,d}\times\TT_\sigma$.

From now on the parameter $\kk$ will be play an important role
in our computations. The next results will deal with big values of
$\kk=\kk(\eps)$ such that $\kk \eps <a$. In particular, in Theorem
\ref{th:CotaDiferencia:lmenor} we will use $\kk = \OO(\log (1/\eps))$.

\begin{theorem}\label{th:CanviFinal:lmenor}
Let $d_2>0$ and $\kk_3>0$ the constants defined in
Theorem \ref{th:ExtensioFinal}, $d_3<d_2$, $\eps_0>0$ small
enough and $\kk_7>\kk_3$ big enough, which might depend on the
previous constants.
Then, for $\eps\in (0,\eps_0)$, $\mu \in B(\mu_0)$ and any
$\kk\geq\kk_7$ such that $\eps\kk<a$, there exists
a real-analytic function
$\CCC:R_{\kk,d_3}\times\TT_\sigma\rightarrow\CC$ such that
$\xi_0(u,\tau)=\eps\ii u-\tau+\CCC(u,\tau)$ is a solution of
\eqref{eq:DifferencePDE} and
\[
\left(\xi_0(u,\tau),\tau\right)=\left(\eps\ii
u-\tau+\CCC(u,\tau),\tau\right)
\]
is injective.

Moreover, there exists a constant $b_{11}>0$ independent of $\mu$,
$\eps$ and $\kk$, such that for $(u,\tau)\in
R_{\kk,d_3}\times\TT_\sigma$,
\[
\begin{split}
\left|\CCC(u,\tau)\right|&\leq
b_{11}|\mu|\eps^{\eta}\\
\left|\pa_u\CCC(u,\tau)\right|&\leq b_{11}\kk\ii|\mu|\eps^{\eta-1}.
\end{split}
\]
\end{theorem}
To study the first order of the difference between the invariant
manifolds, we need a better knowledge of the behavior of the
function $\CCC$  in the inner domains defined in \eqref{def:DominisInnerEnu}.
The next proposition gives the first order asymptotic terms of
$\CCC$ close to $u=ia$, that is in $ D_{\kk,\C}^{\inn,+,u}\cap
D_{\kk,\C}^{\inn,+,s}$. The study close to $u=-ia$ can be done
analogously.
\begin{proposition}\label{coro:Canvi:FirstOrder:lmenor}
Let  $\kk_7$ be given by Theorem
\ref{th:CanviFinal:lmenor} and $\C_1>0$.
Then, for any $\eps_0>0$ and $\kk>\kk_7$ such that $\kk\eps<a$, there exist a
constant
$C(\mu,\eps)$ defined for $(\mu,\eps)\in B(\mu_0) \times (0,\eps_0)$ and
depending real-analytically
in $\mu$ and a constant $b_{12}>0$ such that
$|C(\mu,\eps)|\leq b_{12}|\mu|\eps^\eta$ and, if $(u,\tau)\in \left(
D_{\kk,\C_1}^{\inn,+,u}\cap
D_{\kk,\C_1}^{\inn,+,s}\right)\times\TT_\sigma$,
\[
\left|\CCC(u,\tau)-C(\mu,\eps)\right|\leq
\frac{b_{12}|\mu|\eps^{\eta}}{\kk} .
\]
Moreover, in the case $\eta=0$, there exists a constant $C(\mu)$ such that
$C(\mu,\eps)=C(\mu)+\OO\left(\eps^\nu\right)$ for certain $\nu>0$.
\end{proposition}

The proofs of Theorem \ref{th:CanviFinal:lmenor} and Proposition
\ref{coro:Canvi:FirstOrder:lmenor} are done in Section
\ref{sec:PtFix:canvi:lmenor}.

As we have explained, since $\Delta=T^u-T^s$ is a solution of the
same homogeneous partial differential equation as $\xi_0$ given in
Theorem \ref{th:CanviFinal:lmenor}, there exists a function
$\Upsilon$ such that $\Delta=\Upsilon\circ\xi_0$, which gives
\begin{equation}\label{def:Upsilon:lmenor}
\Delta(u,\tau)=\Upsilon\left(\eps\ii u-\tau+\CCC(u,\tau)\right).
\end{equation}
Since $\Delta$ is $2\pi$-periodic in $\tau$,
we notice that the function $\Upsilon$ is $2\pi$-periodic in its variable.
Therefore, considering the Fourier series of $\Upsilon$ we obtain
\begin{equation}\label{def:Delta:Fourier:lmenor}
\Delta(u,\tau)=\sum_{k\in\ZZ}\Upsilon^{[k]} e^{ik\left(\eps\ii
u-\tau+\CCC(u,\tau)\right)}.
\end{equation}

Now we are going to find the first asymptotic term of $\Delta$. Let us first
observe that
the Melnikov Potential defined in \eqref{def:MelnikovPotentialEnt} can be
defined through
the functions  $\TTT_0^{u,s}$, given in \eqref{def:MigMelnikov}, as
\begin{equation}\label{def:MelnikovPotential}
\TTT_0^u(u,\tau)-\TTT^s_0(u,\tau)=-\mu \eps^{\eta} L(u,\tau).
\end{equation}
Moreover by \eqref{LM},
\begin{equation}\label{def:MelnikovPotential:Fourier}
L(u,\tau)=\sum_{k\in\ZZ}\MM^{[k]}e^{ik\left(\eps\ii u-\tau\right)}.
\end{equation}
In \cite{DelshamsS97} (for the hyperbolic case) and \cite{BaldomaF04} (for the
parabolic case),
it was seen that for $\eta>\ell$, the function $L$ gives the leading term of the
difference between manifolds.
Nevertheless, for the general case $\eta\geq 0$, one has to modify slightly this
function to obtain
the correct first order. Let us define
\begin{equation}\label{def:PrimerOrdre:lmenor}
\Delta_0(u,\tau)=\sum_{k\in\ZZ}\Upsilon_0^{[k]}e^{ik\left(\eps\ii
u-\tau+\CCC(u,\tau)\right)},
\end{equation}
where
\begin{equation}\label{def:PrimerOrdre:FourierCoef:lmenor}
\begin{split}
\Upsilon_0^{[k]}&=-\mu \eps^{\eta} \MM^{[k]}e^{-ik C(\mu,\eps)} \;\;\text{ if
}\;\;k<0\\
\Upsilon_0^{[0]}&=0\\
\Upsilon_0^{[k]}&=-\mu \eps^{\eta} \MM^{[k]}e^{-ik \ol C(\mu,\eps)} \;\;\text{
if }\;\;k>0 ,
\end{split}
\end{equation}
where $C(\mu,\eps)$ is the constant obtained in Proposition
\ref{coro:Canvi:FirstOrder:lmenor}
and $\ol C(\mu,\eps)$ is its complex conjugate. Let us point out that, by
Proposition
\ref{coro:Canvi:FirstOrder:lmenor}, these coefficients satisfy
\[
\Upsilon_0^{[k]}=-\mu \eps^{\eta}
\MM^{[k]}\left(1+\OO\left(|k|\mu\eps^\eta\right)\right).
\]
Next theorem shows that this function $\Delta_0$ gives the first asymptotic
order of \eqref{def:Diferencia:Original}.
{From} now on, in this subsection,
we consider real values of  $\tau\in\TT=\TT_\sigma\cap\RR$. In this setting it
can be easily
seen that the function $\Delta_0$ is real-analytic in $u$.

\begin{theorem}\label{th:CotaDiferencia:lmenor}
Let us  consider  the mean value of $\Upsilon$, $\Upsilon^{[0]}$,
defined in \eqref{def:Delta:Fourier:lmenor}, $s<\nu^\ast$ where $\nu^\ast$ is
the constant defined in Proposition \ref{coro:Varietat:FirstOrder:lmenor} and
$\eps_0>0$ small enough.
Then, there exists a constant $b_{13}>0$ such that for
$\eps\in(0,\eps_0)$ and $\mu\in B(\mu_0)\cap\RR$ and $(u,\tau)\in
\left(R_{s\ln(1/\eps),d_3}\cap\RR\right)\times\TT$, the following
statements are satisfied:
\[
\begin{split}
\left|\Delta(u,\tau)-\Upsilon^{[0]}- \Delta_0(u,\tau)\right|&\leq
\frac{b_{13}|\mu|\eps^{\eta+1-\ell}}{|\ln\eps|}e^{-\dps\tfrac{a}{\eps}}\\
\left|\pa_u\Delta(u,\tau)-\pa_u \Delta_0(u,\tau)\right|&\leq
\frac{b_{13}|\mu|\eps^{\eta-\ell}}{|\ln\eps|}e^{-\dps\tfrac{a}{\eps}}\\
\left|\pa_u^2\Delta(u,\tau)-\pa_u^2 \Delta_0(u,\tau)\right|&\leq
\frac{b_{13}|\mu|\eps^{\eta-1-\ell}}{|\ln\eps|}e^{-\dps\tfrac{a}{\eps}}.
\end{split}\]
\end{theorem}
Let us observe that, using  Lemma \ref{lemma:Melnikov}, the definition
of the coefficients $\Upsilon_0^{[k]}$ in
\eqref{def:PrimerOrdre:FourierCoef:lmenor}
 and Proposition \ref{coro:Canvi:FirstOrder:lmenor}, one can deduce a simpler
leading
term of $\Delta$ in \eqref{def:Diferencia:Original}.
For this purpose let us define the function
\begin{equation}\label{def:Delta00:lmenor}
\Delta_{00}(u,\tau)=\frac{2\mu \eps ^{\eta}}{\eps^{\ell-1}}e^{\dps
-\tfrac{a}{\eps}}\Re\left(f_0e^{iC(\mu,\eps) }e^{-i\left({\dps
\tfrac{u}{\eps}}-\tau+\CCC(u,\tau)\right)}\right),
\end{equation}
where $C(\mu,\eps)$ is  the constant given in Proposition
\ref{coro:Canvi:FirstOrder:lmenor} and $\CCC$ is the function given by
Theorem \ref{th:CanviFinal:lmenor}.

\begin{corollary}\label{coro:Diferencia:Simple:lmenor}
There exists a constant $b_{14}>0$ such that for
$\eps\in(0,\eps_0)$, $\mu\in B(\mu_0)\cap\RR$ and $(u,\tau)\in
\left(R_{s\ln(1/\eps),d_3}\cap\RR\right)\times\TT$, the following
statements are satisfied.
\[
\begin{split}
\left|\Delta(u,\tau)-\Upsilon^{[0]}- \Delta_{00}(u,\tau)\right|&\leq
\frac{b_{14}\left|\mu\right|\eps^{\eta+1-\ell}}{|\ln\eps|}e^{-{\dps\tfrac{a}{
\eps}}}\\
\left|\pa_u\Delta(u,\tau)-\pa_u \Delta_{00}(u,\tau)\right|&\leq
\frac{b_{14}\left|\mu\right|\eps^{\eta-\ell}}{|\ln\eps|}e^{-{\dps\tfrac{a}{\eps}
}}\\
\left|\pa_u^2\Delta(u,\tau)-\pa_u^2 \Delta_{00}(u,\tau)\right|&\leq
\frac{b_{14}\left|\mu\right|\eps^{\eta-1-\ell}}{|\ln\eps|}e^{-{\dps\tfrac{a}{
\eps}}}.\\
\end{split}
\]
\end{corollary}
We devote the rest of this section to prove Theorem
\ref{th:CotaDiferencia:lmenor},
from which, using also Lemma \ref{lemma:Melnikov}, Corollary
\ref{coro:Diferencia:Simple:lmenor} is a direct consequence.

\begin{proof}[Proof of Theorem \ref{th:CotaDiferencia:lmenor}]
For the first part of the proof we consider complex values of
$\mu\in B(\mu_0)$ and later we will restrict to  $\mu\in
B(\mu_0)\cap\RR$.
We define
\[
\wt \Upsilon (\zeta) = \sum_{k\in\ZZ}\wt
\Upsilon^{[k]}e^{ik\zeta},
\]
where $\wt\Upsilon^{[k]}=\Upsilon^{[k]}-\Upsilon^{[k]}_0$.
By \eqref{def:Delta:Fourier:lmenor} and
\eqref{def:PrimerOrdre:lmenor}, the function
$\wt\Delta(u,\tau)=\Delta(u,\tau)-\Delta_0(u,\tau)$ can be written
as
\begin{equation}\label{def:DifResta:lmenor}
\wt\Delta(u,\tau)=\wt \Upsilon\left(\eps\ii
u-\tau+\CCC(u,\tau)\right)=\sum_{k\in\ZZ}\wt
\Upsilon^{[k]}e^{ik\left(\eps\ii u-\tau+\CCC(u,\tau)\right)}.
\end{equation}
Therefore, to obtain the bounds of Theorem \ref{th:CotaDiferencia:lmenor},
it is crucial to bound $\left|\wt\Upsilon^{[k]}\right|$.

The first step is to obtain a bound of $\wt\Delta(u,\tau)$ for
$(u,\tau)\in R_{s\ln\frac{1}{\eps},d_3}\times\TT$. First we bound
this term for $(u,\tau)\in \left(R_{s\ln\frac{1}{\eps},d_3}\cap
D_{s\ln\frac{1}{\eps},\C_1}^{\inn,+,s}\cap
D_{s\ln\frac{1}{\eps},\C_1}^{\inn,+,u}\right)\times\TT$. Recalling
the definitions in \eqref{def:Diferencia:Original}, \eqref{def:MigMelnikov},
\eqref{def:MelnikovPotential}, \eqref{def:MelnikovPotential:Fourier},
\eqref{def:PrimerOrdre:lmenor} and \eqref{def:PrimerOrdre:FourierCoef:lmenor},
we split $\wt\Delta$ as \[\wt\Delta
(u,\tau)=\wt\Delta_1^u (u,\tau)-\wt\Delta_1^s (u,\tau)+\wt\Delta_2
(u,\tau)+\wt\Delta_3 (u,\tau)\] with
\begin{eqnarray}
\wt\Delta_1^{u,s}(u,\tau)&=&T^{u,s}(u,\tau)-T_0(u)-\TTT_0^{u,s}(u,\tau)
\label{def:Lambda1:lmenor}\\
\wt\Delta_2(u,\tau)&=&-\mu \eps^{\eta} \sum_{k<0}\MM^{[k]}e^{ik\left(\eps\ii
u-\tau\right)}\left(1-e^{ik
\left(\CCC(u,\tau)-C(\mu,\eps)\right)}\right)\label{def:Lambda2:lmenor}\\
\wt\Delta_3(u,\tau)&=&-\mu \eps^{\eta} \sum_{k>0}\MM^{[k]}e^{ik\left(\eps\ii
u-\tau\right)}\left(1-e^{ik\left(\CCC(u,\tau)-\ol
C(\mu,\eps)\right)}\right).\label{def:Lambda3:lmenor}
\end{eqnarray}
Applying Proposition \ref{coro:Varietat:FirstOrder:lmenor}, one can see that for
$(u,\tau)\in \left(R_{s\ln\frac{1}{\eps},d_3}\cap
D_{s\ln\frac{1}{\eps},\C_1}^{\inn,+,s}\cap
D_{s\ln\frac{1}{\eps},\C_1}^{\inn,+,u}\right)\times\TT$,
\[
\left|\pa_u\wt\Delta_1^{u,s}(u,\tau)\right|\leq K|\mu|\eps^{\eta-\ell+\nu^\ast},
\]
where $\nu^\ast>0$ is a constant defined in that proposition.

To bound $\wt\Delta_2$, it is enough to apply Lemma \ref{lemma:Melnikov},
Theorem \ref{th:CanviFinal:lmenor} and
Proposition \ref{coro:Canvi:FirstOrder:lmenor} to obtain  that for
$(u,\tau)\in \left(R_{s\ln\frac{1}{\eps},d_3}\cap
D_{s\ln\frac{1}{\eps},\C_1}^{\inn,+,s}\cap
D_{s\ln\frac{1}{\eps},\C_1}^{\inn,+,u}\right)\times\TT$,
\[
\left|\pa_u\wt\Delta_2(u,\tau)\right|\leq
\frac{K|\mu|^2\eps^{2\eta-\ell+s}}{\left|\ln\eps\right|}.
\]
Finally, to bound  $\pa_u\wt\Delta_3$, it is enough to take into
account again Lemma \ref{lemma:Melnikov}, Theorem \ref{th:CanviFinal:lmenor} and
Proposition \ref{coro:Canvi:FirstOrder:lmenor}.
Then, one can see that for
$(u,\tau)\in \left(R_{s\ln\frac{1}{\eps},d_3}\cap
D_{s\ln\frac{1}{\eps},\C_1}^{\inn,+,s}\cap
D_{s\ln\frac{1}{\eps},\C_1}^{\inn,+,u}\right)\times\TT$,
\begin{equation*}
\left|\pa_u\wt\Delta_3(u,\tau)\right|\leq
K|\mu|^2\eps^{2\eta-\ell-s}e^{-\dps\tfrac{2a}{\eps}}.
\end{equation*}
Therefore, from the bounds of $\wt \Delta_1^{u,s}$, $\wt \Delta_2$
and $\wt \Delta_3$ and recalling that by hypothesis $s<\nu^\ast$, we
have that for $(u,\tau)\in \left(R_{s\ln\frac{1}{\eps},d_3}\cap
D_{s\ln\frac{1}{\eps},\C_1}^{\inn,+,s}\cap
D_{s\ln\frac{1}{\eps},\C_1}^{\inn,+,u}\right)\times\TT$,
\begin{equation}\label{eq:CotaDeltaTilde:lmenor:neg}
\left|\pa_u\wt\Delta (u,\tau)\right|\leq
\frac{K|\mu|\eps^{\eta-\ell+s}}{\left|\ln\eps\right|}.
\end{equation}
Reasoning analogously, one can see that  for \[(u,\tau)\in
\left(R_{s\ln\frac{1}{\eps},d_3}\cap
D_{s\ln\frac{1}{\eps},\C_1}^{\inn,-,s}\cap
D_{s\ln\frac{1}{\eps},\C_1}^{\inn,-,u}\right)\times\TT,\] the
function $\pa_u\wt\Delta$ satisfies
\begin{equation}\label{eq:CotaDeltaTilde:lmenor:pos}
\left|\pa_u\wt\Delta(u,\tau)\right|\leq
\frac{K|\mu|\eps^{\eta-\ell+s}}{\left|\ln\eps\right|}.
\end{equation}

Finally, for $(u,\tau)\in
\left(R_{s\ln\frac{1}{\eps},d_3}\cap
D_{\C_1\eps^\ga,\rr_4}^{\out,s}\cap
D_{\C_1\eps^\ga,\rr_4}^{\out,u}\right)\times\TT$, we decompose
$\wt{\Delta}(u,\tau) = (T^{u}(u, \tau) - T_0(u) ) - (T^s(u, \tau) - T_0(u)) -
\Delta_0(u, \tau)$.
Using Theorems \ref{th:Extensio:Trig}, \ref{th:ExtensioFinal},
 and \ref{th:CanviFinal:lmenor} and also Lemma \ref{lemma:Melnikov}, one can
easily see that
\begin{equation*}
 |\pa_u \Delta (u,\tau)| \leq K|\mu| \eps^{\eta + 1 - \gamma(\ell+1)}
\end{equation*}
provided
$|u-ia|\geq \OO(\eps^\ga)$. This bound is smaller than
\eqref{eq:CotaDeltaTilde:lmenor:neg} and
\eqref{eq:CotaDeltaTilde:lmenor:pos} due to the fact that $(\ell+1)(1-\gamma)
>\nu^{\ast} >s$
(see Proposition \ref{coro:Varietat:FirstOrder:lmenor} for the definition of
$\nu^{\ast}$).

Taking into account \eqref{eq:CotaDeltaTilde:lmenor:neg} and
\eqref{eq:CotaDeltaTilde:lmenor:pos}, one can
conclude that for $\mu\in B(\mu_0)\cap\RR$,
\begin{equation}\label{eq:CotaDeltaTilde:lmenor}
\left|\pa_u\wt\Delta(u,\tau)\right|\leq
\frac{K|\mu|\eps^{\eta-\ell+s}}{\left|\ln\eps\right|}.
\end{equation}
The second step of the proof is to consider the change of variables
$(w,\tau)=(u+\eps\CCC(u,\tau),\tau)$. By Theorem
\ref{th:CanviFinal:lmenor}, one can easily see that it is a diffeomorphism
from $R_{s\ln(1/\eps),d_3}\times\TT$ onto its image $\wt R\times\TT$.
Denoting by $\wt\Upsilon'$ the derivative of the function $\wt
\Upsilon$ (see \eqref{def:DifResta:lmenor}), we define the function
\[
\Theta(w,\tau)=\wt\Upsilon' \left(\eps\ii w-\tau\right),
\]
on  $\wt R\times\TT$ which, by construction,
satisfies
\begin{equation}\label{eq:ThetaDelta:lmenor} \Theta
(u+\eps\CCC(u,\tau),\tau)=\left(\frac{1}{\eps}+\pa_u\CCC(u,\tau)\right)\ii
\pa_u\wt\Delta (u,\tau).
\end{equation}
Moreover, as $\Theta(w,\tau)$ is periodic in $\tau$, it can be also
written  as
\[
\Theta(w,\tau)=\sum_{k\in \ZZ}\Theta^{[k]}(w)e^{ik\tau}.
\]
Then, for any $w\in\wt R$, the Fourier coefficients satisfy
\[
ik\wt\Upsilon^{[k]}=\Theta^{[-k]}(w)e^{-ik\dps\tfrac{w}{\eps}}.
\]
Now, taking advantage of the fact that the coefficients
$\wt\Upsilon^{[k]}$ do not depend on $w$, we will obtain sharp bounds for
the coefficients $\wt\Upsilon^{[k]}$ with $k<0$. Since we are
dealing with real analytic functions, the coefficients
$\wt\Upsilon^{[k]}$ with $k>0$ will satisfy the same bounds. Let us
consider $w=w^\ast=u^\ast+\eps\CCC(u^\ast,0)$ with $u^\ast=
i(a-s\eps\ln(1/\eps))$. Then,
\[
\begin{split}
\left|\wt\Upsilon^{[k]}\right|&\leq |k|\ii\sup_{w\in\wt
R}\left|\Theta^{[-k]}(w)\right|e^{-{\dps\tfrac{|k|}{\eps}}\left(a-s\eps\ln\frac{
1}{\eps}\right)-|k|\Im\left(
\CCC\left(u^\ast,0\right)\right)}\\
&\leq |k|\ii\sup_{(w,\tau)\in\wt
R\times\TT}\left|\Theta(w,\tau)\right|e^{-{\dps\tfrac{|k|}{\eps}}
\left(a-s\eps\ln\frac{1}{\eps}\right)-|k|\Im\left(
\CCC\left(u^\ast,0\right)\right)}.
\end{split}
\]
Then, taking into account \eqref {eq:ThetaDelta:lmenor} and Theorem
\ref{th:CanviFinal:lmenor}, we have that for $k<0$,
\[
\left|\wt\Upsilon^{[k]}\right|\leq K\eps\sup_{(u,\tau)\in
R_{s\ln(1/\eps),d_3}\times\TT}\left|\pa_u\wt\Delta(u,\tau)\right|e^{-{\dps\tfrac
{|k|}{\eps}}\left(a-s\eps\ln\frac{1}{\eps}\right)-|k|\Im\left(
\CCC\left(u^\ast,0\right)\right)}.
\]
Therefore, to obtain the bounds for $\wt\Upsilon^{[k]}$ with $k<0$,
it only remains to use bounds \eqref{eq:CotaDeltaTilde:lmenor} and the
properties of $\CCC$ given
in Theorem \ref{th:CanviFinal:lmenor} and Proposition
\ref{coro:Canvi:FirstOrder:lmenor}.
Then, we obtain that for $k<0$
\begin{equation*}
\left| \wt\Upsilon^{[k]}\right|\leq
\frac { K|\mu| \eps^{\eta} e^{-\frac{a}{\eps}}}{|\ln \eps | \eps^{\ell-1}}
e^{-\frac{|k|-1}{\eps} \big (a+\eps s \log \eps +b_{11} |\mu | \eps^{\eta+1}\big
)}.
\end{equation*}
Finally, the bounds of
$\wt\Upsilon^{[k]}$ lead easily to the desired bounds of
$\wt\Delta(u,\tau)$ for
$(u,\tau)\in\left(R_{s\ln(1/\eps),d_3}\cap\RR\right)\times\TT$.
\end{proof}

\subsubsection{Study of the difference between the invariant manifolds for the
case
$\ell-2r\geq 0$}\label{sec:Difference:lmajor}

Recall that when $\ell-2r\geq0$, Hypothesis
\textbf{HP5} becomes $\eta\geq \ell-2r$. For this reason, as we did in Section
\ref{sec:Aproximacio:lmajor},
we will denote $\hat \mu = \mu \eps ^{\eta-\ell+2r}$. Let us emphasize, that the
regular case $\eta>\ell-2r$ in this new setting corresponds to $\hmu\rightarrow
0$ as $\eps\rightarrow 0$.

As we have done for the case $\ell-2r<0$ in
Section \ref{sec:Difference:lmenor}, we consider the function
$\Delta(u,\tau)=T^u(u,\tau)-T^s(u,\tau)$ defined in
\eqref{def:Diferencia:Original}
 in the domain
$R_{\kk,d}=D^s_{\kk,d}\cap D^u_{\kk,d}$ defined in
\eqref{def:DominiRaro:Interseccio} (see also Figure
\ref{fig:BoomInter}).

Now $\Delta$ satisfies the partial differential equation
\begin{equation}\label{pdetilde}
\widetilde\LL_\eps\xi=0,
\end{equation}
where $\widetilde\LL_{\eps}$ is the operator defined in
\eqref{def:OperadorAnulador} and $G$ now is
\begin{equation}\label{def:FuncioLAnulador:lmajor}
\begin{split}
G(u,\tau)=&\dps\frac{1}{2p_0^2(u)}\left(\pa_u T_1^u(u,\tau)+\pa_u
T_1^s(u,\tau)\right)\\
&+\frac{\hmu\eps^{\ell-2r}}{p_0(u)}\int_0^1 \pa_p \wh H_1
\left(q_0(u), p_0(u) +\frac{s\pa_u T_1^u(u,\tau)+(1-s)\pa_u
T_1^s(u,\tau)}{p_0(u)}, \tau\right)\, ds,
\end{split}
\end{equation}
where $\wh H_1$ is the function defined in
\eqref{def:ham:ShiftedOP:perturb} and
$T^{u,s}(u,\tau)=T_0(u)+T_1^{u,s}(u,\tau)$ with $\pa_u
T_0(u)=p_0^2(u)$ and $T_1^{u,s}$ are given in Theorems
\ref{th:Extensio:Trig} and \ref{th:ExtensioFinal}. Let us point out
that the only difference between the function $G$ defined in
\eqref{def:FuncioLAnulador:lmajor} from the one defined in
\eqref{def:FuncioLAnulador:lmenor} is the dependence on the
parameters. The first one depends on $\mu$ and $\eps$ whereas the
second one depends on $\hmu$, which has been defined in terms of
$\mu$ and $\eps$ in \eqref{def:mu:barret}.

As we have done in Section \ref{sec:Difference:lmenor}, to obtain
the asymptotic expression of the difference $\Delta$, we look for a
solution  $\xi_0$ of \eqref{eq:DifferencePDE} of the form
\[
\xi_0(u,\tau)=\eps\ii u-\tau+\CCC(u,\tau)
\]
with $\CCC$ a function $2\pi$-periodic in $\tau$, such that
$(\xi_0(u,\tau),\tau)$ is injective in $R_{\kk,d}\times\TT_\sigma$.
Then, we will write $\Delta$ as $\xi=\Upsilon\circ\xi_0$ for some
function $\Upsilon$.

\begin{theorem}\label{th:CanviFinal}
Let us consider  the constants $d_2>0$ defined in
Theorem \ref{th:ExtensioFinal} and $\kk_6>0$ in Theorem \ref{th:MatchingHJ},
$d_3<d_2$ and $\eps_0>0$ small enough and $\kk_8>\kk_6$ big enough, which
might depend on the previous constants. Then, for $\eps\in
(0,\eps_0)$, $\mu\in B(\mu_0)$ and any $\kk\geq\kk_8$ such that
$\eps\kk<a$, there exists a real-analytic function
$\CCC(u,\tau):R_{\kk,d_3}\times\TT_\sigma\rightarrow\CC$ such that
$\xi_0(u,\tau)=\eps\ii u-\tau+\CCC(u,\tau)$ is solution of
\eqref{pdetilde} and
\[
\left(\xi_0(u,\tau),\tau\right)=\left(\eps\ii
u-\tau+\CCC(u,\tau),\tau\right)
\]
is injective.

Moreover, there exists a constant $b_{15}>0$ independent of $\mu$,
$\eps$ and $\kk$, such that for $(u,\tau)\in
R_{\kk,d_3}\times\TT_\sigma$,
\begin{itemize}
\item If $\ell-2r>0$,
\[
\begin{split}
\left|\CCC(u,\tau)\right|&\leq
\frac{b_{15}\left|\hmu\right|\eps^{\ell-2r}}{\left|u^2+a^2\right|^{\ell-2r}}\\
\left|\pa_u\CCC(u,\tau)\right|&\leq
\frac{b_{15}\left|\hmu\right|\eps^{\ell-2r-1}}{\kk\left|u^2+a^2\right|^{\ell-2r}
}.
\end{split}
\]
\item If $\ell-2r=0$,
\[
\begin{split}
\left|\CCC(u,\tau)\right|&\leq b_{15}\left|\hmu\right|\ln\left|u^2+a^2\right|\\
\left|\pa_u\CCC(u,\tau)\right|&\leq
\frac{b_{15}\left|\hmu\right|}{\left|u^2+a^2\right|}.
\end{split}
\]
\end{itemize}
\end{theorem}
To study the first order of the difference between the invariant
manifolds when $\ell-2r=0$, we need a better knowledge of the
behavior of the function $\CCC$ in the inner domains
\eqref{def:DominisInnerEnu}. The next proposition gives the first
order asymptotic terms of $\CCC$ close to $u=ia$. The study close to
$u=-ia$ can be done analogously.
\begin{proposition}\label{coro:Canvi:FirstOrder}
Assume $\ell=2r$. Let $\C_1$ be a constant as in Theorem
\ref{th:MatchingHJ}. We consider $\C_2>\C_1$ and
\begin{equation}\label{eq:cota:canvi:gamma}
\frac{\q}{\q+1}<\ga<1,
\end{equation}
where $r=\p/\q$ has been defined in Hypothesis \textbf{HP2}.

Then, for any $\eps_0>0$, there exist a constant $C(\hmu,\eps)$
defined for $(\hmu,\eps) \in B(\hmu_0) \times (0,\eps_0)$ depending
real-analytically in
$\hmu$ and a constant $b_{16}>0$ such that $|C(\hmu,\eps)|\leq
b_{16}|\hmu|$ and, if $(u,\tau)\in \left(
D_{\kk_8,\C_2}^{\inn,+,u}\cap
D_{\kk_8,\C_2}^{\inn,+,s}\right)\times\TT_\sigma$,
\[
\left|\CCC(u,\tau)-C(\hmu,\eps)+\mu F_1(\tau)+\hmu^2b \ln
(u-ia)\right|\leq \frac{b_{16}|\hmu|\eps}{|u-ia|}.
\]
We recall that $\ga$ enters in the definitions of
$D_{\kk_8,\C_2}^{\inn,+,u}$ and $D_{\kk_8,\C_2}^{\inn,+,s}$, $\CCC$
is the function given in Theorem \ref{th:CanviFinal} and the
function $F_1$ and the constant $b$ have been defined in
\eqref{def:FunctionsF} and \eqref{def:Constantb} respectively.

Therefore, if we consider the function $g$ given in Theorem
\ref{th:InnerImma}, by Proposition \ref{coro:gInner}, there exists a
constant $b_{17}>0$ such that, if $(u,\tau)\in \left(
D_{\kk_8,\C_2}^{\inn,+,u}\cap
D_{\kk_8,\C_2}^{\inn,+,s}\right)\times\TT_\sigma$,
\[
\left|\CCC(u,\tau)-C(\hmu,\eps)+\hmu^2b \ln\eps -\hmu
g\left(\eps\ii(u-ia),\tau\right)\right|\leq
\frac{b_{17}|\hmu|\eps}{|u-ia|}.
\]
Moreover, there exists a constant $C(\hmu)$ such that  $C(\hmu,\eps)$ satisfies
$C(\hmu,\eps)=C(\hmu)+\OO\left(\eps^\nu\right)$ for a certain $\nu>0$.
\end{proposition}
The proofs of Theorem \ref{th:CanviFinal} and Proposition
\ref{coro:Canvi:FirstOrder} are done in Section
\ref{sec:PtFix:canvi}.

As we have explained in Section \ref{sec:Difference:lmenor}, since
$\Delta$ is a solution of the same homogeneous linear partial
differential equation as $\xi_0$ given by Theorem
\ref{th:CanviFinal}, there exists a $2\pi$-periodic function $\Upsilon$ such
that
$\Delta=\Upsilon\circ\xi_0$, which gives
\begin{equation}\label{def:Delta:xi0}
\Delta(u,\tau)=\Upsilon \left(\eps\ii u-\tau+\CCC(u,\tau)\right).
\end{equation}
and considering its Fourier series we have
\begin{equation}\label{def:Delta:Fourier}
\Delta(u,\tau)=\sum_{k\in\ZZ}\Upsilon^{[k]} e^{ik\left(\eps\ii
u-\tau+\CCC(u,\tau)\right)}.
\end{equation}

Now we are going to find the first asymptotic term of $\Delta$ which
will be strongly related with
$(\psi^{u}_0-\psi^{s}_{0})(\eps^{-1}(u-ia),\tau)$, being
$\psi_0^{u,s}$ the solutions of the inner equation given in Theorem
\ref{th:InnerImma}. We introduce the auxiliary function
\begin{equation}\label{def:Delta0:Fourier}
\Delta_0^+(u,\tau)=\sum_{k<0}\Upsilon_0^{[k]}e^{ik\left(\eps\ii
u-\tau+\CCC(u,\tau)\right)}
\end{equation}
with
\begin{align}
\Upsilon_0^{[k]}&=\frac{C_+^2\hmu}{\eps^{2r-1}}\chi^{[k]}(\hmu)
e^{-\dps\tfrac{|k|a}{\eps}} &\text{ if
}\ell-2r>0\label{def:Dif0:coefsFourier:lmajor}\\
\Upsilon_0^{[k]}&=\frac{C_+^2\hmu}{\eps^{2r-1}}\chi^{[k]}(\hmu)
e^{-{\dps\tfrac{|k|a}{\eps}}-i|k|(-C(\hmu,\eps)+\hmu^2b\ln\eps)}
&\text{ if }\ell-2r=0,\label{def:Dif0:coefsFourier:ligual}
\end{align}
where $\left\{\chi^{k}(\hmu)\right\}_{k<0}$ are the coefficients
given in Theorem \ref{th:InnerImma} and  $C(\hmu,\eps)$ and $b$ are
the constants obtained in Propositions \ref{coro:Canvi:FirstOrder}
and \ref{coro:gInner} respectively. The scaling $C_+^2/\eps^{2r-1}$
comes from the inner change in \eqref{eq:FuncioGeneradoraInner}.

We also introduce \begin{equation*} \Delta_0^-(u,\tau) =
\sum_{k>0}{\Upsilon_0^{[k]}}e^{ik\left(\eps\ii
u-\tau+\CCC(u,\tau)\right)}
\end{equation*}
with
\begin{align}
\Upsilon_0^{[k]}&=\frac{\ol
C_+^2\hmu}{\eps^{2r-1}}\ol\chi^{[-k]}(\hmu)
e^{-\dps\tfrac{|k|a}{\eps}} &\text{ if
}\ell-2r>0\label{def:Dif0:coefsFourier:lmajor:positiu}\\
\Upsilon_0^{[k]}&=\frac{\ol C_+^2\hmu}{\eps^{2r-1}}\ol
\chi^{[-k]}(\hmu) e^{-{\dps\tfrac{|k|a}{\eps}}+i|k|(-\ol
C(\hmu,\eps)+\hmu^2\ol b\ln\eps)} &\text{ if
}\ell-2r=0\label{def:Dif0:coefsFourier:ligual:positiu}.
\end{align}
The function $\Delta_0^-(u,\tau)$  corresponds to the difference of
the solutions of the inner equation close to $u=-ia$ if $\hmu,\tau
\in \RR$. We note that, taking $\tau, \hmu\in \RR$, $\Delta_0^{-}$
is nothing but the complex conjugate of $\Delta_0^{+}$. In fact, as
we know that $\Delta$ is a real analytic function in the $u$
variable for real values of $\hmu,\tau$, we can define $\Delta_0^-$
as the function that satisfies that $\Delta_0=\Delta_0^+ +
\Delta_0^-$ is also a real analytic function in the same sense as
explained before for $\Delta$.

We will see that the first order of $\Delta$ is given by
\begin{equation}\label{Def:Dif:PrimerOrdreSencer}
\Delta_0(u,\tau)=\Delta_0^+(u,\tau)+\Delta_0^-(u,\tau).
\end{equation}
Let us point out that it can be written as
\begin{equation}\label{def:Delta0sencer:Fourier}
\Delta_0(u,\tau)=\sum_{k\in\ZZ\setminus\{0\}}\Upsilon_0^{[k]}e^{ik\left(\eps\ii
u-\tau+\CCC(u,\tau)\right)},
\end{equation}
where $\Upsilon_0^{[k]}$ are defined either by
\eqref{def:Dif0:coefsFourier:lmajor} and
\eqref{def:Dif0:coefsFourier:lmajor:positiu} in the case $\ell-2r>0$
or by \eqref{def:Dif0:coefsFourier:ligual} and
\eqref{def:Dif0:coefsFourier:ligual:positiu} in the case
$\ell-2r=0$. For convenience we introduce $\Upsilon_0^{[0]}=0$. From
now on, in this subsection, we consider real values of
$\tau\in\TT_\sigma\cap\RR$.

\begin{theorem}\label{th:CotaDiferencia}
Let us consider  the mean value of $\Upsilon$, $\Upsilon^{[0]}$,
defined in \eqref{def:Delta:Fourier}, $s<1/\q$, where $r=\p/\q$ is
defined in Hypothesis \textbf{HP2}, and $\eps_0>0$ small enough.
Then, there exists a constant $b_{18}>0$ such that for
$\eps\in(0,\eps_0)$ and $\hmu\in B(\hmu_0)\cap\RR$ and $(u,\tau)\in
\left(R_{s\ln(1/\eps),d_3}\cap\RR\right)\times\TT$, the following
statements are satisfied.
\begin{itemize}
\item If $\ell-2r>0$,
\[
\begin{split}
\left|\Delta(u,\tau)-\Upsilon^{[0]}- \Delta_0(u,\tau)\right|&\leq
\frac{b_{18}|\hmu|}{\eps^{2r-1}|\ln\eps|^{\ell-2r}}e^{-\dps\tfrac{a}{\eps}}\\
\left|\pa_u\Delta(u,\tau)-\pa_u \Delta_0(u,\tau)\right|&\leq
\frac{b_{18}|\hmu|}{\eps^{2r}|\ln\eps|^{\ell-2r}}e^{-\dps\tfrac{a}{\eps}}\\
\left|\pa_u^2\Delta(u,\tau)-\pa_u^2 \Delta_0(u,\tau)\right|&\leq
\frac{b_{18}|\hmu|}{\eps^{2r+1}|\ln\eps|^{\ell-2r}}e^{-\dps\tfrac{a}{\eps}}.\\
\end{split}\]
\item If $\ell-2r=0$,
\[
\begin{split}
\left|\Delta(u,\tau)-\Upsilon^{[0]}- \Delta_0(u,\tau)\right|&\leq
\frac{b_{18}|\hmu|}{\eps^{2r-1}|\ln\eps|}e^{-{\dps\tfrac{a}{\eps}}+\hmu^2\Im
b\ln\eps}\\
\left|\pa_u\Delta(u,\tau)-\pa_u \Delta_0(u,\tau)\right|&\leq
\frac{b_{18}|\hmu|}{\eps^{2r}|\ln\eps|}e^{-{\dps\tfrac{a}{\eps}}+\hmu^2\Im
b\ln\eps}\\
\left|\pa_u^2\Delta(u,\tau)-\pa_u^2 \Delta_0(u,\tau)\right|&\leq
\frac{b_{18}|\hmu|}{\eps^{2r+1}|\ln\eps|}e^{-{\dps\tfrac{a}{\eps}}+\hmu^2\Im
b\ln\eps}.\\
\end{split}
\]
\end{itemize}
\end{theorem}
We observe that $\partial_u \Delta_0$ gives the correct asymptotic
prediction of $\partial_u \Delta$ if $\Upsilon^{[-1]}_0\neq 0$. In
fact, we only need this coefficient to give a simpler leading term
of the asymptotic formula. For this purpose let us define the
function
\begin{equation}\label{def:funciof}
f\left(\hmu\right)=C_+^2 \chi^{[-1]}\left(\hmu\right),
\end{equation}
where $C_+$ is the constant defined in \eqref{eq:SepartriuAlPol} or
\eqref{eq:SepartriuAlPolTrig} and $\chi^{[-1]}(\hmu)$ is the
constant given in Theorem \ref{th:InnerImma}. Let us point out that
the zeros of $f(\hmu)$ correspond to the zeros of $\chi^{[-1]}(\hmu)$. We
define
\begin{align}
\Delta_{00}(u,\tau)&=\frac{2\hmu}{\eps^{2r-1}}e^{\dps
-\tfrac{a}{\eps}}\Re\left(f(\hmu)e^{-i\left({\dps
\tfrac{u}{\eps}}-\tau+\CCC(u,\tau)\right)}\right)&\text{ if
}\ell-2r>0\label{def:Delta00:lmajor}\\
\Delta_{00}(u,\tau)&=\frac{2\hmu}{\eps^{2r-1}}e^{{\dps
-\tfrac{a}{\eps}}}\Re\left(f(\hmu)e^{-i\left(\hmu^2b\ln\eps-C(\hmu,\eps)\right)}
e^{-i\left({\dps
\tfrac{u}{\eps}}-\tau+\CCC(u,\tau)\right)}\right)&\text{ if
}\ell-2r=0,\label{def:Delta00:ligual}
\end{align}
where $b$  is the constant defined in \eqref{def:Constantb},
$C(\hmu,\eps)$ the constant given in Proposition
\ref{coro:Canvi:FirstOrder} and $\CCC$ the function given by
Theorem \ref{th:CanviFinal}.

\begin{corollary}\label{coro:Diferencia:Simple}
There exists a constant $b_{19}>0$ such that for
$\eps\in(0,\eps_0)$, $\hmu\in B(\hmu_0)\cap\RR$ and $(u,\tau)\in
\left(R_{s\ln(1/\eps),d_3}\cap\RR\right)\times\TT$, the following
statements are satisfied.
\begin{itemize}
\item If $\ell-2r>0$,
\[
\begin{split}
\left|\Delta(u,\tau)-\Upsilon^{[0]}- \Delta_{00}(u,\tau)\right|&\leq
\frac{b_{19}|\hmu|}{\eps^{2r-1}|\ln\eps|^{\ell-2r}}e^{-\dps\tfrac{a}{\eps}}\\
\left|\pa_u\Delta(u,\tau)-\pa_u \Delta_{00}(u,\tau)\right|&\leq
\frac{b_{19}|\hmu|}{\eps^{2r}|\ln\eps|^{\ell-2r}}e^{-\dps\tfrac{a}{\eps}}\\
\left|\pa_u^2\Delta(u,\tau)-\pa_u^2 \Delta_{00}(u,\tau)\right|&\leq
\frac{b_{19}|\hmu|}{\eps^{2r+1}|\ln\eps|^{\ell-2r}}e^{-\dps\tfrac{a}{\eps}}.\\
\end{split}\]
\item If $\ell-2r=0$,
\[
\begin{split}
\left|\Delta(u,\tau)-\Upsilon^{[0]}- \Delta_{00}(u,\tau)\right|&\leq
\frac{b_{19}\left|\hmu\right|}{\eps^{2r-1}|\ln\eps|}e^{-{\dps\tfrac{a}{\eps}}
+\hmu^2\Im b\ln\eps}\\
\left|\pa_u\Delta(u,\tau)-\pa_u \Delta_{00}(u,\tau)\right|&\leq
\frac{b_{19}\left|\hmu\right|}{\eps^{2r}|\ln\eps|}e^{-{\dps\tfrac{a}{\eps}}
+\hmu^2\Im b\ln\eps}\\
\left|\pa_u^2\Delta(u,\tau)-\pa_u^2 \Delta_{00}(u,\tau)\right|&\leq
\frac{b_{19}\left|\hmu\right|}{\eps^{2r+1}|\ln\eps|}e^{-{\dps\tfrac{a}{\eps}}
+\hmu^2\Im b\ln\eps}.\\
\end{split}
\]
\end{itemize}
\end{corollary}
We devote the rest of this section to prove Theorem
\ref{th:CotaDiferencia}, from which Corollary
\ref{coro:Diferencia:Simple} is a direct consequence.

%First, we state a well known lemma giving bounds of the Fourier
%coefficients of analytic periodic functions.

%\begin{lemma}\label{lemma:Fourier}
%Let $f(w)$ be an analytic $2\pi$-periodic function defined in the
%strip $|\Im w|\leq\eta$. Then, its  Fourier coefficients satisfy
%\begin{equation}\label{eq:Cota:Fourier}
%\left|f^{[k]}\right|\leq \sup_{|\Im
%w|=\eta}\left|f(w)\right|e^{-|k|\eta}.
%\end{equation}
%\end{lemma}

\begin{proof}[Proof of Theorem \ref{th:CotaDiferencia}]
For the first part of the proof we consider complex values of
$\hmu\in B(\hmu_0)$ and later we will restrict to  $\hmu\in
B(\hmu_0)\cap\RR$. By \eqref{def:Delta:Fourier} and
\eqref{def:Delta0sencer:Fourier}, the function
$\wt\Delta(u,\tau)=\Delta(u,\tau)-\Delta_0(u,\tau)$ can be written
as
\begin{equation}\label{def:DifResta}
\wt\Delta(u,\tau)=\wt \Upsilon\left(\eps\ii
u-\tau+\CCC(u,\tau)\right)=\sum_{k\in\ZZ}\wt
\Upsilon^{[k]}e^{ik\left(\eps\ii u-\tau+\CCC(u,\tau)\right)},
\end{equation}
where $\wt\Upsilon^{[k]}=\Upsilon^{[k]}-\Upsilon^{[k]}_0$.
Therefore, to obtain the bounds of Theorem \ref{th:CotaDiferencia},
it is crucial to bound $\left|\wt\Upsilon^{[k]}\right|$.

The first step is to obtain a bound of $\wt\Delta(u,\tau)$ for
$(u,\tau)\in R_{s\ln\frac{1}{\eps},d_3}\times\TT$. First we bound
this term for $(u,\tau)\in \left(R_{s\ln\frac{1}{\eps},d_3}\cap
D_{s\ln\frac{1}{\eps},\C_2}^{\inn,+,s}\cap
D_{s\ln\frac{1}{\eps},\C_2}^{\inn,+,u}\right)\times\TT$. Recalling
the definitions of \eqref{def:Diferencia:Original},
\eqref{Def:Dif:PrimerOrdreSencer}, \eqref{def:Delta0:Fourier} and
\eqref{eq:DifAsimptInner:HJ}, we split $\wt\Delta$ as \[\wt\Delta
(u,\tau)=\wt\Delta_1^u (u,\tau)-\wt\Delta_1^s (u,\tau)+\wt\Delta_2
(u,\tau)+\wt\Delta_3 (u,\tau)\] with
\begin{align}
\wt\Delta_1^{u,s}(u,\tau)&=T^{u,s}(u,\tau)-\frac{C_+^2}{\eps^{2r-1}}\psi^{u,s}
_0\left(\frac{u-ia}{\eps},\tau\right)\notag\\
&=\frac{C_+^2}{\eps^{2r-1}}\left(\psi^{u,s}\left(\frac{u-ia}{\eps},
\tau\right)-\psi^{u,s}_0\left(\frac{u-ia}{\eps},\tau\right)\right)\label{
def:Lambda1}\\
\wt\Delta_2(u,\tau)&=\frac{C_+^2}{\eps^{2r-1}}\left(\psi_0^{u}\left(\frac{u-ia}{
\eps},\tau\right)-\psi^{s}_0\left(\frac{u-ia}{\eps},
\tau\right)\right)-\Delta_0^+(u,\tau)\label{def:Lambda2}\\
\wt\Delta_3(u,\tau)&=-\Delta_0^-(u,\tau).\label{def:Lambda3}
\end{align}
Applying Theorem \ref{th:MatchingHJ}, one can see that for
$(u,\tau)\in \left(R_{s\ln\frac{1}{\eps},d_3}\cap
D_{s\ln\frac{1}{\eps},\C_2}^{\inn,+,s}\cap
D_{s\ln\frac{1}{\eps},\C_2}^{\inn,+,u}\right)\times\TT$,
\[
\left|\pa_u\wt\Delta_1^{u,s}(u,\tau)\right|\leq
\frac{K\eps^{\frac{1}{\q}-2r}}{\left|\ln\eps\right|^{2r-\frac{1}{\q}}}.
\]
To bound $\wt\Delta_2$, one has to proceed in different ways,
depending on whether $\ell-2r>0$ or $\ell-2r=0$. For the first case,
let us point out that,
\[
\wt\Delta_2(u,\tau)=\sum_{k<0}\ups_0^{[k]}\left(e^{ik\left(\eps\ii
u-\tau+\hmu
g\left(\eps\ii(u-ia),\tau\right)\right)}-e^{ik\left(\eps\ii
u-\tau+\CCC(u,\tau)\right)}\right).
\]
Then, applying Theorems \ref{th:InnerImma} and \ref{th:CanviFinal}
and the mean value theorem one obtains that for $(u,\tau)\in
\left(R_{s\ln\frac{1}{\eps},d_3}\cap
D_{s\ln\frac{1}{\eps},\C_2}^{\inn,+,s}\cap
D_{s\ln\frac{1}{\eps},\C_2}^{\inn,+,u}\right)\times\TT$,
\[
\left|\pa_u\wt\Delta_2(u,\tau)\right|\leq
\frac{K|\hmu|^2\eps^{s-2r}}{\left|\ln\eps\right|^{\ell-2r}}.
\]
For the case $\ell-2r=0$, taking into account the definition of
$\ups_0^{[k]}$ in \eqref{def:Dif0:coefsFourier:ligual},
\[
\wt\Delta_2(u,\tau)=\frac{C_+^2\hmu}{\eps^{2r-1}}\sum_{k<0}\chi^{[k]}
(\hmu)\left(e^{ik\left(\eps\ii
(u-ia)-\tau+\hmu
g\left(\eps\ii(u-ia),\tau\right)\right)}-e^{ik\left(\eps\ii
(u-ia)-\tau+\CCC(u,\tau)-C(\hmu,\eps)+\hmu^2b\ln\eps\right)}\right).
\]
By Theorems \ref{th:InnerImma} and \ref{th:CanviFinal}
and Proposition \ref{coro:Canvi:FirstOrder} for
$(u,\tau)\in \left(R_{s\ln\frac{1}{\eps},d_3}\cap
D_{s\ln\frac{1}{\eps},\C_2}^{\inn,+,s}\cap
D_{s\ln\frac{1}{\eps},\C_2}^{\inn,+,u}\right)\times\TT,$ we have that
\[
\left|\pa_u\wt\Delta_2(u,\tau)\right|\leq
\frac{K|\hmu|^2\eps^{s-2r}}{\left|\ln\eps\right|^{1+\Im\left(\hmu^2
b\right)}}.
\]
Finally, to bound  $\pa_u\wt\Delta_3$, it is enough to take into
account \eqref{eq:DifAsimptInner:HJ}. Then, one can see that for
$(u,\tau)\in \left(R_{s\ln\frac{1}{\eps},d_3}\cap
D_{s\ln\frac{1}{\eps},\C_2}^{\inn,+,s}\cap
D_{s\ln\frac{1}{\eps},\C_2}^{\inn,+,u}\right)\times\TT$,
\begin{align*}
\left|\pa_u\wt\Delta_3(u,\tau)\right|&\leq
K|\hat \mu|\eps^{-s-2r}e^{-\dps\tfrac{2a}{\eps}}&\text{provided }\ell-2r>0\\
\left|\pa_u\wt\Delta_3(u,\tau)\right|&\leq
K|\hat \mu|\eps^{-s-2r}e^{-{\dps\tfrac{2a}{\eps}}+2\Im \left({ \hat \mu} ^2
b\right)\ln\eps +\Im\left( {\hat \mu}^2 b\right)\ln\ln{\dps
\tfrac{1}{\eps}}}&\text{provided }\ell-2r=0.
\end{align*}
Therefore, from the bounds of $\wt \Delta_1^{u,s}$, $\wt \Delta_2$
and $\wt \Delta_3$ and recalling that by hypothesis $s<1/\q$, we
have that for $(u,\tau)\in \left(R_{s\ln\frac{1}{\eps},d_3}\cap
D_{s\ln\frac{1}{\eps},\C_2}^{\inn,+,s}\cap
D_{s\ln\frac{1}{\eps},\C_2}^{\inn,+,u}\right)\times\TT$,
\begin{align*}
\left|\pa_u\wt\Delta(u,\tau)\right|&\leq
\frac{K\eps^{s-2r}}{\left|\ln\eps\right|^{\ell-2r}}&\text{provided
}\ell-2r>0 \\
\left|\pa_u\wt\Delta(u,\tau)\right|&\leq
\frac{K\eps^{s-2r}}{\left|\ln\eps\right|^{1+ \Im\left(\hmu^2
b\right)}}&\text{provided }\ell-2r=0.
\end{align*}
Moreover, taking into account that $\pa_u\wt\Delta (u,\tau)$ depends
analytically on $\hmu$ and moreover satisfies
$\left.\pa_u\wt\Delta(u,\tau)\right|_{\hmu=0}=0$, one can apply
Schwartz Lemma to obtain
\begin{align}
\left|\pa_u\wt\Delta(u,\tau)\right|&\leq
\frac{K|\hmu|\eps^{s-2r}}{\left|\ln\eps\right|^{\ell-2r}}&
\text{provided }\ell-2r>0 \label{eq:CotaDeltaTilde:lmajor:neg}\\
\left|\pa_u\wt\Delta(u,\tau)\right|&\leq
\frac{K|\hmu|\eps^{s-2r}}{\left|\ln\eps\right|^{1+ \Im\left( \hmu^2
b\right)}}&\text{provided
}\ell-2r=0.\label{eq:CotaDeltaTilde:ligual:neg}
\end{align}
Reasoning analogously, one can see that  for $(u,\tau)\in
\left(R_{s\ln\frac{1}{\eps},d_3}\cap
D_{s\ln\frac{1}{\eps},\C_2}^{\inn,-,s}\cap
D_{s\ln\frac{1}{\eps},\C_2}^{\inn,-,u}\right)\times\TT$, the
function $\pa_u\wt\Delta$ satisfies
\begin{align}
\left|\pa_u\wt\Delta(u,\tau)\right|&\leq
\frac{K|\hmu|\eps^{s-2r}}{\left|\Delta_{00}(u,\tau)\ln\eps\right|^{\ell-2r}}&
\text{provided }\ell-2r>0 \label{eq:CotaDeltaTilde:lmajor:pos}\\
\left|\pa_u\wt\Delta(u,\tau)\right|&\leq
\frac{K|\hmu|\eps^{s-2r}}{\left|\ln\eps\right|^{1- \Im\left( \hmu^2
\ol b\right)}}&\text{provided
}\ell-2r=0.\label{eq:CotaDeltaTilde:ligual:pos}
\end{align}
Finally, by Theorems \ref{th:Extensio:Trig}, \ref{th:ExtensioFinal},
\ref{th:InnerImma} and \ref{th:CanviFinal}, one can easily see that
the bound of $\pa_u\wt\Delta (u,\tau)$ for $(u,\tau)\in
\left(R_{s\ln\frac{1}{\eps},d_3}\cap
D_{\C_2\eps^\ga,\rr_4}^{\out,s}\cap
D_{\C_2\eps^\ga,\rr_4}^{\out,u}\right)\times\TT$ is smaller than
\eqref{eq:CotaDeltaTilde:lmajor:neg} and
\eqref{eq:CotaDeltaTilde:lmajor:pos} (case $\ell-2r>0$) and
\eqref{eq:CotaDeltaTilde:ligual:neg} and
\eqref{eq:CotaDeltaTilde:ligual:pos} (case $\ell-2r=0$), provided
$|u-ia|\geq \OO(\eps^\ga)$.

Taking into account \eqref{eq:CotaDeltaTilde:lmajor:neg} and
\eqref{eq:CotaDeltaTilde:lmajor:pos} (case $\ell-2r>0$) and
\eqref{eq:CotaDeltaTilde:ligual:neg} and
\eqref{eq:CotaDeltaTilde:ligual:pos} (case $\ell-2r=0$), one can
conclude that for $\hmu\in B(\hmu_0)\cap\RR$,
\begin{align}
\left|\pa_u\wt\Delta(u,\tau)\right|&\leq
\frac{K|\hmu|\eps^{s-2r}}{\left|\ln\eps\right|^{\ell-2r}}&
\text{provided }\ell-2r>0 \label{eq:CotaDeltaTilde:lmajor}\\
\left|\pa_u\wt\Delta(u,\tau)\right|&\leq
\frac{K|\hmu|\eps^{s-2r}}{\left|\ln\eps\right|^{1+ \hmu^2\Im
b}}&\text{provided }\ell-2r=0.\label{eq:CotaDeltaTilde:ligual}
\end{align}
Analogously to the proof of Theorem \ref{th:CotaDiferencia:lmenor},
the second step is to consider the change of variables
$(w,\tau)=(u+\eps\CCC(u,\tau),\tau)$
and the auxiliary function
\[
\Theta(w,\tau)=\wt\Upsilon' \left(\eps\ii w-\tau\right),
\]
to obtain a bound for the Fourier coefficients of $\wt \Upsilon$:
\[
\left|\wt\Upsilon^{[k]}\right|\leq K\eps\sup_{(u,\tau)\in
R_{s\ln(1/\eps),d_3}\times\TT}\left|\pa_u\wt\Delta(u,\tau)\right|e^{-{\dps\tfrac
{|k|}{\eps}}\left(a-s\eps\ln\frac{1}{\eps}\right)-|k|\Im\left(
\CCC\left(u^\ast,0\right)\right)}.
\]

Therefore, to obtain the bounds for $\wt\Upsilon^{[k]}$ with $k<0$,
it only remains to use bounds \eqref{eq:CotaDeltaTilde:lmajor} and
\eqref{eq:CotaDeltaTilde:ligual} and the properties of $\CCC$ given
in Theorem \ref{th:CanviFinal} and Proposition
\ref{coro:Canvi:FirstOrder}. Then, we obtain that for $k<0$
\begin{align*}
\left| \wt\Upsilon^{[k]}\right|&\leq
\frac{K|\hmu|}{\eps^{2r-1}\left|\ln\eps\right|^{\ell-2r}}e^{-|k|{\dps
\tfrac{a}{\eps}}+(|k|-1)s\ln\frac{1}{\eps}}&\text{provided }\ell-2r>0\\
\left| \wt\Upsilon^{[k]}\right|&\leq
\frac{K|\hmu|}{\eps^{2r-1}\left|\ln\eps\right|}e^{-|k|\left({\dps\tfrac{a}{\eps}
}-\Im
\left( \hmu^2 b\right)\ln\eps\right)+
(|k|-1)\left(s\ln\frac{1}{\eps}+\Im\left( \hmu^2
b\right)\ln\ln\frac{1}{\eps}\right)} &\text{provided }\ell-2r=0.
\end{align*}
Since $\pa_u\wt\Delta(u,\tau)$ and $\CCC(u,\tau)$ are real-analytic
for $(\mu,\tau)\in\RR$, the coefficients $\wt\Upsilon^{[k]}$ for
$k>0$ satisfy the same bounds. Finally, the bounds of
$\wt\Upsilon^{[k]}$ lead easily to the desired bounds of
$\wt\Delta(u,\tau)$ for
$(u,\tau)\in\left(R_{s\ln(1/\eps),d_3}\cap\RR\right)\times\TT$.
\end{proof}

\subsection{Computation of the area of the lobes: proof of
Theorems~\ref{th:MainGeometric:regular} and
\ref{th:MainGeometric:singular}  and Corollaries
\ref{coro:MainGeometric:regular} and \ref{coro:MainGeometric:singular}}
To  prove Theorems~\ref{th:MainGeometric:regular} and
\ref{th:MainGeometric:singular}, we rewrite Corollaries
\ref{coro:Canvi:FirstOrder:lmenor} and \ref{coro:Canvi:FirstOrder} splitting the
results between the regular case $\eta>\ell-2r$ and the singular case
$\eta=\ell-2r$.

\begin{corollary}\label{coro:Diferencia:Simple:regular}
Let us assume $\eta>\ell-2r$. Then, there exists a constant $b_{20}>0$ such that
for
$\eps\in(0,\eps_0)$, $\mu\in B(\mu_0)\cap\RR$ and $(u,\tau)\in
\left(R_{s\ln(1/\eps),d_3}\cap\RR\right)\times\TT$, the following
statements are satisfied.
\[
\begin{split}
\left|\Delta(u,\tau)-\Upsilon^{[0]}- \Delta_{00}(u,\tau)\right|&\leq
\frac{b_{20}\left|\mu\right|\eps^{\eta+1-\ell}}{|\ln\eps|}e^{-{\dps\tfrac{a}{
\eps}}}\\
\left|\pa_u\Delta(u,\tau)-\pa_u \Delta_{00}(u,\tau)\right|&\leq
\frac{b_{20}\left|\mu\right|\eps^{\eta-\ell}}{|\ln\eps|}e^{-{\dps\tfrac{a}{\eps}
}}\\
\left|\pa_u^2\Delta(u,\tau)-\pa_u^2 \Delta_{00}(u,\tau)\right|&\leq
\frac{b_{20}\left|\mu\right|\eps^{\eta-1-\ell}}{|\ln\eps|}e^{-{\dps\tfrac{a}{
\eps}}}.\\
\end{split}
\]
where
\begin{itemize}
\item If $\eta>\eta^\ast$,
\[
\Delta_{00}(u,\tau) =\frac{2\mu \eps ^{\eta}}{\eps^{\ell-1}}e^{\dps
-\tfrac{a}{\eps}}\Re\left(f_0e^{-i\left({\dps
\tfrac{u}{\eps}}-\tau+\CCC(u,\tau)\right)}\right).
\]
\item If $\eta=0$ and $\ell-2r<0$,
\[
\Delta_{00}(u,\tau)=\frac{2\mu}{\eps^{\ell-1}}e^{\dps
-\tfrac{a}{\eps}}\Re\left(f_0e^{iC(\mu) }e^{-i\left({\dps
\tfrac{u}{\eps}}-\tau+\CCC(u,\tau)\right)}\right).
\]
\end{itemize}
\end{corollary}

\begin{corollary}\label{coro:Diferencia:singular}
Let us assume $\ell-2r\geq 0$ and $\eta=\eta^\ast=\ell-2r$. Then, there exists a
constant $b_{21}>0$ such that for
$\eps\in(0,\eps_0)$, $\mu\in B(\mu_0)\cap\RR$ and $(u,\tau)\in
\left(R_{s\ln(1/\eps),d_3}\cap\RR\right)\times\TT$, the following
statements are satisfied.
\begin{itemize}
\item If $\ell-2r>0$,
\[
\begin{split}
\left|\Delta(u,\tau)-\Upsilon^{[0]}- \Delta_{00}(u,\tau)\right|&\leq
\frac{b_{21}|\mu|}{\eps^{2r-1}|\ln\eps|^{\ell-2r}}e^{-\dps\tfrac{a}{\eps}}\\
\left|\pa_u\Delta(u,\tau)-\pa_u \Delta_{00}(u,\tau)\right|&\leq
\frac{b_{21}|\mu|}{\eps^{2r}|\ln\eps|^{\ell-2r}}e^{-\dps\tfrac{a}{\eps}}\\
\left|\pa_u^2\Delta(u,\tau)-\pa_u^2 \Delta_{00}(u,\tau)\right|&\leq
\frac{b_{21}|\mu|}{\eps^{2r+1}|\ln\eps|^{\ell-2r}}e^{-\dps\tfrac{a}{\eps}},
\end{split}\]
where
\[
 \Delta_{00}(u,\tau)=\frac{2\mu}{\eps^{2r-1}}e^{\dps
-\tfrac{a}{\eps}}\Re\left(f(\mu)e^{-i\left({\dps
\tfrac{u}{\eps}}-\tau+\CCC(u,\tau)\right)}\right).
\]
\item If $\ell-2r=0$,
\[
\begin{split}
\left|\Delta(u,\tau)-\Upsilon^{[0]}- \Delta_{00}(u,\tau)\right|&\leq
\frac{b_{21}\left|\mu\right|}{\eps^{2r-1}|\ln\eps|}e^{-{\dps\tfrac{a}{\eps}}
+\mu^2\Im b\ln\eps}\\
\left|\pa_u\Delta(u,\tau)-\pa_u \Delta_{00}(u,\tau)\right|&\leq
\frac{b_{21}\left|\mu\right|}{\eps^{2r}|\ln\eps|}e^{-{\dps\tfrac{a}{\eps}}
+\mu^2\Im b\ln\eps}\\
\left|\pa_u^2\Delta(u,\tau)-\pa_u^2 \Delta_{00}(u,\tau)\right|&\leq
\frac{b_{21}\left|\mu\right|}{\eps^{2r+1}|\ln\eps|}e^{-{\dps\tfrac{a}{\eps}}
+\mu^2\Im b\ln\eps},
\end{split}
\]
where
\[
 \Delta_{00}(u,\tau)=\frac{2\mu}{\eps^{2r-1}}e^{{\dps
-\tfrac{a}{\eps}}}\Re\left(f(\mu)e^{-i\left(\mu^2b\ln\eps-C(\mu)\right)}e^{
-i\left({\dps
\tfrac{u}{\eps}}-\tau+\CCC(u,\tau)\right)}\right).
\]
\end{itemize}
\end{corollary}

Let us fix a transversal Poincar\'e section corresponding to
$\tau=\tau_0\in\RR$. Being $\Upsilon(w)$ in \eqref{def:Upsilon:lmenor} and \eqref{def:Delta:xi0} a $2\pi$-periodic function, we
know that $\Delta(u,\tau_0)$ has critical points which are $\OO(\eps)$-close to each
other. Then, in
$\left(R_{s\ln(1/\eps),d_3}\cap\RR\right)$ there exist almost two of these
points, reducing $\eps$ if necessary. These critical points correspond to
homoclinic orbits of system \eqref{eq:model}. Let us consider two
consecutive zeros  $u_-^*$ and $u_+^*$  in
$\left(R_{s\ln(1/\eps),d_3}\cap\RR\right)$, which depend on
$\tau_0$. Then, taking into account that the change
\eqref{eq:CanviSimplecticSeparatriu} is symplectic,  it preserves
area and recalling the definition of $\Delta$ in
\eqref{def:Diferencia:Original}, the   area of the lobes is given by
\[
 \AAA=\left|\int_{u_-^*}^{u_+^*}\pa_u\Delta(u,
\tau_0)du\,\right|=\left|\Delta(u_+^*,\tau_0)-\Delta(u_-^*,\tau_0)\right|.
\]
First we take $\eta>\ell-2r$ and we prove Theorem \ref{th:MainGeometric:regular}
and Corollary \ref{th:MainGeometric:regular}. The simplest case is when $f_0=0$.
In this case  Corollary \ref{coro:Diferencia:Simple:regular} directly implies
Theorem \ref{th:MainGeometric:regular} since $\Delta_{00}(u,\tau)\equiv 0$.

In the case $f_0\neq 0$ we prove  Theorem \ref{th:MainGeometric:regular} and
Corollary \ref{th:MainGeometric:regular} at the same time. It can be easily seen
that
the consecutive zeros of $\pa_u\Delta_{00}(u,\tau_0)$ (see
\eqref{def:Delta00:lmenor},
\eqref{def:Delta00:lmajor} and \eqref{def:Delta00:ligual}) are also
$\OO(\eps)$-close and therefore taking $\eps$ small enough, in
$\left(R_{s\ln(1/\eps),d_3}\cap\RR\right)$ there exist at least two
consecutive zeros  $u_-$ and $u_+$  in
$\left(R_{s\ln(1/\eps),d_3}\cap\RR\right)$, which again depend on
$\tau_0$. It can be easily checked that the function $\Delta_{00}$
evaluated at these points satisfies
\begin{equation}\label{eq:MesMenys}
\Delta_{00}(u_+,\tau_0)=-\Delta_{00}(u_-,\tau_0)
\end{equation}
 and
\begin{align}
\left|\Delta_{00}(u_\pm,\tau_0)\right|&=2\mu\eps^{\eta+1-\ell}
\left|f_0\right|e^{\dps-\tfrac{a}{\eps}}&\text{ if
}\eta>\eta^\ast\label{def:DeltaHom:regular}\\
\left|\Delta_{00}(u_\pm,\tau_0)\right|&=2\mu\eps^{\eta+1-\ell}
\left|f_0 e^{iC(\mu)}\right|e^{{\dps-\tfrac{a}{\eps}}}
&\text{ if }\ell-2r<0\text{ and }\eta=0.\label{def:DeltaHom:lmenor}
\end{align}

By Corollary \ref{coro:Diferencia:Simple:regular}, since by hypothesis
we have that $f_0\neq 0$, we can apply the implicit function theorem to see
that the zeros $u_-^\ast$ and
$u_+^\ast$ of the function $\pa_u\Delta(u,\tau_0)$ satisfy
\begin{equation}\label{def:PuntsHomo}
u_\pm^\ast=u_\pm+\OO\left(\frac{\eps}{|\ln\eps|^{\nu_\ell}}\right),
\end{equation}
where $\nu_\ell=\ell-2r$ for $\ell>2r$ and $\nu_\ell=1$ for
$\ell\leq 2r$.

Using formulas \eqref{eq:MesMenys}-\eqref{def:PuntsHomo} and the inequalities
given in Corollary \ref{coro:Diferencia:Simple:regular}, one obtains the
asymptotic formula for the area, which finishes the proofs of Theorem
\ref{th:MainGeometric:regular} and Corollary \ref{coro:MainGeometric:regular}.

The proofs of Theorem \ref{th:MainGeometric:singular} and Corollary
\ref{th:MainGeometric:singular} follow the same lines taking into account that
now
\begin{align}
\left|\Delta_{00}(u_\pm,\tau_0)\right|&=2\mu\eps^{\eta+1-\ell}
\left|f(\mu)\right|e^{\dps-\tfrac{a}{\eps}}&\text{ if
}\eta=\eta^\ast\text{ and }\ell-2r>0\label{def:DeltaHom:lmajor}\\
\left|\Delta_{00}(u_\pm,\tau_0)\right|&=2\mu\eps^{\eta+1-\ell}
\left|f(\mu)e^{iC(\mu)}\right|e^{{\dps-\tfrac{a}{\eps}}+\mu^2
\Im b\ln\eps}&\text{ if }\eta=\eta^\ast\text{ and
}\ell-2r=0.\label{def:DeltaHom:ligual}
\end{align}
In this case, given a value of $\mu$, one has to split the proof depending
whether $f(\mu)=0$, and therefore $\Delta_{00}(u,\tau)\equiv 0$, or $f(\mu)\neq
0$.

\begin{remark}
We emphasize that, by hypothesis \textbf{HP3}, the hamiltonian perturbation $H_1$ defined in either \eqref{def:Ham:Original:perturb:poli}
in the polynomial case or \eqref{def:Ham:Original:perturb:trig} in the trigonometric case
it may depend analytically on $\eps$.
We stress that all the results given in this section are also valid in this setting and consequently
Theorems \ref{th:MainGeometric:regular} and \ref{th:MainGeometric:singular} hold true.

Indeed, in this case, what we have is that the $2\pi$-periodically functions
$a_{k,l}(\tau; \eps)$ defining $H_1$ depend analytically on $\eps$ and henceforth the same happens
for the functions $A_{k}(\tau)\equiv A_{k}(\tau;\eps)$ defined in \eqref{def:FuncionsA}.
In this way one has that the inner equation \eqref{eq:HJEqInner2} depends analytically on $\eps$.
Following the proof in \cite{Baldoma06}, it is straightforward to check that the solutions
$\psi_0^{u,s}$ of the inner equation given in Theorem \ref{th:InnerImma} actually also depend analytically on the parameter
$\eps$. Moreover we have the same property for the coefficients $\chi^{[k]}$ defining the difference
$\psi_0^{u} - \psi_0^{s}$. As a consequence, $f(\mu) \equiv f(\mu; \eps) = f(\mu;0) + \OO(\eps)$.
In addition, the constant $b$ given in Proposition \ref{coro:gInner} also depends analytically on
$\eps$ and henceforth $b\equiv b(\eps) = b(0)+ \OO(\eps)$.

After these considerations, it is clear that we can replace $f(\mu;\eps)$ by $f(\mu,0)$ and $b(\eps)$ by $b(0)$ in all the
previous arguments and henceforth the claim is proved.
\end{remark}

\begin{remark}\label{remark:DosSing} The proof that we have just
explained works under the assumed hypotheses (see Section \ref{sec:Hypotheses}), in particular, under Hypothesis \textbf{HP2}, which assumes that there exists only one singularity  on each line $\{\Im u=\pm a\}$. Nevertheless, with little modifications, the same scheme works if there are more singularities on these lines, at least assuming some smallness condition on the perturbation, namely in the
regular case. Let us explain here how, assuming that the perturbation is small enough, the problem can be handled.

Assume that the closest singularities to the real axis of the separatrix are located at $u=\pm \alpha \pm ai$, $\alpha\neq 0$, (and assume moreover that $p_0(u)$ does not vanish to simplify the explanation). To prove the asymptotic formula for the splitting we need to obtain the existence
of two generating functions which parameterize the
perturbed invariant manifolds in a common domain containing points with imaginary part $\Im u=a-\kappa\eps$. The existence of the invariant manifolds close to the fixed point can be proved as in this paper, since the singularities are
far from the domains $D^\ast_{\infty, \rr_1}$. Therefore, Theorem \ref{th:ExistenceCloseInfty} is also valid in
this case (of course Theorem \ref{th:Periodica} is valid as well
since it does not require Hypothesis \textbf{HP2}).

To extend the invariant manifolds to  a common domain  containing points with imaginary part $\Im u=a-\kappa\eps$, we have to modify the outer domains $D_{\rr,\kk}^{\out, u}$ and $D_{\rr,\kk}^{\out, s}$. It is enough, for instance to ``center" the stable domain around the singularity with positive real part (that is, the boundary of the domain intersects the line ${\alpha + t i, t\in \RR}$ at $\alpha \pm ( a-\kappa\eps) i$) and the unstable one around the singularity with negative real part. The corresponding domains intersect in a strip of ``horizontal size" of order $\OO(1)$ but of ``vertical size" size smaller than $a-\kappa \varepsilon$. To achieve that the domains cover a piece of the imaginary axis that contain points with $\Im u= a-\kappa' \varepsilon$ (for some $\kk'>\kk$) one can proceed taking the angle $\beta_1$ of order $\OO(\eps)$. Without any extra technical  work, this worsens the estimates and is the reason  why we  need, under this more general hypothesis,  the perturbation to be small. Namely, we need to take $\eta$ big enough.

Once we have proved the existence of suitable parameterizations of the invariant manifolds in this new outer domain, the proof of the validity of the Melnikov method can be done exactly in the same way as in this paper (namely Theorems \ref{th:CanviFinal:lmenor} and \ref{th:CotaDiferencia:lmenor} are still valid). We have decided not to cover this case in this work due to the considerable length the paper already has.
\end{remark}

\section{Existence of the periodic orbit in the hyperbolic case: proof of
Theorem
\ref{th:Periodica}}\label{sec:periodica}

In this section we prove Theorem \ref{th:Periodica}. We look for a
periodic orbit $(x,y)=(\xp(\tau),\yp(\tau))$ which is close to the
hyperbolic critical point of the unperturbed system $(0,0)$.

By  \textbf{HP1.1}, the differential of
the unperturbed hyperbolic critical point is
\begin{equation}\label{def:DiferencialSella}
\eps\mat=\eps\left(\begin{array}{cc} 0 & 1\\ \lambda^2 &
0\end{array}\right).
\end{equation}
Then, defining $z=(x,y)$ and considering the differential operator
\begin{equation}\label{def:Operador:L0}
\DD_0 z(\tau)=\frac{d}{d\tau}z(\tau),
\end{equation}
we look for the periodic orbit as a $2\pi$-periodic solution of the
following equation,
\begin{equation}\label{eq:Periodica}
\left(\DD_0 - \eps\mat\right)z=\eps F(z,\tau),
\end{equation}
where
\[
F(z,\tau)=\left(\begin{array}{l} \mu\eps^{\eta} \pa_y \wH(x,y,\tau)\\
-\mu\eps^{\eta} \pa_x \wH(x,y,\tau)-\left( V'(x)+\lambda^2 x\right)
\end{array}\right).
\]
We split $F$ in constant, linear and higher order terms with respect to $z$
\begin{equation}
F(z,\tau)= F_0(\tau)+ F_1(\tau)z+ F_2(z,\tau)
\end{equation}
with
\begin{align}
F_0(\tau)&= \left(\begin{array}{l} \mu\eps^{\eta} \pa_y
\wH(0,0,\tau)\\ -\mu\eps^{\eta} \pa_x \wH(0,0,\tau)
\end{array}\right)\label{def:periodica:F0}\\
F_1(\tau)&=\left(\begin{array}{ll} \mu\eps^{\eta} \pa_{yx} \wH(0,0,\tau)
&\mu\eps^{\eta} \pa_{yy} \wH(0,0,\tau)\\
-\mu\eps^{\eta} \pa_{xx} \wH(0,0,\tau) &-\mu\eps^{\eta}
\pa_{xy} \wH(0,0,\tau)\end{array}\right)\label{def:periodica:F1}\\
F_2(z,\tau)&=F(z,\tau)- F_0(\tau)- F_1(\tau)z.
\label{def:periodica:F2}
\end{align}

We devote the rest of the section to obtain a solution of equation
\eqref{eq:Periodica}. First in Section \ref{sec:periodica:banach} we
define a Banach space we will use and we state some technical properties.
Then, in Section \ref{sec:periodica:proof} we prove
Theorem \ref{th:Periodica}.

\subsection{Banach spaces and technical
lemmas}\label{sec:periodica:banach} For analytic functions $z:
\TT_\sigma\rightarrow \CC$,
$z(\tau)=\sum_{k\in\ZZ}z^{[k]}e^{ik\tau}$, we define the Fourier
norm
\[
\| z\|_{\sigma}=\sum_{k\in\ZZ}\left| z^{[k]}\right| e^{|k|\sigma}.
\]
Then, we define the function space endowed with the previous norm
\begin{equation}\label{def:periodica:banach}
\SSS_\sigma=\left\{z:\TT_\sigma\rightarrow\CC; \text{ real-analytic
}, \| z\|_{\sigma}<\infty \right\}
\end{equation}
which is a Banach algebra. We also consider the product space
$\SSS_\sigma\times\SSS_\sigma$ with the induced norm
\[
\left\|(z_1,z_2)\right\|_{1,\sigma}=\left\|z_1\right\|_{\sigma}
+\left\|z_2\right\|_{\sigma}.
\]
\begin{remark}
Let us consider the classical supremmum norm
\[ \|z\|_{\infty, \sigma}=\sup_{\tau\in \ol\TT_\sigma}\left|z(\tau)\right|.
\]
Then, it is a well known fact (see for instance \cite{Sauzin01})
that for any $\sigma_1<\sigma_2$, the supremmum and the Fourier norm
satisfy the following relation
\[
\|z\|_{\sigma_1}<K\left(1+\frac{1}{\sigma_2-\sigma_1}\right)\|z\|_{\infty,
\sigma_2}
\]
Therefore, since we are assuming that there exists $\sigma_0>0$ such
that  the functions $a_{kl}$ defined in
\eqref{def:Ham:Original:perturb:poli} and
\eqref{def:Ham:Original:perturb:trig} are $\CCC^0$ in $\ol
\TT_{\sigma_0}$ and analytic in $\TT_{\sigma_0}$, we can deduce that
for any $\sigma<\sigma_0$ such that $\sigma_0-\sigma$ has a positive
lower bound independent of $\eps$, they satisfy
\[
\|a_{kl}\|_\sigma<K.
\]
We will use this fact without mentioning it, in the rest of the
section and also in Sections \ref{sec:InftyHyperbolic} to
\ref{sec:CanviFinal}.
\end{remark}

Since we deal with vector functions, we also consider the
norm for $2\times 2$ matrices  induced by $\|\cdot\|_{1,\sigma}$.
Let us consider $B=\left(b^{ij}\right)$ a $2\times 2$ matrix such that
$b^{ij}\in \SSS_\sigma$.
Then, the induced matrix norm is given by
\[
\|B\|_{1,\sigma}=\max_{j=1,2}\left\{
\left\|b^{1j}\right\|_\sigma+\left\|b^{2j}\right\|_\sigma\right\}.
\]
The next lemma gives some properties of this norm.
\begin{lemma}\label{lemma:periodica:NormesMatrius}
The following statements are satisfied.
\begin{enumerate}
\item If $h\in\SSS_\sigma\times\SSS_\sigma$ and
$B=\left(b^{ij}\right)$ is a $2\times 2$ matrix with $b^{ij}\in
\SSS_\sigma$, then $Bh\in\SSS_\sigma\times\SSS_\sigma$ and
\[
\left\|Bh\right\|_{1,\sigma}\leq\left\|B\right\|_{1,\sigma}\left\|h\right\|_{1,
\sigma}.
\]
\item If $B_1=\left(b_1^{ij}\right)$ and $B_2=\left(b_2^{ij}\right)$ are
$2\times 2$ matrices which satisfy
$b_1^{ij},b_2^{ij}\in\SSS_\sigma$,
then
%$D=\left(d^{ij}\right)=B_1B_2$ satisfies $d^{ij}\in\SSS_\sigma$ and
\[
\| B_1B_2\|_{1,\sigma}\leq\| B_1\|_{1,\sigma}\| B_2\|_{1,\sigma}.
\]
\end{enumerate}
\end{lemma}
Throughout this section, we will need to solve equations of the form
$(\DD_0-\eps\mat)z=w$.
For that, we will invert the operator $\DD_0-\eps A_0$
acting on $\SSS_\sigma\times\SSS_\sigma$.
Considering the Fourier series of $z(\tau)=(z_1(\tau),z_2(\tau))$, one has that
\[
\DD_0(z)(\tau)=\sum_{\kk\in\ZZ} ikz^{[k]}e^{ik\tau}.
\]
Then, one can invert $\DD_0-\eps\mat$  as
\begin{equation}\label{def:periodica:OperadorIntegral}
\GG_0(w)(\tau)=-\sum_{k\in\ZZ}\frac{1}{k^2+\lambda^2\eps^2}\left(\begin{array}{l
}
ikw^{[k]}_1+\eps w^{[k]}_2\\ \eps\lambda^2 w^{[k]}_1+ik
w^{[k]}_2\end{array}\right)e^{ik\tau}.
\end{equation}

\begin{lemma}\label{lemma:periodica:OperadorIntegral}
The operator $\GG_0:\SSS_\sigma\times\SSS_\sigma\rightarrow
\SSS_\sigma\times\SSS_\sigma$ in
\eqref{def:periodica:OperadorIntegral} is well defined, and for
$w\in\SSS_\sigma\times\SSS_\sigma$,
\[
\left\|\GG_0(w)\right\|_{1,\sigma} \leq
\frac{K}{\eps}\|w\|_{1,\sigma}.
\]
Moreover, if $\langle w\rangle=0$,
\[
\left\|\GG_0(w)\right\|_{1,\sigma} \leq K\|w\|_{1,\sigma}.
\]
\end{lemma}
We finally state a technical lemma which will be used in Section
\ref{sec:periodica:proof}. Its proof is straightforward.

\begin{lemma}\label{lemma:periodica:cotes}
The functions $F_0$, $F_1$ and $F_2$ defined in
\eqref{def:periodica:F0}, \eqref{def:periodica:F1} and
\eqref{def:periodica:F2} respectively satisfy the following
properties.
\begin{enumerate}
\item $F_0\in \SSS_\sigma\times\SSS_\sigma$, $\langle F_0\rangle=0$
and
\[
\left\| F_0\right\|_{1,\sigma}\leq K|\mu|\eps^{\eta}.
\]
\item $F_1=\left(F_1^{ij}\right)$ satisfies
$F_1^{ij}\in\SSS_\sigma$, $\langle F_1^{ij}\rangle=0$ and
\[
\left\| F_1\right\|_{1,\sigma}\leq K|\mu|\eps^{\eta}.
\]
\item If  $z,z'\in B(\nu)\subset\SSS_\sigma$ with
$\nu\ll 1$, then
\[
\left\| F_2(z',\tau)-F_2(z,\tau)\right\|_\sigma\leq K\nu \|
z'-z\|_{\sigma}.
\]
\end{enumerate}
\end{lemma}
\subsection{Proof of Theorem
\ref{th:Periodica}}\label{sec:periodica:proof}

We rewrite Theorem \ref{th:Periodica} in terms of the Banach space
\eqref{def:periodica:banach}.

\begin{proposition}\label{prop:periodica}
Let $\eps_0>0$ small enough. Then, for $\eps\in
(0,\eps_0)$, equation \eqref{eq:Periodica} has a solution
$(\xp,\yp)\in\SSS_\sigma$. Moreover, there exists a constant $b_0>0$
such that
\[
\left\|(\xp,\yp)\right\|_{1,\sigma}\leq b_0|\mu|\eps^{\eta+1}.
\]
\end{proposition}
\begin{corollary}\label{coro:ShiftPeriodica}
The change of variables \eqref{eq:CanviShiftPO} transforms the
Hamiltonian system with Hamiltonian \eqref{def:Ham:Original0} to a
new Hamiltonian system with Hamiltonian
\eqref{def:HamPeriodicaShiftada}.

Moreover, the functions $c_{ij}$ in the definition of
\eqref{def:HamPeriodicaShiftada} (see also
\eqref{def:HamPertorbat:H2}) satisfy
\[
\| c_{ij}\|_\sigma\leq K|\mu|\eps^{\eta}.
\]
\end{corollary}

We devote the rest of the section to prove Proposition
\ref{prop:periodica}. We obtain the solution of equation
\eqref{eq:Periodica} through a fixed point argument. To obtain a
contractive operator, first we have to perform a change of
variables, which actually it is only needed in the case $\ell-2r=0$.

Let us consider a function $\ol F_1$ which satisfies $\langle \ol
F_1\rangle=0$ and $\pa_\tau \ol F_1=F_1$, where $F_1$ is the function
in \eqref{def:periodica:F1}. The function $\ol F_1$ can be defined
as
\[
\ol
F_1(\tau)=\sum_{k\in\ZZ\setminus\{0\}}\frac{1}{ik}F_1^{[k]}e^{ik\tau}
\]
and satisfies
\begin{equation}\label{eq:periodica:cotaF1barra}
\left\| \ol F_1\right\|_{1,\sigma}\leq \left\|
F_1\right\|_{1,\sigma}.
\end{equation}
We perform the change of variables
\begin{equation}\label{eq:periodica:canvi}
z=\left(\Id+\eps \ol F_1(\tau)\right)\ol z
\end{equation}
and then equation \eqref{eq:Periodica} becomes
\begin{equation}\label{eq:periodica:eqModificada}
\left(\DD_0-\eps\mat\right) \ol z=\ol F(\ol z,\tau),
\end{equation}
where
\begin{equation}\label{eq:periodica:Fbarra}
\begin{split}
\ol F(\ol z,\tau)= &\eps \left(\Id +\eps \ol
F_1(\tau)\right)\ii F_0(\tau)\\
&+\eps^2\left(\Id +\eps \ol F_1(\tau)\right)\ii\left(\mat \ol
F_1(\tau)-\ol
F_1(\tau)\mat+ \ol F_1(\tau)F_1(\tau)\right)\ol z\\
&+\eps \left(\Id +\eps \ol F_1(\tau)\right)\ii F_2 \left(\left(\Id
+\eps \ol F_1(\tau)\right)\ol z, \tau\right).
\end{split}
\end{equation}
Since the operator $\GG_0$ defined in
\eqref{def:periodica:OperadorIntegral} is a left inverse of
$\DD_0-\eps \mat$, we look for a solution of  equation
\eqref{eq:periodica:eqModificada} as a fixed point of the operator
\begin{equation}\label{def:periodica:Funcional}
\FF_0=\GG_0\circ \ol F.
\end{equation}

Then Proposition \ref{prop:periodica} follows from the following
lemma.
\begin{lemma}\label{lemma:periodica:PuntFix}
Let $\eps_0>0$ small enough. Then, there exists a
constant $b_0>0$ such that, for $\eps\in(0,\eps_0)$, the operator
$\FF_0$ in \eqref{def:periodica:Funcional} is contractive from $\ol
B\left(b_0|\mu|\eps^{\eta+1}\right)\subset\SSS_\sigma\times\SSS_\sigma$
to itself.

Then, $\FF_0$ has a unique fixed point $\ol z^\ast\in \ol
B\left(b_0|\mu|\eps^{\eta+1}\right)\subset\SSS_\sigma\times\SSS_\sigma$.
\end{lemma}
\begin{proof}
It is easily checked that $\FF_0$ sends
$\SSS_\sigma\times\SSS_\sigma$ into itself. To see that it is
contractive we first consider $\FF_0(0)$, which can be split as
\[
\FF_0(0)=\eps\GG_0\left(F_0\right)-\eps^2\GG_0\left(\left(\Id+\eps\ol
F_1\right)\ii \ol F_1 F_0\right).
\]
By Lemma \ref{lemma:periodica:cotes}, $\langle F_0\rangle=0$ and
$\|F_0\|_{1,\sigma}\leq K|\mu|\eps^{\eta}$. Then, applying Lemma
\ref{lemma:periodica:OperadorIntegral}, one has that
\[
\left\|\GG_0\left(F_0\right)\right\|_{1,\sigma}\leq
K|\mu|\eps^{\eta}.
\]
For the second term, considering also
\eqref{eq:periodica:cotaF1barra}
 and Lemmas
\ref{lemma:periodica:NormesMatrius},
\ref{lemma:periodica:OperadorIntegral} and
\ref{lemma:periodica:cotes}, one can proceed analogously to obtain
\[
\left\|\GG_0\left(\left(\Id+\eps\ol F_1\right)\ii \ol F_1
F_0\right)\right\|_{1,\sigma}\leq K|\mu|\eps^{2\eta-1}.
\]
Therefore, there exists a constant $b_0>0$ such that
\[
\left\|\FF_0(0)\right\|_{1,\sigma}\leq\frac{
b_0}{2}|\mu|\eps^{\eta+1}.
\]
Let us consider now $z^1,z^2\in \ol B\left(
b_0|\mu|\eps^{\eta+1}\right)\subset\SSS_\sigma\times\SSS_\sigma$.
Then, by Lemmas \ref{lemma:periodica:OperadorIntegral},
\ref{lemma:periodica:NormesMatrius} and \ref{lemma:periodica:cotes},
and reducing $\eps$ if necessary,  one can see that,
\[
\begin{split}
\left\|
\FF_0\left(z^2\right)-\FF_0\left(z^1\right)\right\|_{1,\sigma}&\leq
K|\mu|\eps^{\eta+1}\left\|z^2 -z^1\right\|_{1,\sigma}\\
&\leq \frac{1}{2}\left\|z^2- z^1\right\|_{1,\sigma}.
\end{split}
\]
Then, $\FF_0:\ol B\left(b_0|\mu|\eps^{\eta+1}\right)\rightarrow \ol
B\left(
b_0|\mu|\eps^{\eta+1}\right)\subset\SSS_\sigma\times\SSS_\sigma$ and
is contractive. Therefore, it has a unique fixed point $\ol z^\ast$.
\end{proof}

\begin{proof}[Proof of Proposition \ref{prop:periodica}]
It is enough to take
\[
z^\ast(\tau)=\left(\Id+\eps\ol F_1(\tau)\right)\ol z^\ast(\tau),
\]
which satisfies equation \eqref{eq:Periodica} and satisfies the
desired bound (increasing $b_0$ slightly if necessary).
\end{proof}

\section{Local invariant manifolds: proof of
Theorem~\ref{th:ExistenceCloseInfty}}\label{sec:Infty}

Since the proof for both invariant manifolds is analogous, we only
deal with the unstable case. We look for a solution of equation
\eqref{eq:HamJacGeneral} satisfying the asymptotic condition
\eqref{eq:AsymptCondFuncioGeneradora:uns}. We look for it as a
perturbation of the unperturbed separatrix
\begin{equation}\label{def:T0} T_0(u)=\int_{-\infty}^u p^2_0(v)\,dv
\end{equation}
and therefore we work with $T_1(u,\tau)=T(u,\tau)-T_0(u)$.

 Replacing
$T$ in equation \eqref{eq:HamJacGeneral} and taking into account
that $V(q_0(u))=-p^2_0(u)/2$, it is straightforward to see that the
equation for $T_1$ reads
\begin{equation}\label{eq:HJperT1}
\LL_\eps T_1=\FF\left(\pa_u T_1,u,\tau\right),
\end{equation}
where $\LL_\eps$ is the operator defined in \eqref{def:Lde} and
\[
\begin{array}{ll}
\FF(w,u,\tau)=&\dps-\frac{w^2}{2p_0^2(u)}
-\bigg(V(q_0(u)+\xp(\tau))-V(\xp(\tau))-V(q_0(u))-V'(\xp(\tau))q_0(u)\bigg)\\
&\dps-\mu\eps^{\eta}\widehat
H_1\left(q_0(u),p_0(u)+\frac{w}{p_0(u)},\tau\right),
\end{array}
\]
where $\widehat H_1$ is the function defined in
\eqref{def:ham:ShiftedOP:perturb}.

We split $\FF$ into constant, linear and higher order terms in $w$
as
\begin{equation}\label{eq:infty:lmenor:operadorF}
\FF(w,u,\tau)=A(u,\tau)+\left(B_1(u,\tau)+B_2(u,\tau)\right)w+C(w,
u,\tau),
\end{equation}
with
\begin{align}
A(u,
\tau)=&\dps-\left(V(q_0(u)+\xp(\tau))-V(\xp(\tau))-V(q_0(u))-V'(\xp(\tau))q_0(u)
\right)\notag\\
&\dps-\mu\eps^{\eta}\widehat
H_1\left(q_0(u),p_0(u),\tau\right),\label{def:InftyA}\\
B_1(u,\tau)=&-\mu\eps^{\eta}p_0^{-1}(u)\pa_p \widehat
H_1^1(q_0(u),p_0(u),\tau),\label{def:InftyB1}\\
B_2(u,\tau)=&-\mu\eps^{\eta+1}p_0^{-1}(u)\pa_p \widehat
H_1^2(q_0(u),p_0(u),\tau),\label{def:InftyB2}\\
C(w,u,\tau)=&-\frac{w^2}{2p_0^2(u)}-\mu\eps^{\eta}\widehat
H_1\left(q_0(u),p_0(u)+\frac{w}{p_0(u)},\tau\right)\notag\\
&+\mu\eps^{\eta}\frac{w}{p_0(u)}\pa_p \widehat
H_1(q_0(u),p_0(u),\tau)+\mu\eps^{\eta}\widehat
H_1\left(q_0(u),p_0(u),\tau\right),\label{def:InftyC}
\end{align}
where $\widehat H_1^1$ and $\widehat H_1^2$ are the functions
defined in \eqref{def:HamPertorbat:H1} and
\eqref{def:HamPertorbat:H2} respectively.

\subsection{Local invariant manifolds in the hyperbolic
case}\label{sec:InftyHyperbolic}
In this section we prove the existence of suitable representations
of the unstable and stable invariant manifolds in the domains
$D^{u}_{\infty,\rr}\times\TT_\sigma$ and
$D^{s}_{\infty,\rr}\times\TT_\sigma$ respectively under the
hypothesis that the unperturbed Hamiltonian system has a hyperbolic
critical point at the origin.

\subsubsection{Banach spaces and technical
lemmas}\label{Subsection:Infty:Banach}
This subsection is devoted to define the Banach spaces which will be
used in Section \ref{sub:Infty:Hyp:ProofGeneral}. We also state some
of their useful properties.

We define some norms for functions defined in a domain
$D_{\infty,\rr}^u$ with $\rr\geq0$. Given $\alpha\geq 0$, $\rr\geq
0$ and an analytic function $h:D^{u}_{\infty,\rr}\rightarrow \CC$,
we consider
\[
\|h\|_{\alpha, \rr}=\sup_{u\in D^{u}_{\infty,\rr}}\left| e^{-\alpha
u}h(u)\right|.
\]
Moreover for $2\pi$-periodic in $\tau$, analytic functions
$h:D^{u}_{\infty,\rr}\times\TT_\sigma\rightarrow \CC$, we consider the
corresponding Fourier
norm
\[
\|h\|_{\alpha, \rr,\sigma}=
\sum_{k\in\ZZ}\left\|h^{[k]}\right\|_{\alpha,  \rr}e^{|k|\sigma}.
\]
We consider, thus, the following function space

\begin{equation}\label{def:BanachInftyHyp}
\HH_{\alpha, \rr,\sigma}=\{
h:D^{u}_{\infty,\rr}\times\TT_\sigma\rightarrow\CC;\,\,\text{real-analytic},
\|h\|_{\alpha, \rr,\sigma}<\infty\},
\end{equation}
which can be checked that is a Banach space for any fixed $\alpha>0$
and $\sigma>0$.

%If there is no danger of confusion about the definition domain
%$D^{u(i)}_{\infty,\rr}$ we will denote
%\[
%\begin{array}{ccc}
%\|\cdot\|_{\alpha,\sigma}=\|\cdot\|_{\alpha,\sigma,i}&\text{ and
%}&\HH_{\alpha,\sigma}=\HH_{\alpha,\sigma}^{(i)}.
%\end{array}
%\]
In the next lemma, we state some properties of these Banach spaces.
\begin{lemma}\label{lemma:Infty:PropietatsNormes} The following statements hold:
\begin{enumerate}
\item If $\alpha_1\geq \alpha_2\geq 0$, then $\HH_{\alpha_1, \rr,\sigma}\subset
\HH_{\alpha_2, \rr,\sigma}$ and
\[
\|h\|_{\alpha_2, \rr,\sigma}\leq \|h\|_{\alpha_1, \rr,\sigma}.
\]
\item If $\alpha_1,\alpha_2\geq0$, then, for $h\in\HH_{\alpha_1, \rr,\sigma}$
and $g\in\HH_{\alpha_2, \rr,\sigma}$,
we  have that $hg\in\HH_{\alpha_1+\alpha_2, \rr,\sigma}$ and
\[
\|hg\|_{\alpha_1+\alpha_2, \rr,\sigma}\leq \|h\|_{\alpha_1,
\rr,\sigma}\|g\|_{\alpha_2, \rr,\sigma}.
\]
\item Let $\alpha\geq 0$ and  $\rr'>\rr >0$ be such that $\rr'-\rr$ has a
positive lower bound independent of $\eps$. Then for
$h\in\HH_{\alpha,\rr,\sigma}$ we have that $\pa_u
h\in\HH_{\alpha,\rr',\sigma}$ and
\[
\left\| \pa_u h\right\|_{\alpha,\rr',\sigma}\leq
K\|h\|_{\alpha,\rr,\sigma}.
\]
\end{enumerate}
\end{lemma}

Throughout this section we are going to solve equations of the form
$\LL_\eps h=g$, where $\LL_\eps$ is the differential operator
defined in~\eqref{def:Lde}. Note that if $\alpha>0$,
$\mathrm{Ker}\LL_\eps=\{0\}$ and hence $\LL_\eps$ is invertible. It
turns out that its inverse is $\GG_\eps$ defined by
\begin{equation}\label{def:operadorGInfty}
\GG_\eps (h)(u,\tau)=\int_{-\infty}^0h(u+t,\tau+\eps\ii t)\,dt.
\end{equation}
We also introduce
\begin{equation}\label{def:operadorGInftyBarra}
\overline\GG_\eps (h)(u,\tau)=\pa_u\left[\GG_\eps(h)(u,\tau)\right].
\end{equation}
We will consider $\GG_\eps$ defined in $\HH_{\alpha,\rr,\sigma}$
with $\alpha>0$ in order the integral in \eqref{def:operadorGInfty}
to be convergent.
\begin{lemma}\label{lemma:PropietatsGInfty}
Let $\alpha>0$. Then, the operators $\GG_\eps$ and
$\overline\GG_\eps$ in \eqref{def:operadorGInfty} and
\eqref{def:operadorGInftyBarra} respectively  satisfy the following
properties.
\begin{enumerate}
\item $\GG_\eps$ is linear from $\HH_{\alpha,\rr,\sigma}$ to
itself, commutes with $\pa_u$ and
$\LL_\eps\circ\GG_\eps=\mathrm{Id}$.
\item If $h\in\HH_{\alpha, \rr,\sigma}$, then
\[
\left\|\GG_\eps(h)\right\|_{\alpha,\rr,\sigma}\leq
K\|h\|_{\alpha,\rr,\sigma}.
\]
Furthermore, if $\langle h\rangle=0$, then
\[
\left\|\GG_\eps(h)\right\|_{\alpha,\rr,\sigma}\leq
K\eps\|h\|_{\alpha,\rr,\sigma}.
\]
\item If $h\in\HH_{\alpha,\rr,\sigma}$, then $\overline\GG_\eps(h)\in
\HH_{\alpha,\rr,\sigma}$ and
\[
\left\|\overline\GG_\eps(h)\right\|_{\alpha,\rr,\sigma}\leq
K\|h\|_{\alpha,\rr,\sigma}.
\]
\end{enumerate}
\end{lemma}

\begin{proof}
It follows the same lines as the proof of Lemma 5.5 in
\cite{GuardiaOS10}.
\end{proof}

Finally, we state a technical lemma about estimates of the functions
$A$, $B_1$, $B_2$ and $C$ defined in \eqref{def:InftyA},
\eqref{def:InftyB1}, \eqref{def:InftyB2} and \eqref{def:InftyC}
respectively.

\begin{lemma}\label{lemma:Infty:Cotes}
Let $\{\lambda,-\lambda\}$ be the eigenvalues of the hyperbolic
critical point of the unperturbed Hamiltonian system and
$\overline\GG_\eps$ the operator defined in \eqref{def:operadorGInftyBarra}.
Let us fix $\rr_0$ big enough such
that $p_0(u)\neq 0$ in $D^u_{\infty,\rr_0}$ defined in
\eqref{def:DominsInfinit}. Then, for any
$\rr>\rr_0$, the functions $A$, $B_1$, $B_2$ and $C$ defined in
\eqref{def:InftyA}, \eqref{def:InftyB1}, \eqref{def:InftyB2} and
\eqref{def:InftyC} satisfy the following properties,
\begin{enumerate}
\item $A,\pa_u A\in\HH_{2\lambda,\rr,\sigma}$ and satisfy
\begin{equation}\label{eq:Infty:CotesA}
\begin{array}{cc}
\left\|\overline\GG_\eps (A)\right\|_{2\lambda,\rr,\sigma}\leq
K|\mu|\eps^{\eta+1}, &\left\|\pa_u
A\right\|_{2\lambda,\rr,\sigma}\leq K|\mu|\eps^{\eta}.
\end{array}
\end{equation}
\item $B_1, \pa_u B_1, B_2 \in \HH_{0,\rr,\sigma}$ and satisfy
\begin{equation}\label{eq:Infty:CotesB}
\begin{array}{ccc}
\|B_1\|_{0,\rr,\sigma}\leq K|\mu|\eps^{\eta},
&\|\pa_uB_1\|_{0,\rr,\sigma}\leq
K|\mu|\eps^{\eta},&\|B_2\|_{0,\rr,\sigma}\leq K|\mu|\eps^{\eta+1}.
\end{array}
\end{equation}
\item Let $h_1,h_2\in B(\nu)\subset\HH_{2\la,\rr,\sigma}$. Then,
\[
\left\| C(h_2,u,\tau)- C(h_1,u,\tau)\right\|_{2\la,\rr,\sigma}\leq K
\nu \|h_2-h_1\|_{2\la,\rr,\sigma}.
\]
\end{enumerate}
\end{lemma}
\begin{proof}
For the first bounds, we split $A=A_1+A_2+A_3$ as
\begin{align}
A_1(u, \tau)=&-\left(V(q_0(u)+\xp(\tau))-V(\xp(\tau))-V(q_0(u))-V'(\xp(\tau))q_0(u)\right)\label{def:HJ:A1}\\
A_2(u,\tau)=&-\mu\eps^{\eta}\wh H_1^1(q_0(u), p_0(u),\tau)\label{def:HJ:A2}\\
A_3(u,\tau)=&-\mu\eps^{\eta+1}\wh H_1^2(q_0(u),
p_0(u),\tau),\label{def:HJ:A3}
\end{align}
where $\wh H_1^1$ and $\wh H_1^2$ are the functions defined in
\eqref{def:HamPertorbat:H1} and \eqref{def:HamPertorbat:H2}.

For $A_1$, using the mean value theorem and Hypothesis
\textbf{HP1.1}, one can see that
\begin{equation}\label{eq:Formula:A1}
\begin{split}
A_1(u,\tau)=&-q_0^2(u)\int_0^1\left( V^{''}\left(x_p(\tau)+s_1
q_0(u)\right)-V^{''}\left(s_1q_0(u)\right)\right)(1-s_1)\,ds_1\\
=&-q_0^2(u)x_p(\tau)\int_0^1\int_0^1
V^{'''}\left(s_2x_p(\tau)+s_1q_0(u)\right)(1-s_1)\,ds_1ds_2.
\end{split}
\end{equation}
Therefore, $A_1\in\HH_{2\la,\rr,\sigma}$ and
$\|A_1\|_{2\la,\rr,\sigma}\leq K|\mu|\eps^{\eta+1}$. Applying Lemma
\ref{lemma:PropietatsGInfty}, we obtain
$\|\ol\GG_\eps(A_1)\|_{2\la,\rr,\sigma}\leq K|\mu|\eps^{\eta+1}$.

For the other terms, let us point out that, by construction, $\wh
H^1_1$ and $\wh H_1^2$ are quadratic in $(q,p)$ and therefore
$A_2,A_3\in \HH_{2\lambda,\rr,\sigma}$. To bound $\ol\GG_\eps(A_2)$,
using that $\langle A_2\rangle=0$ and taking
into account that $A_2$ is analytic in
$D_{\infty,\rr_0}^u\times\TT_\sigma$ and $\rr>\rr_0$, by Lemmas
\eqref{lemma:Infty:PropietatsNormes} and
\ref{lemma:PropietatsGInfty},
\[
\|\ol\GG_\eps(A_2)\|_{2\la,\rr,\sigma}\leq
K\eps\|A_2\|_{2\la,\rr,\sigma}\leq K|\mu|\eps^{\eta+1}.
\]
On the other hand, since by Corollary \ref{coro:ShiftPeriodica},
$\|A_3\|_{2\la,\rr,\sigma}\leq K|\mu|^2\eps^{2\eta+1}$, we have that
$\|\ol\GG_\eps(A_3)\|_{2\la,\rr,\sigma}\leq K|\mu|^2\eps^{2\eta+1}$.
Therefore
\[
\left\|\overline\GG_\eps (A)\right\|_{2\lambda,\rr,\sigma}\leq
K|\mu|\eps^{\eta+1}.
\]
The bound for $\pa_u A$ can be obtained just differentiating $A_i$,
$i=1,2,3$.

 The other bounds are straightforward.
\end{proof}
\subsubsection{Proof of Theorem \ref{th:ExistenceCloseInfty} in the hyperbolic
case}\label{sub:Infty:Hyp:ProofGeneral} We devote this section to
prove Theorem \ref{th:ExistenceCloseInfty} for the case in which the
unperturbed Hamiltonian has a hyperbolic critical point. First we
rewrite it in terms of the Banach spaces defined in
\eqref{def:BanachInftyHyp}.

\begin{proposition}\label{prop:HJ:hyp:general}
Let  $\{\lambda,-\lambda\}$ be the eigenvalues of the unperturbed
hyperbolic critical point, $\rr_1>0$ big enough and $\eps_0>0$ small
enough. Then, for $\eps\in(0,\eps_0)$, there exists a function
$T_1(u,\tau)$ defined in $D_{\infty,\rr_1}^{u}\times\TT_\sigma$
which satisfies equation \eqref{eq:HJperT1} and the asymptotic
condition \eqref{eq:AsymptCondFuncioGeneradora:uns}. Moreover, there
exists a constant $b_1>0$ such that
\[
\|\pa_u T_1\|_{2\lambda,\rr_1,\sigma}\leq b_1|\mu|\eps^{\eta+1}.
\]
\end{proposition}
Theorem \ref{th:ExistenceCloseInfty} is a straightforward
consequence of this proposition.

Let us observe that the operator $\FF$ defined in
\eqref{eq:infty:lmenor:operadorF} has linear terms in $w$ which are
not small when $\eta=0$. Therefore, if one wants to prove the
existence of $T$ through a fixed point argument, first we must look
for a change of variables. Let us point out that this change of
variables is not necessary for the case $\eta>0$.

\begin{lemma}\label{lemma:infty:ligual:canvi}
Let $\rr_1>\rr_0'>\rr_0>0$, where $\rr_0$ is big enough
such that $p_0(u)\neq 0$ for $u\in D_{\infty,\rr_0}^u$. Then, for
$\eps>0$ small enough, there exists a function
$g\in\HH_{0,\rr_0',\sigma}$ such that $\langle g\rangle=0$ and is
solution of
\begin{equation}\label{eq:Infty:Canvi}
\LL_\eps g=-B_1(v,\tau),
\end{equation}
where $\LL_\eps$ is the operator defined in \eqref{def:Lde} and
$B_1$ is the function defined in \eqref{def:InftyB1}. Moreover, it
satisfies that
\[
\begin{array}{cc}
\|g\|_{0,\rr_0',\sigma}\leq K|\mu|\eps^{\eta+1},&\|\pa_v
g\|_{0,\rr_0',\sigma}\leq K|\mu|\eps^{\eta+1}
\end{array}
\]
and $v+g(v,\tau)\in D^{u}_{\infty, \rr_0}$  for $(v,\tau)\in
D_{\infty, \rr_0'}^{u}\times\TT_\sigma$.

Furthermore, $(u,\tau)=(v+g(v,\tau),\tau)$ is invertible and its
inverse is of the form $(v,\tau)=(u+h(u,\tau),\tau)$, where $h$ is a
function defined for $(u,\tau)\in
D_{\infty,\rr_1}^{u}\times\TT_\sigma$ and satisfies that $h\in
\HH_{0,\rr_1,\sigma}$,
\[
\|h\|_{0,\rr_1,\sigma }\leq K|\mu|\eps^{\eta+1}
\]
and that $u+h(u,\tau)\in D_{\infty,\rr_0'}^{u}$ for $(u,\tau)\in
D_{\infty,\rr_1}^{u}\times\TT_\sigma$.

\end{lemma}
\begin{proof}
From the definition of $B_1$ in \eqref{def:InftyB1} we have that
$\langle B_1\rangle=0$. On the other hand, using the definition of
$\wh H_1^1$ and $\lambda$ in \eqref{def:HamPertorbat:H1} and
\eqref{def:Potencial:Hyperbolic:PropZero} respectively, $B_1$ can be
split as
\[B_1(v,\tau)=B_{10}(\tau)+B_{11}(v,\tau),\] where, using
\eqref{eq:separatrix:hyp:infty}, \[B_{10}(\tau)=\lim_{\Re
v\rightarrow-\infty}B_{1}(v,\tau)=-\mu\eps^{\eta}\left(\frac{a_{11}(\tau)}{
\lambda}+2a_{02}(\tau)\right)\]
and $B_{11}(v,\tau)=B_1(v,\tau)-B_{10}(\tau)$. Both terms have zero
mean. Moreover,  $B_{10}\in \HH_{0,\rr_0',\sigma}$ and satisfies
$\|B_{10}\|_{0,\rr_0',\sigma}\leq K|\mu| \eps^{\eta}$ and $B_{11}\in
\HH_{\lambda,\rr_0',\sigma}$ and satisfies
$\|B_{11}\|_{\lambda,\rr_0',\sigma}\leq K|\mu|\eps^{\eta}$.

Since
$B_{10}(\tau)=\sum_{k\in\ZZ\setminus\{0\}}B_{10}^{[k]}e^{ik\tau}$
has zero average, we can define a $2\pi$-periodic primitive with
zero average as
\[\overline
B_{10}(\tau)=\sum_{k\in\ZZ\setminus\{0\}}\frac{B_{10}^{[k]}}{ik}e^{ik\tau}
\]
which satisfies $\|\overline B_{10}\|_{0,\rr_0',\sigma}\leq
K|\mu|\eps^{\eta}$.

By the linearity of equation \eqref{eq:Infty:Canvi}, we can take $g$
as
\[
g(v,\tau)=-\eps\overline
B_{10}(\tau)-\GG_\eps\left(B_{11}\right)(v,\tau),
\]
where $\GG_\eps$ is the operator defined in
\eqref{def:operadorGInfty}. Moreover, using the first statement of
Lemma \ref{lemma:Infty:PropietatsNormes} and Lemma
\ref{lemma:PropietatsGInfty},
\[
\|g\|_{0,\rr_0',\sigma}\leq \eps\left\|\overline
B_{10}\right\|_{0,\rr_0',\sigma}+\left\|\GG_\eps\left(B_{11}\right)\right\|_{\la
,\rr_0',\sigma}\leq
K|\mu|\eps^{\eta+1}+K\eps\left\|B_{11}\right\|_{\lambda,\rr_0',\sigma}\leq
K|\mu|\eps^{\eta+1}.
\]
Moreover, by Lemma \ref{lemma:PropietatsGInfty},
\[
\pa_v g=-\pa_v\GG_\eps \left(B_{11}\right)=-\GG_\eps \left(\pa_v
B_{11}\right)
\]
and then,
\[
\|\pa_v g\|_{0,\rr_0',\sigma}\leq \|\pa_v
g\|_{\la,\rr_0',\sigma}=\|\GG_\eps \left(\pa_v
B_{11}\right)\|_{\la,\rr_0',\sigma}\leq K\eps \left\|\pa_v
B_{11}\right\|_{\la,\rr_0',\sigma}\leq K|\mu|\eps^{\eta+1}.
\]
Since $\|g\|_{0,\rr_0',\sigma}\leq K|\mu|\eps^{\eta+1}$, we have
that $v+g(v,\tau)\in D_{\infty, \rr_0}^{u}$ for $(v,\tau)\in
D_{\infty, \rr_0'}^{u}\times\TT_\sigma$ provided $\eps$ is small
enough and $\rr_0'>\rr_0$.

To obtain the inverse change and its properties it is
straightforward.
\end{proof}
If we apply the change of variables $u=v+g(v,\tau)$ to equation
\eqref{eq:HJperT1}, one can see that
\[
\wh T_1(v,\tau)=T_1\left(v+g(v,\tau),\tau\right)
\]
is solution of
\begin{equation}\label{eq:HJperT1:ligual}
\LL_\eps \wh T_1=\wh\FF \left(\pa_v \wh T_1\right),
\end{equation}
where
\begin{equation}\label{eq:HJperT1:ligual:RHS}
\dps \wh\FF(h)(v,\tau)=\dps \wh A(v,\tau)+\wh B(v,\tau)
h(v,\tau)+\wh C(h(v,\tau),v,\tau),
\end{equation}
with
\begin{align}
\wh A(v,\tau)&=A\left(v+g(v,\tau),\tau\right)\label{def:InftyHyp:Ahat}\\
 \wh
B(v,\tau)&=\frac{B_1\left(v+g(v,\tau),\tau\right)-B_1(v,\tau)+B_2\left(v+g(v,
\tau),\tau\right)}{1+\pa_v
g(v,\tau)}\label{def:InftyHyp:Bhat}\\
\wh C(w,v,\tau)&=C\left(\frac{1}{1+\pa_v
g(v,\tau)}w,v+g(v,\tau),\tau\right)\label{def:InftyHyp:Chat},
\end{align}
where the functions $A(u,\tau)$, $B_1(u,\tau)$ and $B_2(u,\tau)$
 are defined in \eqref{def:InftyA}, \eqref{def:InftyB1} and \eqref{def:InftyB2}.

We look for $\wh T_1$
by using a fixed point argument for $\pa_v \wh T_1$ instead of $\wh
T_1$ itself. Therefore, we look for a fixed point of the operator
\begin{equation}\label{def:Infty:FullOperator:ligual}
\overline\FF=\overline \GG_\eps\circ\wh \FF,
\end{equation}
where $\overline \GG_\eps$ is the operator in
\eqref{def:operadorGInftyBarra}, in the Banach space
$\HH_{2\lambda,\rr_0', \sigma}$ defined in
\eqref{def:BanachInftyHyp}.

\begin{lemma}\label{lemma:HJ:hyp:ligual}
Let $\rr_0'$ be defined in Lemma \ref{lemma:infty:ligual:canvi} and
$\eps_0>0$ small enough. Then, for $\eps\in(0,\eps_0)$ there exists
a function $\wh T_1(v,\tau)$ defined in
$D_{\infty,\rr_0'}^{u}\times\TT_\sigma$ such that $\pa_v\wh
T_1\in\HH_{2\lambda,\rr_0',\sigma}$ is a fixed point of the operator
\eqref{def:Infty:FullOperator:ligual}. Furthermore, there exists a
constant $b_1>0$ such that,
\[
\left\|\pa_v\wh T_1\right\|_{2\lambda,\rr_0',\sigma}\leq
b_1|\mu|\eps^{\eta+1}.
\]

\end{lemma}
\begin{proof}
It is straightforward to see that $\ol\FF$ is well defined from
$\HH_{2\lambda,\rr_0',\sigma}$ to itself. We are going to prove that
there exists a constant $b_1>0$ such that $\ol\FF$ sends $\ol
B(b_1|\mu|\eps^{\eta+1})\subset \HH_{2\lambda,\rr_0',\sigma}$ to
itself and it is contractive there.

Let us first consider $\ol\FF(0)$. From the definition of $\ol\FF$
in \eqref{def:Infty:FullOperator:ligual} and the definition of
$\wh\FF$ in \eqref{eq:HJperT1:ligual:RHS}, we have that
\[
\overline\FF(0)(v,\tau)=\ol\GG_\eps\left(\wh
A\right)(v,\tau)=\ol\GG_\eps(A)(v,\tau)+\ol\GG_\eps\left(\wh
A-A\right)(v,\tau).
\]
The first term was already bounded in Lemma \ref{lemma:Infty:Cotes}.
For the second one, it is enough to use mean value theorem and
Lemmas \ref{lemma:Infty:Cotes} and \ref{lemma:infty:ligual:canvi} to
bound $\pa_uA$ and $g$ respectively, to obtain
\[
\left\|A(v+g(v,\tau),\tau)-A(v,\tau)\right\|_{2\lambda,\rr_0',\sigma}\leq
K|\mu|^2\eps^{2\eta+1}.
\]
Thus, applying Lemma \ref{lemma:PropietatsGInfty}, there exists
constant a $b_1>0$ such that
\[\left\|\ol\FF(0)\right\|_{2\lambda,\rr_0',\sigma}\leq\frac{b_1}{2}|\mu|\eps^{
\eta+1}.\]

Now, let $h_1,h_2\in\ol
B(b_1|\mu|\eps^{\eta+1})\in\HH_{2\lambda,\rr_0',\sigma}$. Then,
using the properties of $\ol\GG_\eps$ in Lemma
\ref{lemma:PropietatsGInfty} and the definition  of $\wh\FF$ in
\eqref{eq:HJperT1:ligual:RHS}
\[
\begin{split}
\dps\left\| \ol \FF(h_2)-\ol
\FF(h_1)\right\|_{2\lambda,\rr_0',\sigma}&\dps\leq K\left\|
\wh\FF(h_2)-\wh\FF(h_1)\right\|_{2\lambda,\rr_0',\sigma}\\
&\leq \dps K\left\|\wh B\cdot (h_2-h_1)+\wh C(h_2,u,\tau)-\wh
C(h_1,u,\tau) \right\|_{2\lambda,\rr_0',\sigma}.
\end{split}
\]
Taking into account the definitions of $\wh B$ and $\wh C$ in
\eqref{def:InftyHyp:Bhat} and \eqref{def:InftyHyp:Chat} respectively
and applying Lemmas \ref{lemma:Infty:PropietatsNormes},
\ref{lemma:Infty:Cotes} and \ref{lemma:infty:ligual:canvi}, we
 obtain
\[
\dps\left\| \ol \FF(h_2)-\ol
\FF(h_1)\right\|_{2\lambda,\rr_0',\sigma}\leq
K|\mu|\eps^{\eta+1}\|h_2-h_1\|_{2\lambda,\rr_0',\sigma}.
\]
Therefore, reducing $\eps$ if necessary, $\mathrm{Lip}\ol\FF\leq
1/2$ and therefore $\ol\FF$ is contractive from the ball $\ol
B(b_1|\mu|\eps^{\eta+1})\subset\HH_{2\lambda,\rr_0',\sigma}$ into
itself, and  it has a unique fixed point $h^\ast$. Since it
satisfies
\[
\left|h^\ast(v,\tau)\right|\leq b_1|\mu|\eps^{\eta+1}e^{2\la \Re v}
\]
for $(v,\tau)\in D_{\infty,\rr_0'}^u\times\TT_\sigma$, we can take
$\wh T_1$ as
\[
\wh T_1(v,\tau)=\int_{-\infty}^v h^\ast (w,\tau)\,dw.
\]
\end{proof}
Finally, to prove Proposition \ref{prop:HJ:hyp:general}  from Lemma
\ref{lemma:HJ:hyp:ligual}, it is enough to consider the change
$v=u+h(u,\tau)$ obtained in Lemma \ref{lemma:infty:ligual:canvi},
take $T_1(u,\tau)= \wh T_1(u+h(u,\tau),\tau)$ and increase slightly
$b_1$ if necessary.

\subsection{Local invariant manifolds in the parabolic
case}\label{sec:InftyParabolic}

We devote this section to prove the
existence of suitable representations of the unstable and stable
invariant manifolds in the domains
$D_{\infty,\rr}^u\times\TT_\sigma$ and
$D_{\infty,\rr}^s\times\TT_\sigma$ respectively, under the
hypotheses that the unperturbed Hamiltonian system has a parabolic
critical point at the origin. We proceed as we have done in Section
\ref{sec:InftyHyperbolic} for the hyperbolic case, that is, solving
equation \eqref{eq:HJperT1}. Let us point out that in the parabolic
case, by Hypothesis \textbf{HP4.2}, the perturbation is taken in
such a way that the periodic orbit remains at the origin.

\subsubsection{Banach spaces and technical lemmas}
Given $\alpha\geq0$, $\rr\geq 0$  and an analytic function
$h:D^{u}_{\infty,\rr}\rightarrow \CC$, we define
\[
\|h\|_{\alpha,\rr}=\sup_{u\in D^{u}_{\infty,\rr}}\left| u^\alpha
h(u)\right|.
\]
Moreover for $2\pi$-periodic in $\tau$, analytic functions
$h:D^{u}_{\infty,\rr}\times\TT_\sigma\rightarrow \CC$, we define the
corresponding Fourier norm
\[
\|h\|_{\alpha,\rr,\sigma}=\sum_{k\in\ZZ}\left\|h^{[k]}\right\|_{\alpha,\rr}e^{
|k|\sigma}.
\]
We introduce, thus, the following function space

\begin{equation}\label{def:BanachInftyParab}
\PP_{\alpha,\rr,\sigma}=\{
h:D^{u}_{\infty,\rr}\times\TT_\sigma\rightarrow\CC;\,\,\text{real-analytic},
\|h\|_{\alpha,\rr,\sigma}<\infty\},
\end{equation}
which can be checked that is a Banach space for any fixed
$\alpha\geq 0$.

%If there is no danger of confusion about the definition domain
%$D^{u(i)}_{\infty,\rr}$ we will denote
%\[
%\begin{array}{ccc}
%\|\cdot\|_{\alpha,\sigma}=\|\cdot\|_{\alpha,\sigma,i}&\text{ and
%}&\HH_{\alpha,\sigma}=\HH_{\alpha,\sigma}^{(i)}.
%\end{array}
%\]
In the next lemma, we state some properties of these Banach spaces.
\begin{lemma}\label{lemma:Infty:PropietatsNormes:Parab} The following statements
hold:
\begin{enumerate}
\item If $\alpha_1\geq \alpha_2\geq 0$, then $\PP_{\alpha_1,\rr,\sigma}\subset
\PP_{\alpha_2,\rr,\sigma}$ and
\[
\|h\|_{\alpha_2,\rr,\sigma}\leq \|h\|_{\alpha_1,\rr,\sigma}.
\]
\item If $\alpha_1,\alpha_2\geq 0$, then, for $h\in\PP_{\alpha_1,\rr,\sigma}$
and $g\in\PP_{\alpha_2,\rr,\sigma}$,
we  have that $hg\in\PP_{\alpha_1+\alpha_2,\rr,\sigma}$ and
\[
\|hg\|_{\alpha_1+\alpha_2,\rr,\sigma}\leq
\|h\|_{\alpha_1,\rr,\sigma}\|g\|_{\alpha_2,\rr,\sigma}.
\]
%\item If $i\geq 0$ and $\alpha>0$, then for
%$h\in\HH_{\alpha,\sigma}^{(i)}$ we have that $\pa_u
%h\in\HH_{\alpha,\sigma}^{(i+1)}$ and
%\[
%\left\| \pa_u h\right\|_{\alpha,\sigma,i+1}\leq
%K\|h\|_{\alpha,\sigma,i}
%\]
\end{enumerate}
\end{lemma}

As in Section \ref{sec:InftyHyperbolic}, we need to use the
operators  $\GG_\eps$ and $\bar\GG_\eps$ formally defined in
\eqref{def:operadorGInfty} and \eqref{def:operadorGInftyBarra}
respectively.

\begin{lemma}\label{lemma:PropietatsGInfty:parab}
The operators $\GG_\eps$ and  $\overline\GG_\eps$ acting on the
spaces $\PP_{\alpha,\rr,\sigma}$ with $\alpha>1$ satisfy the
following properties.
\begin{enumerate}
\item For any $\alpha>1$, $\GG_\eps
:\PP_{\alpha,\rr,\sigma}\rightarrow\PP_{\alpha-1,\rr,\sigma}$ is
well defined and linear continuous. Moreover, commutes with $\pa_u$
and $\LL_\eps\circ\GG_\eps=\mathrm{Id}$.
\item If $h\in\PP_{\alpha,\rr,\sigma}$ for some $\alpha> 1$, then
\[
\left\|\GG_\eps(h)\right\|_{\alpha-1,\rr,\sigma}\leq
K\|h\|_{\alpha,\rr,\sigma}.
\]
Furthermore, if $h\in\PP_{\alpha,\rr,\sigma}$ for some $\alpha> 0$
and $\langle h\rangle=0$, then
\[
\left\|\GG_\eps(h)\right\|_{\alpha,\rr,\sigma}\leq
K\eps\|h\|_{\alpha,\rr,\sigma}.
\]
\item If $h\in\PP_{\alpha,\rr,\sigma}$ for some $\alpha\geq 1$, then
$\overline\GG_\eps(h)\in
\PP_{\alpha,\rr,\sigma}$ and
\[
\left\|\overline\GG_\eps(h)\right\|_{\alpha,\rr,\sigma}\leq
K\|h\|_{\alpha,\rr,\sigma}.
\]
\end{enumerate}
\end{lemma}

We also state  a technical lemma about properties of the functions
$A$, $B_1$ and $C$ defined in \eqref{def:InftyA},
\eqref{def:InftyB1} and \eqref{def:InftyC} respectively. Notice
that now the function $B_2$ defined in \eqref{def:InftyB2}
satisfies $B_2=0$ since, by hypothesis, the perturbation fixes the
periodic orbit at the origin.

We first fix $\rr_0>0$ such that $p_0(u)$ does not vanish in
$D^{u}_{\infty,\rr_0}$ and we define the constant
\begin{equation}\label{def:alpha0}
\alpha_0=\frac{2n}{m-2}>1,
\end{equation}
where $m$ is the order of the potential
\eqref{def:Potencial:Parabolic} and $n$ is the order of the
perturbation \eqref{def:Ham:Original:perturb:poli}. We observe that
$q_0(u)\in \PP_{\frac{2}{m-2},\rr,\sigma}$ and $p_0(u)\in
\PP_{\frac{m}{m-2},\rr,\sigma}$ for any  $\rr$ big enough and any
$\sigma>0$.

\begin{lemma}\label{lemma:Infty:Cotes:parab}
Let us consider $\rr>\rr_0$. Then, the functions $A$, $B_1$ and $C$
defined in \eqref{def:InftyA}, \eqref{def:InftyB1} and
\eqref{def:InftyC} satisfy the following properties,
\begin{enumerate}
\item $A\in\PP_{\alpha_0,\rr,\sigma}$ and $\pa_u
A\in\PP_{\alpha_0+1,\rr,\sigma}$. Moreover, $\langle A\rangle=\langle
\pa_uA\rangle=0$ and
\begin{equation}\label{eq:Infty:CotesA:parab}
\begin{array}{cc}
\left\| \pa_uA\right\|_{\alpha_0+1,\rr,\sigma}\leq
K|\mu|\eps^{\eta}, &\left\|\overline\GG_\eps
(A)\right\|_{\alpha_0+1,\rr,\sigma}\leq K|\mu|\eps^{\eta+1}.
\end{array}
\end{equation}
\item $B_1 \in \PP_{\frac{2n-m-2}{m-2},\rr,\sigma}$ and $\pa_uB_1 \in
\PP_{\frac{2n-m-2}{m-2}+1,\rr,\sigma}$.  Moreover, they
satisfy
\begin{equation}\label{eq:Infty:CotesB:parab}
\begin{array}{cc}
\|B_1\|_{\frac{2n-m-2}{m-2},\rr,\sigma}\leq
K|\mu|\eps^{\eta},&\|\pa_uB_1\|_{\frac{2n-m-2}{m-2}+1,\rr,\sigma}\leq
K|\mu|\eps^{\eta}.
\end{array}
\end{equation}
\item Let $h_1,h_2\in B(\nu)\subset\PP_{\alpha_0+1,\rr,\sigma}$ with $\nu\ll 1$.
Then,
\[
\left\| C(h_2,u,\tau)-
C(h_1,u,\tau)\right\|_{\alpha_0+1,\rr,\sigma}\leq K \nu
\|h_2-h_1\|_{\alpha_0+1,\rr,\sigma}.
\]
\end{enumerate}
\end{lemma}
\begin{proof}
We prove the lemma in the polynomial case. The trigonometric one can
be done analogously. For the first statement, recall that
in the parabolic case the periodic orbit is located at the origin by
Hypothesis \textbf{HP4.2}. Then
\[
A(u,\tau)=-\mu\eps^{\eta} H_1 (q_0(u),p_0(u),\tau),
\]
where $H_1$ is the function defined in
\eqref{def:Ham:Original:perturb:poli} and has zero mean. On the
other hand, it is clear that the monomial with lowest order as $\Re
u\rightarrow +\infty$ corresponds to $a_{n0}q_0^n(u)$ which behaves
as
\[
a_{n0}(\tau)q_0^n(u)\sim\frac{1}{u^{\alpha_0}}.
\]
Then $A\in \PP_{\alpha_0,\rr,\sigma}$, that implies $\pa_u A\in
\PP_{\alpha_0+1,\rr,\sigma}$ and
\[
\left\|\pa_uA\right\|_{\alpha_0+1,\rr,\sigma}\leq K|\mu|\eps^{\eta}.
\]
Moreover, by Lemma \ref{lemma:PropietatsGInfty:parab},
\[
\left\|\ol\GG_\eps (A)\right\|_{\alpha_0+1,\rr,\sigma}=
\left\|\GG_\eps (\pa_u A)\right\|_{\alpha_0+1,\rr,\sigma}\leq
K|\mu|\eps^{\eta+1}.
\]
For the second statement, let us recall that
\[
B_1(u,\tau)=-\mu\eps^{\eta}\sum_{\substack{i+j=n\\ j\geq
1}}^Na_{ij}(\tau)q_0^i(u)p_0^{j-2}(u).
\]
As $\Re u\rightarrow-\infty$, the monomials of $B_1$ behave as
\[
a_{ij}(\tau)q_0^i(u)p_0^{j-2}(u)\sim
u^{-\left(\frac{2}{m-2}i+\left(\frac{2}{m-2}+1\right)(j-2)\right)}.
\]
Taking into account that $2n-2\geq m$ by Hypothesis \textbf{HP5} and
that $i+j\geq n$ and $j\geq 1$,
\[
\begin{split}
\frac{2}{m-2}i+\left(\frac{2}{m-2}+1\right)(j-2)&=\frac{2}{m-2}(i+j)+j-\frac{2m}
{m-2}\\
&\geq \frac{2n}{m-2}+1-\frac{2m}{m-2}.
\end{split}
\]
Therefore $B_1\in \PP_{\frac{2n-m-2}{m-2},\rr,\sigma}$ and satisfies
$\|B_1\|_{\frac{2n-m-2}{m-2},\rr,\sigma}\leq K|\mu|\eps^{\eta}$. For
$\pa_u B_1$, it is enough to differentiate. For the case $2n-2> m$
we have that $\pa_u B_1\in \PP_{\frac{2n-m-2}{m-2}+1,\rr,\sigma}$.
In the case $2n-2=m$ we have that
\[
\pa_u B_1\in
\PP_{\frac{1}{m-2}+1,\rr,\sigma}\subset\PP_{\frac{2n-m-2}{m-2}+1,\rr,\sigma}.
\]
In both cases, we have that
$\|\pa_uB_1\|_{\frac{2n-m-2}{m-2}+1,\rr,\sigma}\leq
K|\mu|\eps^{\eta}$.

We bound the third term in the polynomial case. We split $C=C_1+C_2$ as
\[
\begin{split}
C_1(w,u,\tau)&=-\frac{w^2}{2p_0^2(u)}\\
C_2(w,u,\tau)&=-\mu\eps^{\eta} \sum_{\substack{i+j=n\\j\geq 1}}^N
a_{ij}(\tau)
q_0^i(u)p_0^j(u)\left(\left(1+\frac{w}{p_0^2(u)}\right)^j-1-j\frac{w}{p_0^2(u)}
\right).
\end{split}
\]
Let $h_1,h_2\in B(\nu)\subset\PP_{\alpha_0+1,\rr,\sigma}$. Then, for
the first term,
\[
\begin{split}
\left\|C_1(h_2,u,\tau)-C_1(h_1,u,\tau)\right\|_{\alpha_0+1,\rr,\sigma}&\leq K
\left\|p_0(u)^{-2}(h_2+h_1)\right\|_{0,\rr,\sigma}\left\|h_2-h_1\right\|_{
\alpha_0+1,\rr,\sigma}\\
&\leq
K\left\|h_2+h_1\right\|_{2m/(m-2),\rr,\sigma}\left\|h_2-h_1\right\|_{\alpha_0+1,
\rr,\sigma}.
\end{split}
\]
By Hypotheses \textbf{HP5}, we have $2n-2\geq m$ which implies
$2m/(m-2)\leq \alpha_0+1$ and therefore
\[
\left\|h_2+h_1\right\|_{2m/(m-2),\rr,\sigma}\leq
\left\|h_2+h_1\right\|_{\alpha_0+1,\rr,\sigma}\leq K\nu.
\]
Reasoning analogously, one can see that
\[
\left\|C_2(h_2,u,\tau)-C_2(h_1,u,\tau)\right\|_{\alpha_0+1,\rr,\sigma}\leq
K|\mu|\eps^{\eta}\nu\left\|h_2-h_1\right\|_{\alpha_0+1,\rr,\sigma}.
\]
\end{proof}

\subsubsection{Proof of Theorem \ref{th:ExistenceCloseInfty} in the
parabolic case}\label{sub:Infty:Parab:ProofGeneral}

We devote this
section to prove Theorem \ref{th:ExistenceCloseInfty} for the case
in which the unperturbed Hamiltonian has a parabolic critical point.
First we rewrite it in terms of the Banach spaces defined in
\eqref{def:BanachInftyParab}.

\begin{proposition}\label{prop:HJ:parab:general}
Let the constant $\alpha_0$ be defined in
\eqref{def:alpha0}, $\rr_1>0$ big enough and  $\eps_0>0$ small
enough. Then, for $\eps\in(0,\eps_0)$, there exists a function
$T_1(u,\tau)$ defined in $D_{\infty,\rr_1}^{u}\times\TT_\sigma$
which satisfies equation \eqref{eq:HJperT1} and the asymptotic
condition \eqref{eq:AsymptCondFuncioGeneradora:uns}. Moreover,
$\pa_uT_1\in\PP_{\alpha_0+1,\rr_1,\sigma}$ and there exists a
constant $b_1>0$ such that
\[
\|\pa_u T_1\|_{\alpha_0+1,\rr_1,\sigma}\leq b_1|\mu|\eps^{\eta+1}.
\]
\end{proposition}
Theorem \ref{th:ExistenceCloseInfty} is a straightforward
consequence of this proposition.

The proof of this proposition follows the same steps as the proof of
Proposition \ref{prop:HJ:hyp:general}.

The first step is to perform a change of variables which reduces the
size of the linear term of $\FF$ in
\eqref{eq:infty:lmenor:operadorF}. This change is not necessary for
the case $\eta>0$.

\begin{lemma}\label{lemma:Parab:canvi}
Let $\rr_0'$ be such that $\rr_0<\rr_0'<\rr_1$. Then, for $\eps>0$ small enough,
there exists a function $g\in\PP_{0,\rr_0',\sigma}$ such that
$\langle g\rangle=0$ and is a solution of \eqref{eq:Infty:Canvi}.
Moreover, it satisfies that
\[
\begin{array}{cc}
\| g\|_{0,\rr_0',\sigma}\leq K|\mu|\eps^{\eta+1}, &\|
\pa_vg\|_{0,\rr_0',\sigma}\leq K|\mu|\eps^{\eta+1},
\end{array}
\]
and $v+g(v,\tau)\in D_{\infty,\rr_0}^u$ for $(v,\tau)\in D_{\infty,
\rr_0'}^u\times\TT_\sigma$.

Furthermore, $(u,\tau)=(v+g(v,\tau),\tau)$ is invertible and its
inverse is of the form $(v,\tau)=(u+h(u,\tau),\tau)$, where $h$ is a
function defined for $(u,\tau)\in
D_{\infty,\rr_1}^u\times\TT_\sigma$ and satisfies that $h\in
\PP_{0,\rr_1,\sigma}$,
\[
\|h\|_{0,\rr_1,\sigma}\leq K|\mu|\eps^{\eta+1}
\]
and that $u+h(u,\tau)\in D^u_{\infty, \rr_0'}$ for $(u,\tau)\in
D_{\infty,\rr_1}^u\times\TT_\sigma$.
\end{lemma}
\begin{proof}
Since  $B_1 \in \PP_{\frac{2n-m-2}{m-2},\rr,\sigma}$ and it might
happen that $\frac{2n-m-2}{m-2}<1$, we cannot apply directly Lemma
\ref{lemma:PropietatsGInfty:parab} to invert $\LL_\eps$. Let us
observe that, by Lemma \ref{lemma:Infty:Cotes:parab}, $\langle
B_1\rangle=0$ and then we can define a function $\ol B_1$ such that
\[
\pa_\tau \ol B_1=B_1\,\,\,\text{ and }\,\,\,\langle \ol
B_1\rangle=0,
\]
which satisfies $\|\ol B_1 \|_{\frac{2n-m-2}{m-2},\rr,\sigma}\leq
K|\mu|\eps^{\eta}$.

We can define $g$ as
\[
g(v,\tau)=-\eps \ol B_1(v,\tau)+\eps \GG_\eps \left(\pa_v\ol
B_1\right)(v,\tau).
\]
Then, applying Lemmas \ref{lemma:PropietatsGInfty:parab} and
\ref{lemma:Infty:Cotes:parab} one obtains the bounds for $g$ and
$\pa_v g$.

 The proof of the other statements is analogous to the
proof of Lemma \ref{lemma:infty:ligual:canvi}.
\end{proof}
As in Section \ref{sub:Infty:Hyp:ProofGeneral}, we define
\[
\wh T_1(v,\tau)=T_1(v+g(v,\tau),\tau),
\]
which is a solution of \eqref{eq:HJperT1:ligual}. Then, we look for
$\pa_v\wh T_1$ as a fixed point of the operator
\eqref{def:Infty:FullOperator:ligual} in the Banach space
$\PP_{\alpha_0+1,\rr_0',\sigma}$.

\begin{lemma}\label{lemma:HJ:parab:lmajor}
Let $\alpha_0$ be the constant defined in
\eqref{def:alpha0} and $\eps_0>0$ small enough. Then, for
$\eps\in(0,\eps_0)$ there exists a function $\wh T_1(v,\tau)$
defined in $D_{\infty,\rr_0'}^{u}\times\TT_\sigma$ such that
$\pa_v\wh T_1\in\PP_{\alpha_0+1,\rr_0',\sigma}$ is a fixed point of
the operator \eqref{def:Infty:FullOperator:ligual}. Furthermore,
there exists a constant $b_1>0$ such that
\[
\left\|\pa_v\wh T_1\right\|_{\alpha_0+1,\rr_0',\sigma,0}\leq
b_1|\mu|\eps^{\eta+1}.
\]
\end{lemma}
\begin{proof}
It is straightforward to see that $\ol\FF$ is well defined from
$\PP_{\alpha_0+1,\rr_0',\sigma}$ to itself. We are going to prove
that there exists a constant $b_1>0$ such that $\ol\FF$ is
contractive in $\ol B(b_1|\mu|\eps^{\eta+1})\subset
\PP_{\alpha_0+1,\rr_0',\sigma}$.

Let us consider first $\ol\FF(0)$. From the definition of
$\overline\FF$ in \eqref{def:Infty:FullOperator:ligual} and the
definition of $\wh\FF$ in \eqref{eq:HJperT1:ligual:RHS}, we have
that
\[
\ol\FF(0)(v,\tau)=\ol\GG_\eps \left(\wh A(v,\tau)\right)=\ol\GG_\eps
\left(A(v,\tau)\right)+\ol\GG_\eps \left(
A(v+g(v,\tau),\tau)-A(v,\tau)\right).
\]
The first term has been bounded in Lemma
\ref{lemma:Infty:Cotes:parab}. For the second one, we apply Lemmas
 \ref{lemma:Infty:Cotes:parab} and \ref{lemma:Parab:canvi} and the mean value
theorem to obtain
\[
\left\|
A(v+g(v,\tau),\tau)-A(v,\tau)\right\|_{\alpha_0+1,\rr_0',\sigma}
\leq \|\pa_u A\|_{\alpha_0+1,\rr_0,\sigma}\|g\|_{0,\rr_0',\sigma}
\leq K|\mu|^2\eps^{2\eta+1}.
\]
Thus, applying Lemma \ref{lemma:PropietatsGInfty:parab}, there
exists a constant $b_1>0$ such that
\[
\left\|\overline\FF(0)\right\|_{\alpha_0+1,\sigma}\leq\frac{b_1}{2}|\mu|\eps^{
\eta+1}.
\]
Let $h_1,h_2\in \ol B(b_1|\mu|\eps^{\eta+1})\subset
\PP_{\alpha_0+1,\rr_0',\sigma}$. Then, using the properties of
$\overline\GG_\eps$ in Lemma~\ref{lemma:PropietatsGInfty:parab} and
the definition of $\wh\FF$ in  \eqref{eq:HJperT1:ligual:RHS},
\[
\begin{split}
\dps\left\| \ol \FF(h_2)-\ol
\FF(h_1)\right\|_{\alpha_0+1,\rr_0',\sigma}& \dps\leq K\left\|
\wh\FF(h_2)-\wh \FF(h_1)\right\|_{\alpha_0+1,\rr_0',\sigma}\\
&\leq \dps K\left\|\wh B\cdot (h_2-h_1)+\wh C(h_2,v,\tau)-\wh
C(h_1,v,\tau) \right\|_{\alpha_0+1,\rr_0',\sigma}.
\end{split}
\]
Taking into account the definitions of $\wh B$ and $\wh C$ in
\eqref{def:InftyHyp:Bhat} and \eqref{def:InftyHyp:Chat}, recalling
that $B_2=0$ and applying Lemmas
\ref{lemma:Infty:PropietatsNormes:Parab},
 \ref{lemma:Infty:Cotes:parab} and \ref{lemma:Parab:canvi},we
 obtain
\[
\dps\left\| \ol \FF(h_2)-\ol
\FF(h_1)\right\|_{\alpha_0+1,\rr_0',\sigma}\leq
K|\mu|\eps^{\eta+1}\|h_2-h_1\|_{\alpha_0+1,\rr_0',\sigma}.
\]
Then,  reducing $\eps$ if necessary, $\mathrm{Lip}\ol\FF<1/2$ and
then $\ol\FF$ is contractive from $\ol
B\left(b_1|\mu|\eps^{\eta+1}\right)\subset \PP_{\alpha_0+1,\sigma}$
to itself and has a unique fixed point $h^\ast$. Moreover, since it
satisfies
\[
\left|h^\ast(v,\tau)\right|\leq b_1 |\mu|\eps^{\eta+1}\frac{1}{|
v|^{\alpha_0+1}}
\]
for $(v,\tau)\in D_{\infty,\rr_0'}^{u}\times\TT_\sigma$, we can
define $\wh T_1$ as
\[
\wh T_1(v,\tau)=\int_{-\infty}^v h^\ast(w,\tau)\,dw.
\]
\end{proof}

To prove Proposition \ref{prop:HJ:parab:general} from Lemma
\ref{lemma:HJ:parab:lmajor}, as we have proceeded in Section
\ref{sub:Infty:Hyp:ProofGeneral}, it is enough to consider the
change of variables $v=u+h(u,\tau)$ obtained in Lemma
\ref{lemma:Parab:canvi}, take $T_1(u,\tau)= \wh
T_1(u+h(u,\tau),\tau)$ and increase slightly $b_1$ if necessary.

\section{Invariant manifolds in the outer domains: proof of Theorems
\ref{th:Extensio:Trig}
and \ref{th:ExtensioFinal}}\label{sec:Extensiogeneral}

\subsection{Invariant manifolds in the outer domains when $p_0(u)\neq 0$:
proof of Theorem \ref{th:Extensio:Trig}}\label{sec:Extensio}

In this section we prove the existence of the invariant manifolds in
the domains $D^{\out,\ast}_{\rr,\kk}\times\TT_\sigma$ for $\ast=u,s$
defined in \eqref{def:DominisOuter} provided $p_0(u)\neq0$ in these
domains. Since the proof for both invariant manifolds is analogous,
we only deal with the unstable case.

First in Section \ref{Subsection:ExistenceQ:Banach} we define some
Banach spaces and we state some technical lemmas. Then, in
Section \ref{subsec:Extensio:trig} we prove Theorem
\ref{th:Extensio:Trig}.

\subsubsection{Banach spaces and technical
lemmas}\label{Subsection:ExistenceQ:Banach}
We start by defining some norms. Given $\nu\in\RR$ and an analytic
function $h:D^{\out,u}_{\rr,\kk}\rightarrow \CC$, where
$D^{\out,u}_{\rr,\kk}$ is the domain defined in
\eqref{def:DominisOuter}, we consider
\[
\|h\|_{\nu,\rr,\kk}=\sup_{u\in D^{\out,u}_{\rr,\kk}}\left|
\left(u^2+a^2\right)^\nu h(u)\right|.
\]
Moreover for $2\pi$-periodic in $\tau$, analytic functions
$h:D^{\out,u}_{\rr,\kk}\times\TT_\sigma\rightarrow \CC$, we consider the
corresponding Fourier
norm
\[
\|h\|_{\nu,\rr,\kk,\sigma}=\sum_{k\in\ZZ}\left\|h^{[k]}\right\|_{\nu,\rr,\kk}e^{
|k|\sigma}.
\]
We consider, thus, the following function space

\begin{equation}\label{def:BanachExtension}
\EE_{\nu,\rr,\kk,\sigma}=\{
h:D^{\out,u}_{\rr,\kk}\times\TT_\sigma\rightarrow\CC;\,\,\text{real-analytic},
\|h\|_{\nu,\rr,\kk,\sigma}<\infty\},
\end{equation}
which can be checked that is a Banach space for any $\nu\in\RR$.

If there is no danger of confusion about the domain
$D_{\rr,\kk}^{\out,u}$, we will denote
\[
\begin{array}{ccc}
\|\cdot\|_{\nu,\sigma}=\|\cdot\|_{\nu,\rr,\kk,\sigma}&\text{ and
}&\EE_{\nu,\sigma}=\EE_{\nu,\rr,\kk,\sigma}.
\end{array}
\]

In the next lemma, we state some properties of these Banach spaces.
In the estimates  we will make explicit the dependence  of the
constants with respect to $\kappa$.
% since we will use them with $\de$ 	depending on $\eps$.
\begin{lemma}\label{lemma:PropietatsNormes} The following statements hold:
\begin{enumerate}
\item If $\nu_1\geq\nu_2$, then
$\EE_{\nu_1,\sigma}\subset\EE_{\nu_2,\sigma}$ and moreover if
$h\in\EE_{\nu_1,\sigma}$,
\[
\|h\|_{\nu_2,\sigma}\leq
K(\kk\eps)^{\nu_2-\nu_1}\|h\|_{\nu_1,\sigma}.
\]
\item If $\nu_1\leq\nu_2$, then
$\EE_{\nu_1,\sigma}\subset\EE_{\nu_2,\sigma}$ and moreover if
$h\in\EE_{\nu_1,\sigma}$,
\[
\|h\|_{\nu_2,\sigma}\leq K\|h\|_{\nu_1,\sigma}.
\]
\item If $h\in\EE_{\nu_1,\sigma}$ and $g\in\EE_{\nu_2,\sigma}$, then
$hg\in\EE_{\nu_1+\nu_2,\sigma}$ and
\[
\|hg\|_{\nu_1+\nu_2,\sigma}\leq
\|h\|_{\nu_1,\sigma}\|g\|_{\nu_2,\sigma}.
\]
\item Let $\rr'<\rr$ be such that $\rr-\rr'$ has a positive lower bound
independent of $\eps$, $\kk'$ and $\kk$ such that  $\kk <\kk'<0$  and
$h\in\EE_{\nu,\rr,\kk,\sigma}$. Then
$\pa_uh\in\EE_{\nu,\rr',\kk',\sigma}$ and satisfies
\[
\|\pa_uh\|_{\nu,\rr',\kk',\sigma}\leq\frac{K}{\eps\left|\kk'-\kk\right|}\|
h\|_{\nu,\rr,\kk,\sigma}.
\]
\end{enumerate}
\end{lemma}

Throughout this section we are going to solve equations of the form
$\LL_\eps h=g$, where $\LL_\eps$ is the differential operator
defined in~\eqref{def:Lde}. Note that $\LL_\eps$ acting on
$\EE_{\nu,\rr}$ is not invertible. Indeed for any smooth
function $f$, $f(u/\eps-\tau)\in\text{Ker}\LL_\eps$. We consider a
left-inverse of the operator $\LL_\eps$, which we call $\GG_\eps$,
defined acting on the Fourier coefficients. Let us consider $u_1,
\bar u_1\in \CC$ the vertices of the domain $D^{\out,u}_{\rr,\kk}$
(see Figure \ref{fig:OuterDomains}). Then, we define $\GG_\eps$ as
\begin{equation}\label{def:operadorG:HJtoParam}
\GG_\eps(h)(u,\tau)=\sum_{k\in\ZZ}\GG_\eps(h)^{[k]}(u)e^{ik\tau},
\end{equation}
where its Fourier coefficients are given by
\begin{align*}
\dps\GG_\eps (h)^{[k]}(u)&=\int_{\bar u_1}^u e^{ik\eps\ii
(t-u)}h^{[k]}(t)\,dt& \text{ for }k<0\\
\dps\GG_\eps (h)^{[0]}(u)&=\int_{- \rr}^u h^{[0]}(t)\,dt&
\\
\dps\GG_\eps (h)^{[k]}(u)&= \int_{u_1}^u e^{ik\eps\ii
(t-u)}h^{[k]}(t)\,dt& \text{ for }k> 0.
\end{align*}

\begin{remark}
Let us observe that the definition of the operator $\GG_\eps$
depends on the domain, since in its definition we use its vertices
$u_1$, $\bar u_1$ and also $\rr$.
\end{remark}

%We define the inverse of $\LL_\eps$ acting on the Fourier
%coefficients. Let us consider the points $u_1$ and $\bar u_1$
%defined in ??, then we define $\GG_\eps$ as
%\begin{equation}\label{def:operadorGExtensio}
%\GG_\eps (h)^{[k]}(u)=\left\{\dps\begin{array}{ll} \dps\int_{\bar
%u_1}^u e^{ik\eps\ii (t-u)}h^{[k]}(t)dt& \text{ for }k<0\\\dps
%\int_{u_1}^u e^{ik\eps\ii (t-u)}h^{[k]}(t)dt& \text{ for }k\geq 0
%\end{array}\right..
%\end{equation}\marginpar{ho trec?}
\begin{lemma}\label{lemma:PropietatsGExtensio}
The operator $\GG_\eps$ in \eqref{def:operadorG:HJtoParam} satisfies
the following properties.
\begin{enumerate}
\item If $h\in \EE_{\nu,\sigma}$ for some $\nu\geq 0$, then $\GG_\eps(h)\in
\EE_{\nu,\sigma}$ and
\[
\|\GG_\eps(h)\|_{\nu,\sigma}\leq K\|h\|_{\nu,\sigma}.
\]
Furthermore, if $\langle h\rangle=0$,
\[
\left\|\GG_\eps(h)\right\|_{\nu,\sigma}\leq
K\eps\left\|h\right\|_{\nu,\sigma}.
\]
\item If $h\in\EE_{\nu,\sigma}$ for some $\nu>1$, then $\GG_\eps(h)\in
\EE_{\nu-1,\sigma}$ and
\[
\left\|\GG_\eps(h)\right\|_{\nu-1,\sigma}\leq K\|h\|_{\nu,\sigma}.
\]
\item If $h\in\EE_{\nu,\sigma}$ for some $\nu\in (0,1)$, then $\GG_\eps(h)\in
\EE_{0,\sigma}$ and
\[
\left\|\GG_\eps(h)\right\|_{0,\sigma}\leq K\|h\|_{\nu,\sigma}.
\]

\item If $h\in\EE_{\nu,\sigma}$ for some $\nu\geq0$, then $\GG_\eps(\pa_uh)\in
\EE_{\nu,\sigma}$ and
\[
\left\|\GG_\eps(\pa_uh)\right\|_{\nu,\sigma}\leq
K\|h\|_{\nu,\sigma}.
\]
\item If $h\in\XX_{\nu,\sigma}$ for some $\nu\geq 0$,
$\LL_\eps\circ\GG_\eps (h)=h$ and
\[
\GG_\eps\circ\LL_\eps(h)(v,\tau)=h(v,\tau)-\sum_{k<0}e^{ik\eps\ii
(-u_1-u)}h^{[k]}(-u_1)-h^{[0]}(u_0)-\sum_{k>0}e^{ik\eps\ii
(u_1-u)}h^{[k]}(u_1).
\]
\item If $h\in\XX_{\nu,\sigma}$ for some $\nu\geq 0$,
$\LL_\eps\circ\GG_\eps (h)=h$ and
\[
\GG_\eps\circ\LL_\eps(h)(v,\tau)=h(v,\tau)-\sum_{k<0}e^{ik\eps\ii
(-u_1-u)}h^{[k]}(-u_1)-h^{[0]}(u_0)-\sum_{k>0}e^{ik\eps\ii
(u_1-u)}h^{[k]}(u_1).
\]
\end{enumerate}
\end{lemma}

\begin{proof}
It is a consequence of Lemma 5.5 in \cite{GuardiaOS10}.
\end{proof}

\subsubsection{Proof of
Theorem~\ref{th:Extensio:Trig}}\label{subsec:Extensio:trig}

We prove Theorem \ref{th:Extensio:Trig}, by looking for the analytic
continuation of the function $T_1=T-T_0$ obtained in Propositions
\ref{prop:HJ:hyp:general} and \ref{prop:HJ:parab:general} as a
solution of equation \eqref{eq:HJperT1}. First we rewrite the result in
terms of the Banach spaces defined in \eqref{def:BanachExtension}.

\begin{proposition}\label{prop:extensio:trig}
Let $\rr_1$ be the constant introduced in Theorem
\ref{th:ExistenceCloseInfty} and let $\rr_2>\rr_1$,
$\eps_0>0$ small enough and $\kk_1>0$ big enough. Then, for $\eps\in
(0,\eps_0)$, there exists a function
$T_1\in\EE_{\ell+1,\rr_2,\kk_1,\sigma}$ which satisfies equation
\eqref{eq:HJperT1} and is the analytic continuation of the analytic
function  $T_1$ obtained in Propositions \ref{prop:HJ:hyp:general}
and \ref{prop:HJ:parab:general}. Moreover, there exists a constant
$b_2>0$ such that
\[
\left\|\pa_u T_1\right\|_{\ell+1,\rr_2,\kk_1,\sigma}\leq
b_2|\mu|\eps^{\eta+1}.
\]
\end{proposition}
This proposition gives the existence of the invariant manifolds in
$D^{\out,*}_{\rr_2,\kk_1}\times\TT_\sigma$, $*= u, s$.

We devote the rest of the section to prove Proposition
\ref{prop:extensio:trig}.

First, we state a technical lemma about properties of the functions
$A$, $B_1$, $B_2$ and $C$ defined in \eqref{def:InftyA},
\eqref{def:InftyB1}, \eqref{def:InftyB2} and \eqref{def:InftyC}
respectively.

\begin{lemma}\label{lemma:Extensio:Trig:Cotes}
Let $\rr>0$ and $\kk>0$. Then, the functions $A$,
$B_1$, $B_2$ and $C$ defined in \eqref{def:InftyA},
\eqref{def:InftyB1}, \eqref{def:InftyB2} and \eqref{def:InftyC}
satisfy the following properties.
\begin{enumerate}
\item $A\in \EE_{\ell,\rr,\kk,\sigma}$ and $\pa_u
A\in\EE_{\ell+1,\rr,\kk,\sigma}$. Moreover $\pa_uA$
satisfies
\begin{equation}\label{eq:Infty:CotesA:outer}
\begin{split}
\left\|\pa_u A\right\|_{\ell+1,\rr,\kk,\sigma}&\leq
K|\mu|\eps^{\eta}\\
 \left\|\GG_\eps
(\pa_uA)\right\|_{\ell+1,\rr,\kk,\sigma}&\leq K|\mu|\eps^{\eta+1}.
\end{split}
\end{equation}
\item If $\ell-2r<0$, $B_1, \pa_u B_1, B_2 \in \EE_{0,\rr,\kk,\sigma}$ and
satisfy
$\langle B_1\rangle=0$ and
\begin{equation}\label{eq:Infty:CotesB:outer:lmenor}
\begin{split}
\dps\|B_1\|_{0,\rr,\kk,\sigma}&\leq K|\mu|\eps^{\eta}\\
\dps\|\pa_uB_1\|_{\max\{0,\ell-2r+1\},\rr,\kk,\sigma}&\leq
K|\mu|\eps^{\eta}\\
\dps\|B_2\|_{0,\rr,\kk,\sigma}&\leq K|\mu|^2\eps^{2\eta+1}.
\end{split}
\end{equation}
\item If $\ell-2r\geq 0$, $B_1, B_2 \in \EE_{\ell-2r,\rr,\kk,\sigma}$, $\pa_u
B_1 \in \EE_{\ell-2r+1,\rr,\kk,\sigma}$ and satisfy
$\langle B_1\rangle=0$ and
\begin{equation}\label{eq:Infty:CotesB:outer}
\begin{split}
\dps\|B_1\|_{\ell-2r,\rr,\kk,\sigma}&\leq K|\mu|\eps^{\eta}\\
\dps\|\pa_uB_1\|_{\ell-2r+1,\rr,\kk,\sigma}&\leq
K|\mu|\eps^{\eta}\\
\dps\|B_2\|_{\ell-2r,\rr,\kk,\sigma}&\leq K|\mu|^2\eps^{2\eta+1}.
\end{split}
\end{equation}
\item Let us consider $h_1,h_2\in B(\nu)\subset\EE_{\ell+1,\rr,\kk,\sigma}$
with $\nu\ll 1$. Then,
\begin{itemize}
\item If $\ell-2r<0$,
\[
\left\| C(h_2,u,\tau)-
C(h_1,u,\tau)\right\|_{\ell+1,\rr,\kk,\sigma}\leq K
\frac{\nu}{\eps^{\max\{0,\ell-2r+1\}}}\|h_2-h_1\|_{\ell+1,\rr,\kk,\sigma}.
\]
\item If $\ell-2r\geq0$,
\[
\left\| C(h_2,u,\tau)-
C(h_1,u,\tau)\right\|_{2\ell-2r+2,\rr,\kk,\sigma}\leq K
\nu\|h_2-h_1\|_{\ell+1,\rr,\kk,\sigma}.
\]
\end{itemize}
\end{enumerate}
\end{lemma}
\begin{proof}
For the first bounds, we split $A=A_1+A_2+A_3$, where $A_i$,
$i=1,2,3$, are the functions defined in \eqref{def:HJ:A1},
\eqref{def:HJ:A2} and \eqref{def:HJ:A3} respectively.

Using \eqref{eq:Formula:A1} and \eqref{eq:PotencialInfinit}, one can
see that
$A_1\in\EE_{r+1,\rr,\de,\sigma}\subset\EE_{\ell+1,\rr,\de,\sigma}$
and
\begin{equation}\label{eq:Cota:A1}
\|A_1\|_{\ell+1,\rr,\de,\sigma}\leq\|A_1\|_{r+1,\rr,\de,\sigma}\leq
K|\mu|\eps^{\eta+1}.
\end{equation}
 Applying
Lemma \ref{lemma:PropietatsGInfty}, we obtain
$\|\GG_\eps(\pa_uA_1)\|_{\ell+1,\rr,\de,\sigma}\leq
K|\mu|\eps^{\eta+1}$.

Moreover, by the definition of
$\ell$, $A_2,A_3\in \EE_{\ell,\rr,\de,\sigma}$. Therefore
$\pa_uA_2,\pa_uA_3\in \EE_{\ell+1,\rr,\de,\sigma}$ and satisfy
$\|\pa_u A_2\|_{\ell+1,\rr,\de,\sigma}\leq K|\mu|\eps^{\eta}$ and
$\|\pa_u A_3\|_{\ell+1,\rr,\de,\sigma}\leq K|\mu|^2\eps^{2\eta+1}$.

To bound $\GG_\eps(A_2)$, let us point out that $\langle
A_2\rangle=0$ and then, by Lemma \ref{lemma:PropietatsGInfty},
\[
\left\|\GG_\eps\left(\pa_uA_2\right)\right\|_{\ell+1,\rr,\de,\sigma}\leq
K\eps\|\pa_uA_2\|_{\ell+1,\rr,\de,\sigma}\leq K|\mu|\eps^{\eta+1}.
\]
Applying again  Lemma \ref{lemma:PropietatsGInfty} we have
$\|\GG_\eps(\pa_uA_3)\|_{\ell+1,\rr,\de,\sigma}\leq
K|\mu|^2\eps^{2\eta+1}$. Therefore
\[
\left\|\GG_\eps (\pa_uA)\right\|_{\ell+1,\rr,\de,\sigma}\leq
K|\mu|^2\eps^{2\eta+1}.
\]
The other bounds are straightforward.
\end{proof}

To prove Proposition \ref{prop:extensio:trig}, we proceed as in the
proofs of Propositions \ref{prop:HJ:hyp:general} and
\ref{prop:HJ:parab:general}. That is, we first perform a change of
variables which reduces the size of the linear terms of $\FF$ in
\eqref{eq:infty:lmenor:operadorF}. Notice that in order to prove
Proposition \ref{prop:extensio:trig} we could look for this
change as the analytic continuation of the changes obtained in
Lemmas \ref{lemma:infty:ligual:canvi} and \ref{lemma:Parab:canvi}.
Nevertheless, since we want the proof of Theorem
\ref{th:Extensio:Trig} be also valid for Theorem
\ref{th:ExtensioFinal}, we look for a change $g$ which is not
necessarily continuation of the one obtained in Lemmas
\ref{lemma:infty:ligual:canvi} and \ref{lemma:Parab:canvi}.

\begin{lemma}\label{lemma:Extensio:Trig:canvi}
Let $\kk_1>\kk_0'>\kk_0>0$ and
$\rr_1''>\rr_1'>\rr_2>\rr_0'$, where $\rr_0'$ is the constant
introduced in Lemmas \ref{lemma:infty:ligual:canvi} and
\ref{lemma:Parab:canvi}. Then, for $\eps>0$ small enough and
$\kk_0'$ big enough, there exists a function
$g$ which is solution of
\eqref{eq:Infty:Canvi} and satisfies:
\begin{itemize}
\item If $\ell-2r<0$, $g \in\EE_{0,\rr_1',\kk_0',\sigma}$
and
\[
\begin{split}
\dps\|g\|_{0,\rr_1',\kk_0',\sigma}&\leq K|\mu|\eps^{\eta+1}\\
\dps\|\pa_v g\|_{0,\rr_1',\kk_0',\sigma}&\leq K|\mu|\eps^{\eta+1}.
\end{split}
\]
\item If $\ell-2r\geq0$, $g \in\EE_{\ell -2r,\rr_1',\kk_0',\sigma}$
and
\[
\begin{split}
\dps\|g\|_{\ell-2r,\rr_1',\kk_0',\sigma}&\leq K|\mu|\eps^{\eta+1}\\
\dps\|\pa_v g\|_{\ell-2r+1,\rr_1',\kk_0',\sigma}&\leq
K|\mu|\eps^{\eta+1}.
\end{split}
\]
\end{itemize}
Moreover, $v+g(v,\tau)\in D^{\out,u}_{\rr_1'',\kk_0}$ for
$(v,\tau)\in D^{\out,u}_{\rr_1',\kk_0'}\times\TT_\sigma$.

Furthermore, the change of variables $(u,\tau)=(v+g(v,\tau),\tau)$
is invertible and its inverse is of the form
$(v,\tau)=(u+h(u,\tau),\tau)$. The function $h$ is defined in the
domain $D_{\rr_2,\kk_1}^{\out,u}\times\TT_\sigma$ and it satisfies
\begin{itemize}
\item If $\ell-2r<0$
\[
\dps\|h\|_{0,\rr_2,\kk_1,\sigma}\leq K|\mu|\eps^{\eta+1}.
\]
\item If $\ell-2r\geq 0$
\[
\dps\|h\|_{\ell-2r,\rr_2,\kk_1,\sigma}\leq K|\mu|\eps^{\eta+1}.
\]
\end{itemize}
Moreover, $u+h(u,\tau)\in D^{\out,u}_{\rr_1',\kk_0'}$ for
$(u,\tau)\in D^{\out,u}_{\rr_2,\kk_1}\times\TT_\sigma$.
\end{lemma}
In the case $\ell-2r<0$ we need more precise bounds of both functions $g$ and
$h$
restricted to the inner domain $D_{\kk_1,c}^{\inn,+,u}$ defined in
\eqref{def:DominisInnerEnu}.
These bounds are given in the next corollary.
\begin{corollary}\label{coro:Extensio:CotaCanvigInner}
Let us assume $\ell-2r<0$ and let $c_1>0$. Then, the functions $g$ and $h$
obtained in Lemma \ref{lemma:Extensio:Trig:canvi},
restricted to the inner domain $D_{\kk_1,c_1}^{\inn,+,u}$, satisfy the following
bounds
\[
\sup |g (u,\tau)|_{(u,\tau)  \in D_{\kk_1,c_1}^{\inn,+,u}\times \TT_\sigma }\leq
K|\mu|\eps^{\eta+1+\nu_1^*}\quad\text{ and }\quad\sup |h(u,\tau) |_{(u,\tau)
\in D_{\kk_1,c_1}^{\inn,+,u}\times \TT_\sigma } \leq K|\mu|\eps^{\eta+1+\nu_1^*}
\]
with $\nu_1^*=\min\{(2r-\ell)\gamma,1\}$.
\end{corollary}
\begin{proof}[Proof of Lemma \ref{lemma:Extensio:Trig:canvi} and Corollary
\ref{coro:Extensio:CotaCanvigInner}]
To define $g$, let us recall first that, by Lemma
\ref{lemma:Extensio:Trig:Cotes}, $\langle B_1\rangle=0$. Then we can
define a function $\ol B_1$ such that $\pa_\tau\ol B_1=B_1$ and
$\langle \ol B_1\rangle=0$.
Then, one can see that a solution of equation
\eqref{eq:Infty:Canvi}, can be given by
\begin{equation}\label{def:FuncioG:extensio}
g(v,\tau)=-\eps\ol B_1(v,\tau)+\eps\GG_\eps(\pa_v\ol B_1)(v,\tau),
\end{equation}
where $\GG_\eps$ is the integral operator defined in
\eqref{def:operadorG:HJtoParam}.

By Lemma \ref{lemma:Extensio:Trig:Cotes} one has:
if $\ell-2r\geq 0$,
\begin{equation}\label{eq:cotaB1barraCanvi}
\begin{split}
\left\|\ol B_1\right\|_{\ell-2r,\rr_2,\kk_0',\sigma}&\leq
K|\mu|\eps^{\eta}\\
\left\|\pa_v\ol B_1\right\|_{\ell-2r+1,\rr_2,\kk_0',\sigma}&\leq
K|\mu|\eps^{\eta},
\end{split}
\end{equation}
if  $-1 \le \ell-2r < 0$,
\begin{equation}\label{eq:cotaB1barraCanvi1}
\begin{split}
\left\|\ol B_1\right\|_{0,\rr_2,\kk_0',\sigma}&\leq
K|\mu|\eps^{\eta}\\
\left\|\pa_v\ol B_1\right\|_{\ell-2r+1,\rr_2,\kk_0',\sigma}&\leq
K|\mu|\eps^{\eta}.
\end{split}
\end{equation}
and finally, if $\ell-2r < -1$
\begin{equation}\label{eq:cotaB1barraCanvi2}
\begin{split}
\left\|\ol B_1\right\|_{0,\rr_2,\kk_0',\sigma}&\leq
K|\mu|\eps^{\eta}\\
\left\|\pa_v\ol B_1\right\|_{0,\rr_2,\kk_0',\sigma}&\leq
K|\mu|\eps^{\eta}.
\end{split}
\end{equation}

From these inequalities, using Lemma
\ref{lemma:PropietatsGExtensio} we conclude that:
\[
\left\|  g(v,\tau)+\eps\ol
B_1(v,\tau)\right\|_{\max\{\ell-2r+1,0\},\rr_2,\kk_0',\sigma} \le
K\mu \eps ^{\eta +2},
\]
which, together with  \eqref{eq:cotaB1barraCanvi} when $\ell-2r \ge 0$ and with
\eqref{eq:cotaB1barraCanvi1}  and \eqref{eq:cotaB1barraCanvi2} when
$\ell-2r < 0$, gives the desired bounds for $g$.
For the proof of the bound of $\pa_v g$ it is enough to apply again
Lemmas \ref{lemma:PropietatsGExtensio} and
\ref{lemma:Extensio:Trig:Cotes} and \eqref{eq:cotaB1barraCanvi}.

The rest of the statements are straightforward.

To proof Corollary \ref{coro:Extensio:CotaCanvigInner} we just need to use the
definition of $B_1$ in \eqref{def:InftyB1},
and observe that it has a singularity or order $\ell-2r$ if $\ell-2r \ge 0$ and
a zero of order $2r-\ell$ if $\ell-2r \le 0$.
\end{proof}

Once we have the change $g$, we proceed as in Section
\ref{sub:Infty:Hyp:ProofGeneral}, defining
\begin{equation}\label{def:TtoTbarra}
\wh T_1(v,\tau)=T_1(v+g(v,\tau),\tau)
\end{equation}
which is solution of \eqref{eq:HJperT1:ligual}, that is:
$$
\LL_\eps \wh T_1=\wh\FF \left(\pa_v \wh T_1\right).
$$
We look for it using
a fixed point argument on $\pa_v\wh T_1$. Nevertheless, since we
want $\pa_u T_1$ to be the analytic continuation of the function
$\pa_u T_1$ obtained in Propositions \ref{prop:HJ:hyp:general} and
\ref{prop:HJ:parab:general}, we have to impose \emph{initial
conditions}. Nevertheless, since we invert $\LL_\eps$ by using the
operator $\GG_\eps$ defined in \eqref{def:operadorG:HJtoParam}
adapted to the domain $D_{\rr_1',\de}^{\out, u}\times \TT_\sigma$,
we consider a different initial condition depending on the Fourier
coefficient. Recall that we are looking for $\pa_v\wh T_1$ defined
in $D_{\rr_1',\de}^{\out, u}\times \TT_\sigma$. Thus, we define
\begin{equation}\label{def:Extensio:A0}
\begin{split}
A_0(v,\tau)=&\sum_{k<0}\pa_v\wh T_1^{[k]}\left(\ol v_1\right)
e^{-ik\eps^{-1}(v-\ol v_1)}e^{ik\tau}\\
&+\sum_{k>0}\pa_v\wh T_1^{[k]}\left( v_1\right)
e^{-ik\eps^{-1}(v- v_1)}e^{ik\tau}\\
&+ \pa_v\wh T_1^{[0]}(-\rr_1'),
\end{split}
\end{equation}
where $v_1,\ol v_1$ are the vertices of the outer domain
$D_{\rr_1',\de}^{\out, u}$ (see Figure \ref{fig:OuterDomains}) and
$\pa_v\wh T_1$ can be obtained differentiating
\eqref{def:TtoTbarra}, since $T_1$ is already known in a
neighborhood of these points. Note that $v_1, \ol v_1, \rr_1'\in
D_{\infty,\rr_1}^u$. Applying the bounds obtained in Propositions
\ref{prop:HJ:hyp:general} and \ref{prop:HJ:parab:general} and Lemma
\ref{lemma:Extensio:Trig:canvi}, one can see that
\begin{equation}\label{eq:Cota:CondIni:ExtTrig}
\|A_0\|_{0,\rr_1',\kk_0',\sigma}\leq K|\mu|\eps^{\eta+1}.
\end{equation}
Let us define $S(v,\tau)$ as the solution of
\[
S(v,\tau)=A_0(v,\tau)+\GG_\eps\left(\pa_v \wh \FF(S)\right)(v,\tau),
\]
where $\GG_\eps$ and $\wh\FF$ are the operators defined in
\eqref{def:operadorG:HJtoParam} and \eqref{eq:HJperT1:ligual:RHS} respectively.
Let us point out that the definition of $\wh \FF$ involves the
functions $\wh A$, $\wh B$ and $\wh C$ defined in
\eqref{def:InftyHyp:Ahat}, \eqref{def:InftyHyp:Bhat}
 and \eqref{def:InftyHyp:Chat}.
 Even if we keep the same notation, now the definitions involve the function $g$
obtained in Lemma \ref{lemma:Extensio:Trig:canvi} instead of the ones given in
Lemmas \ref{lemma:infty:ligual:canvi} and Lemma \ref{lemma:Parab:canvi}.

We will see that $S$ is the analytic continuation of the function
$\pa_uT_1(v+g(v,\tau),\tau)(1+\pa_v g(v,\tau))\ii$, where $T_1$ is  obtained
from
Propositions \ref{prop:HJ:hyp:general} and
\ref{prop:HJ:parab:general}.

Thus, we look for a fixed point
$S\in\EE_{\ell+1,\rr_1',\kk_0',\sigma}$ of the operator
\begin{equation}\label{def:Extensio:Trig:Operador:Sencer}
\JJ(S)(v,\tau)=A_0(v,\tau)+\GG_\eps\left(\pa_v \wh
\FF(S)\right)(v,\tau).
\end{equation}
\begin{lemma}\label{lemma:Extensio:Trig:General}
Let $\eps_0>0$ be small enough and $\kk_0'>\kk_0$ big
enough. Then, for $\eps\in (0,\eps_0)$, there exists a function
$S\in\EE_{\ell+1,\rr_1',\kk_0',\sigma}$ defined in
$D^{\out,u}_{\rr_1',\kk_0'}\times\TT_\sigma$ such that it is a fixed
point of the operator \eqref{def:Extensio:Trig:Operador:Sencer} and
is the analytic continuation of the function
$\pa_uT_1(v+g(v,\tau),\tau)(1+\pa_v g(v,\tau))\ii$, where $T_1$ is obtained from
Propositions \ref{prop:HJ:hyp:general} and
\ref{prop:HJ:parab:general} and $g$ is given in Lemma
\ref{lemma:Extensio:Trig:canvi}.

Moreover, there exists a constant
$b_2>0$ such that
\[
\|S\|_{\ell+1,\rr_1',\kk_0',\sigma}\leq b_2|\mu|\eps^{\eta+1}.
\]
\end{lemma}
\begin{proof}
We recall that, during the proof,  $g$ is the function given in Lemma
\ref{lemma:Extensio:Trig:canvi}.

It is straightforward to see that $\JJ$ is well defined from
$\EE_{\ell+1,\rr_1',\de,\sigma}$ to itself. We are going to prove
that there exists a constant $b_2>0$ such that $\JJ$ is contractive
in $\ol
B(b_2|\mu|\eps^{\eta+1})\subset\EE_{\ell+1,\rr_1',\kk_0',\sigma}$.

First we deal with $\JJ(0)$. From the definition of $\JJ$ in
\eqref{def:Extensio:Trig:Operador:Sencer} and the definition of
$\wh\FF$ in \eqref{eq:HJperT1:ligual:RHS}, we have
\[
\JJ(0)(v,\tau)=A_0(v,\tau)+\GG_\eps\left(\pa_v \wh A(v,\tau)\right),
\]
where $\wh A$
%and $A$ are
is the function in \eqref{def:InftyHyp:Ahat}.
%and \eqref{def:InftyA}, with $g$ given in Lemma\ref{lemma:Extensio:Trig:canvi}.

Taking into account the definition of $\wh A$, we split $\JJ(0)$  as
\[
\JJ(0)(v,\tau)=A_0(v,\tau)+\GG_\eps\left(\pa_v
A(v,\tau)\right)+\GG_\eps\left(\pa_v \left[A(v+g(v,\tau),\tau)-
A(v,\tau)\right]\right),
\]
where $A$ is given in  \eqref{def:InftyA}.
The first term has already been bounded in
\eqref{eq:Cota:CondIni:ExtTrig} and the second one in Lemma
\ref{lemma:Extensio:Trig:Cotes}. For the third one, using $\rr''_1$
introduced in Lemma \ref{lemma:Extensio:Trig:canvi}, and applying
Lemmas \ref{lemma:PropietatsGExtensio},
\ref{lemma:Extensio:Trig:Cotes} and \ref{lemma:Extensio:Trig:canvi}
and the mean value theorem,
\[
\begin{split}
\left\|\GG_\eps\left(\pa_v \left[A(v+g(v,\tau),\tau)-
A(v,\tau)\right]\right)\right\|_{\ell+1,\rr_1',\kk_0',\sigma}&\leq
\left\|A(v+g(v,\tau),\tau)- A(v,\tau)\right\|_{\ell+1,\rr_1',\kk_0',\sigma}\\
&\leq
\|\pa_uA\|_{\ell+1,\rr_1'',\kk_0\eps,\sigma}\|g\|_{0,\rr_1',\kk_0',\sigma}\\
&\leq K|\mu|^2\eps^{2\eta+1}
\end{split}
\]
Thus, there exists a constant $b_2>0$ such that
\[
\left\|\JJ(0)\right\|_{\ell+1,\rr_1',\kk_0',\sigma}\leq
\frac{b_2}{2}|\mu|\eps^{\eta+1}.
\]
Now let $h_1,h_2\in\ol
B(b_2|\mu|\eps^{\eta+1})\subset\EE_{\ell+1,\rr_1',\kk_0',\sigma}$.
Using the definitions of $\JJ$ and $\wh \FF$ in
\eqref{def:Extensio:Trig:Operador:Sencer} and
\eqref{eq:HJperT1:ligual:RHS} respectively, and applying Lemma
\ref{lemma:PropietatsGExtensio},
\[
\begin{split}
\left\|\JJ(h_2)-\JJ(h_1)\right\|_{\ell+1,\rr_1',\kk_0',\sigma}&\leq
K\left\|\wh\FF(h_2)-\wh\FF(h_1)\right\|_{\ell+1,\rr_1',\kk_0',\sigma}\\
&\leq K\left\|\wh B\cdot(h_2-h_1)+\wh C(h_2,v,\tau)-\wh
C(h_1,v,\tau)\right\|_{\ell+1,\rr_1',\kk_0',\sigma}.
\end{split}
\]
To bound the Lipschitz constant of $\JJ$, one has to take into account the
definitions of $\wh B$ and $\wh C$ in \eqref{def:InftyHyp:Bhat} and
\eqref{def:InftyHyp:Chat} respectively,
%with $g$ given in Lemma \ref{lemma:Extensio:Trig:canvi},
and to apply Lemmas
\ref{lemma:Extensio:Trig:Cotes} and \ref{lemma:Extensio:Trig:canvi}.
We bound it in different ways depending whether $\ell-2r<0$ or
$\ell-2r\geq 0$. In the first case we obtain
\[
\left\|\JJ(h_2)-\JJ(h_1)\right\|_{\ell+1,\rr_1',\kk_0',\sigma}\leq
K|\mu|\eps^{\eta+1-\max\{0,\ell-2r+1\}}\left\|h_2-h_1\right\|_{\ell+1,\rr_1',
\kk_0',\sigma},
\]
and in the second,
\[
\left\|\JJ(h_2)-\JJ(h_1)\right\|_{\ell+1,\rr_1',\kk_0',\sigma}\leq
K|\mu|\frac{\eps^{\eta-(\ell-2r)}}{\left(\kk_0'\right)^{\ell-2r+1}}
\left\|h_2-h_1\right\|_{\ell+1,\rr_1',\kk_0',\sigma}.
\]
Therefore, since $\eta\geq\max\{0,\ell-2r\}$,  taking $\eps<\eps _0$ and
$\kk_0'$ big enough, $\Lip\JJ<1/2$ and then $\JJ$ is
contractive in $\ol
B(b_2|\mu|\eps^{\eta+1})\subset\EE_{\ell+1,\rr_1',\kk_0',\sigma}$
and it has a unique fixed point $S(v,\tau)$.

Now, we have to prove that $S(v,\tau)$ is the analytic continuation
 of the function $\wt S(v,\tau)=\pa_u
T_1(v+g(v,\tau),\tau)(1+\pa_vg(v,\tau))\ii$  obtained from
Propositions \ref{prop:HJ:hyp:general} and
\ref{prop:HJ:parab:general}.
First let us observe that the operator \eqref{def:Extensio:Trig:Operador:Sencer}
is well defined for functions in
$\left(D_{\infty,\rr_1}^u\cap
D_{\rr_1',\kk_0'}^{\out,u}\right)\times\TT_\sigma$. Moreover,  both
functions $S(v,\tau)$ and $\wt S(v,\tau)$
are defined in $\left(D_{\infty,\rr_1}^u\cap
D_{\rr_1',\kk_0'}^{\out,u}\right)\times\TT_\sigma$ and for
$(v,\tau)$ in this domain both are fixed points of the operator
\eqref{def:Extensio:Trig:Operador:Sencer} and
\[
\left\| \wt S \right\|_{\ell+1,\sigma} \le b_1\mu \eps^{\eta+1}.
\]

Then, using the norms
defined in Section \ref{Subsection:ExistenceQ:Banach} but for
functions defined in $\left(D_{\infty,\rr_1}^u\cap
D_{\rr_1',\kk_0'}^{\out,u}\right)\times\TT_\sigma$, one can see that
\[
\begin{split}
\left\|S(v,\tau)-\wt
S(v,\tau)\right\|_{\ell+1,\sigma}&\leq\left\|\JJ\left(S(v,
\tau)\right)-\JJ\left(\wt S(v,\tau)\right)\right\|_{\ell+1,\sigma}\\
&\leq \frac{1}{2}\left\|S(v,\tau)-\wt
S(v,\tau)\right\|_{\ell+1,\sigma}.
\end{split}
\]
Then $S(v,\tau)=\wt S(v,\tau)$ for
$(v,\tau)\in\left(D_{\infty,\rr_1}^u\cap
D_{\rr_1',\kk_0'}^{\out,u}\right)\times\TT_\sigma$ and  $S(v,\tau)$
is the analytic continuation of the function $\pa_u
T_1(v+g(v,\tau),\tau)(1+\pa_vg(v,\tau))\ii$ to
$D_{\rr_1',\kk_0'}^{\out,u}\times\TT_\sigma$. Finally, one can easily recover
$\wh T_1$ from $S$.
\end{proof}

\begin{proof}[Proof of Proposition \ref{prop:extensio:trig}]
To prove Proposition \ref{prop:extensio:trig} from Lemma
\ref{lemma:Extensio:Trig:General}, it is enough to consider the
change of variables $v=u+h(u,\tau)$  obtained in Lemma
\ref{lemma:Extensio:Trig:canvi} and to take $T_1(u,\tau)=\wh
T_1(u+h(u,\tau),\tau)$ which by construction is the analytic
continuation of the function $T_1$ obtained in Propositions
\ref{prop:HJ:hyp:general} and \ref{prop:HJ:parab:general}.
\end{proof}

\subsection{Invariant manifolds in the outer domains in the general case: proof
of Theorems \ref{th:HJtoParam}, \ref{th:Extensio}, \ref{th:ParamtoHJ} and
\ref{th:ExtensioFinal}}\label{sec:Extensio:polinomial}

We devote
this section to prove the existence of the invariant manifolds in
the outer domains, in the general case, that is assuming that $p_0(u)$
can vanish. We split the proofs into
Theorems \ref{th:HJtoParam}, \ref{th:Extensio}, \ref{th:ParamtoHJ}
and \ref{th:ExtensioFinal}.

\subsubsection{The variational equation along the separatrix}
In order to prove the existence of the perturbed stable and unstable
invariant manifolds in certain domains, we will need to consider a
real-analytic fundamental matrix solution  of the variational
equations along the unperturbed separatrix
\begin{equation}\label{eq:variacional}
\dot\xi=A(u)\xi,
\end{equation}
where
\begin{equation}\label{eq:MatriuLinealSeparatriu}
A(u)=\left(\begin{array}{cc}0&1\\\dps
-V''\left(q_0(u)\right)&0\end{array}\right)
\end{equation}
and $(q_0(u), p_0(u))$ is the parameterization of the unperturbed
separatrix given in Hypothesis \textbf{HP2}.

It is a well known fact that the derivative of the parameterization
of the separatrix, that is $(p_0(u),\dot p_0(u))$ (recall that $\dot
q_0(u)=p_0(u)$), is a solution of \eqref{eq:variacional}. A second
independent solution can be given by $(\zeta(u),\dot \zeta (u))$,
where
\begin{equation}\label{def:SolVariacional2}
\zeta(u)=p_0(u)\int_{u_0}^u\frac{1}{p^2_0(v)}\,dv,
\end{equation}
where $u_0\in\RR$ is such that $p_0(u_0)\neq 0$. We consider then
the following fundamental matrix
\begin{equation}\label{def:MatriuFonamentalReal}
\Phi(u)=\left(\begin{array}{cc} p_0(u) &\zeta (u)\\\dot p_0(u) &\dot
\zeta (u)\end{array}\right).
\end{equation}

\begin{remark}
Notice that the function $\zeta$ defined in
\eqref{def:SolVariacional2} is well defined and analytic even if
$p_0(u)$ can vanish for some $u\in\CC$ and even that \emph{a priori} it could
seem that the integral depends on the path of integration.

Indeed, since $\ddot p_0(u)=-V''(q_0(u))p_0(u)$, one can see that
the Taylor expansion around any zero $u^\ast\in\CC$ of $p_0(u)$ is
of the form
\[
p_0(u)=\dot
p_0\left(u^\ast\right)\left(u-u^\ast\right)+\OO\left(u-u^\ast\right)^3
\]
(observe that $\dot p_0(u^\ast)\neq 0)$) and then, the residue of
the integrand appearing in the definition of $\zeta$ in
\eqref{def:SolVariacional2} is zero. Finally,
even if the integral might be divergent if one takes $u^\ast$ as the
upper limit of integration, $\lim_{u\to
u^\ast}\zeta(u)=-1/\dot{p}_0(u^\ast)$.
\end{remark}

\subsubsection{Proof of Theorem \ref{th:HJtoParam}}\label{sec:HJtoParam}
In this section we prove the existence of a change of variables
which allow us to obtain a parameterization of the invariant
manifolds which satisfies equation \eqref{eq:PDEParametritzacions}
from the parameterization obtained in Theorem
\ref{th:ExistenceCloseInfty}.

It is straightforward to see that the functions defined in \eqref{eq:HJtoParam}
satisfy equation \eqref{eq:PDEParametritzacions} provided $\UU^u$ satisfies
\begin{equation}\label{eq:PDEHJtoParam}
\LL_\eps h=M\left(v+h(v,\tau),\tau\right),
\end{equation}
where
\begin{equation}\label{eq:PDEHJtoParamRHS}
M(u,\tau)=\frac{1}{p_0^2(u)}\pa_uT_1(u,\tau)+\frac{\mu\eps^{\eta}}{p_0(u)}\pa_p
\wh H_1\left(q_0(u),p_0(u)+\frac{1}{p_0(u)}\pa_u
T_1(u,\tau),\tau\right),
\end{equation}
$\wh H_1$ is the Hamiltonian defined in
\eqref{def:ham:ShiftedOP:perturb} and $T_1$ is the function obtained
in Proposition \ref{prop:HJ:hyp:general}.

Decomposing the right hand side of equation \eqref{eq:PDEHJtoParam}
into constant, linear and higher order terms in $h$, it can be
rewritten as
\begin{equation}\label{eq:PDEHJtoParam2}
\LL_\eps h=\Mm(h),
\end{equation}
where
\begin{equation}\label{eq:PDEtoParamRHS2}
\Mm(h)(v,\tau)=M(v,\tau)+\left(N_1(v,\tau)+N_2(v,\tau)\right)
h(v,\tau)+R(h(v,\tau),v,\tau)
\end{equation}
and
\begin{align}
N_1(v,\tau)&=\mu\eps^{\eta}\pa_v\left[\frac{1}{p_0(v)}\pa_p \wh
H_1^1\left(q_0(v),p_0(v),\tau\right)\right]\label{def:PDEtoPam:N1} \\
N_2(v,\tau)&=\pa_vM(v,\tau)-N_1(v,\tau)\label{def:PDEtoPam:N2} \\
R(h,v,\tau)&=M(v+h,\tau)-\pa_v
M(v,\tau)h-M(v,\tau)\label{def:PDEtoPam:R},
\end{align}
where $\wh H_1^1$ and $M$ are defined in
\eqref{def:HamPertorbat:H1} and \eqref{eq:PDEHJtoParamRHS}
respectively.

We now define appropriate Banach spaces. For analytic functions $h:
\tro_{\rr_3,\rr_4}^u\times\TT_\sigma\rightarrow\CC$, where
$\tro_{\rr_3,\rr_4}^u$ is the domain defined in
\eqref{def:TransDomainInfty}, we define the Fourier norm
\[
\|h\|_\sigma=\sum_{k\in\ZZ}\left\|h^{[k]}\right\|_{\infty}e^{|k|\sigma},
\]
where $\|\cdot\|_\infty$ is the classical supremum norm in
$\tro_{\rr_3,\rr_4}^u$. We consider
 the following function space
\begin{equation}\label{def:HJtoParam:Banach}
\AAA_\sigma=\left\{h:
\tro_{\rr_3,\rr_4}^u\times\TT_\sigma\rightarrow \CC;\,\,
\text{real-analytic}, \|h\|_\sigma<\infty\right\}
\end{equation}
which is straightforward to see that is a Banach algebra.

Throughout this section we will need to solve equations of the form
$\LL_\eps h=g$, where $\LL_\eps$ is the differential operator
defined in \eqref{def:Lde}. We take the operator $\GG_\eps$ defined
in \eqref{def:operadorG:HJtoParam} as right inverse of $\LL_\eps$.
In Section \ref{Subsection:ExistenceQ:Banach} it was applied to
functions belonging to $\EE_{\nu,\rr,\de,\sigma}$ (see
\eqref{def:BanachExtension}) but it is clear that it can also be
applied to functions in $\AAA_\sigma$ if we take as the constant integration
limits of
the Fourier coefficients of $\GG_\eps$  as $v_1$, $\ol v_1$,
the vertices of the domain $\tro_{\rr_3,\rr_4}^u$, and $-\rr_4$
(see Figure \ref{fig:TransInfty}).

\begin{lemma}\label{lemma:PropietatsG:HJtoParam}
The operator $\GG_\eps$ in \eqref{def:operadorG:HJtoParam} satisfies
the following properties.
\begin{itemize}
\item $\GG_\eps$ is linear from $\AAA_\sigma$ to itself and satisfies
$\LL_\eps\circ\GG_\eps=\mathrm{Id}$.
\item If $h\in\AAA_\sigma$, then
\[
\|\GG_\eps(h)\|_{\sigma}\leq K\|h\|_\sigma.
\]
Furthermore, if $\langle h\rangle=0$, then
\[
\|\GG_\eps(h)\|_{\sigma}\leq K\eps\|h\|_\sigma.
\]
\end{itemize}
\end{lemma}

Finally, we  state  a technical lemma which gives some properties of
the functions $M$,  $N_1$, $N_2$ and $R$ defined in
\eqref{eq:PDEHJtoParamRHS}, \eqref{def:PDEtoPam:N1},
\eqref{def:PDEtoPam:N2} and \eqref{def:PDEtoPam:R} respectively.
\begin{lemma}\label{lemma:HJtoParam:Cotes}
The functions $M$,  $N_1$, $N_2$ and $R$ defined in
\eqref{eq:PDEHJtoParamRHS}, \eqref{def:PDEtoPam:N1},
\eqref{def:PDEtoPam:N2} and \eqref{def:PDEtoPam:R} satisfy the
following properties:
\begin{enumerate}
\item $M\in\AAA_\sigma$ and  satisfies
\begin{equation}\label{eq:HJtoParam:CotesM}
\begin{array}{cc}
\dps\|M\|_{\sigma}\leq K|\mu|\eps^{\eta}, &\dps\|\GG_\eps
(M)\|_{\sigma}\leq K|\mu|\eps^{\eta+1}.
\end{array}
\end{equation}
\item $N_1,N_2 \in \AAA_{\sigma}$. Moreover they satisfy $\langle N_1\rangle=0$
and
\begin{equation}\label{eq:HJtoParam:CotesN}
\begin{array}{cc}
\dps \|N_1\|_{\sigma}\leq K|\mu|\eps^{\eta},&\dps
\|N_2\|_{\sigma}\leq K|\mu|\eps^{\eta+1}.
\end{array}
\end{equation}
\item Let us consider $h_1,h_2\in B(\nu)\subset\AAA_\sigma$ with
$\nu\ll 1$. Then,
\[
\left\|R(h_2,v,\tau)-R(h_1,v,\tau)\right\|_\sigma \leq K\nu
\|h_2-h_1\|_\sigma.
\]
\end{enumerate}
\end{lemma}
\begin{proof}
The first bound is straightforward taking into account the bounds
for  $c_{lk}$ and  $T_1$ obtained in  Corollary
\ref{coro:ShiftPeriodica} and Propositions \ref{prop:HJ:hyp:general}
and \ref{prop:HJ:parab:general}. For the second one, one has to
split $M$ as $M=M_1+M_2$, where
\[
M_1(u,\tau)=\mu\eps^{\eta}\frac{1}{p_0(u)} \pa_p\wh H_1^1
(q_0(u),p_0(u),\tau),
\]
where $\wh H_1^1$ is the Hamiltonian in \eqref{def:HamPertorbat:H1},
and $M_2=M-M_1$. Since $\langle M_1\rangle=0$ and satisfies
$\|M_1\|_\sigma\leq K|\mu|\eps^{\eta}$, by Lemma
\ref{lemma:PropietatsG:HJtoParam} we have that
$\|\GG_\eps(M_1)\|_\sigma\leq K|\mu|\eps^{\eta+1}$. On the other
hand, by the bound of $c_{lk}$ in  Corollary
\ref{coro:ShiftPeriodica} and the bound of $T_1$ given by
Proposition \ref{prop:HJ:hyp:general}, $M_2$ satisfies
$\|M_2\|_\sigma\leq K|\mu|\eps^{\eta+1}$, and therefore $\|\GG_\eps
(M_2)\|_\sigma\leq K|\mu|\eps^{\eta+1}$.

The bounds of $N_1$, $N_2$ and $R$ can be obtained analogously
taking into account the definition of $M$ in
\eqref{eq:PDEHJtoParamRHS} and that $R$ is quadratic in $h$.
\end{proof}

We split Theorem
\ref{th:HJtoParam} in the following proposition and corollary, which
are rewritten in terms of the Banach space defined in
\eqref{def:HJtoParam:Banach}. Theorem  \ref{th:HJtoParam} follows
directly from those results.

\begin{proposition}\label{prop:HJtoParam}
Let $\rr_1$  be the constant considered in Proposition
\ref{prop:HJ:hyp:general} and let us consider $\rr_3$ and $\rr_4$
such that $\rr_4>\rr_3>\rr_1$  and $\eps_0>0$ small enough (which
might depend on $\rr_i$, $i=1,3,4$). Then, for $\eps\in(0,\eps_0)$
there exists a function $\UU^u\in \AAA_\sigma$ defined in
$\tro_{\rr_3,\rr_4}^u\times\TT_\sigma$ that satisfies equation
\eqref{eq:PDEHJtoParam2}. Moreover,  for
$(v,\tau)\in\tro_{\rr_3,\rr_4}^u\times\TT_\sigma$,
$v+\UU^u(v,\tau)\in
 D^{u}_{\infty,\rr_1}$ and there
exists a constant $b_3>0$ such that
\[
\|\UU^u\|_{\sigma}\leq b_3|\mu|\eps^{\eta+1}.
\]
\end{proposition}

\begin{corollary}\label{coro:ParameterizationOuter:Tecnic}
Let  us consider the constants $\rr_3$ and $\rr_4$ given by
Proposition \ref{prop:HJtoParam} and $\eps_0>0$ small enough. Then,
for $\eps\in(0,\eps_0)$ there exist parameterizations of the
invariant manifolds
\[(Q^u(v,\tau),P^u(v,\tau))=(q_0(v)+Q^u_1(v,\tau),p_0(v)+P^u_1(v,\tau))\]
which are solution of equation \eqref{eq:PDEParametritzacions}.
Moreover $(Q^u_1,P^u_1)\in \AAA_\sigma\times\AAA_\sigma$ are defined
in $\tro_{\rr_3,\rr_4}^u\times\TT_\sigma$  and there exists a
constant $b_4>0$ such that
\[
\begin{array}{l}
\left\| Q^u_1\right\|_\sigma\leq b_4|\mu|\eps^{\eta+1}\\
\left\| P^u_1\right\|_\sigma\leq b_4|\mu|\eps^{\eta+1}.
\end{array}
\]
\end{corollary}
The proof of this corollary is a straightforward consequence of
Proposition \ref{prop:HJtoParam}.

We prove Proposition \ref{prop:HJtoParam} by using a fixed point
argument. Nevertheless, the operator $\MM$ in
\eqref{eq:PDEtoParamRHS2} has linear terms in $h$ which are not
small when $\eta=0$. Therefore, we have first to consider a change
of variables to obtain a contractive operator. For this purpose, let
us consider $\ol N_1=\GG_\eps (N_1)$, where $\GG_\eps$ is the
operator in \eqref{def:operadorG:HJtoParam} and $N_1$ the function
in \eqref{def:PDEtoPam:N1}. Taking into account that $\langle
N_1\rangle=0$ and applying Lemmas \ref{lemma:HJtoParam:Cotes} and
\ref{lemma:PropietatsG:HJtoParam}, we have that
\begin{equation}\label{eq:CotaN1barra}
\left\| \ol
N_1\right\|_\sigma=\left\|\GG_\eps(N_1)\right\|_\sigma\leq
K|\mu|\eps^{\eta+1}.
\end{equation}
Then, we consider the change
\begin{equation}\label{eq:HJtoPDE:canvi:ligual}
h=\left(1+\ol N_1\right)\ol h
\end{equation}
which, by \eqref{eq:CotaN1barra}, is invertible for $(v,\tau)\in
\tro_{\rr_3,\rr_4}^u\times\TT_\sigma$. By \eqref{eq:PDEHJtoParam2}
and \eqref{eq:HJtoPDE:canvi:ligual}, $\ol h$ is solution of
\[
\LL_\eps\ol h=\Mm^*(\ol h),
\]
where
\begin{equation}\label{eq:HJtoParam:OperadorM1:ligual}
\Mm^* \left(\ol h\right)(v,\tau)= \wh M(v,\tau)+\wh N (v,\tau)\ol
h(v,\tau)+\wh R\left(\ol h(v,\tau),v,\tau\right)
\end{equation}
with
\begin{align}
\wh M(v,\tau)&=\left(1+\ol N_1(v,\tau)\right)\ii
M(v,\tau)\label{def:HJtoParam:Mbarret}\\
\wh N(v,\tau)&=\left(1+\ol N_1(v,\tau)\right)\ii N_1(v,\tau) \ol
N_1(v,\tau)+N_2(v,\tau)\label{def:HJtoParam:Nbarret}\\
\wh R(\ol h,v,\tau)&=\left(1+\ol N_1(v,\tau)\right)\ii
R\left(\left(1+\ol N_1(v,\tau)\right)\ol
h,v,\tau\right).\label{def:HJtoParam:Rbarret}
\end{align}
To find a solution of this equation, we look for a fixed point $\ol
h\in\AAA_\sigma$ of the operator
\begin{equation}\label{eq:HJtoParam:Functional:ligual}
\ol\Mm=\GG_\eps\circ\Mm^*,
\end{equation}
where $\GG_\eps$ and $\Mm^*$ are the operators
\eqref{def:operadorG:HJtoParam} and
\eqref{eq:HJtoParam:OperadorM1:ligual}. Then, Proposition
\ref{prop:HJtoParam} is a consequence of the following lemma.

\begin{lemma}\label{lemma:HJtoParam:ligual}
Let us consider $\eps_0>0$ small enough. Then, for
$\eps\in(0,\eps_0)$, there exists a function $\overline
h\in\AAA_\sigma$ defined in $\tro_{\rr_3,\rr_4}^u\times\TT_\sigma$,
such that it is a fixed point of the operator
\eqref{eq:HJtoParam:Functional:ligual}. Moreover, it satisfies
\[
\left\| \overline h\right\|_\sigma\leq K|\mu|\eps^{\eta+1}
\]
and  then $u=v+\left(1+\ol N_1(v,\tau)\right)\ol h(v,\tau)\in
D^{u}_{\infty,\rr_1}$ for
$(v,\tau)\in\tro_{\rr_3,\rr_4}^u\times\TT_\sigma$.
\end{lemma}
\begin{proof}
It is straightforward to see that the operator $\ol \Mm$ sends
$\AAA_\sigma$ to itself. We are going to prove that there exists a
constant $b_3>0$ such that $\ol\Mm$ is contractive in $\ol
B(b_3|\mu|\eps^{\eta+1})\subset\AAA_\sigma$.

Let us consider first $\overline\Mm(0)=\GG_\eps\circ\Mm^*(0)$. From
the definitions of $\Mm^*$ and $\wh M$ in
\eqref{eq:HJtoParam:OperadorM1:ligual} and
\eqref{def:HJtoParam:Mbarret} respectively, we have that
\[
\ol \Mm(0)=\GG_\eps(\Mm^*)=\GG_\eps\left(\left(1+\ol N_1\right)\ii
M\right)=\GG_\eps\left(M\right)-\GG_\eps\left(\left(1+\ol
N_1\right)\ii\ol N_1M\right).
\]
The first term has already been bounded in Lemma
\ref{lemma:HJtoParam:Cotes}, and satisfies
$\|\GG_\eps(M)\|_\sigma\leq K|\mu|\eps^{\eta+1}$. For the second one
has to take into account Lemma \ref{lemma:PropietatsG:HJtoParam},
and then \eqref{eq:CotaN1barra} and  Lemma
\ref{lemma:HJtoParam:Cotes},  to obtain
\[
\left\|\GG_\eps\left(\left(1+\ol N_1\right)\ii\ol
N_1M\right)\right\|_\sigma\leq K \left\|\ol
N_1\right\|_\sigma\left\|M\right\|_\sigma\leq
K|\mu|^2\eps^{2\eta+1}.
\]
Therefore, there exists a constant $b_3>0$ such that
\[
\left\|\ol\MM(0)\right\|_\sigma\leq \frac{b_3}{2}|\mu|\eps^{\eta+1}.
\]
Let us consider now $\ol h_1,\ol  h_2\in \ol
B(b_3|\mu|\eps^{\eta+1})\subset\AAA_\sigma$. Then using the
properties of $\GG_\eps$ given in Lemma
\ref{lemma:PropietatsG:HJtoParam} and the definition of $\Mm^*$ in
\eqref{eq:HJtoParam:OperadorM1:ligual},
\[
\begin{array}{ll}
\dps\left\|\ol\Mm(\ol h_2)-\ol\Mm(\ol h_1)\right\|_\sigma&\dps\leq
K\left\|\Mm^*(\ol h_2)-\Mm^*(\ol h_1)\right\|_\sigma\\
&\dps\leq K\left\| \wh N(v,\tau)(\ol h_2-\ol h_1)+\wh R(\ol
h_2,v,\tau)-\wh R(\ol h_1,v,\tau)\right\|_\sigma.
\end{array}
\]
Taking into account the definitions of $\wh N$ and $\wh R$ in
\eqref{def:HJtoParam:Nbarret} and \eqref{def:HJtoParam:Rbarret} and
applying Lemma \ref{lemma:HJtoParam:Cotes} and bound
\eqref{eq:CotaN1barra}, one obtains
\[
\dps\left\|\ol\Mm(\ol h_2)-\ol\Mm(\ol h_1)\right\|_\sigma\leq
K|\mu|\eps^{\eta+1}\|\ol h_2-\ol h_1\|_\sigma.
\]
Therefore, reducing $\eps$ if necessary, $\mathrm{Lip}\ol\Mm\leq
1/2$ and therefore $\ol\Mm$ is contractive from the ball
$B(b_3|\mu|\eps^{\eta+1})\subset\AAA_\sigma$ into itself and it has
a unique fixed point $\overline h$.
\end{proof}

\begin{proof}[Proof of Proposition \ref{prop:HJtoParam}]
To prove Proposition \ref{prop:HJtoParam} from Lemma
\ref{lemma:HJ:hyp:ligual}, it is enough to undo the change of
variables \eqref{eq:HJtoPDE:canvi:ligual} to obtain
$\UU^u=\left(1+\ol N_1\right)\ol h$. Then, using bound
\eqref{eq:CotaN1barra} and increasing slightly $b_3$ if necessary,
we obtain the bound for $\UU^u$.
\end{proof}

\subsubsection{Proof of
Theorem~\ref{th:Extensio}}\label{subsec:Extensio:general}
We prove Theorem~\ref{th:Extensio} looking for a solution
of~\eqref{eq:PDEParametritzacions} through a fixed point argument,
taking the parameterizations of the invariant manifolds as
perturbations of the parameterizations of the unperturbed
separatrix. Since we only deal with the unstable manifold, we omit
the superscript $u$. We consider
\[
\left(\begin{array}{c} Q(v,\tau)\\P(v,\tau)\end{array}\right)=
\left(\begin{array}{c}
q_0(v)+Q_1(v,\tau)\\p_0(v)+P_1(v,\tau)\end{array}\right)
\]
and thus we  look for $(Q_1,P_1)$ as solutions of
\begin{equation}\label{eq:PDEParametritzacionsLinealitzat}
\left(\LL_\eps-A(u)\right)\left(\begin{array}{l}Q_1\\P_1\end{array}
\right)=\KK\left(\begin{array}{l}Q_1\\P_1\end{array}\right),
\end{equation}
where $\LL_\eps$ is the operator defined in \eqref{def:Lde}, $A$ is
the matrix defined in \eqref{eq:MatriuLinealSeparatriu},
\[
\KK(\xi)(u,\tau)=\left(\begin{array}{l}\dps\mu\eps^{\eta}\pa_p\widehat
H_1\left(q_0(u)+\xi_1,p_0(u)+\xi_2,\tau\right)\\
\dps G(\xi_1)(u,\tau)-\mu\eps^{\eta}\pa_q\widehat
H_1\left(q_0(u)+\xi_1,p_0(u)+\xi_2,\tau\right)\end{array}\right)
\]
and
\begin{equation}\label{def:Extensio:G}
G(\xi_1)(u,
\tau)=-\left(V'(\xp(\tau)+q_0(u)+\xi_1)-V'(\xp(\tau))-V'(q_0(u))-V''(q_0(u))\xi_
1\right),
\end{equation}
where for shortness we have put $\xi_1$ and $\xi_2$ for
$\xi_1(u,\tau)$ and $\xi_2(u,\tau)$.

We decompose $\KK$ considering constant, linear and higher order terms in $\xi$
as
\begin{equation}\label{def:ExtensionOperator}
\KK(\xi)(u,\tau)=L(u,\tau)+\left(M_1(u,\tau)+M_2(u,\tau)\right)\xi(u,
\tau)+N(\xi)(u,\tau)
\end{equation}
with
\begin{align}
L(u,\tau)&=\mu\eps^{\eta}\left(\begin{array}{c}\pa_p\widehat
H_1(q_0(u),p_0(u),\tau)\\
-\pa_q\widehat
H_1(q_0(u),p_0(u),\tau)\end{array}\right)+\left(\begin{array}{c}0\\
G(0)(u,\tau)\end{array}\right)\label{def:Extension:L}\\
M_1(u,\tau)&=\mu\eps^{\eta}\left(\begin{array}{cc}\dps\pa_{qp}\widehat
H^1_1\left(q_0(u),p_0(u),\tau\right)&\pa_{pp}\widehat
H^1_1(q_0(u),p_0(u),\tau)\\
\dps -\pa_{qq}\widehat
H^1_1\left(q_0(u),p_0(u),\tau\right)&-\pa_{qp}\widehat
H^1_1(q_0(u),p_0(u),\tau)
\end{array}\right)\label{def:Extension:M1}\\
M_2(u,\tau)&=\mu\eps^{\eta+1}\left(\begin{array}{cc}\dps\pa_{qp}\widehat
H^2_1\left(q_0(u),p_0(u),\tau\right)&\pa_{pp}\widehat
H^2_1(q_0(u),p_0(u),\tau)\\
\dps -\pa_{qq}\widehat
H^2_1\left(q_0(u),p_0(u),\tau\right)&-\pa_{qp}\widehat
H^2_1(q_0(u),p_0(u),\tau)
\end{array}\right)\label{def:Extension:M2}\\
N(\xi)(u,\tau)&=L(u,\tau)+\left(M_1(u,\tau)+M_2(u,\tau)\right)\xi(u,
\tau)-\KK(\xi)(u,\tau)\label{def:Extension:N}.
\end{align}

First step is to define the following function space
\[
\YY_\sigma=\left\{h:\wt
D^{\out,u}_{\rr,d,\kk}\times\TT\rightarrow\CC;\text{ real-analytic},
\|h\|_{\sigma}<\infty\right\},
\]
where $\wt D^{\out,u}_{\rr,d,\kk}$ is the domain defined in
\eqref{def:DominOuterParam} and
\begin{equation}\label{def:FourierNorm}
\|h\|_\sigma=\sum_{k\in\ZZ}\left\|h^{[k]}\right\|_\infty
e^{|k|\sigma},
\end{equation}
where $\|\cdot\|_\infty$ is the classical supremmum norm. It is a
well known fact that  this function space is a Banach algebra (see
for instance \cite{Sauzin01}). We also define the product space
\begin{equation}\label{def:BanachExtension:vec}
\YY_{\sigma}\times\YY_{\sigma}=\{ h=(h_1,h_2):\wt
D^{\out,u}_{\rr,d,\kk}\times\TT_\sigma\rightarrow\CC^2;\,\,\text{real-analytic},
\|h\|_{\sigma}=\|h_1\|_{\sigma}+\|h_2\|_{\sigma}<\infty\}.
\end{equation}
Since we deal with the Banach space
$\YY_{\sigma}\times\YY_{\sigma}$, it is also useful to consider the
norm for $2\times 2$ matrices induced by $\|\cdot\|_{\sigma}$. Let
$B=\left(b^{ij}\right)$ be a $2\times 2$ matrix such that
$b^{ij}\in \YY_{\sigma}$. Then, the induced norm with respect to the
norm of $\YY_{\sigma}\times\YY_{\sigma}$, which we also denote
$\|\cdot\|_{\sigma}$ abusing notation, is given by
\begin{equation}
\|B\|_{\sigma}=\max_{j=1,2}\left\{\left\|b^{1j}\right\|_{\sigma}+\left\|b^{2j}
\right\|_{\sigma}\right\}.
\end{equation}
The next lemma gives some properties of this induced norm.
\begin{lemma}\label{lemma:Extensio:MatrixNorms} The following statements are
satisfied
\begin{enumerate}
\item If $h\in \YY_{\sigma}\times\YY_{\sigma}$ and
$B=(b^{ij})$ is a $2\times 2$ matrix with $b^{ij}\in \YY_{\sigma}$,
then $Bh\in \YY_{\sigma}\times\YY_{\sigma}$ and
\[
\|Bh\|_{\sigma}\leq\|B\|_{\sigma}\|h\|_{\sigma}.
\]
\item If $B_1=(b_1^{ij})$ and $B_2=(b_2^{ij})$ are $2\times 2$ matrices which
satisfy $b_1^{ij}\in \YY_{\sigma}$ and $b_2^{ij}\in \YY_{\sigma}$
respectively, then $B_3=(b_3^{ij})=B_1B_2$ satisfies $b_3^{ij}\in
\EE_{\sigma}$ and
\[
\|B_3\|_{\sigma}\leq\|B_1\|_{\sigma}\|B_2\|_{\sigma}.
\]
\end{enumerate}
\end{lemma}

Second step is to look for a right inverse of $\LL_\eps-A(u)$, where
$A$ is defined in \eqref{eq:MatriuLinealSeparatriu}. To obtain it we
use the operator $\GG_\eps$ defined in
\eqref{def:operadorG:HJtoParam}, which is well defined for functions
belonging to $\YY_\sigma$, if we take $u_1,\ol u_1$ the vertices of
the domain $\wt D^{\out,u}_{\rr,d,\kk}$ defined in
\eqref{def:DominOuterParam} (see Figure \ref{fig:OuterParam}).
Recalling that $\Phi$ defined in \eqref{def:MatriuFonamentalReal}
satisfies $\LL_\eps \Phi=A\Phi$, we can define a right inverse of
$\LL_\eps-A(v)$ as
\begin{equation}\label{def:operadorGbarraExtensio}
\wh\GG_\eps(h)=\Phi\GG_\eps(\Phi\ii h), \qquad\text{ for
}\quad
h=\left(\begin{array}{c}h_1\\h_2\end{array}\right).
\end{equation}
\begin{lemma}\label{lemma:PropietatsGbarraExtensio}
The operator $\wh \GG_\eps$ in \eqref{def:operadorGbarraExtensio}
satisfies the following properties.
\begin{enumerate}
\item If $h\in\YY_{\sigma}\times\YY_{\sigma}$, then $\wh\GG_\eps(h)\in
\YY_{\sigma}\times\YY_{\sigma}$ and
\[
\left\|\wh\GG_\eps(h)\right\|_{\sigma}\leq K\|h\|_{\sigma}.
\]
\item Furthermore, if $\langle h\rangle=0$,  then
\[
\left\|\wh\GG_\eps(h)\right\|_{\sigma}\leq K\eps\|h\|_{\sigma}.
\]
\end{enumerate}
\end{lemma}

We rewrite  Theorem \ref{th:Extensio} in terms of equation
\eqref{eq:PDEParametritzacionsLinealitzat} and the Banach spaces
defined in \eqref{def:BanachExtension:vec}.

\begin{proposition}\label{prop:extensio}
Let $\rr_4$  and $\kk_1$ be the constant considered in Proposition
\ref{prop:HJtoParam} and \ref{prop:extensio:trig} and let also $d_0>0$ and
$\eps_0>0$ small enough. Then, for
$\eps\in(0,\eps_0)$ there exist functions
$(Q_1,P_1)\in\YY_{\sigma}\times\YY_{\sigma}$ which satisfy equation
\eqref{eq:PDEParametritzacionsLinealitzat} and are the analytic
continuation of the functions $(Q_1,P_1)$ obtained in Corollary
\ref{coro:ParameterizationOuter:Tecnic}. Moreover, there exists a
constant $b_5>0$ such that
\[
\|(Q_1,P_1)\|_{\sigma}\leq b_5|\mu|\eps^{\eta+1}.
\]
\end{proposition}
Before proving the proposition, we state and prove the following
technical lemma.
\begin{lemma}\label{lemma:Extension:cotes}
The functions $L$, $M_1$, $M_2$ and $N$ defined in
\eqref{def:Extension:L}, \eqref{def:Extension:M1},
\eqref{def:Extension:M2} and  \eqref{def:Extension:N} respectively,
have the following properties,
\begin{enumerate}
\item $L\in \YY_{\sigma}\times\YY_{\sigma}$ and
satisfies
\[
\| L \|_{\sigma}\leq K|\mu|\eps^{\eta},\qquad
\| \wh\GG_\eps(L)\|_{\sigma}\leq K|\mu|\eps^{\eta+1}.
\]
\item $M_1=\left(m_1^{ij}\right)$ and $M_2=\left(m_2^{ij}\right)$ satisfy
$m_1^{ij},m_2^{ij}\in \YY_{\sigma}\times\YY_{\sigma}$, $\langle M_1\rangle=0$
and
\[
\|M_1\|_{\sigma}\leq K|\mu|\eps^{\eta},\qquad
\|M_2\|_{\sigma}\leq K|\mu|^2\eps^{2\eta+1}.
\]
\item If $\xi,\xi'\in B(\nu)\subset \YY_{\sigma}\times\YY_{\sigma}$, then
\[
\left\| N(\xi')-N(\xi)\right\|_{\sigma}\leq K\nu\left\|
\xi'-\xi\right\|_{\sigma}.
\]
\end{enumerate}
\end{lemma}
\begin{proof}
For the first statement let us split $L$ as $L=L_1+L_2+L_3$ with
\[
L_i(u,\tau)=\left(\begin{array}{c}\mu\eps^{\eta+i-1}\pa_p \wh
H_1^i(q_0(u),p_0(u),\tau)\\
-\mu\eps^{\eta+i-1}\pa_q \wh
H_1^i(q_0(u),p_0(u),\tau)\end{array}\right),\,\,i=1,2
\]
and
\[
L_3(u,\tau)=\left(\begin{array}{c}0\\
G(0)(u,\tau)\end{array}\right),
\]
 where $\wh H_1^1$, $\wh H_1^2$ and $G$ are the functions defined in
\eqref{def:HamPertorbat:H1}, \eqref{def:HamPertorbat:H2} and
\eqref{def:Extensio:G} respectively. One can easily see that
$L_1,L_2\in \YY_{\sigma}\times\YY_{\sigma}$,  $\langle L_1\rangle=0$
and $\|L_1\|_{\sigma}\leq K|\mu|\eps^{\eta}$ and, using Corollary
\ref{coro:ShiftPeriodica}, also that $\|L_2\|_{\sigma}\leq
K|\mu|^2\eps^{2\eta+1}$. Thus, applying Lemma
\ref{lemma:PropietatsGbarraExtensio} one obtains
$\|\wh\GG_\eps(L_i)\|_{\sigma}\leq K|\mu|\eps^{\eta+1}$ for $i=1,2$.

To obtain analogous properties for $L_3$, it is enough to apply Mean
Value Theorem to obtain
\[
L_3(u,\tau)=\left(\begin{array}{c} 0\\\dps
-\int_0^1V^{'''}\left(s_1x_p(\tau)+s_2q_0(u)\right)ds_1ds_2
q_0(u)x_p(\tau)\end{array}\right).
\]
Then, $\|L_3\|_{\sigma}\leq K|\mu|\eps^{\eta+1}$. Therefore,
applying Lemma \ref{lemma:PropietatsGbarraExtensio} we have that
$\|\wh\GG_\eps(L_3)\|_{\sigma}\leq K|\mu|\eps^{\eta+1}$. This
finishes the proof of the first statement.

The proof of the other statements is straightforward.
\end{proof}

To prove Proposition \ref{prop:extensio}, first one has to perform a
change of variables to equation
\eqref{eq:PDEParametritzacionsLinealitzat} to obtain a contractive
operator. In fact, this change is only necessary in the case
$\eta=0$. Let us consider
\begin{equation}\label{def:M1barra:Extensio}
\overM=\left(\ol m_1^{ij}\right)\qquad\text{ with }\qquad \ol
m_1^{ij}=\GG_\eps\left(m_1^{ij}\right),
\end{equation}
where $\GG_\eps$ is the operator defined in
\eqref{def:operadorG:HJtoParam} and $M_1=\left(m_1^{ij}\right)$ is
the matrix defined in \eqref{def:Extension:M1}. By Lemmas
\ref{lemma:Extension:cotes} and \ref{lemma:PropietatsGExtensio}, one
can see that
\begin{equation}\label{eq:Extensio:cotaMbarra}
\|\overM\|_{\sigma}\leq K|\mu|\eps^{\eta+1}.
\end{equation}
We consider the change of variables
\begin{equation}\label{def:extensio:xibarra}
\xi=\left(\Id+\overM\right)\overline\xi
\end{equation}
which is invertible. Using
\eqref{eq:PDEParametritzacionsLinealitzat} and
\eqref{def:extensio:xibarra}, $\overline\xi$ is solution of equation
\begin{equation}\label{eq:PDE:extensio:xibarra}
\left(\LL_\eps-A(u)\right)\overline\xi=\wh\KK(\overline \xi),
\end{equation}
where
\begin{equation}\label{eq:Extensio:OperadorRHS:xibarra}
\wh\KK(\overline \xi) =\wh L+\wh M \ol \xi+\wh N\left(\ol\xi\right)
\end{equation}
with
\begin{align}
\wh L&=\dps\left(\Id+\overM\right)\ii L \label{def:Extensio:Lbarra}\\
\wh M&=\left(\Id+\overM\right)\ii\left(M_1\overM+ A\overM -\overM
A+M_2\left(\Id+\overM\right)\right)\label{def:Extensio:Mbarra}\\
\wh N(\ol \xi)&=\dps\left(\Id+\overM\right)\ii
N\left(\left(\Id+\overM\right)\ol\xi\right)\label{def:Extensio:Nbarra}.
\end{align}

Since we want to obtain the analytic continuation of the
parameterizations of the manifolds obtained in Corollary
\ref{coro:ParameterizationOuter:Tecnic}, we need to impose
\emph{initial conditions}. Nevertheless, since we invert
$\LL_\eps-A(u)$  by using the operator $\wh\GG_\eps$ in
\eqref{def:operadorGbarraExtensio} which is defined acting on the
Fourier coefficients, we need to consider a different initial
condition depending on the Fourier coefficient, that is in $u_1$ or
in $\bar u_1$ (see Figure \ref{fig:OuterParam}). Thus, we define the
following function
\begin{equation}\label{def:ExtensioCondInicial}
\begin{split}
L_0(v,\tau)=&\sum_{k<0}\Phi(v)\Phi\ii(\ol v_1)\ol \xi^{[k]}\left(\ol
v_1\right)e^{-ik\eps\ii (v-\bar v_1)}e^{ik\tau}
\\
&+\sum_{k\geq 0}\Phi(v)\Phi\ii( v_1)\ol \xi^{[k]}\left(
v_1\right)e^{-ik\eps\ii
(v-v_1)}e^{ik\tau}\\
&+\Phi(v)\Phi\ii\left( -\rr_4\right)\xi^{[0]}\left(-\rr_4\right).
\end{split}
\end{equation}
Recall that $\ol\xi(v,\tau)$ is already known for $v=v_1,\ol
v_1,-\rr_4$ using \eqref{def:extensio:xibarra},
\eqref{def:M1barra:Extensio} and Corollary
\ref{coro:ParameterizationOuter:Tecnic}.

\begin{lemma}\label{lemma:Extensio:CondInicial}
The function $L_0(u,\tau)$ in \eqref{def:ExtensioCondInicial}
satisfies de following properties:
\begin{itemize}
\item $\left(\LL_\eps-A(v)\right) L_0=0$, where $\LL_\eps$ is the operator in
\eqref{def:Lde}.
\item $L_0\in
\YY_{\sigma}\times\YY_{\sigma}$ and
\[\|L_0\|_{\sigma}\leq K|\mu|\eps^{\eta+1}.\]
\end{itemize}
\end{lemma}

The function $\ol\xi$ satisfies equation
\eqref{eq:PDE:extensio:xibarra} and the initial conditions on the
Fourier coefficients $L_0$ in \eqref{def:ExtensioCondInicial} if and
only if it is solution of the integral equation
\[
\left(\begin{array}{l}Q_1\\P_1\end{array}\right)=L_0+\wh\GG_\eps
\circ \KK \left(\begin{array}{l}Q_1\\P_1\end{array}\right),
\]
where $\wh\GG_\eps$ and  $\KK$ are the operators defined in
\eqref{def:operadorGbarraExtensio} and \eqref{def:ExtensionOperator}
respectively. Thus, we look for a fixed point
$\xi=(Q_1,P_1)\in\YY_{\sigma}\times\YY_{\sigma}$ of the operator
\begin{equation}\label{def:ExtensionFullOperator}
\ol\KK= L_0+\wh\GG_\eps\circ \wh\KK.
\end{equation}
Therefore, Proposition \ref{prop:extensio} is a straightforward
consequence of the following lemma.

\begin{lemma}\label{lemma:Extension_lmajor}
Let $\eps_0>0$ be small enough. Then, for
$\eps\in(0,\eps_0)$, there exists a function
$\ol\xi\in\YY_{\sigma}\times\YY_{\sigma}$ defined in $\wt D^{\out,
u}_{\rr_4,d_0,\kk_1}\times \TT_\sigma$ such that is a fixed point of
the operator \eqref{def:ExtensionFullOperator} and satisfies
\[
\left\|\ol\xi\right\|_{\sigma}\leq b_5|\mu|\eps^{\eta+1}.
\]
for a certain constant $b_5>0$ independent of $\eps$ and $\mu$.
Moreover, $\xi=(\Id+\overM)\overline\xi$, where $\overM$ is the
function defined in \eqref{def:M1barra:Extensio}, is the analytic
continuation of the function $\xi=(Q_1,P_1)$ obtained in Corollary
\ref{coro:ParameterizationOuter:Tecnic}.
\end{lemma}

\begin{proof}
To prove the lemma, first we see that there exists a constant
$b_5>0$ such that the operator $\bar\KK$ in
\eqref{def:ExtensionFullOperator} is contractive from $\ol
B(b_5|\mu|\eps^{\eta+1})\subset \YY_{\sigma}\times\YY_{\sigma}$ to
itself and thus that it has a fixed point. Then, we will see that
 $\xi=(\Id+\overM)\overline\xi$, where $\overM$ is the
function defined in \eqref{def:M1barra:Extensio}, is the analytic
continuation of the parameterizations of the manifolds which have
been obtained in Corollary \ref{coro:ParameterizationOuter:Tecnic}.

Let us first consider $\ol\KK(0)$. Using the definitions of
$\ol\KK$, $\wh\KK$ and $\wh L$ in \eqref{def:ExtensionFullOperator},
\eqref{def:ExtensionOperator} and \eqref{def:Extensio:Lbarra}, we
have that
\[
\begin{split}
\ol\KK(0)&= L_0+\wh\GG_\eps\left(\wh L\right)\\
&= L_0+\wh\GG_\eps\left(
L\right)-\wh\GG_\eps\left(\overM\left(\Id+\overM\right)\ii L\right).
\end{split}
\]

From Lemmas~\ref{lemma:Extensio:CondInicial},
\ref{lemma:PropietatsGbarraExtensio} and
\ref{lemma:Extension:cotes}, and applying also  the bound of
$\overM$ in \eqref{eq:Extensio:cotaMbarra}, it is straightforward to
see that $\|\ol\KK(0)\|_{\sigma}\leq K|\mu|\eps^{\eta+1}$, and thus
there exists a constant $b_5>0$ such that
$\|\ol\KK(0)\|_{\sigma}\leq b_5|\mu|\eps^{\eta+1}/2$.

Let us consider now $\ol\xi^1,\ol\xi^2\in \ol B(
b_5|\mu|\eps^{\eta+1})\subset \YY_{\sigma}\times\YY_{\sigma}$. Then
using the definitions of $\ol\KK$ and $\wh\KK$ in
\eqref{def:ExtensionFullOperator} and
\eqref{eq:Extensio:OperadorRHS:xibarra}, and applying Lemma
\ref{lemma:PropietatsGbarraExtensio},
\begin{align*}
\left\|\ol\KK\left(\ol\xi^1\right)-\ol\KK\left(\ol\xi^2\right)\right\|_{\sigma}
\leq
&K\left\|\wh\KK\left(\ol\xi^1\right)-\wh\KK\left(\ol\xi^2\right)\right\|_{\sigma
}\\
\leq&\dps K\left\|\wh M\left(\ol \xi^2-\ol \xi^1\right)+\wh
N\left(\ol \xi^1\right)-\wh N\left(\ol
\xi^2\right)\right\|_{\sigma}.
\end{align*}
Then, using the definitions of $\wh M$ and $\wh N$ in
\eqref{def:Extensio:Mbarra} and \eqref{def:Extensio:Nbarra} and
applying Lemma \ref{lemma:Extension:cotes} and bound
\eqref{eq:Extensio:cotaMbarra}, one can see that
\[
\left\|\ol\KK\left(\ol\xi^1\right)-\ol\KK\left(\ol\xi^2\right)\right\|_{\sigma}
\leq
K|\mu|\eps^{\eta+1}\left\|\ol\xi^1-\ol\xi^2\right\|_{\sigma}.
\]
Therefore, reducing $\eps$ if necessary, $\Lip \ol\KK<1/2$ and then
$\ol\KK$ is contractive  from $\ol B( b_5|\mu|\eps^{\eta+1})\subset
\YY_{\sigma}\times\YY_{\sigma}$ to itself and it has a unique fixed
point $\ol\xi$.

To prove that $\xi=(\Id+\ol M_1)\ol\xi$ is the analytic continuation
of the function $\xi=(Q_1,P_1)$ obtained in Corollary
\ref{coro:ParameterizationOuter:Tecnic}, one can proceed as in the
proof of Lemma \ref{lemma:Extensio:Trig:General}.
\end{proof}

\begin{proof}[Proof of Proposition \ref{prop:extensio}]
It is enough to undo the change \eqref{def:extensio:xibarra}. For the
bound of $\xi=(Q_1,P_1)$ it is enough to consider the bound of
$\overM$ in \eqref{eq:Extensio:cotaMbarra} and the bound of $\ol\xi$
in Lemma \ref{lemma:Extension_lmajor} and increase slightly $b_5$ if
necessary.
\end{proof}

\subsubsection{Proof of Theorem \ref{th:ParamtoHJ}}\label{sec:Transicio}
This section is devoted to obtain a parameterization of the
invariant manifolds of the form \eqref{eq:ParameterizationHJ} in the
domains \eqref{def:DominisRaros}. To this end, we look for changes
of variables $v=u+\VV^{u,s}(u,\tau)$ which satisfy
\eqref{eq:CanviParamToHJ}.

Since the proof of Theorem \ref{th:ParamtoHJ} is analogous for both
invariant manifolds, we only deal with the unstable case and we omit
the superscript $u$ to simplify notation.

Writing $Q(v,\tau)=q_0(v)+Q_1(v,\tau)$, equation
\eqref{eq:CanviParamToHJ} reads
\[
q_0\left(u+\VV(u,\tau)\right)-q_0(u)=-Q_1\left(u+\VV(u,\tau),\tau\right).
\]
Taking into account that $\dot q_0(u)=p_0(u)$, to obtain a solution
of this equation is equivalent to obtain a fixed point of the
operator
\begin{equation}\label{def:Functional:Transition}
\NNN(h)(u,\tau)=-\frac{1}{p_0(u)}\left(Q_1\left(u+h(u,\tau),
\tau\right)+q_0\left(u+h(u,\tau)\right)-q_0(u)-p_0(u)h(u,\tau)\right).
\end{equation}

Let the function space
\begin{equation}\label{def:Transition:Banach}
\QQ_{\kk,d,\sigma}=\left\{h:
\tro_{\kk,d}^{\out,u}\times\TT_\sigma\rightarrow \CC;\,\,
\text{real-analytic}, \|h\|_{\kk,d,\sigma}<\infty\right\},
\end{equation}
where $\|\cdot\|_{\kk,d,\sigma}$ is the Fourier norm defined in
\eqref{def:FourierNorm} but applied to functions defined in
$\tro_{\kk,d}^{\out,u}\times\TT_\sigma$.

We split Theorem \ref{th:ParamtoHJ} in the following proposition and
corollary, which are written in terms of the Banach space defined in
\eqref{def:Transition:Banach}.

\begin{proposition}\label{prop:Transition}
Let us consider  the constant $\kk_1$ given in Proposition
\ref{prop:extensio}, $d_0>d_1>0$,  $\kk_2>\kk_1$ and $\eps_0>0$
small enough, which might depend on the previous constants. Then,
there exists a constant $b_6>0$ such that for $\eps\in(0,\eps_0)$
and $\kk_1$ and $\kk_2$ big enough, the operator $\NNN$ is
contractive from $\ol
B\left(b_6|\mu|\eps^{\eta+1}\right)\subset\QQ_{\sigma}$ to itself.

Then, $\NNN$ has a unique fixed point $\VV\in \ol
B\left(b_6|\mu|\eps^{\eta+1}\right)\subset\QQ_{\sigma}$, which
satisfies that $u+\VV(u,\tau)\in \tro_{\kk_1,d_0}^{\out,u}$ for
$(u,\tau)\in \tro_{\kk_2,d_1}^{\out,u}\times\TT_\sigma$.
\end{proposition}

\begin{corollary}\label{coro:Transition:CotesGeneradora}
There exists a function
$T:\tro_{\kk_2,d_1}^{\out,u}\times\TT_\sigma\rightarrow\CC$ such
that
\[
\pa_uT(u,\tau)=p_0(u)P(u+\VV(u,\tau),\tau),
\]
where $P$ and $\VV$ are the functions obtained in Theorem
\ref{th:Extensio} and Proposition \ref{prop:Transition}
respectively, and satisfies equation \eqref{eq:HamJacGeneral}.
Moreover, it belongs to $\QQ_{\sigma}$ and satisfies
\[
\left\| \pa_u T-\pa_u T_0\right\|_{\kk_2,d_1,\sigma}\leq
b_7|\mu|\eps^{\eta+1}.
\]
for certain constant $b_7>0$.
\end{corollary}

We devote the rest of this section to prove Proposition
\ref{prop:Transition} and Corollary
\ref{coro:Transition:CotesGeneradora}.

\begin{proof}[Proof of Proposition \ref{prop:Transition}]
The operator $\NNN$ sends $\QQ_{\kk_2,d_1,\sigma}$ to itself. To see
that exists a constant $b_6>0$ such that $\NNN$ is contractive in
$\ol
B\left(b_6|\mu|\eps^{\eta+1}\right)\subset\QQ_{\kk_2,d_1,\sigma}$,
we first consider $\NNN(0)$. By Proposition \ref{prop:extensio},
there exists a constant $b_6>0$ such that
\[
\left\|\NNN(0)\right\|_{\kk_2,d_1,\sigma}=\left\|p_0\ii(v)Q_1(v,\tau)\right\|_{
\kk_2,d_1,\sigma}\leq
\frac{b_6}{2}|\mu|\eps^{\eta+1}.
\]
To see that $\NNN$ is contractive, let $h_1,h_2\in \ol
B\left(b_6|\mu|\eps^{\eta+1}\right)\subset\QQ_{\kk_2,d_1,\sigma}$.
By Proposition \ref{prop:extensio}, we know that $Q_1(u,\tau)$ is
defined in $\tro_{\kk_1,d_0}^{\out,u}$ and satisfies
$\|Q_1\|_{\kk_1,d_0,\sigma}\leq K |\mu|\eps^{\eta+1}$ in this
 domain. Applying Cauchy estimates in the nested
domains $\tro_{2\kk_1,d_0/2}^{\out,u}\subset
\tro_{\kk_1,d_0}^{\out,u}$, one has that
\[
\|\pa_vQ_1\|_{2\kk_1,d_0/2,\sigma}\leq \frac{K}{\kk_1}
\mu\eps^{\eta}.
\]
Then, defining $h^s(v,\tau)=sh_2(v,\tau)+(1-s)h_1(v,\tau)$ for
$s\in(0,1)$, using the mean value theorem, increasing $\kk_1$ if
necessary and taking $\kk_2>2\kk_1$,
\[
\begin{split}
\left\| \NNN(h_2)-\NNN(h_1)\right\|_{\kk_2,d_1,\sigma}\leq&
\dps\left\|p_0\ii(v)\int_0^1\left(\pa_u
Q_1\left(v+h^s,\tau\right)+p_0(v+h^s)-p_0(v)\right)ds\right\|_{\kk_2,d_1,\sigma}
\\
&\cdot\left\|h_2-h_1\right\|_{\kk_2,d_1,\sigma}\\
\leq&\dps
\frac{K|\mu|\eps^{\eta}}{\kk_1}\left\|h_2-h_1\right\|_{\kk_2,d_1,\sigma}\\
\leq& \frac{1}{2}\left\|h_2-h_1\right\|_{\sigma}.
\end{split}
\]
Then, $\NNN: \ol B\left(b_6|\mu|\eps^{\eta+1}\right)\rightarrow \ol
B\left(b_6|\mu|\eps^{\eta+1}\right)\subset\QQ_{\kk_2,d_1,\sigma}$
and is contractive. Therefore, it has a unique fixed point which
satisfies the properties stated in Proposition
\ref{prop:Transition}.
\end{proof}

\begin{proof}[Proof of Corollary
\ref{coro:Transition:CotesGeneradora}] Proposition
\ref{prop:Transition}, gives a parameterization of the form
\[
(q,p)=(Q(u+\VV(u,\tau),\tau),P(u+\VV(u,\tau),\tau))=(q_0(u),P(u+\VV(u,\tau),
\tau)).
\]
We want to have a parameterization of the form
\eqref{eq:ParameterizationHJ}, where  $T$ is a function which
satisfies \eqref{eq:HamJacGeneral}. To recover this function it is
enough to point out that, since we want it to be solution of
\eqref{eq:HamJacGeneral}, we know its gradient
\[
(\pa_u T(u,\tau), \pa_\tau T(u,\tau))
=\left(p_0(u)P(u+\VV(u,\tau),\tau), -\eps\ol
H(u,p_0(u)P(u+\VV(u,\tau),\tau),\tau\right).
\]
Then, it is enough to check the compatibility condition
\begin{equation}\label{eq:canvi:compat}
\pa_\tau\left[p_0(u)P(u+\VV(u,\tau),\tau)\right]=-\pa_u\left[\eps\ol
H\left(u,p_0(u)P(u+\VV(u,\tau),\tau),\tau\right)\right].
\end{equation} Differentiating equation \eqref{eq:CanviParamToHJ}, one has
that $\VV$ satisfies
\[
\begin{split}
\pa_vQ(u+\VV(u,\tau),\tau)\left(1+\pa_u\VV(u,\tau)\right)&=p_0(u)\\
\pa_vQ(u+\VV(u,\tau),\tau)\pa_\tau\VV(u,\tau)+\pa_\tau
Q(u+\VV(u,\tau),\tau)&=0
\end{split}
\]
Then, using this equalities and  equation
\eqref{eq:PDEParametritzacions}, one can prove
\eqref{eq:canvi:compat}.

Finally, recalling that $\pa_uT_0(u)=p^2_0(u)$ and
$P(v,\tau)=p_0(v)+P_1(v,\tau)$ and applying Proposition
\ref{prop:extensio} and the mean value theorem,
\[
\begin{split}
\left\| \pa_u T-\pa_u T_0\right\|_{\kk_2,d_1,\sigma}&\leq\left\| p_0(u)
\left(P_1(u+\VV(u,\tau),\tau)+p_0(u+\VV(u,\tau))-p_0(u)\right)
\right\|_{\sigma}\\
&\leq b_7|\mu|\eps^{\eta+1}.
\end{split}
\]
\end{proof}

\subsubsection{Proof of Theorem
\ref{th:ExtensioFinal}}\label{subsec:ExtensioFinal} The proof of
Theorem \ref{th:ExtensioFinal} follows the same steps as the proof
of Theorem \ref{th:Extensio:Trig}. For this reason, in this section
we only explain which are the main differences.

First, let us point out that the operator $\GG_\eps$ defined in
\eqref{def:operadorG:HJtoParam} can be also applied to functions
defined in $D^u_{\kk_3,d_2}\times\TT_\sigma$ if one takes as
$u_1,\bar u_1$ the vertices of $D^u_{\kk_3,d_2}$ (see Figure
\ref{fig:BoomerangDomains}) and as $\rr$ the left endpoint of the
interval $D^u_{\kk_3,d_2}\cap\RR$. Now the paths of integration
cannot be straight lines. Nevertheless, it is easy to see that
$\GG_\eps$ satisfies the same properties as the ones stated in Lemma
\ref{lemma:PropietatsGExtensio} but applied to functions defined in
the new domain.

Then, if one considers Banach spaces analogous to
$\EE_{\nu,\sigma}$, with $\nu>0$, given in
\eqref{def:BanachExtension}, for functions defined in
$D^u_{\kk_3,d_2}\times\TT_\sigma$, one can prove Proposition
\ref{prop:extensio:trig}, but looking for the function $T_1$ as the
analytic continuation of the function obtained in Corollary
\ref{coro:Transition:CotesGeneradora} instead of the function $T_1$
obtained in Proposition \ref{prop:HJ:hyp:general} and Proposition
\ref{prop:HJ:parab:general}.

The rest of the proof follows the same lines as the proof of
Proposition \ref{prop:extensio:trig}.

\subsection{The first asymptotic term of the invariant manifolds near the
singularities for the
case $\ell =2r$}\label{subsec:Extensio:ProvaExpansioT:ligual}

In the case $\eta=0$ and $\ell-2r=0$,  we need a
better knowledge of the first asymptotic terms of the invariant
manifolds close the singularities of the unperturbed separatrix
$u=\pm ia$. In the next result, we obtain them for the unstable
invariant manifold close to $u=ia$. The other cases can be done
analogously.

For real, $2\pi$-periodic in $\tau$, analytic functions
$h:D_{\kk_3,d_2}^u\times\TT_\sigma \to \CC$,
we define the Fourier norm
\begin{equation*}
\Vert h \Vert_{\nu,\sigma} = \sum_{k\in \ZZ} \sup_{(u,\tau) \in
D_{\kk_3,d_2}^u\times\TT_\sigma} |(u^2 +a ^2)^{\nu} h^{[k]}(u)| e^{|k| \sigma}
\end{equation*}
being, as usual, $h^{[k]}$ the $k$-Fourier coefficient of $h$.

The next proposition will be used later in Section \ref{sec:CanviFinal}.
\begin{proposition}\label{coro:TuPropSing}
Let us assume $\ell-2r=0$,
and let $Q_j$ and $F_j$ be the functions defined in
\eqref{def:FunctionsQ} and \eqref{def:FunctionsF} respectively (see
also Remark \ref{remark:DefQIntrinseca}) and the constant $C_+$
given in \eqref{eq:SepartriuAlPol} and
\eqref{eq:SepartriuAlPolTrig}.

Then, there exists a real analytic function
$\xi :D_{\kk_3,d_2}^u\times\TT_\sigma \to \CC$,
satisfying that:
\begin{equation*}
\Vert \xi \Vert_{2r+1-1/q,\sigma} \leq K|\mu |\eps^{\eta+1},
\end{equation*}
where $r=\p/\q$ has been
defined in Hypothesis \textbf{HP2} and, for $(u,\tau)\in
D_{\kk_3,c_2}^u\times\TT_\sigma$, the functions $T^u$
obtained respectively in Proposition \ref{prop:extensio:trig} (case $\p_0(u)
\neq 0$)
and Proposition \ref{prop:extensio} (general case),
are such that
\begin{equation}\label{eq:ExpansioT1u}
\left\|\pa_uT_1(u,\tau)-\frac{2r\mu\eps^{\eta+1}
C_+^2}{(u-ia)^{2r+1}}\left(F_0(\tau)+\mu\langle Q_0F_1\rangle\right)
+\xi(u,\tau)\right\|_{2r+2,\sigma}\leq
K|\mu|\eps^{\eta+2}.
\end{equation}
\end{proposition}
\begin{proof}
We prove Proposition \ref{coro:TuPropSing} in the polynomial case. Taking into
account
Remark \ref{remark:DefQIntrinseca}, the proof of the trigonometric
case is completely analogous.

We only deal with the case $p_0(u)\neq 0$ being the other case analogous.
For this reason we will only take into account the previous results in this
case.
In fact we will see that Proposition \ref{coro:TuPropSing} is also valid for
$(u,\tau) \in D_{\rr_1',\kk_0'}^{out,u}$
where $\rr_1'$ and $\kk_0'$ are the constants for which Proposition
\ref{prop:extensio:trig} holds.

We first obtain the asymptotic
expansion for the function $\pa_v \wh T_1(v,\tau)$ obtained in
Proposition \ref{prop:extensio:trig}, which is defined for
$(v,\tau)\in D_{\rr_1',\kk_0'}^{\out,u}\times\TT_\sigma$ and then we
use the change variables $v=u+h(u,\tau)$ defined in Lemma
\ref{lemma:Extensio:Trig:canvi}.

To obtain the asymptotic expansion, we decompose $\pa_v\wh T_1$ into
several parts taking into account that $\pa_v \wh T_1$ is a fixed
point of the operator $\JJ$ in
\eqref{def:Extensio:Trig:Operador:Sencer} and that we know
explicitly $\JJ(0)$. We use the functions $A_i$ defined in
\eqref{def:HJ:A1}, \eqref{def:HJ:A2}, \eqref{def:HJ:A3}
respectively, the change of variables $g$ obtained in Lemma
\ref{lemma:Extensio:Trig:canvi} and the operator $\JJ$ in
\eqref{def:Extensio:Trig:Operador:Sencer}. We take
\[
\pa_v\wh T_1=\sum_{i=1}^7 D_i(v,\tau)
\]
with
\begin{align}
D_1(v,\tau)&=A_0(v,\tau)\label{def:ExpT1:D1}\\
D_2(v,\tau)&=\GG_\eps\left(\pa_vA_1(v+g(v,\tau),\tau)\right)\label{def:ExpT1:D2}
\\
D_3(v,\tau)&=\GG_\eps\left(\pa_v A_2(v,\tau)\right)\label{def:ExpT1:D3}\\
D_4(v,\tau)&=\GG_\eps\left(\pa_v\left[
\pa_vA_2(v,\tau)g(v,\tau)\right]\right)\label{def:ExpT1:D4}\\
D_5(v,\tau)&=\GG_\eps\left(\pa_v\left[ A_2(v+g(v,\tau),\tau)-\pa_v
A_2(v,\tau)g(v,\tau)-A_2(v,\tau)\right]\right)\label{def:ExpT1:D5}\\
D_6(v,\tau)&=\GG_\eps\left(\pa_v
\left[A_3(v+g(v,\tau),\tau)\right]\right)\label{def:ExpT1:D6}\\
D_7(v,\tau)&=\JJ\left(\pa_v\wh
T_1\right)(v,\tau)-\JJ\left(0\right)(v,\tau)\label{def:ExpT1:D7}.
\end{align}
Let us point out that the sum of the first six terms is $\JJ(0)$. We
bound each term.  For the second to the fifth terms, we follow the proof
of Lemma \ref{lemma:Extensio:Trig:Cotes}, where the functions $A_1$,
$A_2$ and $A_3$ have been bounded.

To bound \eqref{def:ExpT1:D1}, it is enough to recall that, by
\eqref{eq:Cota:CondIni:ExtTrig}, $D_1\in
\EE_{0,\rr_2,\kk_1,\sigma}\subset\EE_{2r+1-1/\q,\rr_2,\kk_1,\sigma}$,
to obtain
\[
\|D_1\|_{2r+1-\frac{1}{\q},\sigma}\leq \|D_1\|_{0,\sigma}\leq
K|\mu|\eps^{\eta+1}.
\]
To bound \eqref{def:ExpT1:D2}, we apply the bound of $A_1$ obtained
in \eqref{eq:Cota:A1} and use $r\geq 1$ to see that
$D_2\in\EE_{r+1,\sigma}\subset\EE_{2r+1-1/\q,\sigma}$ and
\[
\|D_2\|_{2r+1-\frac{1}{\q},\sigma}\leq\|D_2\|_{r+1,\sigma}\leq
K|\mu|\eps^{\eta+1}.
\]
Since $\langle A_2\rangle=0$,  we can define a function $\ol A_2$
such that $\pa_\tau \ol A_2=A_2$ and $\langle \ol A_2\rangle=0$.
Moreover, one can write
\[
\begin{split}
D_3&=\GG_\eps(\pa_v A_2)=\GG_\eps\left(\pa^2_{\tau v}\ol A_2\right)\\
&=\eps\GG_\eps\left(\LL_\eps\left(\pa_v\ol
A_2\right)\right)-\eps\GG_\eps \left(\pa^2_v \ol A_2\right).
\end{split}
\]
Then, using the definition of $\GG_\eps$ in
\eqref{def:operadorG:HJtoParam} and applying Lemma
\ref{lemma:PropietatsGExtensio}, one can see that there exists a
function $\wt\xi_3\in \EE_{0,\sigma}\subset\EE_{2r+1-1/\q,\sigma}$,
which satisfies,
\[
\|\wt \xi_3\|_{2r+1-\frac{1}{\q},\sigma}\leq
K\|\wt\xi_3\|_{0,\sigma}\leq K|\mu|\eps^{\eta+1},
\]
such that
\[
\left\|D_3-\eps\pa_v\ol A_2-\wt\xi_3\right\|_{2r+2,\sigma}\leq
K|\mu|\eps^{\eta+2}.
\]
Moreover, recalling the definition of $A_2$ in \eqref{def:HJ:A2} and
defining functions $\ol a_{kl}$ such that
\begin{equation}\label{def:Primis:as}
\pa_\tau \ol a_{kl}=0\,\,\text{ and }\,\,\langle \ol a_{kl}\rangle=0
\end{equation}
we have that
\[
\pa_v\ol A_2(v,\tau)=-\mu\sum_{2\leq k+l\leq N}\ol a_{kl}(\tau)\pa_v
\left(q_0(v)^k p_0(v)^l\right).
\]
Then, recalling the definition of the functions $Q_j$ and $F_j$ in
\eqref{def:FunctionsQ} and \eqref{def:FunctionsF} and the constant
$C_+$ in \eqref{eq:SepartriuAlPol}, $\pa_v\ol A_2$ satisfies
\[
\eps\pa_v\ol A_2(v,\tau)=\frac{2r\mu\eps^{\eta+1}
C_+^2F_0(\tau)}{(v-ia)^{2r+1}}+\OO\left(\frac{\mu\eps^{\eta+1}}{(v-ia)^{
2r+1-\frac{1}{\q}}}\right).
\]
Therefore, there exists
$\xi_3\in\EE_{2r+1-1/\q,\rr_2,\kk_1,\sigma}$ satisfying
\[
\| \xi_3\|_{2r+1-\frac{1}{\q},\sigma}\leq K|\mu|\eps^{\eta+1},
\]
 such
that
\[
\left\|D_3(v,\tau)-\frac{2r\mu\eps^{\eta+1} C_+^2
F_0(\tau)}{(v-ia)^{2r+1}}-\xi_3(v,\tau)\right\|_{2r+2,\sigma}\leq
K|\mu|\eps^{\eta+2}.
\]
To bound \eqref{def:ExpT1:D4}, we  first subtract its averaged term.
Then, using Lemma \ref{lemma:Extensio:Trig:canvi} to bound $g$ and
$\pa_v g$, Lemma \ref{lemma:Extension:cotes} to bound the first and
second derivatives of $A_2$ and Lemma
\ref{lemma:PropietatsGExtensio}, we obtain
\[
\left\|D_4-\GG_\eps\left(\pa_v\langle\pa_v A_2\cdot
g\rangle\right)\right\|_{2r+2,\sigma}\leq K|\mu|^2\eps^{\eta+2}.
\]
On the other hand, using the definition of $\GG_\eps$ in
\eqref{def:operadorG:HJtoParam}
\[
\GG_\eps\left(\pa_v\langle \pa_v A_2\cdot g\rangle\right)(v)=\langle
\pa_v A_2\cdot g\rangle(v)-\langle\pa_v A_2\cdot g\rangle(-\rr_1').
\]
To obtain its leading term,  first we look for the first order of
the function $g$ given in \eqref{def:FuncioG:extensio}. Using the
definition of $B_1$ in \eqref{def:InftyB1}, the functions
\eqref{def:Primis:as}, the bounds of $\pa_v B_1$ in
\eqref{eq:Infty:CotesB} and Lemma \ref{lemma:PropietatsGExtensio},
we have that
\begin{equation}\label{eq:Expansio:canvi}
\left\| g(v,\tau)-\mu\eps^{\eta+1}\sum_{\substack{2\leq k+l\leq
N\\l\geq 1}}l\ol a_{kl}(\tau)q_0(v)^k
p_0(v)^{l-2}\right\|_{1,\sigma}\leq K|\mu|\eps^{\eta+2}.
\end{equation}
Then, using the functions $Q_j$ and $F_j$ defined in
\eqref{def:FunctionsQ} and \eqref{def:FunctionsF} respectively, and
taking into account the definition of $A_2$ in \eqref{def:HJ:A2},
there exists a function
$\xi_4\in\EE_{2r+1-1/\q,\rr_2,\kk_1,\sigma}$ satisfying
\[
\| \xi_4\|_{2r+1-\frac{1}{\q},\sigma}\leq K|\mu|\eps^{\eta+1},
\]
such that
\[
\GG_\eps\left(\pa_v\langle\pa_v A_2\cdot
g\rangle\right)=\frac{2r\mu^2\eps^{2\eta+1} C_+^2\langle Q_0
F_1\rangle}{(v-ia)^{2r+1}}+\xi_4(u,\tau).
\]
Therefore,  one can see that
\[
\left\|D_4(v,\tau)-\frac{2r\mu^2\eps^{2\eta+1} C_+^2 \langle Q_0
F_1\rangle}{(v-ia)^{2r+1}}-\xi_4(u,\tau)\right\|_{2r+2,\sigma}\leq
K|\mu|^2\eps^{2\eta+2}.
\]
For \eqref{def:ExpT1:D5}, it is enough to apply Lemmas
\ref{lemma:PropietatsGExtensio} and \ref{lemma:Extensio:Trig:canvi},
the definition of $A_2$ and the mean value theorem, to obtain
\[
\|D_5\|_{2r+2,\sigma}\leq K|\mu|^3\eps^{3\eta+2}.
\]
To bound \eqref{def:ExpT1:D6}, let us recall the definitions of
$A_3$ and $\wh H_1^2$ in \eqref{def:HJ:A3} and
\eqref{def:HamPertorbat:H2}. Then, it is enough to apply Lemma
\ref{lemma:PropietatsGExtensio}, to obtain
\[
\| D_6\|_{2r+1-\frac{1}{\q},\sigma}\leq\| D_6\|_{2r,\sigma}\leq
K|\mu|\eps^{\eta+1}.
\]
Finally, for \eqref{def:ExpT1:D7}, it is enough to take into account
the definitions of $\JJ$ and $\wh \FF$ in
\eqref{def:Extensio:Trig:Operador:Sencer} and
\eqref{eq:HJperT1:ligual:RHS} and apply Lemmas
\ref{lemma:PropietatsGExtensio}, \ref{lemma:Extensio:Trig:Cotes} and
\ref{lemma:Extensio:Trig:General}, which give,
\[
\begin{split}
\left\|\JJ\left(\pa_v\wh
T_1\right)-\JJ\left(0\right)\right\|_{2r+2,\sigma}&\leq
\left\|\wh\FF\left(\pa_v\wh
T_1\right)-\wh\FF\left(0\right)\right\|_{2r+2,\sigma}\\
&\leq \left\|\wh B\cdot\pa_v \wh T_1+\wh C\left(\pa_v \wh T_1,v,\tau\right)-\wh
C\left(0,v,\tau\right)\right\|_{2r+2,\sigma}\\
&\leq K|\mu|\eps^{\eta+1}\left\|\pa_v\wh
T_1\right\|_{2r+1,\sigma}\leq K|\mu|^2\eps^{2\eta+2}.
\end{split}
\]
Considering all the bounds of $D_i$,  we define
\[
\xi(u,\tau)=D_1(u,\tau)+D_2(u,\tau)+\xi_3(u,\tau)+\xi_4(u,\tau)+D_6(u,\tau)
\]
Then, $\xi\in\EE_{2r+1-1/\q,\sigma}$ satisfying
\[
\| \xi\|_{2r+1-\frac{1}{\q},\sigma}\leq K|\mu|\eps^{\eta+1},
\]
and then we have
\begin{equation}\label{eq:T1xi}
\left\|\pa_v\wh T_1(v,\tau)-\frac{2r\mu\eps^{\eta+1}
C_+^2}{(v-ia)^{2r+1}}\left(F_0(\tau)+\mu\langle
Q_0F_1\rangle\right)-\xi(u,\tau)\right\|_{2r+2,\sigma}\leq
K|\mu|\eps^{\eta+2}.
\end{equation}
To finish the proof of Proposition \ref{coro:TuPropSing}, one has to
consider the change of variables $v=u+h(u,\tau)$ defined in Lemma
\ref{lemma:Extensio:Trig:canvi} to obtain
\[
\pa_u T_1(u,\tau)=(1+\pa_u h(u,\tau))\ii \pa_v \wh
T_1(u+h(u,\tau),\tau).
\]
Then, the bounds of $h$ and $\pa_uh$ in Lemma
\ref{lemma:Extensio:Trig:canvi} and \eqref{eq:T1xi}, finish the
proof of the proposition.
\end{proof}

\section{Approximation of the invariant manifolds in the inner
domains.}\label{sec:matching}

\subsection{Case $\ell <2r$ : proof of Proposition
\ref{coro:Varietat:FirstOrder:lmenor}}
\label{subsec:Extensio:ProvaExpansioT:lmenor}

We prove the results stated in Proposition \ref{coro:Varietat:FirstOrder:lmenor}
concerning the unstable
manifold. The proof of the results concerning the stable one follows the same
lines.
To obtain the bound of $\pa_u T_1^u(u,\tau)-\pa_u\TTT^u_0(u,\tau)$, we first
bound
$\pa_v \wh T^u_1(v,\tau)-\pa_v\TTT^u_0(v,\tau)$ where $\wh T^u_1$ is the
function obtained in Theorems \ref{th:Extensio:Trig} and \ref{th:ExtensioFinal},
which is defined for
$(v,\tau)\in D_{\kk_3,d_2}^u\times\TT_\sigma$, and $\TTT^u_0$ is the function
defined in
\eqref{def:MigMelnikov}. Then, we will use the change of variables
$v=u+h(u,\tau)$ defined in Lemma
\ref{lemma:Extensio:Trig:canvi} to obtain the bound stated in Proposition
\ref{coro:Varietat:FirstOrder:lmenor}.

Let us define first $v_3$ and $v_4$ the leftmost and rightmost vertices of the
inner domain
$D_{\kk_3,c_1}^{\inn,+,u}$ (see Figure \ref{fig:Inners}). Then, we can define
the operator
\begin{equation}\label{def:operadorGtitlla:MillorCota}
\wt\GG_\eps(h)(v,\tau)=\sum_{k\in\ZZ}\wt\GG_\eps(h)^{[k]}(v)e^{ik\tau},
\end{equation}
where its Fourier coefficients are given by
\begin{align*}
\dps\wt\GG_\eps (h)^{[k]}(v)&=\int_{v_3}^v e^{ik\eps\ii
(t-v)}h^{[k]}(t)\,dt& \text{ for }k>0\\
\dps\wt\GG_\eps (h)^{[0]}(v)&=\int_{v_4}^v h^{[0]}(t)\,dt&
\\
\dps\wt\GG_\eps (h)^{[k]}(v)&= \int_{v_4}^v e^{ik\eps\ii
(t-v)}h^{[k]}(t)\,dt& \text{ for }k< 0.
\end{align*}
It can be easily seen that this operator satisfies analogous properties to the
ones satisfied by
the operator $\GG_\eps$ defined in \eqref{def:operadorG:HJtoParam}, which are
given in Lemma
\ref{lemma:PropietatsGExtensio}. Let us consider also the Fourier expansions
\[
h_1(v,\tau)=H_1(q_0(v),p_0(v),\tau)=\sum_{k\in\ZZ}
H_1^{[k]}(v)e^{ik\tau}\quad\text{ and }\quad \wh
A(v,\tau)=A(v+g(v,\tau),\tau)=\sum_{k\in\ZZ}\wh A^{[k]}(v)e^{ik\tau},
\]
where $H_1$ is the function defined in \eqref{def:Ham:Original:perturb:poli} and
\eqref{def:Ham:Original:perturb:trig}, $A$ is the function defined in
\eqref{def:InftyA} and $g$ has been
given in Lemma \ref{lemma:Extensio:Trig:canvi}.

First, we observe that, since $\pa_v \wh T_1 = \JJ(\pa_v \wh T_1)$, where the
operator $\JJ$ is defined in
\eqref{def:Extensio:Trig:Operador:Sencer},
\[
\pa_v\wh T_1 (v,\tau) = \wt\GG_\eps (\pa_v A )(v,\tau) + \sum_{i=1}^4
N_i(v,\tau)
\]
with:
\begin{align}
N_1(v,\tau)=&A_0(v,\tau)\label{def:ExpT1:N1}\\
N_2(v,\tau)=&\JJ\left(\pa_v\wh
T_1\right)(v,\tau)-\JJ\left(0\right)(v,\tau)\label{def:ExpT1:N2}.\\
N_3(v,\tau)=&-\wt\GG_\eps (\pa_v \wh A) (v,\tau) +\GG_\eps(\pa_v \wh A)
(v,\tau)\label{def:ExpT1:N3}\\
N_4(v,\tau)=&\wt\GG_\eps (\pa_v \wh A) (v,\tau) -\wt\GG_\eps(\pa_v  A)
(v,\tau).\label{def:ExpT1:N4}
\end{align}

Second we split $\pa_v\TTT^u_0$ as:
\[
\pa_v\TTT^u_0 = -\mu \eps ^{\eta}\wt\GG_\eps (\pa_v h_1)(v,\tau) - N_5,
\]
where
\begin{eqnarray}
N_5(v,\tau)&=&\mu\eps^\eta\sum_{k>0}\int_{-\infty}^{v_3} e^{ik\eps\ii(t-v)}
\pa_v H^{[k]}_1(t) \,dt\notag \nonumber\\
&&+\mu\eps^\eta\sum_{k\leq 0}\int_{-\infty}^{v_4} e^{ik\eps\ii(t-v)}\pa_v
H^{[k]}_1(t)\,dt .\label{def:ExpT1:N5}
\end{eqnarray}

Finally, we use the definition of $A$ in \eqref{def:InftyA} and $\wh H_1$ in
\eqref{def:ham:ShiftedOP:perturb}, and the fact that, as the periodic orbit does
not depend on $v$,
\[
\pa_v\left(V(x_p(\tau))+H_1(x_p(\tau),y_p(\tau),\tau)\right)=0
\]
to obtain

\begin{align}
\wt\GG_\eps (\pa_v A )(v,\tau)+\mu \eps ^{\eta}\wt\GG_\eps (\pa_v h_1) (v,\tau)
=& -y_p(\tau)p_0(u)+x_p(\tau)\dot p_0(u)\\
&+N_6 + N_7+N_8
\end{align}
with
\begin{eqnarray}
N_6 & = &-\mu
\eps^{\eta}\wt\GG_\eps\pa_v\Big(H_1(q_0(v)+x_p(\tau),p_0(v)+y_p(\tau),
\tau)-H_1(q_0(v),p_0(v),\tau)\Big)
\label{def:ExpT1:N6} \\
N_7 & = &
-\wt\GG_\eps\pa_v\Big(V(q_0(u)+x_p(\tau))-V(q_0(u))-V'(q_0(u))x_p(\tau)\Big)
\label{def:ExpT1:N7} \\
N_8 & = & \wt\GG_\eps\pa_v\Big(-V'(q_0(u))x_p(\tau)+V'(x_p(\tau))q_0(u)
\nonumber \\
&&+ \mu\eps^\eta\left(q_0(u)\pa_x H_1(x_p(\tau),y_p(\tau),\tau)+p_0(u)\pa_y
H_1(x_p(\tau),y_p(\tau),\tau)\right)\Big)\nonumber\\
&&+ y_p(\tau)p_0(u)-x_p(\tau)\dot p_0(u). \label{def:ExpT1:N8}
\end{eqnarray}
Finally we obtain:
\[
\pa_v\wh T_1 (v,\tau) -\pa_v\TTT^u_0 = -y_p(\tau)p_0(u)+ x_p(\tau)\dot
p_0(u)+\sum_{i=1}^8 N_i(v,\tau).
\]

Now, we proceed  to bound $N_1 ,\dots ,N_8$.

To bound $N_1$ in  \eqref{def:ExpT1:N1}, it is enough to recall that, by
\eqref{eq:Cota:CondIni:ExtTrig}, $N_1\in \EE_{0,\rr_1',\kk_0',\sigma}$
and
\[
\|N_1\|_{0,\sigma}\leq K|\mu|\eps^{\eta+1}.
\]

For $N_2$ in \eqref{def:ExpT1:N2}, it is enough to consider the bound of $\pa_v
\wh T_1$ given in Proposition \ref{prop:extensio:trig} and the Lipschitz
constant of the operator  $\JJ$ in
\eqref{def:Extensio:Trig:Operador:Sencer} restricted to the ball $\ol
B(|\mu|\eps^{\eta+1})\subset \EE_{\ell+1,\rr_1',\kk_0',\sigma}$, which has been
obtained in the proof of Lemma \ref{lemma:Extensio:Trig:General}. Then,
\[
 \begin{split}
 \| N_2\|_{0,\sigma}\leq & K\frac{\eps^{-(\ell+1)}}{(\kappa_0')^{\ell +1}}
\|N_2\|_{\ell+1,\sigma}\\
\leq & K|\mu|\eps^{-(\ell+1)+\eta+1-\max\{0,\ell-2r+1\}}\left\| \pa_v\wh
T_1\right\|_{\ell+1,\sigma}\\
\leq & K|\mu|^2\eps^{2\eta-\ell+1-\max\{0,\ell-2r+1\}}.
 \end{split}
\]

To bound $N_3$ in \eqref{def:ExpT1:N3} we observe that $\langle N_3 \rangle=0$
and
\begin{align*}
 N_3^{[k]} (v) & = e^{ik\eps^{-1}(v_3-v)}\int_{u_1}^{v_3} e^{ik\eps^{-1}
(t-v_3)}\left(\pa_v \wh A^{[k]} \right)(t)\,dt &
\text{for }\;\;\; k>0\\
N_3^{[0]}(v)&=\wh A^{[0]}(v)-\wh A^{[0]}(v_4)\\
N_3^{[k]} (v) & = e^{ik\eps^{-1}(v_4-v)}\int_{\bar{u}_1}^{v_4} e^{ik\eps^{-1}
(t-v_4)}\left(\pa_v \wh A^{[k]}\right)(t)\,dt &
\text{for }\;\;\; k<0.
\end{align*}
Taking into account that the operator $\wt\GG_\eps$ satisfies also the
properties of the operator $\GG_\eps$ given in Lemma
\ref{lemma:PropietatsGExtensio},  and using the bounds of $g$ and $\pa_vA$ given
in  Lemmas \ref{lemma:Extensio:Trig:canvi} and \ref{lemma:Extensio:Trig:Cotes}
respectively, we  obtain the following bounds. For $k\neq 0$,
\[
\begin{split}
\left\|N_3^{[k]}\right\|_{0,\sigma} &\leq \left\|\wt \GG_\eps \left(\pa_v \wh
A^{[k]}(v)e^{ik\tau}\right)\right\| _{0,\sigma} \\
&\leq K\eps \left\| \pa_v \wh A^{[k]}(v)e^{ik\tau}\right\| _{0,\sigma} \\
&\leq K\eps^{1-(\ell+1)\gamma} \left\| \pa_v \wh A^{[k]}(v)e^{ik\tau}\right\|
_{\ell+1,\sigma} \\
&\leq K|\mu|\eps^{\eta+1-(\ell+1)\gamma}.
\end{split}
\]
For $k=0$, we have that
\[
\begin{split}
 \|N_3^{[0]}\|_{0,\sigma}& \leq K\left\|\wh A^{[0]} \right\| _{0,\sigma} \\
&\leq K\eps^{-\ell\ga}\left\|\wh A^{[0]} \right\| _{\ell,\sigma} \leq
K|\mu|\eps^{\eta-\ell\ga}.
\end{split}
\]
Finally, note that in the case $\ell=0$, we have that the change $g$ obtained in
 Lemma \ref{lemma:Extensio:Trig:canvi} satisfies $g=0$. Then $\wh A=A$, which
implies $\langle \wh A\rangle=0$. Therefore when $\ell=0$ we have that
$N_3^{[0]}=0$. Taking this fact into account, we can bound $N_3$ by
\[
  \|N_3\|_{0,\sigma}\leq K|\mu|\eps^{\eta-\ell+\nu_2^\ast},
\]
where
\[
 \nu_2^\ast=\left\{\begin{array}{ll} \ell(1-\gamma) &\text{if }\ell>0\\
                    1-\ga &\text{if }\ell=0.
                   \end{array}\right.
\]
For $N_4$ in \eqref{def:ExpT1:N4}, one has to consider the bound of $\pa_v A$
given in Lemma \ref{lemma:Extensio:Trig:Cotes} and the bound of $g$ restricted
to the inner domain given in Corollary \ref{coro:Extensio:CotaCanvigInner}.
Then, using again the  bounds analogous to the ones given in Lemma
\ref{lemma:PropietatsGExtensio}, but to the operator $\wt\GG_\eps$
\[
\| N_4\|_{0,\sigma}\leq K\| \wh A -A \|_{0,\sigma} \leq K \|\pa_v
A\|_{0,\sigma}\|g\|_{0,\sigma}\leq K|\mu|^2\eps^{2\eta-\ell+\nu_1^{\ast}}
\]
with $\nu_1^{\ast}$ is defined in Corollary \ref{coro:Extensio:CotaCanvigInner}.

For $N_5$ in \eqref{def:ExpT1:N5}, it is enough to take into account that
$\langle h_1\rangle=0$, that $h_1$ has a ramified point of order $\ell$ at
$u=ia$ and that both $v_3$ and $v_4$  satisfy
$|v_i-ia|=\OO\left(\eps^\gamma\right)$, $i=3,4$. Then, bounding the integrals as
in Lemma \ref{lemma:PropietatsGInfty} and \ref{lemma:PropietatsGInfty:parab},
one has that
\[
\|N_5\|_{0,\sigma}\leq K|\mu|\eps^{\eta+1}  \|\partial _v h_1 \|_{0,\sigma} \leq
K|\mu|\eps^{\eta+1-\ga(\ell+1)}.
\]

To bound $N_6$ in \eqref{def:ExpT1:N6} we first use the mean value theorem to
obtain
\[
 \left\|H_1(q_0(v)+x_p(\tau),p_0(v)+y_p(\tau),\tau)-H_1(q_0(v),p_0(v),
\tau)\right\|_{0,\sigma}\leq
|\mu| \eps^{\eta-\ell+r}.
\]
Then, using that $\wt \GG_\eps$ has similar properties to the ones given in
Lemma \ref{lemma:PropietatsGExtensio} for the operator  $\GG_\eps$ we obtain
\[
\|N_6\|_{0,\sigma}\leq K|\mu|^2\eps^{2\eta-\ell+r}.
\]

The bound for $N_7$ in \eqref{def:ExpT1:N7} comes from applying the mean bound
theorem to the function
\[
V(q_0(u)+x_p(\tau))-V(q_0(u))-V'(q_0(u))x_p(\tau)
\]
and using that $V''(q_0(u))$ has a pole of second order, the bound of the
periodic orbit and the properties of  $\wt \GG_\eps$. Then, we obtain
\[
\|N_7\|_{0,\sigma} \leq K
\|V(q_0(u)+x_p(\tau))-V(q_0(u))-V'(q_0(u))x_p(\tau)\|_{0,\sigma} \leq
K|\mu|^2\eps^{2\eta} =K |\mu|^2\eps^{(\eta-\ell)+(\eta+\ell)}.
\]

To bound $N_8$ in \eqref{def:ExpT1:N8}, we write it as
\[
N_8= \wt\GG_\eps\left(\pa_v N_8^0\right ) + y_p(\tau)p_0(u)-x_p(\tau)\dot p_0(u)
\]
with
\[
\begin{split}
N_8^0(v,\tau)=&-V'(q_0(u))x_p(\tau)+V'(x_p(\tau))q_0(u)\\
&+\mu\eps^\eta\left(q_0(u)\pa_x H_1(x_p(\tau),y_p(\tau),\tau)+p_0(u)\pa_y
H_1(x_p(\tau),y_p(\tau),\tau)\right).
\end{split}
\]
Using that $-V'(q_0(u))=\dot p_0(u)$, $\dot q_0(u)=p_0(u)$ and that the periodic
orbit satisfies equations \eqref{eq:ode:original:lent}, one has
\[
 \begin{split}
N_8^0(v,\tau)=& \dot p_0(u)x_p(\tau)-\eps^{-1}\pa_\tau y_p(\tau)
q_0(u)-p_0(u)y_p(\tau)+\eps^{-1}\pa_\tau x_p(\tau) p_0(u)\\
=&-\LL_\eps (y_p(\tau)q_0(u))+\LL_\eps (x_p(\tau)p_0(u)).
 \end{split}
\]
Therefore $N_8$ can be written as
\[
\begin{split}
N_8=&\wt\GG_\eps\pa_v\LL_\eps\left(-y_p(\tau)q_0(u)+x_p(\tau)p_0(u)\right)\\
&+y_p(\tau)p_0(u)-x_p(\tau)\dot p_0(u)\\
=&\wt\GG_\eps\LL_\eps\left(-y_p(\tau)p_0(u)+x_p(\tau)\dot p_0(u)\right)\\
&-(-y_p(\tau)p_0(u)+x_p(\tau)\dot p_0(u)).
\end{split}
\]
Then, using that $\wt \GG_{\eps}$ satisfies an analogous property to the one
given  for $\GG_\eps$ in the last item of Lemma \ref{lemma:PropietatsGExtensio}:
\[
 \|N_8\|_{0,\sigma}\leq K|\mu|\eps^{\eta+1-(r+1)\gamma}.
\]
Now, choosing  $\gamma$ such that
\[
 1-(r+1)\gamma>-\ell,
\]
that is,
\[
 \gamma<\frac{\ell+1}{r+1}
\]
and considering all the bounds of $N_i$ and taking
\[
\nu^\ast=\min\left\{\nu_2^\ast, \nu_1^{\ast},1-\max\{0,\ell-2r+1\},r,\ell,
\ell+1- (r+1) \ga \right\},
\]
we obtain
\[
\left\|\pa_v\wh T_1(v,\tau)-\pa_v\TTT_0(v,\tau)\right\|_{0,\sigma}\leq
K|\mu|\eps^{\eta-\ell+\nu^\ast}.
\]
To finish the proof of Proposition \ref{coro:Varietat:FirstOrder:lmenor}, it is
enough to
consider the change of variables $v=u+h(u,\tau)$ defined in Lemma
\ref{lemma:Extensio:Trig:canvi}
and its bounds restricted to the inner domains given in Corollary
\ref{coro:Extensio:CotaCanvigInner}.

\subsection{Case $\ell \geq 2r$: proof of Theorem \ref{th:MatchingHJ}}

This section is devoted to obtain good approximations of the
invariant manifolds in the \emph{inner domains} defined in
\eqref{def:DominisInnerEnu} for the case $\ell \geq 2r$.

First in Section \ref{sec:matching:Banach} we define the Banach
spaces that will be used in the forthcoming sections and we state
some technical lemmas. In Section \ref{sec:matching:HJ} we prove
Theorem \ref{th:MatchingHJ}.

\subsubsection{Banach spaces and technical
lemmas}\label{sec:matching:Banach}

We start by defining some norms. Given $\nu\in\RR$ and an analytic
function $h:\DD_{\kk,\C}^{\inn, +,u}\rightarrow\CC$, where
$\DD_{\kk,\C}^{\inn, +,u}$ is the domain defined in
\eqref{def:DominisInnerEnu}, we consider
\[
\|h\|_{\nu,\kk,\C}=\sup_{z\in \DD_{\kk,\C}^{\inn, +,u}}\left| z^\nu
h(z)\right|.
\]
Then, for analytic functions $h:\DD_{\kk,\C}^{\inn,
+,u}\times\TT_\sigma\rightarrow\CC$ which are $2\pi$-periodic in
$\tau$, we define the corresponding Fourier norm
\[
\|h\|_{\nu,\kk,\C,\sigma}=\sum_{k\in\ZZ}\|h^{[k]}\|_{\nu,\kk,\C}e^{|k|\sigma}
\]
and the function space
\begin{equation}\label{def:matching:banach}
\ZZZ_{\nu,\kk,\C,\sigma}=\left\{h:\DD_{\kk,\C}^{\inn,
+,u}\times\TT_\sigma\rightarrow\CC; \text{ analytic},
\|h\|_{\nu,\kk,\C,\sigma}<\infty\right\}
\end{equation}
which can be checked that is a Banach space for any $\nu\in\RR$.

If there is no danger of confusion about the definition domain
$\DD_{\kk,\C}^{\inn, +,u}$ we will denote
\[
\begin{array}{ccc}
\|\cdot\|_{\nu,\sigma}=\|\cdot\|_{\nu,\kk,\C,\sigma}&\text{ and
}&\ZZZ_{\nu,\sigma}=\ZZZ_{\nu,\kk,\C,\sigma}.
\end{array}
\]

The next lemma gives some properties of these Banach spaces.
\begin{lemma}\label{lemma:matching:propsnormes}
Let $\C,\kappa>0$.
\begin{enumerate}
\item If $\nu_1\leq \nu_2$,
$\ZZZ_{\nu_2,\sigma}\subset\ZZZ_{\nu_1,\sigma}$. Moreover,
\[
\|h\|_{\nu_2,\sigma}\leq
\frac{K}{\kk^{\nu_2-\nu_1}}\|h\|_{\nu_1,\sigma}.
\]
\item If $h\in\ZZZ_{\nu_1,\sigma}$ and $g\in\ZZZ_{\nu_2,\sigma}$, then
$hg\in\ZZZ_{\nu_1+\nu_2,\sigma}$ and
\[
\|hg\|_{{\nu_1+\nu_2},\sigma}\leq
\|h\|_{\nu_1,\sigma}\|g\|_{\nu_2,\sigma}.
\]
\item Let $h\in\ZZZ_{\nu,\kk,\C,\sigma}$ and $\wh \C<\C$,
then, $\pa_xh\in\XX_{\nu,2\kk,\wh \C,\sigma}$ and
\[
\|\pa_x h\|_{\nu,2\kk,\wh \C,\sigma}\leq\frac{K}{\kk}\|
h\|_{\nu,\kk,\C,\sigma}.
\]
\end{enumerate}
\end{lemma}

Throughout this section we are going to solve equations of the form
$\LL h=g$ and $\LL h=\pa_z g$, where
\begin{equation}\label{def:L}
\LL=\pa_z+\pa_\tau.
\end{equation}
To solve these equations we consider operators $\GG$ and $\ol\GG$,
which are defined ``acting on the Fourier coefficients''.

\begin{figure}[h]
\begin{center}
\psfrag{D}{$\DD_{\kk,\C}^{\inn,+,u}$}\psfrag{I}{$\tri_{ \C,\ol
\C}^{+,u}$}\psfrag{z1}{$z_1$}\psfrag{z2}{$z_2$}
\psfrag{b0}{$\beta_0$}\psfrag{b1}{$\beta_1$}\psfrag{b2}{$\beta_2$}
\psfrag{a}{$ia$}\psfrag{a1}{$-ia$}\psfrag{a2}{$i(a-\kk\eps)$}
\includegraphics[height=6cm]{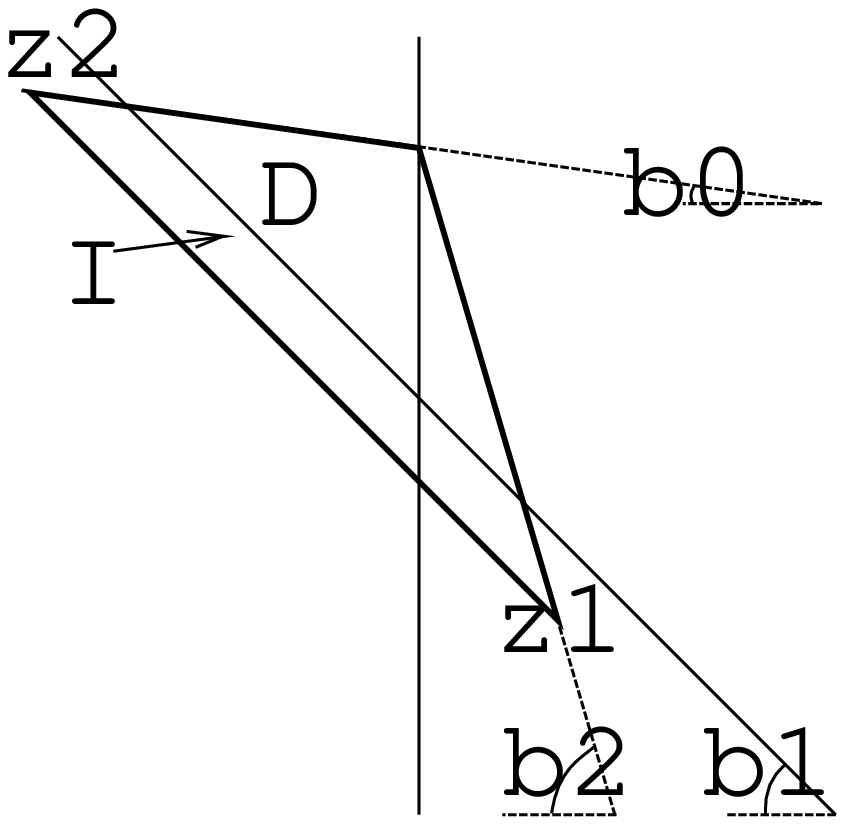}
\end{center}
\caption{\figlabel{fig:InnerVarInner} The \emph{inner domain}
$\DD_{\kk,\C}^{\inn,+,u}$ defined in \eqref{def:DominisInnerEnz} and
the transition domain $\tri_{ \C,\ol \C}^{+,u}$ defined in
\eqref{def:DominisTransInner}.}
\end{figure}

Let us consider $z_1$ and $z_2$ the vertices of the inner domain
$\DD_{\kk,\C}^{\inn,+,u}$ (see Figure \ref{fig:InnerVarInner}). As
we have done in Section \ref{sec:HJtoParam} to invert the operator
$\LL_\eps=\eps\ii\pa_\tau+\pa_v$, we invert $\LL$ integrating from
$z_1$ or $z_2$ depending on the harmonic.

We define the operators
\begin{equation}\label{def:OperadorIntegral:Matching}
\GG(h)(z,\tau)=\sum_{k\in\ZZ}\GG(h)^{[k]}(z)e^{ik\tau},
\end{equation}
where the Fourier coefficients are given by
\begin{align*}
\GG(h)^{[k]}(z)&=\int_{z_1}^z e^{-ik (z-s)}h^{[k]}(s)\,ds &\text{
for
}k<0\\
\GG(h)^{[k]}(z)&=\int_{z_2}^z e^{-ik (z-s)}h^{[k]}(s)\,ds &\text{
for }k\geq 0
\end{align*}
and
\begin{equation}\label{def:OperadorIntegralBarra:Matching}
\ol\GG(h)(z,\tau)=\sum_{k\in\ZZ}\ol\GG(h)^{[k]}(z)e^{ik\tau},
\end{equation}
where its Fourier coefficients are given by
\begin{align*}
\ol\GG(h)^{[k]}(z)&=h^{[k]}(z)-e^{-ik
(z-z_1)}h^{[k]}(z_1)-ik\int_{z_1}^z e^{-ik (z-s)}h^{[k]}(s)\,ds
&\text{ for
}k<0\\
\ol\GG(h)^{[0]}(z)&=h^{[0]}(z)-h^{[0]}(z_2) &\\
\ol\GG(h)^{[k]}(z)&=h^{[k]}(z)-e^{-ik
(z-z_2)}h^{[k]}(z_2)-ik\int_{z_2}^z e^{-ik (z-s)}h^{[k]}(s)\,ds
&\text{ for }k >0.
\end{align*}
The next lemma gives some properties of these operators. Its proof
is analogous to the one of Lemma 5.5 in \cite{GuardiaOS10}.
\begin{lemma}\label{lemma:matching:operador}
Let $\kk,\C,\nu>0$ and $\gamma\in (0,1)$. Then,
\begin{enumerate}
\item The operator $\GG:
\ZZZ_{\nu+1,\sigma}\rightarrow\ZZZ_{\nu,\sigma}$ is well defined.
Moreover, if $h\in\ZZZ_{\nu+1,\sigma}$,
\[
\left\|\GG(h)\right\|_{\nu,\sigma}\leq K\|h\|_{\nu+1,\sigma}.
\]
\item The operator $\GG:
\ZZZ_{\nu,\sigma}\rightarrow\ZZZ_{\nu,\sigma}$ is well defined.
Moreover, if $h\in\ZZZ_{\nu,\sigma}$,
\[
\left\|\GG(h)\right\|_{\nu,\sigma}\leq
K\eps^{\gamma-1}\|h\|_{\nu,\sigma}.
\]
\item The operator $\ol\GG:
\ZZZ_{\nu,\sigma}\rightarrow\ZZZ_{\nu,\sigma}$ is well defined.
Moreover, if $h\in\ZZZ_{\nu,\sigma}$,
\[
\left\|\ol\GG(h)\right\|_{\nu,\sigma}\leq K\|h\|_{\nu,\sigma}.
\]
\end{enumerate}
\end{lemma}

\subsubsection{Proof of Theorem \ref{th:MatchingHJ}}\label{sec:matching:HJ}

We rewrite Theorem \ref{th:MatchingHJ} in terms of the Banach space
\eqref{def:matching:banach}.
\begin{proposition}\label{prop:matching:HJ}
Let $\ga\in(0,\ga_2)$, where
\begin{equation}\label{def:gamma2}
\ga_2=\frac{\q(\ell-2r+1)}{\q(\ell-2r+1)+1},
\end{equation}
$\C_1>0$, $\eps_0>0$ small
enough and $\kk_6>\max\{\kk_3,\kk_5\}$ big enough, where $\kk_5$ are the
constants
defined in Theorems \ref{th:ExtensioFinal} and \ref{th:InnerImma} respectively.
Let,
\[
\varphi=\psi^u-\psi_0^u,
\]
where $\psi^u$ is the function in
\eqref{eq:FuncioGeneradoraInner}and  $\psi_0$ is the function
obtained in Theorem \ref{th:InnerImma}. Then, for $\eps\in
(0,\eps_0)$, we have $\varphi\in
\ZZZ_{2r-\frac{1}{\q},\kk_6,\C_1,\sigma}$ and there exists a
constant $b_{10}>0$ such that
\[
\left\| \pa_z\varphi\right\|_{2r-\frac{1}{\q},\kk_6,
\C_1,\sigma}\leq {b_{10}}\eps^{\frac{1}{\q}},
\]
where $r=\p/\q$ has been defined in \eqref{eq:SepartriuAlPol}.
\end{proposition}

\begin{remark}
We emphasize that Proposition \ref{prop:matching:HJ} implies straightforwardly
Theorem \ref{th:MatchingHJ}.
Indeed, we observe that the only restriction is about the range of values of
$\ga \in (0,\ga_2)$.
Let us denote by $D_{\ga}^{\inn}$ the inner domain defined by $\ga$. It is clear
that, if
$\ga \geq \ga_2 >\ga_1$, then $D_{\ga}^{\inn} \subset D_{\ga_1}^{\inn}$ and
henceforth the result holds also for
values of $\ga \geq \ga_2$.

We need to impose this condition about $\ga$ just for technical reasons.
\end{remark}
In the proof of this proposition we will refer several times to the
bounds given in Theorem \ref{th:InnerImma}. In fact, we need these
bounds expressed in terms of the Fourier norm, which are given in
Proposition 4.8 of \cite{Baldoma06}, instead of the ones given in
this theorem, which use the classical supremmum norm.

Let us point out that using the bounds of Proposition 4.8 of
\cite{Baldoma06}  and Corollary
\ref{coro:Transition:CotesGeneradora} leads to a bound of
$\pa_z\varphi$ of order 1 with respect to $\eps$. Nevertheless, this
bound is too rough to prove later the asymptotic formula for the
splitting of separatrices and therefore we will need the improved
estimates given in Proposition \ref{prop:matching:HJ}.

The proof of  Proposition \ref{prop:matching:HJ} goes as follows.
First in Section \ref{subsec:matching:HJtoPDE} we obtain a
(non-homogeneous) linear partial differential equation satisfied by
$\varphi=\psi-\psi_0$. Then, in Section
\ref{subsec:matching:initialcondition}, we obtain quantitative
estimates of $\pa_z\varphi$ in the \emph{transition domain}
$\tri_{\C,\ol \C}^{+,u}$ defined as
\begin{equation}\label{def:DominisTransInner}
\tri_{ \C,\ol \C}^{\pm,u}=\left\{z\in\CC; ia+\eps z\in
D^{\out,u}_{\rho_2,\ol\C\eps^\gamma}\cap D_{\kk, \C}^{\inn,
\pm,u}\right\},
\end{equation}
where $\ast=u,s$ (see Figure \ref{fig:InnerVarInner}), which allow
us to obtain an integral equation satisfied by $\pa_z\varphi$.
Finally, in Sections \ref{subsec:matching:fixedPoint:lmajor} and
\ref{subsec:matching:fixedPoint:ligual}  we obtain the improved
bound for $\pa_z\varphi$ for the cases $\ell-2r>0$ and $\ell-2r=0$
respectively, proving Proposition \ref{prop:matching:HJ}.

\paragraph{The Hamilton-Jacobi equation}\label{subsec:matching:HJtoPDE}

First we look for the equation satisfied by
\begin{equation}\label{def:perturbacioInner}
\varphi=\psi-\psi_0.
\end{equation}
Subtracting the Hamilton-Jacobi equations \eqref{eq:HJGeneralInner}
and \eqref{eq:HJEqInner}, one obtains
\[
\pa_\tau\varphi
+\HH(\pa_z\psi_0+\pa_z\varphi,z,\tau)-\HH_0(\pa_z\psi_0,z,\tau)=0.
\]
Taking into account that we already know the existence of $\varphi$,
we know that it is also solution of
\begin{equation}\label{eq:PDEMatchin:enZ}
\LL\varphi=\WW(\pa_z\varphi,z,\tau),
\end{equation}
where $\LL$ is the operator defined in \eqref{def:L} and
\begin{equation}\label{eq:matcing:RHS}
\WW(w,z,\tau)=-L(z,\tau)-\left(Q_1(\tau)\frac{\hmu}{z^{\ell-2r}}+M(z,
\tau)\right)w,
\end{equation}
where $Q_1$ is the function defined in \eqref{def:FunctionsQ} and
\begin{align}
L(z,\tau)&=
\HH(\pa_z\psi_0,z,\tau)-\HH_0(\pa_z\psi_0,z,\tau)\label{def:matching:L}\\
M(z,\tau)&=\int_0^1 \pa_w
\HH\left(\pa_z\psi_0(z,\tau)+s\pa_z\varphi(z,\tau),z,\tau\right)ds-1-
Q_1(\tau)\frac{\hmu}{z^{\ell-2r}},\label{def:matching:M}
\end{align}
where $\HH$ and $\HH_0$ are the Hamiltonians defined in
\eqref{Hamiltonia:varInner} and \eqref{Hamiltonia:EqInner}
respectively. Even if $M$ depends on $\varphi$, since its existence
is already known, $M$ can be seen as a function depending on the
variables $z$ and $\tau$, and then equation
\eqref{eq:PDEMatchin:enZ} can be seen as a linear equation. This
fact simplifies considerably the obtention of the estimates for
$\varphi$.

Let us point out that the term $\hmu Q_1(\tau )z^{-(\ell-2r)}$ in
\eqref{eq:matcing:RHS} behaves in a completely different way in the cases
$\ell-2r>0$ and $\ell-2r=0$, since in the first case is small for
$z\in \DD_{\kk,c}^{\inn,+,u}$ and in the second is not. For this
reason, we split the proof of Proposition \ref{prop:matching:HJ} into
these two cases.

Finally in this section, we state the following lemma, which gives
some properties of the functions involved in equation
\eqref{eq:PDEMatchin:enZ}.
\begin{lemma}\label{lemma:matching:cotes}
Let $\kk\geq\kk_5$ and $\C>0$. The functions $L$
and $M$ defined in \eqref{def:matching:L} and \eqref{def:matching:M}
respectively, satisfy the following properties.
\begin{enumerate}
\item $L\in \ZZZ_{2r-\frac{1}{\q},\kk,\C,\sigma}$ and satisfies
\[
\left\| L\right\|_{2r-\frac{1}{\q},\kk,\C,\sigma}\leq
K\eps^{\frac{1}{\q}}.
\]
\item $M\in \ZZZ_{0,\kk,\C,\sigma}$ and satisfies
\[
\left\| M\right\|_{0,\kk,\C,\sigma}\leq \frac{K}{\kk^{\ell-2r+1}}.
\]
\end{enumerate}
\end{lemma}
\begin{proof}
We prove the lemma in the polynomial case. The trigonometric case
can be done analogously taking into account Remark \ref{remark:DefQIntrinseca}.

First we bound $L$. Using the definitions of $\HH$, $\ol H$, $\wh H$
and $\HH_0$ in \eqref{Hamiltonia:varInner}, \eqref{def:Hbarra},
\eqref{def:HamPeriodicaShiftada} and \eqref{Hamiltonia:EqInner}
respectively, we split it as $L=L_1+L_2+L_3+L_4$ with
\[
\begin{split}
L_1(z,\tau)=&\frac{1}{2}\left(\frac{C_+^2}{\eps^{2r}p_0^2(ia+\eps
z)}-z^{2r}\right)\left(\pa_z\psi_0\right)^2\\
L_2(z,\tau)=&\frac{\eps^{2r}}{C_+^2}\left(V(q_0\left(ia+\eps
z)+x_p(\tau)\right)-V\left(x_p(\tau)\right)-V'\left(x_p(\tau)\right)q_0(ia+\eps
z)\right)\\
&-\frac{1}{2z^{2r}}\\
L_3(z,\tau)=&\frac{\hmu\eps^\ell}{C_+^2}\wh H_1^1\left(q_0(ia+\eps
z),C_+^2\eps^{-2r}\pa_z\psi_0(z,\tau),\tau\right)\\
&-\frac{\hmu}{z^\ell}\sum_{(r-1)k+rl=\ell}a_{kl}(\tau)\frac{C_+^{k+l-2}}{(1-r)^k
}\left(z^{2r}\pa_z\psi_0(z,\tau)\right)^l\\
L_4(z,\tau)=&\frac{\hmu\eps^{\ell+1}}{C_+^2}\wh
H_1^2\left(q_0(ia+\eps z),C_+^2\eps^{-2r}\pa_z\psi_0(z,\tau)\right).
\end{split}
\]
Taking into account the properties of $p_0(u)$ in
\eqref{eq:SepartriuAlPol} and Theorem \ref{th:InnerImma}, one can
see that
\[
\left\| L_1\right\|_{2r-\frac{1}{\q},\kk,\C,\sigma}\leq
K\eps^{\frac{1}{\q}}.
\]
For $L_2$ one has to take into account that $V(q_0(u))=-p^2_0(u)/2$,
use \eqref{eq:PotencialInfinit} and the bound of $x_p(\tau)$ in
Proposition \ref{prop:periodica}. Then, one obtains
\[
\left\| L\right\|_{2r-\frac{1}{\q},\kk,\C,\sigma}\leq
K\eps^{\frac{1}{\q}}.
\]
To bound the third term, using the definition of $\wh H_1^1$ in
\eqref{def:HamPertorbat:H1} and also \eqref{eq:SepartriuAlPol}, one
can rewrite it as
\[
\begin{split}
L_3(z,\tau)=&\hmu\eps^{\ell-(r-1)k-rl}\sum_{2\leq k+l\leq
N}a_{kl}(\tau)\frac{C_+^{k+l-2}}{(1-r)^k}\left(\frac{1}{z^{r-1}}+
\OO\left(\frac{\eps^\frac{1}{\q}}{z^{r-1-\frac{1}{\q}}}
\right)\right)^k\left(z^r\pa_z\psi\right)^l\\
&-\frac{\hmu}{z^\ell}\sum_{(r-1)k+rl=\ell}a_{kl}(\tau)\frac{C_+^{k+l-2}}{(1-r)^k
}\left(z^{2r}\pa_z\psi_0(z,\tau)\right)^l.
\end{split}
\]
Then, it is easy to see that $L_3\in
\ZZZ_{\ell-\frac{1}{\q},\kk,\C,\sigma}\subset\ZZZ_{2r-\frac{1}{\q},\kk,\C,\sigma
}$
and
\[
\left\| L_3\right\|_{2r-\frac{1}{\q},\kk,\C,\sigma}\leq K\left\|
L_3\right\|_{\ell-\frac{1}{\q},\kk,\C,\sigma}\leq
K\eps^{\frac{1}{\q}}.
\]
The bound of $L_4$ is straightforward.

For the bound of $M$, we split it as $M=M_1+M_2+M_3$ with
\[
\begin{split}
M_1(z,\tau)&=\pa_w\HH_0\left(\pa_z\psi_0,z,\tau\right)-
Q_1(\tau)\frac{\hmu}{z^{\ell-2r}}-1\\
M_2(z,\tau)&=\int_0^1\left(\pa_w\HH_0\left(\pa_z\psi_0+s\pa_z\varphi,z,
\tau\right)-\pa_w\HH_0\left(\pa_z\psi_0,z,\tau\right)\right)ds\\
M_3(z,\tau)&=\int_0^1\left(\pa_w\HH\left(\pa_z\psi_0+s\pa_z\varphi,z,
\tau\right)-\pa_w\HH_0\left(\pa_z\psi_0+s\pa_z\varphi,z,\tau\right)\right)ds
\end{split}
\]
and we bound each term.

Taking into account the definitions of $\HH_0$ and $Q_j$ in
\eqref{Hamiltonia:EqInner} and \eqref{def:FunctionsQ} respectively,
and the properties of $\psi_0$ given by Theorem \ref{th:InnerImma},
one can see that $M_1\in \ZZZ_{\ell-2r+1,\kk,\C,\sigma}$ and
$\|M_1\|_{\ell-2r+1,\kk,\C,\sigma}\leq K$, which implies
\[
\|M_1\|_{0,\kk,\C,\sigma}\leq\frac{K}{\kk^{\ell-2r+1}}.
\]
For the second term, let us recall that, using the definition of
$T_0$ in \eqref{def:T00}, by Theorems \ref{th:Extensio:Trig} (see
also Section \ref{subsec:ExtensioFinal}) and \ref{th:InnerImma}, we
have an \emph{a priori} estimate for $\pa_z\varphi$,
\[
\|\pa_z\varphi\|_{\ell+1,\kk,\C,\sigma}\leq K.
\]
Then, it is enough to apply again the mean value theorem and the
bounds of $\psi_0$ in Theorem \ref{th:InnerImma} to obtain
\[
\|M_2\|_{0,\kk,\C,\sigma}\leq\frac{K}{\kk^{\ell-2r+1}}.
\]
For $M_3$, it is enough to proceed as in the bound for $L$ to obtain
\[
\|M_3\|_{0,\kk,\C,\sigma}\leq K\eps^{\frac{\ga}{\q}}.
\]
\end{proof}

\paragraph{The initial condition in the transition
domains}\label{subsec:matching:initialcondition}

To obtain better estimates of $\pa_z\varphi$ we  use an integral equation. To
obtain
it from \eqref{eq:PDEMatchin:enZ} we need initial conditions.
Therefore, we take constants $\C_1<\C_0'<\C_0$ an we look for them
in the transition domains $\tri_{\C_0, \C_0'}^{+,u}\times
\TT_\sigma$, defined in \eqref{def:DominisTransInner} (see also
Figure \ref{fig:InnerVarInner}). In this domain, the next lemma
gives sharp estimates for the function $\pa_z\varphi$. We abuse
notation and we use the norms defined in Section
\ref{sec:Transicio}, even if here the suprema are taken in
$\tri_{\C_0,\C_0'}^{+,u}$.
\begin{lemma}\label{lemma:matching:CondInicial:general}
Let $\gamma\in (0,\gamma_2)$, where $\ga_2$ is defined
in \eqref{def:gamma2}, and $\eps_0>0$ small enough. Then, for
$\eps\in(0,\eps_0)$, the function $\pa_z\varphi$ restricted to
$\tri_{\C_0,\C_0'}^{+,u}$ satisfies
\[
\left\|\pa_z\varphi\right\|_{0,\sigma}\leq
K\eps^{2r(1-\gamma)+\frac{\gamma}{\q}}.
\]
\end{lemma}

\begin{proof}
Considering the functions $T=T_0+T_1$,  obtained in Proposition
\ref{prop:extensio:trig} (see also Section
\ref{subsec:ExtensioFinal}), and
\[
\psi_0(z,\tau)=-\frac{1}{(2r-1)z^{2r-1}}+\hmu\ol\psi_0(z,\tau)+K,
\] obtained in Theorem \ref{th:InnerImma}, and recalling
that $\pa_u T_0(u)=p_0^2(u)$, we split $\pa_z\varphi$ as
\[
\begin{split}
\pa_z\varphi(z,\tau)=&\pa_z\psi(z,\tau)-\pa_z\psi_0(z,\tau)\\
=&\eps^{2r}C_+^2\bigg(\pa_uT(\eps z+ia,\tau)-\pa_uT_0(\eps
z+ia)\bigg)\\
&+\left(\eps^{2r}C_+^2p_0^2(\eps
z+ia)-\frac{1}{z^{2r}}\right)-\hmu\pa_z\ol\psi_0(z,\tau).
\end{split}
\]
We bound each term. For the first term it is enough to apply the
result obtained in Proposition \ref{prop:extensio:trig} to obtain
\[
\left\|\eps^{2r}C_+^2\Big(\pa_uT(\eps z+ia,\tau)-\pa_uT_0(\eps
z+ia)\Big)\right\|_{0,\sigma}\leq K\eps^{(1-\ga)(\ell+1)}.
\]
Then, since $\ga\in (0,\ga_2)$, $(\ell+1)(1-\gamma)\geq
2r(1-\gamma)+\frac{\gamma}{\q}$, we obtain the desired bound. For
the second term we use \eqref{eq:SepartriuAlPol}. Finally, the bound
of the third term is a direct consequence of Proposition 4.8 of
\cite{Baldoma06}. This proposition states the same results of
Theorem \ref{th:InnerImma} but bounds $\ol\psi_0(z,\tau)$ using
Fourier norms instead of using classical supremum norm.
\end{proof}

\paragraph{The fixed point equation for
$\ell-2r>0$}\label{subsec:matching:fixedPoint:lmajor}

In this section we prove Proposition \ref{prop:matching:HJ} under
the hypothesis $\ell-2r>0$. Let us define $\ppi=\pa_z\varphi$,
which, using \eqref{eq:PDEMatchin:enZ}, is solution of
\begin{equation}\label{eq:PDEmatching:derivada}
\left(\LL\phi\right)(z,\tau)=\pa_z\left[\WW(\phi(z,\tau),z,\tau)\right],
\end{equation}
where $\LL=\pa_\tau+\pa_z$ and $\WW$ is the operator defined in
\eqref{eq:matcing:RHS}. We use this equation to obtain bounds for
$\ppi$.

To invert the operator $\LL=\pa_\tau+\pa_z$, we consider the
operator $\ol\GG$ defined in
\eqref{def:OperadorIntegralBarra:Matching}. Since the operator
$\ol\GG$ is defined acting on the Fourier harmonics, we impose a
different initial condition for each one. Recall that for the negative
harmonics we integrate from $z_1\in \DD_{\kk_5',\C_0}^{u,+}$ and for
the positive and zero harmonics from $z_2\in
\DD_{\kk_5',\C_0}^{u,+}$ (see Figure \ref{fig:InnerVarInner}) for a
fixed $\kk_5'>\kk_5$. Then, we define the function
\begin{equation}\label{def:matching:CondIni}
W_0(z,\tau)=\sum_{k<0}\pa_z\varphi^{[k]}(z_1)e^{-ik(z-z_1)}e^{ik\tau}+
\sum_{k\geq 0}\pa_z\varphi^{[k]}(z_2)e^{-ik(z-z_2)}e^{ik\tau},
\end{equation}
where $\pa_z\varphi$ is the function bounded in Lemma
\ref{lemma:matching:CondInicial:general}. The next lemma, whose
proof is straightforward, gives some properties of this function.
\begin{lemma}\label{lemma:matching:CondIni:lmajor}
The function $W_0$ defined in \eqref{def:matching:CondIni}
satisfies:
\begin{enumerate}
\item $\LL W_0=0$, where $\LL=\pa_\tau+\pa_z$.
\item $W_0\in \ZZZ_{2r-\frac{1}{\q},\sigma}$ and
\[
\left\|W_0\right\|_{2r-\frac{1}{\q},\sigma}\leq
K\eps^{\frac{1}{\q}}.
\]
\end{enumerate}
\end{lemma}
Then, the function  $\phi$ is a  solution  of the integral equation
\[
\phi=W_0+\ol\GG\circ \WW(\phi).
\]
We use a fixed point argument to obtain good estimates of $\phi$. We
study $\phi\in\ZZZ_{2r-\frac{1}{\q},\sigma}$ as a fixed point of the
operator
\begin{equation}\label{def:matching:FuncionalPtFix}
\ol \WW=W_0+\ol\GG\circ \WW.
\end{equation}
\begin{lemma}\label{lemma:matching:PuntFix}
Let $\gamma\in (0,\gamma_2)$, $\eps_0$ small enough and
$\kk_5'>\kk_5$ big enough. Then, for $\eps\in (0,\eps_0)$, the
operator $\ol\WW$ is contractive from
$\ZZZ_{2r-\frac{1}{\q},\sigma}$ to itself.

Then, there exists a constant $b_{10}>0$ such that $\ppi$, the unique
fixed point of $\ol\WW$, satisfies
\[
\| \ppi\|_{2r-\frac{1}{\q},\sigma} \leq b_{10}\eps^{\frac{1}{\q}}.
\]
\end{lemma}
\begin{proof}
$\ol \WW$ sends $\ZZZ_{2r-\frac{1}{\q},\sigma}$ to itself. To see
that $\ol\WW$ is contractive from $\ZZZ_{2r-\frac{1}{\q},\sigma}$ to
itself, let us consider $\phi_1,\phi_2\in
\ZZZ_{2r-\frac{1}{\q},\sigma}$. Then, applying Lemmas
\ref{lemma:matching:operador}  and \ref{lemma:matching:cotes} and
the definition of $\WW$ in \eqref{eq:matcing:RHS}, and increasing
$\kk_5'>0$ if necessary,
\[
\begin{split}
\left\|\ol\WW(\phi_2)-\ol\WW(\phi_1)\right\|_{2r-\frac{1}{\q},\sigma}&\leq
K\left\|\WW(\phi_2)-\WW(\phi_1)\right\|_{2r-\frac{1}{\q},\sigma}\\
&\leq
K\left\|\left(Q_1(\tau)\frac{\hmu}{z^{\ell-2r}}+M(z,\tau)\right)
\cdot
(\phi_2-\phi_1)\right\|_{2r-\frac{1}{\q},\sigma}\\
 &\leq
\frac{K}{(\kk_5')^{\ell-2r}}\left\|\phi_2-\phi_1\right\|_{2r-\frac{1}{\q},\sigma
}\\
&\leq
\frac{1}{2}\left\|\phi_2-\phi_1\right\|_{2r-\frac{1}{\q},\sigma}.
\end{split}
\]
Then $\ol\WW$ is contractive from $ \ZZZ_{2r-\frac{1}{\q},\sigma}$
to itself, and then it has a unique fixed point $\ppi$.

To obtain a bound for $\ppi$, it is enough to take into account that
$\|\ppi\|_{2r-\frac{1}{\q},\sigma}\leq
2\|\ol\WW(0)\|_{2r-\frac{1}{\q},\sigma}$. By the definition of
$\ol\WW$  in \eqref{def:matching:FuncionalPtFix}, we have that
$\ol\WW(0)=W_0+\ol\GG( L)$. Then, applying Lemmas
\ref{lemma:matching:operador}, \ref{lemma:matching:cotes} and
\ref{lemma:matching:CondInicial:general}, there exists a constant
$b_{10}>0$ such that
\[
\left\|\ol\WW(0)\right\|_{2r-\frac{1}{\q},\sigma}\leq
\left\|W_0\right\|_{2r-\frac{1}{\q},\sigma}+\left\|\ol\GG\left(
L\right)\right\|_{2r-\frac{1}{\q},\sigma}\leq
\frac{b_{10}}{2}\eps^{\frac{1}{\q}}.
\]
Let us point out that since the fixed point of $\ol\WW$ is unique in
$\ZZZ_{2r-\frac{1}{\q},\sigma}$, the obtained function $\ppi$ must
coincide with $\phi=\psi^u-\psi_0^u$, where $\psi^u$ is the function
defined in \eqref{eq:FuncioGeneradoraInner} and $\psi_0^u$ is the
one given in Theorem \ref{th:InnerImma}.
\end{proof}

\paragraph{The fixed point equation for
$\ell-2r=0$}\label{subsec:matching:fixedPoint:ligual}

We devote this section to prove Proposition \ref{prop:matching:HJ}
under the hypothesis $\ell-2r=0$. Now, the term $\hmu
Q_1(\tau)z^{-(\ell-2r)}=\hmu Q_1(\tau)$ in $\WW$ (see
\eqref{eq:matcing:RHS}) is  not small. Then, following
\cite{Baldoma06}, the first step is to perform the change of
variables
\begin{equation}\label{def:CanviInner}
z=x+\hmu F_1(\tau),
\end{equation}
where $F_1$ is the function defined in \eqref{def:FunctionsF}. Then,
we define
\[
\wh \varphi(x,\tau)=\varphi\left(x+\hmu F_1(\tau),\tau\right),
\]
which satisfies equation
\begin{equation}\label{eq:PDEMatchin:enZ:ligual}
\LL\wh\varphi=\wh\WW(\pa_x\wh\varphi,x,\tau),
\end{equation}
with
\begin{equation}\label{eq:matcing:RHS:2}
\wh\WW(w,x,\tau)=L(x+\hmu F_1(\tau),\tau)+M(x+\hmu F_1(\tau),\tau)w.
\end{equation}
We study this equation through a fixed point argument, as we have
done in Section \ref{subsec:matching:fixedPoint:lmajor}. Then, we
define $\wh\ppi=\pa_x\wh\varphi$, which is a solution of
\[
\LL\wh\ppi=\pa_x\left[\wh\WW(\pa_x\wh\ppi,x,\tau)\right].
\]
Let us take $\C_0''\in (\C_0',\C_0)$ and $\kk_5''>\kk_5$. Then, we
look for $\wh\ppi$ defined for $(x,\tau)\in \DD_{\kk_5'',
\C_0''}^{\inn,+,u}\times\TT_\sigma$.

To invert the operator $\LL=\pa_\tau+\pa_x$, we consider the
operator $\ol\GG$ defined in
\eqref{def:OperadorIntegralBarra:Matching} and initial conditions as
we have done in Section \ref{subsec:matching:fixedPoint:lmajor}.
Thus, we define
\begin{equation}\label{def:matching:CondIni:ligual}
\begin{split}
\wh W_0(x,\tau)=&\sum_{k<0}\pa_z\varphi^{[k]}(x_1+\hmu
F_1(\tau))e^{-ik(x-x_1)}e^{ik\tau}\\
&+ \sum_{k\geq 0}\pa_z\varphi^{[k]}(x_2+\hmu
F_1(\tau))e^{-ik(x-x_2)}e^{ik\tau},
\end{split}
\end{equation}
where $x_1$ and $x_2$ are the vertices of $\DD_{\kk_5'',
\C_0''}^{\inn,+,u}$. Since $\C_0''\in (\C_0',\C_0)$, $x_1,x_2\in
\tri_{\C_0,\C_0'}^{+,u}$ and then $\pa_z\varphi$ is already defined
in $x_i+\mu F_1(\tau), i=1,2$ and moreover, we can use the bounds in
Lemma \ref{lemma:matching:CondInicial:general}. Then, it is
straightforward to see that $\wh W_0$ satisfies the same properties
as the function $W_0$ given in Lemma
\ref{lemma:matching:CondIni:lmajor}.

The function  $\wh\ppi$ is a  solution  of the integral equation
\[
\wh\phi=\wh W_0+\ol\GG\circ \wh\WW(\wh\phi).
\]
We study $\wh\ppi\in\ZZZ_{2r-\frac{1}{\q},\sigma}$ as a fixed point
of the operator
\begin{equation}\label{def:matching:FuncionalPtFix:ligual}
\wt \WW=\wh W_0+\ol\GG\circ\wh \WW.
\end{equation}
\begin{lemma}\label{lemma:matching:PuntFix:igual}
Let $\gamma\in (0,\gamma_2)$, $\eps_0>0$ small enough
and $\kk_5''>\kk_5$ big enough. Then,  for $\eps\in (0,\eps_0)$, the
operator $\wt\WW$ is contractive from $\ZZZ_{2r-\frac{1}{\q},
\kk_5'', \C_0'',\sigma}$ to itself.

Then, there exists a constant $b_{10}>0$ such that $\wh\ppi$, the
unique fixed point of $\wt\WW$, satisfies
\[
\| \wh\ppi\|_{2r-\frac{1}{\q},\kk_5'', \C_0'',\sigma} \leq
b_{10}\eps^{\frac{1}{\q}}.
\]
\end{lemma}
\begin{proof}
The proof of this lemma is completely analogous to the proof of
Lemma \ref{lemma:matching:PuntFix}. The only fact that one has to
take into account is that the functions $L(x+\hmu F_1(\tau),\tau)$
and $M(x+\hmu F_1(\tau),\tau)$ satisfy the same properties as
$L(z,\tau)$ and $M(z,\tau)$, which are given in Lemma
\ref{lemma:matching:cotes}.
\end{proof}
To prove Proposition \ref{prop:matching:HJ} for $\ell-2r=0$, it is
enough to undo the change of variables \eqref{def:CanviInner}. Then,
taking $\phi(z,\tau)=\wh\phi(x-\hmu F_1(\tau),\tau)$, we recover
$\pa_z\varphi$ which is defined for $(z,\tau)\in
D_{\kk_6,\C_1}^{\inn,+,u}\times\TT_\sigma$, where $\C_1<\C_0''$ and
$\kk_6>\kk_5''$.

\section{An injective solution of the partial differential equation
$\wt{\mathcal{L}}_{\eps} \xi =0$}\label{sec:CanviFinal}

In this section we prove the existence and provide useful properties of a
solution $\xi_0$ of the
equation $\wt{\mathcal{L}}_{\eps} \xi =0$ (see \eqref{eq:DifferencePDE}) of the
form
\[
\xi_0(u,\tau)=\eps\ii u-\tau+\CCC(u,\tau).
\]
The function $\CCC$ must satisfy
\begin{equation}\label{eq:General:Diff}
\LL_\eps\CCC(u,\tau)=\FF\left(\CCC\right)(u,\tau),
\end{equation}
where $\LL_\eps$ is the operator in \eqref{def:Lde},
\begin{equation}\label{eq:Operador:CanviFinal}
\FF(\CCC)(u,\tau)=-\eps\ii G(u,\tau)-G(u,\tau)\pa_u \CCC(u,\tau)
\end{equation}
and $G$ is the function defined in
\eqref{def:FuncioLAnulador:lmenor} (case $\ell-2r<0$) and
\eqref{def:FuncioLAnulador:lmajor} (case $\ell-2r\geq 0$). We devote
the rest of the section to obtain a solution of this equation in both cases.

\subsection{Banach spaces and technical lemmas}\label{sec:Banach:canvi}

This section is devoted to define
the Banach spaces and to state some technical lemmas which will be
used in Sections \ref{sec:PtFix:canvi:lmenor} and
\ref{sec:PtFix:canvi}.

We start by defining some norms. Given $\nu\geq 0$ and an analytic
function $h:R_{\kk,d}\rightarrow\CC$, where $R_{\kk,d}$ is the
domain defined in \eqref{def:DominiRaro:Interseccio}, we consider
\[
\begin{split}
\|h\|_{\nu,\kk,d}&=\sup_{u\in R_{\kk,d}}\left| \left(u^2+a^2\right)^\nu
h(u)\right|\\
\|h\|_{\ln,\kk,d}&=\sup_{u\in
R_{\kk,d}}\left|\ln\ii\left|u^2+a^2\right|\cdot h(u)\right|.
\end{split}
\]
Moreover for $2\pi$-periodic in $\tau$, analytic functions
$h:R_{\kk,d}\times\TT_\sigma\rightarrow \CC$, we consider the corresponding
Fourier
norms
\[
\begin{split}
\|h\|_{\nu,\kk,d,\sigma}&=\sum_{k\in\ZZ}\left\|h^{[k]}\right\|_{\nu,\kk,d}e^{
|k|\sigma}\\
\|h\|_{\ln,\kk,d,\sigma}&=\sum_{k\in\ZZ}\left\|h^{[k]}\right\|_{\ln,\kk,d}e^{
|k|\sigma}.
\end{split}
\]
We consider, thus, the following function spaces
\begin{equation}\label{def:BanachCanvi}
\begin{split}
\XX_{\nu,\kk,d,\sigma}&=\{
h:R_{\kk,d}\times\TT_\sigma\rightarrow\CC;\,\,\text{real-analytic},
\|h\|_{\nu,\kk,d,\sigma}<\infty\}\\
\XX_{\ln,\kk,d,\sigma}&=\{
h:R_{\kk,d}\times\TT_\sigma\rightarrow\CC;\,\,\text{real-analytic},
\|h\|_{\ln,\kk,d,\sigma}<\infty\},
\end{split}
\end{equation}
which can be checked that are a Banach spaces.

If there is no danger of confusion about the definition domain
$R_{\kk,d}$ we will denote
\[
\begin{array}{ccc}
\|\cdot\|_{\nu,\sigma}=\|\cdot\|_{\nu,\kk,d,\sigma}&\text{ and
}&\XX_{\nu,\sigma}=\XX_{\nu,\kk,d,\sigma}.
\end{array}
\]

In the next lemma, we state some properties of these Banach spaces.
\begin{lemma}\label{lemma:PropietatsNormes:canvi} The following statements hold:
\begin{enumerate}
\item If $\nu_1\geq\nu_2\geq 0$,
$\XX_{\nu_1,\sigma}\subset\XX_{\nu_2,\sigma}$ and moreover if
$h\in\XX_{\nu_1,\sigma}$,
\[
\|h\|_{\nu_2,\sigma}\leq
K(\kk\eps)^{\nu_2-\nu_1}\|h\|_{\nu_1,\sigma}.
\]
\item If $0\leq \nu_1\leq\nu_2$,
$\XX_{\nu_1,\sigma}\subset\XX_{\nu_2,\sigma}$ and moreover if
$h\in\XX_{\nu_1,\sigma}$,
\[
\|h\|_{\nu_2,\sigma}\leq K\|h\|_{\nu_1,\sigma}.
\]
\item If $h\in\XX_{\nu_1,\sigma}$ and $g\in\XX_{\nu_2,\sigma}$, then
$hg\in\XX_{\nu_1+\nu_2,\sigma}$ and
\[
\|hg\|_{\nu_1+\nu_2,\sigma}\leq
\|h\|_{\nu_1,\sigma}\|g\|_{\nu_2,\sigma}.
\]
\item Let $d>d'>0$ be such that $d-d'$ has a positive lower bound independent of
$\eps$, and $h\in
\XX_{\nu,\kk,d,\sigma}$. Then, $\pa_uh\in \XX_{\nu,2\kk,d',\sigma}$
and satisfies
\[
\|\pa_uh\|_{\nu,2\kk,d',\sigma}\leq\frac{K}{\kk\eps}
\|h\|_{\nu,\kk,d,\sigma}.
\]
\end{enumerate}
\end{lemma}

Throughout this section we are going to solve equations of the form
$\LL_\eps h=g$, where $\LL_\eps$ is the operator defined in
\eqref{def:Lde}. To find a right-inverse of this operator in
$R_{\kk,d}$ let us consider $u_1=i(a-\kk\eps)$ and  $u_0$ the left
endpoint of $R_{\kk,d}\cap\RR$. Then, we define the operator
$\GG_\eps$ as
\begin{equation}\label{def:operador:canvi}
\GG_\eps(h)(u,\tau)=\sum_{k\in\ZZ}\GG_\eps(h)^{[k]}(u)e^{ik\tau},
\end{equation}
 where its Fourier coefficients are given by
\begin{align*}
\dps \GG_\eps(h)^{[k]}(u)&=\int_{-u_1}^u
e^{ik\eps\ii(v-u)}h^{\left[k\right]}(v)\,dv&\textrm{ if }k<0\\
\dps \GG_\eps(h)^{[0]}(u)&= \int_{u_0}^{u} h^{\left[0\right]}(v)\,dv&\\
\dps \GG_\eps(h)^{[k]}(u)&= -\int^{u_1}_u
e^{ik\eps\ii(v-u)}h^{\left[k\right]}(v)\,dv&\textrm{ if }k>0,
\end{align*}
where we make the integrals along any path contained in $R_{\kk,d}$.

Let us point that we will apply this operator to functions defined
in $R_{\kk,d}\times\TT_\sigma$ with different values of $\kk$ and
$d$ and then the definition of $\GG_\eps$ depends on the domain.

\begin{lemma}\label{lemma:Canvi:Operador}
The operator $\GG_\eps$ in \eqref{def:operador:canvi} satisfies the
following properties.
\begin{enumerate}
\item If $h\in \XX_{\nu,\sigma}$ for some $\nu\geq 0$, then
$\GG_\eps(h)\in \XX_{\nu,\sigma}$ and
\[
\|\GG_\eps(h)\|_{\nu,\sigma}\leq K\|h\|_{\nu,\sigma}.
\]
Furthermore, if $\langle h\rangle=0$,
\[
\left\|\GG_\eps(h)\right\|_{\nu,\sigma}\leq
K\eps\left\|h\right\|_{\nu,\sigma}.
\]
\item If $h\in\XX_{\nu,\sigma}$ for some $\nu>1$, then $\GG_\eps(h)\in
\XX_{\nu-1,\sigma}$ and
\[
\left\|\GG_\eps(h)\right\|_{\nu-1,\sigma}\leq K\|h\|_{\nu,\sigma}.
\]
\item If $h\in\XX_{\nu,\sigma}$ for some $\nu\in (0,1)$, then $\GG_\eps(h)\in
\XX_{0,\sigma}$ and
\[
\left\|\GG_\eps(h)\right\|_{0,\sigma}\leq K\|h\|_{\nu,\sigma}.
\]
\item If $h\in\XX_{1,\sigma}$, then $\GG_\eps(h)\in
\XX_{\ln,\sigma}$ and
\[
\left\|\GG_\eps(h)\right\|_{\ln,\sigma}\leq K\|h\|_{1,\sigma}.
\]

\item If $h\in\XX_{\nu,\sigma}$ for some $\nu\geq 0$, then $\GG_\eps(\pa_u h)\in
\XX_{\nu,\sigma}$ and
\[
\left\|\GG_\eps(\pa_u h)\right\|_{\nu,\sigma}\leq
K\|h\|_{\nu,\sigma}.
\]
\item If $h\in\XX_{\nu,\sigma}$ for some $\nu\geq 0$, then $\pa_u\GG_\eps(h)\in
\XX_{\nu,\sigma}$ and
\[
\left\|\pa_u\GG_\eps(h)\right\|_{\nu,\sigma}\leq
K\|h\|_{\nu,\sigma}.
\]
\item If $h\in\XX_{\nu,\sigma}$ for some $\nu\geq 0$,
$\LL_\eps\circ\GG_\eps (h)=h$ and
\[
\GG_\eps\circ\LL_\eps(h)(v,\tau)=h(v,\tau)-\sum_{k<0}e^{ik\eps\ii
(-u_1-u)}h^{[k]}(-u_1)-h^{[0]}(u_0)-\sum_{k>0}e^{ik\eps\ii
(u_1-u)}h^{[k]}(u_1).
\]
\end{enumerate}
\end{lemma}
\begin{proof}
The first four statements are straightforward. For the fifth one, one
has to integrate by parts and for the sixth one has to apply
Leibnitz rule.
\end{proof}
\subsection{Case $\ell <2r$: proof of Theorem \ref{th:CanviFinal:lmenor} and
Proposition \ref{coro:Canvi:FirstOrder:lmenor}}\label{sec:PtFix:canvi:lmenor}

\subsubsection{Proof of Theorem \ref{th:CanviFinal:lmenor}}

Theorem \ref{th:CanviFinal:lmenor} is a straightforward consequence
of the following proposition.
\begin{proposition}\label{prop:canvi:lmenor}
Let $d_2>0$ and $\kk_3>0$  be defined in
Theorem \ref{th:ExtensioFinal}, $d_3<d_2$, $\eps_0>0$ small enough
and $\kk_7>\kk_3$ big enough, which might depend on the previous
constants. Then, for $\eps\in(0,\eps_0)$ and any $\kk\geq \kk_7$
such that $\eps\kk<a$, there exists a function $\CCC:
R_{\kk,d_3}\times\TT_\sigma\rightarrow \CC$ that satisfies equation
\eqref{eq:General:Diff}.

Moreover,
\[
\left(\xi_0(u,\tau),\tau\right)=\left(\eps\ii
u-\tau+\CCC(u,\tau),\tau\right)
\]
is injective and there exists a constant $b_{11}>0$ independent of
$\eps$, $\mu$ and $\kk$ such that
\[
\begin{split}
\left\|\CCC \right\|_{0,\sigma}&\leq b_{11}|\mu|\eps^{\eta}\\
\left\|\pa_u\CCC \right\|_{0,\sigma}&\leq
b_{11}\kk\ii|\mu|\eps^{\eta-1}.
\end{split}
\]
\end{proposition}

To prove this proposition, first we split $G$ into several terms.
Recall that, since $\ell-2r<0$, the perturbation $\wh H_1$ in
\eqref{def:ham:ShiftedOP:perturb} is a polynomial of degree one in
$p$. Then, $G$ can be split as $G=G_1+G_2+G_3$ with
\begin{align}
G_1(u,\tau)&=\mu\eps^{\eta}p_0(u)\ii\pa_p \wh
H_1^1\left(q_0(u),p_0(u),\tau\right)\label{def:Canvi:G1:lmenor}\\
G_2(u,\tau)&=\mu\eps^{\eta+1}p_0(u)\ii\pa_p \wh
H_1^2\left(q_0(u),p_0(u),\tau\right)\label{def:Canvi:G2:lmenor}\\
G_3(u,\tau)&=\frac{\pa_u T_1^s(u,\tau)+\pa_u
T_1^u(u,\tau)}{2p^2_0(u)}.\label{def:Canvi:G3:lmenor}
\end{align}
The next lemma gives several properties of these functions.
\begin{lemma}\label{lemma:Canvi:cotes:lmenor}
Let us consider any $\kk>\kk_3$ and $d<d_2$, where $\kk_3$ and $d_2$
are the constants given in Theorem \ref{th:ExtensioFinal}. Then, the
functions $G_1$, $G_2$ and $G_3$ defined in
\eqref{def:Canvi:G1:lmenor}, \eqref{def:Canvi:G2:lmenor} and
\eqref{def:Canvi:G3:lmenor} respectively, have the following
properties.
\begin{enumerate}
\item $G_1\in \XX_{0,\sigma}$ and it satisfies $\langle
G_1\rangle=0$ and
\[
\begin{split}
\left\|G_1\right\|_{0,\sigma}&\leq K|\mu|\eps^\eta\\
\left\|\pa_vG_1\right\|_{\max\{\ell-2r+1,0\},\sigma}&\leq
K|\mu|\eps^\eta.
\end{split}
\]
\item $G_2\in \XX_{0,\sigma}$ and it satisfies
\[
\left\|G_2\right\|_{0,\sigma}\leq K|\mu|\eps^{\eta+1}.
\]
\item $G_3\in \XX_{\max\{\ell-2r+1,0\},\sigma}$ and it satisfies
\[
\left\|G_3\right\|_{\max\{\ell-2r+1,0\},\sigma}\leq
K|\mu|\eps^{\eta+1}.
\]
\end{enumerate}
\end{lemma}
\begin{proof}
The proof of the statements about $G_1$ and $G_2$ are
straightforward, using the bounds obtained
in Corollary \ref{coro:ShiftPeriodica} for $G_2$. For $G_3$, one has to take
into account the bounds for $T_1^u$ obtained in Proposition
\ref{prop:extensio:trig} and the analogous bounds that $T_1^s$
satisfies.
\end{proof}

To prove Proposition \ref{prop:canvi:lmenor}, we first perform a
change of variables which reduces the linear terms of equation
\eqref{eq:General:Diff}.
\begin{lemma}\label{lemma:canvi:canvi:lmenor}
Let $\kk_7>\kk_3'>\kk_3$ and $d_3<d_2'<d_2$. Then, for
$\eps>0$ small enough, there exists a function $g$ which is solution
of the equation
\[
\LL_\eps g(v,\tau)=G_1(v,\tau),
\]
where $G_1$ is the function defined in \eqref{def:Canvi:G1:lmenor}.
Moreover, it satisfies that
\[
\|g\|_{0,\kk_3',d_2',\sigma}\leq K|\mu|\eps^{\eta+1},\qquad
\|\pa_vg\|_{\max\{\ell-2r+1,0\},\kk_3',d_2',\sigma}\leq
K|\mu|\eps^{\eta+1}
\]
and that $u=v+g(v,\tau)\in R_{\kk_3,d_2}$ for $(v,\tau)\in
R_{\kk_3',d_2'}\times\TT_\sigma$.

Moreover, the change $(u,\tau)=(v+g(v,\tau),\tau)$ is invertible
and its inverse is of the form $(v,\tau)=(u+h(u,\tau),\tau)$. The
function $h$ is defined in the domain
$R_{\kk_7,d_3}\times\TT_\sigma$ and it satisfies
\[
\|h\|_{0,\kk_7,d_3,\sigma}\leq K|\mu|\eps^{\eta+1}
\]
and that $u+h(u,\tau)\in R_{\kk_3',d_2'}$ for $(u,\tau)\in
R_{\kk_7,d_3}\times\TT_\sigma$.
\end{lemma}
Furthermore, we need precise bounds of both functions $g$ and $h$ restricted to
the inner domain
$D_{\kk_7,c}^{\inn,+,u}$ defined in \eqref{def:DominisInnerEnu}.
These bounds are given in next corollary, whose proof is straightforward. We
abuse notation and
we use the norms defined in Section \ref{sec:Banach:canvi} for functions
restricted to the inner domain.
\begin{corollary}\label{coro:Canvi:CotaCanvigInner}
Let $\C_1>0$ be the constant defined in Corollary
\ref{coro:Extensio:CotaCanvigInner}
and let also $\C_2>\C_1$. Then, the functions $g$ ad $h$ obtained in Lemma
\ref{lemma:canvi:canvi:lmenor} restricted to the inner domains
$D_{\kk_3',\C_1}^{\inn,+,u}$  and
$D_{\kk_7,\C_2}^{\inn,+,u}$ respectively satisfy the following bounds
\[
\dps\|g\|_{0,\kk_3',d_2',\sigma}\leq K|\mu|\eps^{\eta+1+(2r-\ell)\gamma}\qquad
\text{ and }\qquad\dps\|h\|_{0,\kk_7,d_3,\sigma}\leq
K|\mu|\eps^{\eta+1+(2r-\ell)\gamma}.
\]
\end{corollary}
\begin{proof}[Proof of Lemma \ref{lemma:canvi:canvi:lmenor}]
From Lemma \ref{lemma:Canvi:cotes:lmenor}, $\langle G_1\rangle=0$
and then we can define a function $\ol G_1$ such that
\begin{equation}\label{def:canvi:G1:barra:lmenor}
\pa_\tau \ol G_1=G_1\,\,\text{ and }\,\,\langle \ol G_1\rangle=0,
\end{equation}
which satisfies
\begin{equation}\label{eq:canvi:Cota:G1barra:lmenor}
\begin{split}
\left\|\ol G_1\right\|_{0,\kk_3',d_2',\sigma}&\leq K|\mu|\eps^\eta\\
\left \|\pa_v\ol
G_1\right\|_{\max\{\ell-2r+1,0\},\kk_3',d_2',\sigma}&\leq
K|\mu|\eps^\eta.
\end{split}
\end{equation}
Then, we can define $g$ as
\begin{equation}\label{def:canvi:canvi:lmenor}
g(v,\tau)=\eps\ol G_1(v,\tau)-\eps\GG_\eps\left(\pa_v\ol
G_1\right)(v,\tau),
\end{equation}
where $\GG_\eps$ is the operator defined in
\eqref{def:operador:canvi} adapted to the domain
$R_{\kk_3',d_2'}\times\TT_\sigma$.

Finally, applying Lemmas
\ref{lemma:Canvi:cotes:lmenor} and \ref{lemma:Canvi:Operador}, one
obtains the bounds for $g$ and $\pa_v g$. The other statements are
straightforward.
\end{proof}

We perform the change of variables $u=v+g(v,\tau)$ given in Lemma
\ref{lemma:canvi:canvi:lmenor} to equation
\eqref{eq:Operador:CanviFinal} and we obtain
\begin{equation}\label{eq:Canvi:PDE:Cbarret:lmenor}
\LL_\eps\wh\CCC=\wh\FF\left(\wh\CCC\right),
\end{equation}
where $\wh\CCC$ is the unknown
\begin{equation}\label{eq:CtoCbarret:lmenor}
\wh\CCC(v,\tau)=\CCC(v+g(v,\tau),\tau)
\end{equation}
and
\begin{equation}\label{def:canvi:operador:RHS:lmenor}
\wh\FF(h)=M(v,\tau)+N(v,\tau)\pa_v h
\end{equation}
with
\begin{align}
M(v,\tau)&=-\eps\ii G\left(v+g(v,\tau),\tau\right)\label{def:canvi:M:lmenor}\\
N(v,\tau)&=-\frac{G\left(v+g(v,\tau),\tau\right)-G_1(v,\tau)}{1+\pa_v
g(v,\tau)}.\label{def:canvi:N:lmenor}
\end{align}
Next lemma gives some properties of these functions
\begin{lemma}\label{lemma:canvi:lmenor:cotesMN}
The functions $M$ and $N$ defined in \eqref{def:canvi:M:lmenor} and
\eqref{def:canvi:N:lmenor} satisfy the following properties.
\begin{itemize}
\item $\GG_\eps(M)\in \XX_{0,\kk_3',d_2',\sigma}$ and it satisfies
\[
\left\|\GG_\eps(M)\right\|_{0,\kk_3',d_2',\sigma}\leq
K|\mu|\eps^\eta.
\]
\item $\langle M\rangle \in \XX_{\max\{\ell-2r+1,0\},\kk_3',d_2',\sigma}$ and it
satisfies
\[
\left\|\langle
M\rangle\right\|_{\max\{\ell-2r+1,0\},\kk_3',d_2',\sigma}\leq
K|\mu|\eps^{\eta}.
\]
\item $\pa_vM\in \XX_{\max\{\ell-2r+1,0\},\kk_3',d_2',\sigma}$ and it satisfies
\[
\left\|\pa_vM\right\|_{\max\{\ell-2r+1,0\},\kk_3',d_2',\sigma}\leq
K|\mu|\eps^{\eta-1}.
\]
\item The function $M$ restricted to $\left( D_{\kk,\C_1}^{\inn,+,u}\cap
D_{\kk,\C_1}^{\inn,+,s}\right)\times\TT_\sigma$ satisfies
\[
\begin{split}
\left\|M\right\|_{0,\kk_3',d_2',\sigma}&\leq
K|\mu|\eps^{\eta+\nu-1}\\
\left\|\langle M\rangle
\right\|_{\max\{\ell-2r+1,0\},\kk_3',d_2',\sigma}&\leq
K|\mu|\eps^{\eta},
\end{split}
\]
where
\begin{equation}\label{def:canvi:lmenor:exponentnu}
\nu=\min\{1-\max\{\ell-2r+1,0\}, (2r-\ell)\gamma\}.
\end{equation}
\item $N\in \XX_{\max\{\ell-2r+1,0\},\kk_3',d_2',\sigma}$ and it satisfies
\[
\begin{split}
\|N\|_{\max\{\ell-2r+1,0\},\kk_3',d_2',\sigma}&\leq K|\mu|\eps^{\eta+1}\\
\left\|\pa_v N\right\|_{\max\{\ell-2r+1,0\},\kk_3',d_2',\sigma}&\leq
K\frac{|\mu|\eps^\eta}{\kk_3'}.
\end{split}
\]
\end{itemize}
\end{lemma}
\begin{proof}
We split $M$ as
$M=M_1+M_2$ with
\[
\begin{split}
M_1(v,\tau)&=-\eps\ii G_1(v,\tau)\\
M_2(v,\tau)&=-\eps\ii\left(G_1\left(v+g(v,\tau),\tau\right)-G_1(v,
\tau)+G_2\left(v+g(v,\tau),
\tau\right)+G_3\left(v+g(v,\tau),\tau\right)\right).
\end{split}
\]
Then, for the first statement it is enough to use the properties of the
functions $G_1$, $G_2$ and $G_3$
given by Lemma \ref{lemma:Canvi:cotes:lmenor} and apply also Lemmas
\ref{lemma:Canvi:Operador}, \ref{lemma:PropietatsNormes:canvi} and
\ref{lemma:canvi:canvi:lmenor}. For the second and the third one has to apply
again Lemmas
\ref{lemma:Canvi:cotes:lmenor}, \ref{lemma:PropietatsNormes:canvi} and
\ref{lemma:canvi:canvi:lmenor}, taking also into account for the second that
$\langle M_1\rangle=0$.
Besides, these lemmas, for the fourth statement, one has to consider also the
bound of the change $g$ in the
inner domain, which is given in Corollary \ref{coro:Canvi:CotaCanvigInner}. For
the last statement, it is
enough to apply again Lemmas  \ref{lemma:Canvi:cotes:lmenor},
\ref{lemma:PropietatsNormes:canvi} and
\ref{lemma:canvi:canvi:lmenor}.
\end{proof}

With the bounds obtained in Lemma \ref{lemma:canvi:lmenor:cotesMN},
we can look for a solution of equation
\eqref{eq:Canvi:PDE:Cbarret:lmenor}  through a fixed point argument.
For that purpose, we define the operator
\begin{equation}\label{def:canvi:operador:lmenor}
\wt \FF=\GG_\eps\circ\wh\FF,
\end{equation}
where $\GG_\eps$ and $\wh\FF$ are the operators defined in
\eqref{def:operador:canvi} and \eqref{def:canvi:operador:RHS:lmenor}
respectively. For convenience, we rewrite $\wh\FF$ as
\begin{equation}\label{def:Operador:canvi:v2:lmenor}
\wh\FF(h)(u,\tau)=M(u,\tau)+\pa_v
\left(N(v,\tau)h(v,\tau)\right)-\pa_v N(v,\tau)h(v,\tau).
\end{equation}

\begin{lemma}\label{lemma:canvi:lmenor:PtFix}
Let $\eps_0>0$ be small enough and $\kk_3'>\kk_3$ big
enough. Then, the operator $\wt\FF$ defined in
\eqref{def:canvi:operador:lmenor} is contractive from
$\XX_{0,\kk_3',d_2',\sigma}$ to itself.

Thus, it has a unique fixed point, which moreover satisfies
\[
\begin{split}
\left\|\wh\CCC\right\|_{0,\kk_3',d_2',\sigma}&\leq K|\mu|\eps^\eta\\
\left\|\pa_v\wh\CCC\right\|_{0,\kk_3',d_2',\sigma}&\leq
K\frac{|\mu|\eps^{\eta-1}}{\kk_3'}.
\end{split}
\]
\end{lemma}
\begin{proof}
To see that $\wt\FF$ is contractive, let $h_1,h_2\in
\XX_{0,\kk_3',d_2',\sigma}$. Then, recalling the definition of
$\wt\FF$ and $\wh\FF$ in \eqref{def:canvi:operador:lmenor} and
\eqref{def:Operador:canvi:v2:lmenor} respectively and applying
Lemmas \ref{lemma:Canvi:Operador},
\ref{lemma:PropietatsNormes:canvi} and
\ref{lemma:canvi:lmenor:cotesMN},
\[
\begin{split}
\left\|\wt\FF(h_2)-\wt\FF(h_1)\right\|_{0,\kk_3',d_2',\sigma}&\leq
\left\|\GG_\eps\pa_v\left(N\cdot(h_2-h_1)\right)\right\|_{0,\kk_3',d_2',\sigma}
+\left\|\GG_\eps\left(\pa_vN\cdot(h_2-h_1)\right)\right\|_{0,\kk_3',d_2',\sigma}
\\
&\leq K
\left\|N\right\|_{0,\kk_3',d_2',\sigma}\left\|h_2-h_1\right\|_{0,\kk_3',d_2',
\sigma}+
K\left\|\pa_v
N\right\|_{\max\{\ell-2r+1,0\},\kk_3',d_2',\sigma}\left\|h_2-h_1\right\|_{0,
\kk_6',d_2',\sigma}\\
&\leq\frac{K|\mu|\eps^\eta}{\kk_3'}\left\|h_2-h_1\right\|_{0,\kk_3',d_2',\sigma}
.
\end{split}
\]
Then, increasing $\kk_3'$ if necessary, $\wt\FF$ is contractive from
$\XX_{0,\kk_3',d_2',\sigma}$ to itself and then it has a unique fixed point.

To obtain a bound for the fixed point $\wh\CCC$, it is enough to
recall that
\[
\left\|\wh\CCC\right\|_{0,\kk_3',d_2',\sigma}\leq
2\left\|\wt\FF(0)\right\|_{0,\kk_3',d_2',\sigma}.
\]
By the definition of $\wt\FF$ in \eqref{def:canvi:operador:lmenor},
$\wt\FF(0)=\GG_\eps(M)$. Then, applying Lemma
\ref{lemma:canvi:lmenor:cotesMN}, we obtain the bound for $\wh\CCC$.
For the bound of $\pa_v\wh\CCC$ it is enough to reduce slightly the
domain and apply the fourth statement of Lemma
\ref{lemma:PropietatsNormes:canvi}.
\end{proof}
\begin{proof}[Proof of Proposition \ref{prop:canvi:lmenor}]
To recover $\CCC$ from $\wh\CCC$ it is enough to consider the change
of variables $v=u+h(u,\tau)$ obtained in Lemma
\ref{lemma:canvi:canvi:lmenor}, which is defined for $(u,\tau)\in
R_{\kk_7,d_3}\times\TT_\sigma$ with $\kk_7>\kk_3'$ and $d_3<d_2'$.
Applying this change, one obtains $\CCC$ which satisfies the bounds
of $\CCC$ and $\pa_u\CCC$ stated in Proposition
\ref{prop:canvi:lmenor}. To check that $(\xi_0(u,\tau),\tau)$ is
injective, it is enough to see that for $(u_1,\tau),(u_2,\tau)\in
R_{\kk_7,d_3}\times\TT_\sigma$,
\[
\eps\ii u_2-\tau+\CCC(u_2,\tau)=\eps\ii u_1-\tau+\CCC(u_1,\tau)
\]
implies $u_2=u_1$. To prove this fact, it is enough to take into
account the just obtained bound of $\pa_u\CCC$, which gives
\[
\begin{split}
\left|u_2-u_1\right|&=\eps
\left|\CCC(u_2,\tau)-\CCC(u_1,\tau)\right|\\
&\leq \frac{K|\mu|\eps^\eta}{\kk_7}|u_2-u_1|.
\end{split}
\]
Then, increasing $\kk_7$ if necessary, one can see that  $u_2=u_1$.
\end{proof}

\subsubsection{Proof of Proposition
\ref{coro:Canvi:FirstOrder:lmenor}}\label{Proof:Existence:C:lmenor}
To prove Proposition \ref{coro:Canvi:FirstOrder:lmenor} it is enough
to study the first asymptotic terms of the function $\wh\CCC$
obtained in Lemma \ref{lemma:canvi:canvi:lmenor}. For that purpose,
we define
\begin{equation}\label{def:canvi:Mtilde:lmenor}
\wt M(v,\tau)=M(v,\tau)-\langle M\rangle(v)
\end{equation}
and we split $\wh\CCC$ as $\wh\CCC=E_1+E_2+E_3$ with
\begin{align}
E_1(v)&=\GG_\eps\left(\langle
M\rangle\right)(v)\label{def:canvi:lmenor:E1}\\
E_2(v,\tau)&=\GG_\eps\left(\wt
M\right)(v,\tau)\label{def:canvi:lmenor:E2}\\
E_3(v,\tau)&=\wt \FF\left(\wh \CCC\right)-\wt
\FF\left(0\right)\label{def:canvi:lmenor:E3}.
\end{align}
Let us point out that the sum of the first two terms corresponds to
$\wt\FF(0)$. We study each term separately. We abuse notation and we
use the same norms as in the previous section but now for functions
defined in  $\left( D_{\kk,\C_1}^{\inn,+,u}\cap
D_{\kk,\C_1}^{\inn,+,s}\right)\times\TT_\sigma$.

For $E_1$, using the definition of $\GG_\eps$ in
\eqref{def:operador:canvi}, one has that
\[
E_1(v,\tau)=\int_{v_0}^v \langle M\rangle (w)\,dw
\]
and then, if we consider $v_1=i(a-\kk_3'\eps)$ the upper vertex of the
domain $R_{\kk_3',d_3}$ (see Figure \ref{fig:BoomInter}), we can
define
\begin{equation}\label{def:C:ViaIntegrals}
C(\mu,\eps)=\int_{v_0}^{v_1} \langle M\rangle (w)\,dw,
\end{equation}
which by Lemmas \ref{lemma:Canvi:Operador} and \ref{lemma:canvi:lmenor:cotesMN}
satisfies
\[
\left\| C(\mu,\eps)\right\|_{0,\sigma}\leq K|\mu|\eps^\eta.
\]
Then
\[
\left\| E_1-C(\mu,\eps)\right\|_{0,\sigma}\leq
K|\mu|\eps^{\eta+(2r-\ell)\ga}.
\]
To bound $E_2$ defined in \eqref{def:canvi:lmenor:E2}, we first
recall that $\langle \wt M\rangle=0$. Then we can define a
function $\ol M$ such that
\[
\pa_\tau \ol M=\wt M\qquad\text{ and }\qquad \langle \ol M\rangle=0,
\]
which satisfies that for $(v,\tau)\in\left(
D_{\kk,\C_1}^{\inn,+,u}\cap
D_{\kk,\C_1}^{\inn,+,s}\right)\times\TT_\sigma$,
\[
\left\|\ol M\right\|_{0,\sigma}\leq K|\mu|\eps^{\eta+\nu-1},
\]
where $\nu$ is the constant defined in
\eqref{def:canvi:lmenor:exponentnu}. Then, we can write $E_2$ as
\[
E_2=\eps\GG_\eps\circ\LL_\eps (\ol M)-\eps\GG_\eps \left(\pa_v\ol
M\right)
\]
and therefore, by Lemma \ref{lemma:Canvi:Operador},
\[
\left\| E_2\right\|_{0,\sigma}\leq  K|\mu|\eps^{\eta+\nu}.
\]
For $E_3$ in \eqref{def:canvi:lmenor:E3}, it is enough to consider
the bound of the Lipschitz constant of the operator $\wt\FF$ given
in the proof of Lemma \ref{lemma:canvi:lmenor:PtFix}, which gives
\[
\|E_3\|_{0,\sigma}\leq K\frac{|\mu|\eps^{2\eta}}{\kk_3'}.
\]
Thus, we have that
\[
\left\| \wh\CCC-C(\mu,\eps)\right\|_{0,\sigma}\leq
K\frac{|\mu|\eps^{\eta}}{\kk_3'}.
\]
To finish the proof of Proposition
\ref{coro:Canvi:FirstOrder:lmenor}, it is enough to consider the
change of variables $v=u+h(u,\tau)$ obtained in Lemma
\ref{lemma:canvi:canvi:lmenor}. Since $h$ restricted to the inner domains
satisfies the bounds
given in Corollary \ref{coro:Canvi:CotaCanvigInner}, this change of variables
does not change the asymptotic
first order of $\CCC$.

\subsubsection{An asymptotic formula for
$C(\mu,\eps)$}\label{Proof:Existence:C0:lmenor}
When $\eta=0$, the constant $C(\mu,\eps)$ considered in Theorem
\ref{th:MainGeometric:regular} satisfies that $\lim_{\eps\rightarrow
0}C(\mu,\eps)=C_0(\mu)$ for a certain function $C_0(\mu)$ analytic in $\mu$. We
devote this section to prove this fact.
This proof follows the same lines as the one of Proposition
\ref{coro:Canvi:FirstOrder:lmenor} in Section \ref{Proof:Existence:C:lmenor}
and, therefore, we only sketch it. Recall that throughout this section we assume
$\eta=0$.

We split the constant $C(\mu,\eps)$ as
$C(\mu,\eps)=C^1(\mu,\eps)+C^2(\mu,\eps)+C^3(\mu,\eps)$ and we  obtain the
corresponding first orders in $\eps$, which we call $C_0^i(\mu)$ for $i=1,2,3$.
Then, the function $C_0(\mu)$ will be given by
$C_0(\mu)=C_0^1(\mu)+C_0^2(\mu)+C_0^3(\mu)$.

Recall that $C(\mu,\eps)$ has been defined as \eqref{def:C:ViaIntegrals} where
$v_0$ is the left endpoint of $R_{\kk_3',d_3}\cap\RR$ , $v_1=i(a-\kk_3'\eps)$ is
the upper vertex of the
domain $R_{\kk_3',d_3}$ (see Figure \ref{fig:BoomInter}) and $M$ is the function
defined in \eqref{def:canvi:M:lmenor}. To obtain the constants $C^i$ we split
$M$ as $M=M^1+M^2+M^3$ with
\begin{equation}\label{def:Mi:lmenor}
M^i(v,\tau)=-\eps\ii G_i\left(v+g(v,\tau),\tau\right)\qquad\text{ for }\qquad
i=1,2,3,
\end{equation}
where $G_i$, $i=1,2,3$, are the functions defined in
\eqref{def:Canvi:G1:lmenor}, \eqref{def:Canvi:G2:lmenor} and
\eqref{def:Canvi:G3:lmenor} and $g$ is the function obtained in Lemma
\ref{lemma:canvi:canvi:lmenor}. Then,
\[
 C^i(\mu,\eps)=\int_{v_0}^{v_1}\left\langle M^i\right\rangle (v)\,dv.
\]

To define $C_0^1$, we expand $M^1$ with respect to $\eps$. Using the formulas
\eqref{def:canvi:canvi:lmenor} for $g$ and  \eqref{def:Canvi:G1:lmenor} for
$G_1$, one can easily see that for $(v,\tau)\in R_{\kk_3',d_3}\times\TT_\sigma$,
\[
 M^1(v,\tau)=-\eps\ii G_1(v,\tau)-\pa_v G_1(v,\tau) \ol
G_1(v,\tau)+\OO\left(\frac{\mu\eps}{(v-ia)^{\max\{0,2-\nu_1\}}}\right)
\]
for certain $\nu_1>0$. Recall that by Lemma \ref{lemma:Canvi:cotes:lmenor}, we
have that $\langle G_1\rangle=0$ and therefore this first term does not
contribute to $C_1(\mu,\eps)$. The second term, that is $-\pa_v G_1(v,\tau) \ol
G_1(v,\tau)$, is independent of $\eps$. Moreover, using the properties of $G_1$
stated in  Lemma \ref{lemma:Canvi:cotes:lmenor}, one can see that it can be
analytically extended to reach $v=ia$ and that it satisfies
\[
 -\pa_v G_1(v,\tau) \ol
G_1(v,\tau)=\OO\left(\frac{\mu}{(v-ia)^{\max\{0,1-\nu_1'\}}}\right)
\]
for certain $\nu_1'>0$. Therefore, one can define
\begin{equation}\label{def:C01:lmenor}
 C_0^1(\mu)=-\int_{v_0}^{ia}\left\langle  \pa_v G_1(v,\tau) \ol
G_1(v,\tau)\right\rangle dv,
\end{equation}
which is a constant independent of $\eps$. Finally it can be easily seen that
\begin{equation}\label{eq:Cota:Resta:C10:lmenor}
 \left|C^1(\mu,\eps)-C_0^1(\mu)\right|\leq K|\mu|\eps^{\nu_1''}
\end{equation}
for a suitable $\nu_1''>0$.

To obtain $C_0^2(\mu)$, let us first point out that, following the proof of
Theorem \ref{th:Periodica}, one can see that the parameterization of the
periodic orbit satisfies
\begin{equation}\label{def:Expansio:Periodica}
 \left(x_p(\tau),y_p(\tau)\right)=\left(\eps x^0_p(\tau),\eps
y^0_p(\tau)\right)+\OO\left(\mu\eps^2\right),
\end{equation}
where $\left( x^0_p(\tau), y^0_p(\tau)\right)$ is independent of $\eps$. Using
this fact, one can easily deduce that the functions $c_{kl}$ involved in the
definition of $\wh H_1^2$ in \eqref{def:HamPertorbat:H2} satisfy
\[
 c_{kl}(\tau)=c_{kl}^0(\tau)+\OO(\mu\eps),
\]
for adequate functions $c_{kl}^0(\tau)$ independent of $\eps$. Therefore, $\wh
H_1^2$  satisfies
\begin{equation}\label{def:Expansio:H2}
 \wh H_1^2(q,p,\tau)=\eps\wh H_1^{20}(q,p,\tau)+\eps^2\wh H_1^{22}(q,p,\tau),
\end{equation}
where $\wh H_1^{20}(q,p,\tau)$ is independent of $\eps$. Taking into account the
definition of $M_2$ in \eqref{def:Mi:lmenor} and recalling that for
$(v,\tau)\in R_{\kk_3',d_3}\times\TT_\sigma$,
\[
 p_0(v)\ii \wh H_1^{20}(q_0(v),p_0(v),\tau)=\OO\left((v-ia)^{2r-\ell}\right),
\]
we can define
\begin{equation}\label{def:C02:lmenor}
 C_0^2(\mu)=-\mu\int_{v_0}^{ia}\left\langle  p_0(v)\ii \wh
H_1^{20}(q_0(v),p_0(v),\tau)\right\rangle dv.
\end{equation}
Then, the constant $C_0^2(\mu)$ is independent of $\eps$. Moreover, using Lemmas
\ref{lemma:Canvi:Operador} and \ref{lemma:Canvi:cotes:lmenor} and
\ref{lemma:canvi:canvi:lmenor}, one can see that
\begin{equation}\label{eq:Cota:Resta:C20:lmenor}
 \left|C^2(\mu,\eps)-C_0^2(\mu)\right|\leq K|\mu|\eps^{\nu_2},
\end{equation}
for certain constant $\nu_2>0$.

To obtain $C_3^0(\mu)$ we need a careful study of the function $G_3$ in
\eqref{def:Canvi:G3:lmenor}. To this end, we have to expand asymptotically the
functions $\pa_v\wh T_1^{u,s}(v,\tau)$ obtained in Theorems
\ref{th:Extensio:Trig} and \ref{th:ExtensioFinal}. To obtain this expansion we
consider equation  \eqref{eq:HJperT1:ligual} for $(v,\tau)\in
R_{\kk_3',d_3}\times\TT_\sigma$.

As a first step we expand the function $A(u,\tau)$ defined in
\eqref{def:InftyA}. It can be seen that it satisfies
\[
 A(u,\tau)=A^0(u,\tau)+\eps
A^1(u,\tau)+\OO\left(\frac{\mu\eps^2}{(v-ia)^{\ell}}\right),
\]
where
\begin{align}
 A^0(u,\tau)&=-\mu \wh
H_1^1\left(q_0(u),p_0(u),\tau\right)\label{def:A:Expansio:A0:lmenor}\\
A^1(u,\tau)&=-\mu \wh
H_1^{20}\left(q_0(u),p_0(u),\tau\right)-V'(q_0(u))x_p^0(\tau)+\la^2x_p^0(\tau),
\label{def:A:Expansio:A1:lmenor}
\end{align}
where $\wh H_1^1$, $\wh H_1^{20}$ and $x_p^0$ are the functions defined in
\eqref{def:HamPertorbat:H1},   \eqref{def:Expansio:H2} and
\eqref{def:Expansio:Periodica} respectively, and $\lambda$ is the constant
defined in Hypothesis \textbf{HP1.1}. Recall that in the parabolic case, we have
that $x_p^0(\tau)=0$. It is clear that both $A^0$ and $A^1$ are independent of
$\eps$.

From this expansion, one can deduce the expansion of the function $\wh A$
defined in \eqref{def:InftyHyp:Ahat}. Let us first recall that the change of
variables $g$ obtained in Lemma \ref{lemma:Extensio:Trig:canvi} can be written
as
\[
 g(v,\tau)=-\eps\ol
B_1(v,\tau)+\OO\left(\frac{\mu\eps^2}{(v-ia)^{\max\{1+\ell-2r,0\}}}\right),
\]
where $\ol B_1$ is the function defined on the proof of Lemma
\ref{lemma:Extensio:Trig:canvi}, which is independent of $\eps$.

Therefore,
\[
\wh A(v,\tau)=\wh A^0(v,\tau)+\eps\wh
A^1(v,\tau)+\OO\left(\frac{\mu\eps^2}{(v-ia)^{\ell+2+2(\ell-2r)}}\right),
\]
with
\[
\begin{split}
 \wh A^0(v,\tau)&= A^0(v,\tau)\\
\wh A^1(v,\tau)&=A^1(v,\tau)-\pa_v A^0(v,\tau)\ol B_1(v,\tau).
\end{split}
\]
Using this fact and the properties of the functions $\wh B$ and $\wh C$ in
\eqref{def:InftyHyp:Bhat} and \eqref{def:InftyHyp:Chat}, one can see that the
functions  $\wh T_1^{u,s}(v,\tau)$ obtained in Theorems \ref{th:Extensio:Trig}
and \ref{th:ExtensioFinal} satisfy that
\[
 \pa_v\wh T_1^{u,s}(v,\tau)=\eps \pa_v\wh
T_1^{0}(v,\tau)+\OO\left(\frac{\mu\eps^2}{(v-ia)^{\max\{0,2+\ell-\nu_3\}}}
\right)
\]
for certain $\nu_3>0$. The first order  $\pa_v\wh T_1^{0}(v,\tau)$ is defined by
$\pa_v\wh T_1^{0}(v,\tau)=\pa_v\ol A^0(v,\tau)+\langle \wh A^1\rangle (v)$,
where $\ol A^0$ is a function satisfying that $\pa_\tau \ol A^0=A^0$ and
$\langle A^0\rangle=0$. Then,  $\pa_v\wh T_1^{0}(v,\tau)$ is independent of
$\eps$ and can be analytically extended to reach $v=ia$.

Taking into account the properties of the change $g$ stated in Lemma
\ref{lemma:Extensio:Trig:canvi}, one can see that the function $\pa_u
T_1(u,\tau)$ has the same expansion as the function $\pa_v\wh T_1(v,\tau)$.

We can define
\begin{equation}\label{def:C03:lmenor}
 C_0^3(\mu)=-\int_{v_0}^{ia}\left\langle p_0(v)^{-2}\pa_v
T_1^0(v,\tau)\right\rangle dv,
\end{equation}
which is a constant independent of $\eps$. Doing little effort, it can be seen
also that
\begin{equation}\label{eq:Cota:Resta:C30:lmenor}
 \left|C^3(\mu,\eps)-C_0^3(\mu)\right|\leq K|\mu|\eps^{\nu_3'}
\end{equation}
for certain $\nu_3'>0$.

Finally, it is enough to define  $C_0(\mu)=C_0^1(\mu)+C_0^2(\mu)+C_0^3(\mu)$
where $C_0^i(\mu)$ are the constants defined in \eqref{def:C01:lmenor},
\eqref{def:C02:lmenor} and \eqref{def:C03:lmenor}. It is straightforward to see
that $C_0(\mu)$ is an entire function. Moreover, by
\eqref{eq:Cota:Resta:C10:lmenor}, \eqref{eq:Cota:Resta:C20:lmenor} and
\eqref{eq:Cota:Resta:C30:lmenor}, it is clear that
\[
 \lim_{\eps\rightarrow 0}C(\mu,\eps)=C_0(\mu).
\]

\subsection{Case $\ell \geq 2r$: Proof of Theorem \ref{th:CanviFinal}
and Proposition \ref{coro:Canvi:FirstOrder}}\label{sec:PtFix:canvi}

\subsubsection{Proof of Theorem \ref{th:CanviFinal}}
Theorem \ref{th:CanviFinal} is a straightforward consequence of the following
proposition.
\begin{proposition}\label{prop:canvi}
Let $d_2>0$ and $\kk_6>0$  be defined in
Theorem \ref{th:ExtensioFinal} and Proposition
\ref{prop:matching:HJ}, $d_3<d_2$, $\eps_0>0$ small enough and
$\kk_8>\kk_6$ big enough, which might depend on the previous
constants. Then, for $\eps\in(0,\eps_0)$ and any $\kk\geq \kk_8$
such that $\eps\kk<a$, there exists a function $\CCC:
R_{\kk,d_3}\times\TT_\sigma\rightarrow \CC$ that satisfies equation
\eqref{eq:General:Diff}.

Moreover,
\[
\left(\xi_0(u,\tau),\tau\right)=\left(\eps\ii
u-\tau+\CCC(u,\tau),\tau \right)
\]
is injective and there exists a constant $b_{15}>0$ such that
\begin{itemize}
\item If $\ell-2r>0$,
\[
\begin{split}
\left\|\CCC \right\|_{\ell-2r,\sigma}&\leq b_{15}|\hmu|\eps^{\ell-2r}\\
\left\|\pa_u\CCC \right\|_{\ell-2r,\sigma}&\leq
b_{15}\kk\ii|\hmu|\eps^{\ell-2r-1}.
\end{split}
\]
\item If $\ell-2r=0$,
\[
\begin{split}
\left\|\CCC \right\|_{\ln,\sigma}&\leq b_{15}|\hmu|\\
\left\|\pa_u\CCC \right\|_{1,\sigma}&\leq b_{15}|\hmu|.
\end{split}
\]
\end{itemize}
\end{proposition}

We split the proof into the two cases : $\ell-2r>0$ and $\ell-2r=0$.

Nevertheless we need to state some useful properties of the function $G$
defined in \eqref{def:FuncioLAnulador:lmajor}.

\paragraph{Properties of the function $G$}\label{sec:Canvi:lemes}

We decompose the function $G$ in \eqref{def:FuncioLAnulador:lmajor}
as $G=G_1+G_2+G_3+G_4$ with
\begin{align}
G_1(u,\tau)&=\hmu\eps^{\ell-2r}p_0(u)\ii\pa_p \wh
H_1^1\left(q_0(u),p_0(u),\tau\right)\label{def:Canvi:G1}\\
G_2(u,\tau)&=\hmu\eps^{\ell-2r+1}p_0(u)\ii\pa_p \wh
H_1^2\left(q_0(u),p_0(u),\tau\right)\label{def:Canvi:G2}\\
G_3(u,\tau)&=\frac{1}{2}\left(1+\hmu\eps^{\ell-2r}\pa^2_p \wh
H_1^1\left(q_0(u),p_0(u),\tau\right)\right)\frac{\pa_u T_1^s(u,\tau)+\pa_u
T_1^u(u,\tau)}{p^2_0(u)}\label{def:Canvi:G3}\\
G_4(u,\tau)&=G(u,\tau)-G_1(u,\tau)-G_2(u,\tau)-G_3(u,\tau),\label{def:Canvi:G4}
\end{align}
where $\wh H_1^1$ and $\wh H_1^2$ are the functions defined in
\eqref{def:HamPertorbat:H1} and \eqref{def:HamPertorbat:H2}. The
next lemma gives some properties of these functions.

\begin{lemma}\label{lemma:canvi:CotesBesties}
Let $\kk>\kk_6$ and $d<d_2$, where $\kk_6$ an $d_0$
are the constants defined in Theorems \ref{prop:matching:HJ} and
\ref{th:ParamtoHJ}. Then, the functions $G_i$, $i=1,2,3,4$, defined
in \eqref{def:Canvi:G1}, \eqref{def:Canvi:G2}, \eqref{def:Canvi:G3}
and \eqref{def:Canvi:G4} respectively, have the following
properties.
\begin{enumerate}
\item $G_1\in \XX_{\ell-2r,\sigma}$ and satisfies $\langle
G_1\rangle=0$ and
\[
\left\|G_1\right\|_{\ell-2r,\sigma}\leq K|\hmu|\eps^{\ell-2r}.
\]
Moreover
\begin{itemize}
\item If $\ell-2r>0$,  $\pa_uG_1\in \XX_{\ell-2r+1,\sigma}$ and
satisfies
\[
\left\|\pa_u G_1\right\|_{\ell-2r+1,\sigma}\leq
K|\hmu|\eps^{\ell-2r}.
\]
\item If $\ell-2r=0$,  $\pa_uG_1\in \XX_{1-\frac{1}{\q},\sigma}$ and
satisfies
\[
\left\|\pa_u G_1\right\|_{1-\frac{1}{\q},\sigma}\leq K|\hmu|.
\]
\end{itemize}
\item $G_2\in \XX_{\ell-2r,\sigma}$ and satisfies
\[
\left\|G_2\right\|_{\ell-2r,\sigma}\leq
K|\hmu|^2\eps^{2(\ell-2r)+1}.
\]
Moreover
\begin{itemize}
\item If $\ell-2r>0$,  $\pa_uG_2\in \XX_{\ell-2r+1,\sigma}$ and
satisfies
\[
\left\|\pa_u G_2\right\|_{\ell-2r+1,\sigma}\leq
K|\hmu|^2\eps^{2(\ell-2r)+1}.
\]
\item If $\ell-2r=0$,  $\pa_uG_2\in \XX_{1-\frac{1}{\q},\sigma}$ and
satisfies
\[
\left\|\pa_u G_2\right\|_{1-\frac{1}{\q},\sigma}\leq K|\hmu|^2\eps.
\]
\end{itemize}
\item $G_3\in \XX_{\ell-2r+1,\sigma}$ and satisfies
\[
\left\|G_3\right\|_{\ell-2r+1,\sigma}\leq K|\hmu|\eps^{\ell-2r+1}.
\]
Moreover,
\begin{itemize}
\item If $\ell-2r>0$,  $\pa_uG_3\in \XX_{\ell-2r+1,\sigma}$ and
satisfies
\[
\left\|\pa_u G_3\right\|_{\ell-2r+1,\sigma}\leq
K\kk\ii|\hmu|\eps^{\ell-2r}.
\]
\item If $\ell-2r=0$,  $\pa_uG_3\in \XX_{2,\sigma}$ and
satisfies
\[
\left\|\pa_u G_3\right\|_{2,\sigma}\leq K|\hmu|\eps.
\]
\end{itemize}
\item $G_4, \pa_u G_4\in \XX_{3(\ell-2r)+2,\sigma}$ and satisfy
\[
\begin{split}
\left\|G_4\right\|_{3(\ell-2r)+2,\sigma}&\leq
K|\hmu|^3\eps^{3(\ell-2r)+2}\\
 \left\|\pa_u
G_4\right\|_{3(\ell-2r)+2,\sigma}&\leq
K\kk\ii|\hmu|^3\eps^{3(\ell-2r)+1}.
\end{split}
\]
\end{enumerate}
\end{lemma}

\begin{proof}
The proof of the statements about $G_1$ and $G_2$ are
straightforward, taking into account, for $G_2$, the bounds obtained
in Corollary \ref{coro:ShiftPeriodica}. For $G_3$, one has to take
into account the bounds for $T_1^{u}$ obtained in Proposition
\ref{prop:extensio:trig} and Corollary
\ref{coro:Transition:CotesGeneradora} and the analogous bounds that
$T_1^s$ satisfies.  To obtain the bound for its derivative, one can
apply the fourth statement of Lemma
\ref{lemma:PropietatsNormes:canvi}. Analogously, one can obtain the
bounds for $G_4$ and $\pa_u G_4$.
\end{proof}

\paragraph{Case $\ell-2r>0$}\label{sec:PtFix:canvi:lmajor}

To prove Proposition \ref{prop:canvi} for $\ell-2r>0$, we look for
$\CCC$ as a fixed point of the operator
\begin{equation}\label{def:Functional:Canvi}
\ol\FF=\GG_\eps\circ\FF,
\end{equation}
where $\GG_\eps$ and $\FF$ are the operators defined in
\eqref{def:operador:canvi} and \eqref{eq:Operador:CanviFinal}
respectively. For convenience, we rewrite $\FF$ as
\begin{equation}\label{def:Operador:canvi:v2}
\FF(\CCC)(u,\tau)=-\eps\ii G(u,\tau)-\pa_u
\left(G(u,\tau)\CCC(u,\tau)\right)+\pa_uG(u,\tau)\CCC(u,\tau).
\end{equation}
Then Proposition \ref{prop:canvi} is a  consequence of the following
lemma.
\begin{lemma}\label{lemma:canvi:general}
Let $\eps_0>0$ be small enough and $\kk_8>\kk_6$ big
enough. Then, for $\eps\in(0,\eps_0)$ and any $\kk\geq \kk_8$ such
that $\eps\kk<a$, the operator $\ol\FF$ defined in
\eqref{def:Functional:Canvi} is contractive from
$\XX_{\ell-2r,\sigma}$ to itself.

Then, it has a unique fixed point $\CCC\in \XX_{\ell-2r,\sigma}$,
which moreover satisfies
\[
\begin{split}
\|\CCC\|_{\ell-2r,\sigma}&\leq K|\hmu|\eps^{\ell-2r}\\
\|\pa_u\CCC\|_{\ell-2r,\sigma}&\leq K\kk\ii|\hmu|\eps^{\ell-2r-1}.
\end{split}
\]
\end{lemma}

Before proving Lemma \ref{lemma:canvi:general}, we state the following technical
lemma
about  the properties of the function $G$ defined in
\eqref{def:FuncioLAnulador:lmajor}.

\begin{lemma}\label{lemma:Canvi:cotes}
Let us assume $\ell-2r>0$. Then, the function $G$ defined in
\eqref{def:FuncioLAnulador:lmajor} has the following properties:
\begin{enumerate}
\item $G\in\XX_{\ell-2r,\sigma}$ and satisfies
\[
\|G\|_{\ell-2r,\sigma}\leq K|\hmu|\eps^{\ell-2r}.
\]
\item $\pa_u G\in\XX_{\ell-2r+1,\sigma}$ and satisfies
\[
\|\pa_u G\|_{\ell-2r+1,\sigma}\leq K|\hmu|\eps^{\ell-2r}.
\]
\item $\GG_\eps(G)\in\XX_{\ell-2r,\sigma}$ and satisfies
\[
\|\GG_\eps(G)\|_{\ell-2r,\sigma}\leq K|\hmu|\eps^{\ell-2r+1}.
\]
\end{enumerate}
\end{lemma}
\begin{proof}
The bounds for $G$ and $\pa_u G$ are a direct consequence of Lemma
\ref{lemma:canvi:CotesBesties}. To obtain the bound for
$\GG_\eps(G)$, it is enough to apply Lemma
\ref{lemma:Canvi:Operador} and to take into account that $\langle
G_1\rangle=0$.
\end{proof}

Using the bounds given in this lemma, we can prove Lemma
\ref{lemma:canvi:general}.
\begin{proof}[Proof of Lemma \ref{lemma:canvi:general}]
Let $\CCC_1,\CCC_2\in \XX_{\ell-2r,\sigma}$. By
definition of $\FF$ in \eqref{def:Operador:canvi:v2} and
Lemmas \ref{lemma:PropietatsNormes:canvi}, \ref{lemma:Canvi:Operador} and
\ref{lemma:Canvi:cotes}
\[
\begin{split}
\left\|\ol\FF(\CCC_2)-\ol\FF(\CCC_1)\right\|_{\ell-2r,\sigma}&\leq
\left\|\GG_\eps\left(\pa_u\left(G\cdot\left(\CCC_2-\CCC_1\right)\right)\right)
\right\|_{\ell-2r,\sigma}+\left\|\GG_\eps\left(\pa_u
G\cdot\left(\CCC_2-\CCC_1\right)\right)\right\|_{\ell-2r,\sigma}\\
&\leq K\left\|
G\right\|_{0,\sigma}\left\|\CCC_2-\CCC_1\right\|_{\ell-2r,\sigma}+K\left\|
\pa_uG\right\|_{1,\sigma}\left\|\CCC_2-\CCC_1\right\|_{\ell-2r,\sigma}\\
&\leq
\frac{K|\hmu|}{\kk_8^{\ell-2r}}\left\|\CCC_2-\CCC_1\right\|_{\ell-2r,\sigma}.
\end{split}
\]
Then, increasing $\kk_8$ if necessary, $\ol\FF$ is contractive from
$\XX_{\ell-2r,\sigma}$ to itself, and then it has a unique fixed
point $\CCC\in\XX_{\ell-2r,\sigma}$.

To obtain a bound for the fixed point $\CCC$ it is enough to recall
that
\[
\left\|\CCC\right\|_{\ell-2r,\sigma}\leq
2\left\|\ol\FF(0)\right\|_{\ell-2r,\sigma}.
\]
By the definition of $\ol\FF$ in \eqref{def:Functional:Canvi},
$\ol\FF(0)=-\eps\ii\GG_\eps(G)$. Then, applying Lemma
\ref{lemma:Canvi:cotes}, we obtain the bound for $\CCC$. Finally, to
obtain the bound for $\pa_u\CCC$ it is enough to reduce slightly the
domain and apply the fourth statement of Lemma
\ref{lemma:PropietatsNormes:canvi}.
\end{proof}
\begin{proof}[Proof of Proposition \ref{prop:canvi} for $\ell-2r>0$]
To prove Proposition \ref{prop:canvi} from Lemma
\ref{lemma:canvi:general}, it only remains to check that
$(\xi_0(u,\tau),\tau)$ is injective in
$R_{\kk,d_3}\times\TT_\sigma$. We prove this fact as in the proof of
Proposition \ref{prop:canvi:lmenor}, that is, we check that if
\[
\eps\ii u_2-\tau+\CCC(u_2,\tau)=\eps\ii u_1-\tau+\CCC(u_1,\tau)
\]
for $(u_1,\tau),(u_2,\tau)\in R_{\kk,d_3}\times\TT_\sigma$, then we
have that $u_2=u_1$. Indeed, by the bound of $\pa_u\CCC$ given in Lemma
\ref{lemma:canvi:general},
\[
\begin{split}
\left|u_2-u_1\right|&=\eps
\left|\CCC(u_2,\tau)-\CCC(u_1,\tau)\right|\\
&\leq \frac{K|\hmu|}{\kk_8^{\ell-2r+1}}|u_2-u_1|.
\end{split}
\]
Then, increasing $\kk_8$ if necessary, one can see that  $u_2=u_1$.
\end{proof}

\paragraph{Case $\ell-2r=0$}\label{sec:PtFix:canvi:ligual}

We will prove Proposition \ref{prop:canvi} under
the hypothesis $\ell-2r=0$. Now, as happened in Section
\ref{sec:PtFix:canvi:lmenor}, the linear term $G_1$ in
\eqref{def:Canvi:G1} of $\FF$ in \eqref{eq:Operador:CanviFinal} is
not small. Then, we  perform again a change of variables.

\begin{lemma}\label{lemma:canvi:canvi}
Let $\kk_8>\kk_6'>\kk_6$ and $d_3<d_2'<d_2$. Then, for
$\eps>0$ small enough, there exists a function $g$ which is solution
of the equation
\[
\LL_\eps g(v,\tau)=G_1(v,\tau),
\]
where $G_1$ is the function defined in \eqref{def:Canvi:G1}.
Moreover, it satisfies that
\[
\|g\|_{0,\kk_6',d_2',\sigma}\leq K|\hmu|\eps,\qquad
\|\pa_vg\|_{1-\frac{1}{\q},\kk_6',d_2',\sigma}\leq K|\hmu|\eps
\]
and that $u=v+g(v,\tau)\in R_{\kk_6,d_2}$ for $(v,\tau)\in
R_{\kk_6',d_2'}\times\TT_\sigma$.

Furthermore, the change $(u,\tau)=(v+g(v,\tau),\tau)$ is invertible
and its inverse is of the form $(v,\tau)=(u+h(u,\tau),\tau)$. The
function $h$ is defined in the domain
$R_{\kk_8,d_3}\times\TT_\sigma$, satisfies
\[
\|h\|_{0,\kk_8,d_3,\sigma}\leq K|\hmu|\eps
\]
and that $u+h(u,\tau)\in R_{\kk_6',d_2'}$ for $(u,\tau)\in
R_{\kk_8,d_3}\times\TT_\sigma$.
\end{lemma}
\begin{proof}
From Lemma \ref{lemma:canvi:CotesBesties}, $\langle G_1\rangle=0$
and then we can define a function $\ol G_1$ such that
\begin{equation}\label{def:canvi:G1:barra}
\pa_\tau \ol G_1=G_1\qquad\text{ and }\qquad \langle \ol G_1\rangle=0,
\end{equation}
which satisfies
\begin{equation}\label{eq:canvi:Cota:G1barra}
\begin{split}
\left\|\ol G_1\right\|_{0,\kk_6',d_2',\sigma}&\leq K|\hmu|\\
\left \|\pa_v\ol
G_1\right\|_{1-\frac{1}{\q},\kk_6',d_2',\sigma}&\leq K|\hmu|.
\end{split}
\end{equation}
Then, we can define $g$ as
\begin{equation}\label{def:canvi:canvi}
g(v,\tau)=\eps\ol G_1(v,\tau)-\eps\GG_\eps\left(\pa_v\ol
G_1\right)(v,\tau),
\end{equation}
where $\GG_\eps$ is the operator defined in
\eqref{def:operador:canvi} adapted to the domain
$R_{\kk_6',d_2'}\times\TT_\sigma$.

Finally, applying Lemma
\ref{lemma:canvi:CotesBesties} and \ref{lemma:Canvi:Operador}, one
obtains the bounds for $g$ and $\pa_v g$. The other statements are
straightforward.
\end{proof}

We perform the change of variables $u=v+g(v,\tau)$ given in Lemma
\ref{lemma:canvi:canvi} to equation \eqref{eq:Operador:CanviFinal}
and we obtain
\begin{equation}\label{eq:Canvi:PDE:Cbarret}
\LL_\eps\wh\CCC=M(v,\tau)+N(v,\tau)\pa_v\wh\CCC,
\end{equation}
where $\wh\CCC$ is the unknown
\begin{equation}\label{eq:CtoCbarret}
\wh\CCC(v,\tau)=\CCC(v+g(v,\tau),\tau)
\end{equation}
and
\begin{align}
M(v,\tau)&=-\eps\ii G\left(v+g(v,\tau),\tau\right)\label{def:canvi:M}\\
N(v,\tau)&=-\frac{G\left(v+g(v,\tau),\tau\right)-G_1(v,\tau)}{1+\pa_v
g(v,\tau)}.\label{def:canvi:N}
\end{align}
Moreover, we want to have the first order terms in $\wh\CCC$, coming
from $G_1$, $G_2$ and $G_3$, in a explicit form. For this purpose, we
define
\begin{equation}\label{def:C0barret}
\begin{split}
\wh\CCC_0(v,\tau)=&-\ol G_1(v,\tau)-\eps\ii \GG_\eps\left(\langle
\pa_vG_1g\rangle\right)(v) \\
&-\eps\ii \GG_\eps\left(\langle G_2+G_3\rangle\right)(v),
\end{split}
\end{equation}
where $\ol G_1$ is the function defined in
\eqref{def:canvi:G1:barra}, $g$ is the function given by Lemma
\ref{lemma:canvi:canvi} and $G_2$ and $G_3$ are the functions
defined in \eqref{def:Canvi:G2} and \eqref{def:Canvi:G3}
respectively. The next lemma, whose proof is straightforward
applying Lemmas \ref{lemma:Canvi:Operador},
\ref{lemma:canvi:CotesBesties} and \ref{lemma:canvi:canvi}, gives
some properties of $\wh\CCC_0$.

\begin{lemma}\label{lemma:canvi:C0}
The function $\wh\CCC_0$ defined in \eqref{def:C0barret} satisfies
that
\[
\left\|\wh\CCC_0\right\|_{\ln,\kk_6',d_2',\sigma}\leq K|\hmu|, \qquad
\left\|\pa_v\wh\CCC_0\right\|_{1,\kk_6',d_2',\sigma}\leq K|\hmu|.
\]
\end{lemma}

Then, we define
\[
\wh\CCC_1=\wh\CCC-\wh\CCC_0.
\]
Taking into account equation \eqref{eq:Canvi:PDE:Cbarret},
$\wh\CCC_1$ is a solution of
\begin{equation}\label{eq:Canvi:PDE:Cbarret1}
\LL_\eps\wh\CCC_1=\wh \FF\left(\wh\CCC_1\right),
\end{equation}
where
\begin{equation}\label{eq:Canvi:OperadorRHS:Cbarret1}
\wh\FF(h)=\wh M(v,\tau)+N(v,\tau)\pa_v h
\end{equation}
and
\begin{equation}\label{def:canvi:Mbarret}
\wh M(v,\tau)=M(v,\tau)-\LL_\eps\wh\CCC_0+N(v,\tau)\pa_v\wh\CCC_0.
\end{equation}
We obtain $\wh\CCC_1$ through a fixed point argument. For this
purpose we define the operator
\begin{equation}\label{def:OperadorPtFix:ligual}
\wt \FF=\GG_\eps\circ \wh\FF,
\end{equation}
where $\wh\FF$ and $\GG_\eps$ are the operators defined
\eqref{eq:Canvi:OperadorRHS:Cbarret1} and
\eqref{def:operador:canvi}. For convenience, we rewrite it as
\begin{equation}\label{eq:Canvi:OperadorRHS:Cbarret1:v2}
\wh\FF(h)(v,\tau)=\wh
M(v,\tau)+\pa_v\left(N(v,\tau)h(v,\tau)\right)-\pa_v
N(v,\tau)h(v,\tau).
\end{equation}
\begin{lemma}\label{lemma:canvi:PtFix:ligual}
Let us consider $\eps_0>0$ small enough and $\kk_6'>\kk_6$ big
enough. Then, the operator $\wt\FF$ is contractive from
$\XX_{1,\kk_6',d_2',\sigma}$ to itself.

Thus, it has a unique fixed point, which moreover satisfies that
\[
\begin{split}
\left\|\wh\CCC_1\right\|_{1,\kk_6',d_2',\sigma}&\leq K|\hmu|\eps\\
\left\|\pa_v\wh\CCC_1\right\|_{1,\kk_6',d_2',\sigma}&\leq
K\frac{|\hmu|}{\kk_6'}.
\end{split}
\]
\end{lemma}
Before proving this lemma, we state the following lemma, whose proof
is postponed to the end of this section.
\begin{lemma}\label{lemma:canvi:Cotes3}
The functions $\wh M$ and $N$ defined in \eqref{def:canvi:Mbarret}
and \eqref{def:canvi:N} respectively, satisfy the following
properties.
\begin{itemize}
\item $\GG_\eps (\wh M)\in \XX_{1,\kk_6',d_2',\sigma}$ and satisfies
\[
\left\| \GG_\eps (\wh M)\right\|_{1,\kk_6',d_2',\sigma}\leq
K|\hmu|\eps.
\]
\item $N, \pa_v N\in \XX_{1,\kk_6',d_2',\sigma}$ and satisfy
\[
\begin{split}
\left\| N\right\|_{1,\kk_6',d_2',\sigma}&\leq K|\hmu|\eps\\
\left\| \pa_v N\right\|_{1,\kk_6',d_2',\sigma}&\leq
K\frac{|\hmu|}{\kk_6'}.
\end{split}
\]
\end{itemize}
\end{lemma}

\begin{proof}[Proof of Lemma \ref{lemma:canvi:PtFix:ligual}]
The operator $\wt \FF$ sends $\XX_{1,\kk_6',d_2',\sigma}$ to itself.
Let $h_1,h_2\in
\XX_{1,\kk_6',d_2',\sigma}$. Then, recalling the definitions of
$\wt\FF$ and $\wh\FF$ in \eqref{def:OperadorPtFix:ligual} and
\eqref{eq:Canvi:OperadorRHS:Cbarret1:v2} and applying Lemmas
\ref{lemma:Canvi:Operador} and \ref{lemma:canvi:Cotes3}, one can see
that
\[
\begin{split}
\left\|\wt\FF(h_2)-\wt\FF(h_1)\right\|_{1,\kk_6',d_2',\sigma}&\leq
\left\|\GG_\eps\pa_v\left(N\cdot(h_2-h_1)\right)\right\|_{1,\kk_6',d_2',\sigma}
+\left\|\GG_\eps\left(\pa_vN\cdot(h_2-h_1)\right)\right\|_{1,\kk_6',d_2',\sigma}
\\
&\leq
K\left\|N\right\|_{0,\kk_6',d_2',\sigma}\left\|h_2-h_1\right\|_{1,\kk_6',d_2',
\sigma}+K\left\|\pa_vN\right\|_{1,\kk_6',d_2',\sigma}\left\|h_2-h_1\right\|_{1,
\kk_6',d_2',\sigma}\\
&\leq\frac{K|\hmu|}{\kk_6'}\left\|h_2-h_1\right\|_{1,\kk_6',d_2',\sigma}.
\end{split}
\]
and therefore, increasing $\kk_6'$ if necessary, $\wt\FF$ is
contractive in $\XX_{1,\kk_6',d_2',\sigma}$ and has a unique fixed
point $\wh\CCC_1$. To obtain bounds for $\wh\CCC_1$ it is enough to
recall that
\[
\left\|\wh\CCC_1\right\|_{1,\kk_6',d_2',\sigma}\leq
2\left\|\wt\FF(0)\right\|_{1,\kk_6',d_2',\sigma}.
\]
By the definition of $\wt\FF$ in \eqref{def:OperadorPtFix:ligual},
$\wt\FF(0)=\GG_\eps(\wh M)$. Then, it is enough to apply Lemma
\ref{lemma:canvi:Cotes3} to obtain
\[
\left\|\wh\CCC_1\right\|_{1,\kk_6',d_2',\sigma}\leq K|\hmu|\eps.
\]
For the bound of $\pa_v\wh\CCC_1$ it is enough to apply the fourth
statement of Lemma \ref{lemma:PropietatsNormes:canvi} and rename
$\kk_6'$.
\end{proof}

\begin{proof}[Proof of Proposition \ref{prop:canvi} for $\ell-2r=0$]
By Lemmas \ref{lemma:canvi:C0} and \ref{lemma:canvi:PtFix:ligual},
we have that there exists a constant $b_{15}>0$ such that
\[
\begin{split}
\left\|\wh\CCC \right\|_{\ln,\sigma}&\leq b_{15}|\hmu|\\
\left\|\pa_v\wh\CCC \right\|_{1,\sigma}&\leq b_{15}|\hmu|.
\end{split}
\]
To recover $\CCC$ it is enough to consider the change of variables
$v=u+h(u,\tau)$ obtained in Lemma \ref{lemma:canvi:canvi}, which is
defined for $(u,\tau)\in R_{\kk_8,d_3}\times \TT_\sigma$ with
$\kk_8>\kk_6'$ and $d_3<d_2'$. Applying this change, one obtains
$\CCC$ which satisfies the bounds stated in Proposition
\ref{prop:canvi}. To check that $(\xi_0(u,\tau),\tau)$ is injective,
one can proceed as in the proof of Proposition \ref{prop:canvi} for
$\ell-2r>0$. Finally let us point out that it is easy to see that
this proposition is also satisfied taking any $\kk\ge \kk_8$ such that
$\eps\kk<a$.
\end{proof}

It only remains to prove Lemma \ref{lemma:canvi:Cotes3}.

\begin{proof}[Proof of Lemma \ref{lemma:canvi:Cotes3}]
We start by proving the second statement. Let us split the function $N$
defined in \eqref{def:canvi:N} as $N=N_1+N_2$ with
\begin{align}
N_1(v,\tau)&=-\left(1+\pa_v g(v,\tau)\right)\ii
\left(G_1(v+g(v,\tau),\tau)-G_1(v,\tau)\right)\label{def:canvi:N1}\\
N_2(v,\tau)&=-\left(1+\pa_v g(v,\tau)\right)\ii
\left(G_2(v+g(v,\tau),\tau)+G_3(v+g(v,\tau),\tau)+G_4(v+g(v,\tau),
\tau)\right).\label{def:canvi:N2}
\end{align}
To bound $N_1$, we apply Lemmas \ref{lemma:canvi:canvi} and
\ref{lemma:canvi:CotesBesties} and the mean value theorem, obtaining
\[
\| N_1\|_{1-\frac{1}{\q},\kk_6',d_2',\sigma}\leq K|\hmu|^2\eps.
\]
Applying the same lemmas, one can see that
\[
\| N_2\|_{1,\kk_6',d_2',\sigma}\leq K|\hmu|\eps
\]
which gives the bound for $N$. To obtain the bound for $\pa_v N$ it
is enough to apply the fourth statement of Lemma
\ref{lemma:PropietatsNormes:canvi} and to rename $\kk_6'$.

For the first statement, taking into account the definitions of $\wh
M$ and $M$ in \eqref{def:canvi:M} and \eqref{def:canvi:Mbarret}
respectively, and using the functions $G_i$, $i=1,2,3,4$,  and $\ol
G_1$ defined in \eqref{def:Canvi:G1}, \eqref{def:Canvi:G2},
\eqref{def:Canvi:G3}, \eqref{def:Canvi:G4} and
\eqref{def:canvi:G1:barra}, let us decompose $\wh M$ as
\[
\wh M(v,\tau)=\sum_{i=1}^6 \wh M_i (v,\tau)
\]
with
\begin{align}
\wh M_1(v,\tau)&=\pa_v\ol G_1(v,\tau)-\eps\ii\left(\pa_v
G_1(v,\tau)g(v,\tau)-\langle
\pa_vG_1 g\rangle(v)\right)\label{def:canvi:M1}\\
\wh M_2(v,\tau)&=-\eps\ii \left( G_1(v+g(v,\tau),\tau)-G_1(v,\tau)-\pa_v
G_1(v,\tau)g(v,\tau)\right)
\label{def:canvi:M2}\\
\wh M_3(v,\tau)&=-\eps\ii\left( G_2(v,\tau)+G_3(v,\tau)-\langle G_2
+G_3\rangle(v)\right)
\label{def:canvi:M3}\\
\wh M_4(v,\tau)&=-\eps\ii\left( G_2(v+g(v,\tau),\tau)+G_3(v+g(v,\tau),\tau)-
G_2(v,\tau)-
G_3(v,\tau)\right)\label{def:canvi:M4}\\
\wh M_5(v,\tau)&=-\eps\ii G_4(v+g(v,\tau),\tau)\label{def:canvi:M5}\\
\wh M_6(v,\tau)&=N(v,\tau)\pa_v
\wh\CCC_0(v,\tau).\label{def:canvi:M6}
\end{align}
We bound each term. For the first one, by Lemmas
\ref{lemma:canvi:canvi} and \ref{lemma:canvi:CotesBesties}, we have
that $\wh
M_1\in\XX_{1-\frac{1}{\q},\kk_6',d_2',\sigma}\subset\XX_{1,\kk_6',d_2',\sigma}$.
Moreover, taking also into account \eqref{eq:canvi:Cota:G1barra},
\[
\left\|\wh M_1\right\|_{1,\kk_6',d_2',\sigma}\leq K|\hmu|
\]
and therefore, since $\langle\wh M_1\rangle=0$, by Lemma
\ref{lemma:Canvi:Operador},
\[
\left\|\GG_\eps\left(\wh
M_1\right)\right\|_{1,\kk_6',d_2',\sigma}\leq K|\hmu|\eps.
\]
For the term \eqref{def:canvi:M2}, it is enough to apply Lemmas
\ref{lemma:canvi:canvi} and Taylor's formula to obtain $\wh
M_2\in\XX_{2-\frac{1}{\q},\kk_6',d_2',\sigma}\subset\XX_{2,\kk_6',d_2',\sigma}$
and
\[
\left\|\wh M_2\right\|_{2,\kk_6',d_2',\sigma}\leq K|\hmu|^3\eps.
\]
Then, applying again Lemma \ref{lemma:Canvi:Operador}, we have that,
\[
\left\|\GG_\eps\left(\wh
M_2\right)\right\|_{1,\kk_6',d_2',\sigma}\leq K|\hmu|^3\eps.
\]
To bound \eqref{def:canvi:M3}, it is enough to apply Lemma
\ref{lemma:canvi:CotesBesties} to see that $M_3\in
\XX_{1,\kk_6',d_0',\sigma}$ and
\[
\left\|\wh M_3\right\|_{1,\kk_6',d_2',\sigma}\leq K|\hmu|
\]
which, using that $\langle \wh M_3\rangle=0$, implies
\[
\left\|\GG_\eps\left(\wh
M_3\right)\right\|_{1,\kk_6',d_2',\sigma}\leq K|\hmu|\eps.
\]
Applying the mean value theorem, using the definition of $G_3$ in
\eqref{def:Canvi:G3} and Proposition \ref{coro:TuPropSing}, and the
definition of $G_2$ in \eqref{def:Canvi:G2},  Lemmas
\ref{lemma:canvi:canvi} and \ref{lemma:canvi:CotesBesties}, one can
see that $\wh M_4$ in \eqref{def:canvi:M4} satisfies
\[
\left\|\wh M_4\right\|_{2,\kk_6',d_2',\sigma}\leq K|\hmu|^2\eps
\]
and then,
\[
\left\|\GG_\eps\left(\wh
M_4\right)\right\|_{1,\kk_6',d_2',\sigma}\leq K|\hmu|^2\eps.
\]
For $\wh M_5$ in \eqref{def:canvi:M5}, it is enough to notice that,
by Lemma \ref{lemma:canvi:CotesBesties} and
\ref{lemma:Canvi:Operador},
\[
\left\|\wh M_5\right\|_{2,\kk_6',d_2',\sigma}\leq K|\hmu|^3\eps
\]
and
\[
\left\|\GG_\eps\left(\wh
M_5\right)\right\|_{1,\kk_0',d_2',\sigma}\leq K|\hmu|^3\eps.
\]
Finally, for the last term \eqref{def:canvi:M6}, one has to apply
the bound of $N$ already obtained and Lemma \ref{lemma:canvi:C0}, to
see that
\[
\left\|\wh M_6\right\|_{2,\kk_6',d_2',\sigma}\leq
\left\|N\right\|_{1,\kk_6',d_0',\sigma}\left\|
\pa_v\wh\CCC_0\right\|_{1,\kk_6',d_2',\sigma}\leq K|\hmu|^2\eps.
\]
Then, by Lemma \ref{lemma:Canvi:Operador}, we have that,
\[
\left\|\GG_\eps\left(\wh
M_6\right)\right\|_{1,\kk_6',d_2',\sigma}\leq K|\hmu|^2\eps.
\]
Joining all these bounds, we prove the first statement of Lemma
\ref{lemma:canvi:Cotes3}
\end{proof}

\subsubsection{Proof of Proposition \ref{coro:Canvi:FirstOrder}}
To prove Proposition \ref{coro:Canvi:FirstOrder}, it is enough to
obtain the first asymptotic terms of the function $\wh \CCC_0$
obtained in Lemma \ref{lemma:canvi:C0}. From them, we  can deduce
the first order terms of $\wh\CCC=\wh \CCC_0+\wh\CCC_1$, where
$\wh\CCC_1$ is the function bounded in Lemma
\ref{lemma:canvi:PtFix:ligual}, and from them, using
\eqref{eq:CtoCbarret}, the ones of $\CCC$.

Recall that $\wh\CCC_0$ has been defined in \eqref{def:C0barret} as
$\wh\CCC_0=E_1+E_2+E_3+E_4$ with
\begin{align}
E_1(v,\tau)=&-\ol G_1(v,\tau)\label{def:canvi:E1}\\
E_2(v)=&-\eps\ii \GG_\eps\left(\langle
\pa_vG_1g\rangle\right)(v) \label{def:canvi:E2}\\
E_3(v)= &-\eps\ii \GG_\eps\left(\langle
G_2\rangle\right)(v)\label{def:canvi:E3}\\
E_4(v)=&-\eps\ii \GG_\eps\left(\langle
G_3\rangle\right)(v),\label{def:canvi:E4}
\end{align}
where $G_1$, $G_2$, $G_3$ and $\ol G_1$ are the functions defined in
\eqref{def:Canvi:G1}, \eqref{def:Canvi:G2}, \eqref{def:Canvi:G3} and
\eqref{def:canvi:G1:barra} respectively and $g$ is the function
given by Lemma \ref{lemma:canvi:canvi}.

We analyze each of the four terms $E_i$ that give $\wh\CCC_0$ for
$(v,\tau)\in \left( D_{\kk_6',\C_1}^{\inn, +,u}\cap
D_{\kk_6',\C_1}^{\inn, +,s}\right)\times \TT_\sigma$. For the first
one \eqref{def:canvi:E1}, it is enough to recall that, by
definition, the function $F_1$ defined in \eqref{def:FunctionsF}
satisfies that
\[
\hmu F_1(\tau)=\ol G_1(ia,\tau)
\]
and therefore,
\[
E_1(v,\tau)=-\ol G_1(v,\tau)=-\hmu
F_1(\tau)+\OO(v-ia)^{\frac{1}{\q}}.
\]
Then, using \eqref{eq:cota:canvi:gamma} and that $|v-ia|\leq
K\eps^{\ga}$,
\[
\left\|E_1+\mu F_1\right\|_{1,\sigma}\leq K|\hmu|\eps.
\]
For the second term, let us recall that by \eqref{def:canvi:canvi}
and applying Lemma \ref{lemma:Canvi:Operador}, we have that the
function $g$, obtained in Lemma \ref{lemma:canvi:canvi}, satisfies
\[
\left\| g-\eps\ol G_1(v,\tau)\right\|_{1-\frac{1}{\q},\sigma}\leq
K|\hmu|\eps^2.
\]
Then, by Lemma \ref{lemma:canvi:CotesBesties}, one can see that
\[
 \left\|\pa_v\left(g-\eps\ol
G_1(v,\tau)\right)\right\|_{2-\frac{2}{\beta},\sigma}\leq
K|\hmu|\eps^2
\]
and therefore, using Lemma \ref{lemma:Canvi:Operador},
\[
 \left\|\eps\ii\GG_\eps\left(\pa_v\left(g-\eps\ol
G_1(v,\tau)\right)\right)\right\|_{1,\sigma}\leq K|\hmu|\eps.
\]

Now it remains to bound, the first order of $E_3$, which is given by
\[
-\hmu\int_{v_0}^v \langle\pa_vG_1\ol G_1\rangle (w)\,dw,
\]
where we recall that $v_0\in R_{\kk_6',d_3}$.

Since $\langle\pa_vG_1\ol G_1\rangle=\OO(v-ia)^{1-\frac{1}{\q}}$, we
can define the constant
\[
C_2(\mu)=-\hmu\int_{v_0}^{ia} \langle\pa_vG_1\ol G_1\rangle (w)\,dw
\]
and then, using \eqref{eq:cota:canvi:gamma}, one has that
\[
\left\| E_2-C_2(\hmu)\right\|_{1,\sigma}\leq K|\hmu|^2\eps.
\]
For the third term, by the definitions of $G_2$ in
\eqref{def:Canvi:G2} and $\GG_\eps$ in \eqref{def:operador:canvi},
we have that
\[
\begin{split}
E_3(v)&=-\hmu
\int_{v_0}^v \langle \wh H_1^2\rangle(w)\,dw\\
&=-\hmu \int_{v_0}^{ia} \langle \wh
H_1^2\rangle(w)\,dw+\OO(v-ia)^\frac{1}{\q}.
\end{split}
\]
Then, proceeding as for $E_2$, we define
\[
C_3(\mu,\eps)=-\hmu \int_{v_0}^{ia} \langle \wh H_1^2\rangle(w)\,dw
\]
and using \eqref{eq:cota:canvi:gamma}, we have that
\[
\left\| E_3-C_3(\mu,\eps)\right\|_{1,\sigma}\leq K|\hmu|\eps.
\]
To bound $E_4$, using Proposition \ref{coro:TuPropSing}, we decompose
$G_3$ into two terms as $G_3=G_3^1+G_3^2$, with
\[
G_3^1(v,\tau)=\left(1+\hmu\pa_p^2\wh
H_1^1\left(q_0(u),p_0(u),\tau\right)\right)p_0(u)^{-2}\left(\frac{2r\hmu\eps
C_+^2}{(v-ia)^{2r+1}}\left(F_0(\tau)+\hmu \langle Q_0
F_1\rangle\right)+\xi(u,\tau)\right)
\]
and $G_3^2=G_3-G_3^1$. By Proposition \ref{coro:TuPropSing},
$\|G_3^2\|_{2,\sigma}\leq K\hmu\eps^2$ and therefore
\[
\left\|\eps\ii \GG_\eps\left(\langle
G_3^2\rangle\right)\right\|_{2,\sigma}\leq K|\hmu|\eps.
\]
For the other term, using the definitions of $\wh H_1^1$,  $b$,
$Q_j$ and $F_j$ in \eqref{def:HamPertorbat:H1},
\eqref{def:Constantb}, \eqref{def:FunctionsQ} and
\eqref{def:FunctionsF}, and recalling that by Proposition
\ref{coro:TuPropSing}, $\xi\in \XX_{1-\frac{1}{\q},\sigma}$, there
exist a function $\wh\xi\in\XX_{1-\frac{1}{\q},\sigma}$, such that
\[
\left\langle G_3^1\right\rangle(v)=\frac{b\hmu^2\eps}{v-ia}+\wh \xi(v,\tau).
\]
Then, one can see that there exists a constant  $C_4(\hmu,\eps)$
satisfying $|C_4(\hmu,\eps)|\leq K|\hmu|$, such that,
\[
\left\| E_4(v)+b\hmu^2\ln
(v-ia)-C_4(\hmu,\eps)\right\|_{1,\sigma}\leq K|\hmu|\eps.
\]
Taking $C=C_2+C_3+C_4$ one obtains that
\[
\left\|\wh \CCC(v,\tau)+\hmu F_1(\tau)-C(\hmu,\eps)+b\hmu^2\ln
(v-ia)\right\|_{1,\sigma}\leq K|\hmu|\eps.
\]
To finish the proof of Proposition \ref{coro:Canvi:FirstOrder}, it
is enough to consider the change of variables $v=u+h(u,\tau)$
obtained in Lemma \ref{lemma:canvi:canvi}, which does not change the
asymptotic first order of $\CCC$. Let us note that to see that $C(\mu,\eps)$ has
a well defined limit when $\eps\rightarrow 0$ one can easily proceed as we have
done in the case $\ell-2r<0$ in Section \ref{Proof:Existence:C0:lmenor}.

\section*{Acknowledgements}
I.B. acknowledges the support of the Spanish Grant MEC-FEDER
MTM2006-05849/Consolider, the Spanish Grant MTM2010-16425
and the Catalan SGR grant 2009SGR859 and E.F. the support of the Spanish Grant MEC-FEDER
MTM2006-05849/Consolider, the Spanish Grant MTM2010-16425
and the Catalan grant CIRIT 2005 SGR01028. M. G and T.M.S. have been partially supported by the Spanish MCyT/FEDER grant
MTM2009-06973 and the Catalan SGR grant 2009SGR859. 
In addition, the research of M. G. has been supported by the
Spanish PhD grant FPU AP2005-1314. Part of this work was done while M. G. was doing stays in the Departments of Mathematics of the University of Maryland at College Park and the Pennsylvania State University.  He wants to thank these institutions for their hospitality and support, and specially thank Vadim Kaloshin for making these stays possible.

\bibliography{references}
\bibliographystyle{alpha}
\end{document}